\documentclass[a4paper,10pt]{article} 
\usepackage{amsmath,amssymb,amsthm,graphicx}
\usepackage[latin1]{inputenc}
\usepackage[T1]{fontenc}
\usepackage[english]{babel}
\usepackage{dsfont}
\usepackage{xspace}
\usepackage[hypcap=false]{caption}

\usepackage{bbm}
\numberwithin{equation}{section}

\theoremstyle{plain}
\newtheorem{theo}{Theorem}[section] \newtheorem{defi}[theo]{Definition}
\newtheorem{lemm}[theo]{Lemma} \newtheorem{prop}[theo]{Proposition}
\newtheorem{coro}[theo]{Corollary} 

\theoremstyle{definition}
\newtheorem{rema}[theo]{Remark}
\newtheorem*{rema*}{Remark}
\newtheorem*{ex*}{Example}
\newtheorem{ex}[theo]{Example}

\usepackage[citecolor=blue,colorlinks=true]{hyperref}

\makeatletter
\newcommand{\pushright}[1]{\ifmeasuring@#1\else\omit\hfill$\displaystyle{#1}$\fi\ignorespaces}
\newcommand{\pushleft}[1]{\ifmeasuring@#1\else\omit$\displaystyle{#1}$\hfill\fi\ignorespaces}
\makeatother

\newcommand{\N}{\mathbb N}                   

\newcommand{\Q}{\mathbb Q}                   

\newcommand{\R}{\mathbb R}                   

\newcommand{\C}{\mathbb C}                   

\newcommand{\eps}{\varepsilon}

\newcommand{\schtroumpf}{subsectorial\xspace}
\newcommand{\Schtroumpf}{Subsectorial\xspace}

\let \Re \relax
\let \Im \relax
\newcommand{\Im}{\mathrm{Im}}

\DeclareMathOperator{\Sp}{Sp}

\DeclareMathOperator{\tr}{tr}
\DeclareMathOperator{\ran}{ran}
\DeclareMathOperator{\Id}{Id}
\DeclareMathOperator{\rank}{rank}
\DeclareMathOperator{\mult}{mult}
\DeclareMathOperator{\vect}{span}
\DeclareMathOperator{\loc}{loc}

\DeclareMathOperator{\Leb}{Leb}

\newcommand{\Eik}{\mathrm{Eik}}
\newcommand{\Op}{\mathrm{Op}}
\newcommand{\Re}{\mathrm{Re}}
\newcommand{\Ker}{\mathrm{Ker}\,}
\let \S  \relax
\newcommand{\S}{\mathbb{S}}

\usepackage{soul}

\newcommand{\opnorm}[1]{{\left\vert\kern-0.25ex\left\vert\kern-0.25ex\left\vert #1 \right\vert\kern-0.25ex\right\vert\kern-0.25ex\right\vert}}

\newcommand{\ope}{\mathcal P}

\newcommand{\opec}{\mathcal P}

\newcommand{\qpec}{\mathcal Q}

\newcommand{\opeg}{A}

\newcommand{\bspe}{\lambda}

\newcommand{\vp}{\mu}

\newcommand{\para}{N}

\newcommand{\cq}{\mathsf{a}}

\newcommand{\opq}{\mathcal K_{\cq}}

\usepackage{todonotes}

\begin{document}

\author{ Paul Alphonse, Jean-Marc Bouclet and Matthieu L\'eautaud}
\title{Short-time parametrix for the Fokker--Planck semigroup and applications}
\date{\today}

\maketitle

\begin{abstract} 
We  construct a short-time parametrix for the Fokker--Planck semigroup in Euclidean space. 
Among possible applications, we obtain smoothing and localization properties of the semigroup, the derivation of an approximate short-time polar decomposition for the semigroup, the construction of a parametrix of the resolvent, pseudospectral estimates, and estimates of the asymptotics of the number of eigenvalues of the Fokker--Planck operator.
As a step in the proofs, we introduce a class of operators, which we call \schtroumpf operators, to which the Fokker--Planck operator belongs, and describe some of their functional analytic and spectral properties.
\end{abstract}

\tableofcontents

\section{Introduction and main results}

\subsection{Context}

The Fokker--Planck\footnote{also called Kramers--Fokker--Planck or kinetic-Fokker--Planck in the literature.} equation is a central equation in kinetic theory, modeling non-equilibrium statistical mechanics.
It takes the following form:
\begin{equation}\label{VlasovFokkerPlanck}
	m \partial_t f + m v \cdot \partial_x f -( \nabla_x V) (x) \cdot \partial_v f - \frac{\gamma}{\beta} \mbox{div}_v \big( ( \partial_v + m \beta v ) f \big) = 0,
\end{equation}
on $(\mathbb{R}_+)_t \times {\mathbb R}^n_x \times {\mathbb R}^n_v $, with $ n \in \N^* $, see e.g.~\cite[eq. (1.18)]{Risken89} or~\cite[eq. (249)]{Chandresekhar}.
Equation~\eqref{VlasovFokkerPlanck} describes the time evolution of the distribution function $f=f(t,x,v)$ in position $x$ and velocity $v$, modeling the motion of particles in an external field.
The constant $m>0$ is the mass of a particle, $\beta>0$ is the temperature, $ -( \nabla_x V) $ is the external force produced by the external potential $V:\R^n\to \R$, and $\gamma>0$ is the friction constant. 
The set of equilibria for this equation is spanned by the Maxwellian $ {\mathcal M} $ defined on $\mathbb R^{2n}$ by 
\[
	{\mathcal M}(x,v) := \exp \left[ - \beta \Big( m \frac{|v|^2}{2} + V (x) \Big)  \right].
\]
Near this equilibrium, the natural change of unknown function $ f = {\mathcal M}^{\frac{1}{2}}u$ 
turns (\ref{VlasovFokkerPlanck}) into the equation
\begin{equation}\label{pourCauchyproblem}
	m \partial_tu + {\mathcal T}u + {\mathcal H}u = 0, 
\end{equation}
where
\begin{equation}\label{definittransportharmonique}
	{\mathcal T} := m v \cdot \partial_x - \nabla V (x) \cdot \partial_v, \qquad {\mathcal H} =  \frac{\gamma}{\beta} \left(-  \Delta_v + \frac{m^2 \beta^2}{4} |v|^2 - \frac{m \beta n}{2} \right).
\end{equation}
Here, ${\mathcal T}$ is the Hamilton vector field generated by the Hamiltonian $m\frac{|v|^2}{2}+  V (x)$ on $\R^{2n}$ (with momentum given by $mv$),
coming from classical dynamics, and ${\mathcal H}$ is a harmonic oscillator, coming from the stochastic derivation of the equation, see e.g.~\cite[eq. (249)]{Chandresekhar}.

The generator of the evolution in~\eqref{pourCauchyproblem} is what we call, from now on, the Fokker--Planck operator
\begin{equation}\label{e:def-FP-operator}
	{\mathcal P} := {\mathcal T} + {\mathcal H} . 
\end{equation}
From the mathematical perspective, the Fokker--Planck operator $\mathcal{P}$ is a very interesting differential operator: it is neither elliptic, nor selfadjoint, nor a small perturbation of those two.
If $ P $ is an elliptic selfadjoint (pseudo)differential operator, well-known procedures allow to compute pseudodifferential approximations of the resolvent $ (P-z)^{-1} $ (if $z$ is away from the spectrum and the numerical range of the principal symbol) and of the heat semigroup $ e^{-t P} $ as $t \rightarrow 0^+$ (see the discussion after Remark \ref{remarquefonda} below).
The selfadjointness is not necessary to get a pseudodifferential parametrix of the resolvent: elliptic operators with a principal symbol taking values on strict sector of the complex plane can also be handled when the spectral parameter $z$ is away from this sector. The ellipticity is neither necessary  and can be sometimes relaxed for operators with symbols in the H\"ormander classes $ S^m_{\rho,\delta} $, allowing to get pseudodifferential parametrices under certain conditions (see e.g.~\cite[Section 22.1]{Hormander3}). 
Nevertheless, a lot of relevant operators fail to satisfy those conditions, and this is the case for the Fokker--Planck operator $\mathcal{P}$, which even fails to be sectorial as an operator on $L^2(\R^{2n})$ since its numerical range, namely,
\[
	N(\mathcal{P}) := \left\{ \frac{\langle\mathcal{P}f , f \rangle_{L^2(\R^{2n})}}{\|f\|_{L^2(\R^{2n})}^2}\ \bigg\vert\ f \in D(\ope) \right\},
\]
where $D(\ope)$ will be properly defined below, is a sub-region of the half-plane $\big\{z \in \C\ \vert\ \Re(z)\geq 0 \big\}$ with which it may coincide.

The purpose of the present work is to construct a fairly simple and explicit short-time parametrix of the semigroup $ (e^{-\frac{t}{m} {\mathcal P}})_{t \geq0} $ associated to the Fokker--Planck operator $\mathcal P$ on $ {\mathbb R}^{2n} $, by mean of an elementary Fourier integral operator with complex phase (which is quadratic in the momentum variables). Applications of the parametrix presented in this article include:
\begin{itemize}
    \item the derivation of an approximate short-time polar decomposition for the semigroup $(e^{-\frac{t}{m} \mathcal P})_{t\geq0} $, that is to say, $ e^{-\frac{t}{m} \mathcal P} \approx U_t S_t$, where $U_t$ is approximately unitary and $S_t$ is approximately selfadjoint,
	\item the characterization of the smoothing and localization properties of the semigroup $(e^{-\frac{t}{m}\mathcal P})_{t\geq0}$,
	\item the construction of a parametrix of the resolvent $ ({\mathcal P}-z)^{-1} $ when $z$ belongs to some region of the half plane $ \{ \Re(z) \geq 0 \} $ (recall that it might coincide with the numerical range of $ {\mathcal P} $) and thus the existence of non trivial spectrum free region of this half plane,
	\item the derivation of estimates on the asymptotics of the number of eigenvalues of the operator $\mathcal{P}$. 
\end{itemize}
Moreover, most of these results can be proven in the spaces $ L^p(\mathbb R^{2n}) $ and not just in $ L^2(\mathbb R^{2n}) $.
Along the proofs of spectral properties, we identify a class of operators, which we call \schtroumpf operators, to which the Fokker--Planck operator $\mathcal{P}$ belongs, and describe some of their functional analytic properties.

We point out that the present text is exclusively devoted to the short-time regime. A considerable amount of works have considered the long-time asymptotics, {\it i.e.} the question of convergence to equilibrium, for kinetic equations; regarding the Fokker--Planck equation, the reader may for instance consult \cite{desvilla,HerauNier,mischmou}  (confining potentials),  \cite{wang} (decaying potentials) and \cite{carramisch} (compact domains) and the references therein.

\subsection{The parametrix}

Before describing our results, let us  first describe the appropriate functional setting for the Fokker-Planck operator $\mathcal{P}$ and properly define the semigroup $ (e^{-\frac{t}{m} {\mathcal P}} )_{t\geq0}$, which is the central object of the article.

\begin{theo} \label{theoreme-semigroup-intro}
Assume that $V \in C^2(\R^n;\R)$ satisfies $\partial^\alpha V \in L^\infty(\mathbb R^n)$ for all $\alpha $ such that $|\alpha| = 2$. Let $ p \in (1,\infty) $. 
Then, the operator $ {\mathcal P} $ in~\eqref{e:def-FP-operator}--\eqref{definittransportharmonique} defined as an unbounded operator on $L^p(\mathbb R^{2n})$ with domain $D(\mathcal{P}) = C^\infty_c(\R^{2n})$ is closable, and setting
\begin{align*}
	D ({\mathcal P}_{\rm min}) & := \big\{ f \in L^p(\mathbb R^{2n})  \ | \ \exists (f_n)_n \in C^\infty_c(\mathbb R^{2n})^{\mathbb N}, \exists g\in L^p(\mathbb R^{2n}),\,f_n \to f,\, \mathcal{P}f_n\rightarrow g\big\}, \\
	D ({\mathcal P}_{\rm max}) & := \big\{ f \in L^p ({\mathbb R}^{2n}) \ | \  {\mathcal P} f\in L^p(\mathbb R^{2n}) \big\} ,
\end{align*}
we have 
\[
	D( \overline{\mathcal P} ) = D ({\mathcal P}_{\rm min}) = D ({\mathcal P}_{\rm max}).
\]
Finally, the operator $ \overline{\mathcal P} $ defined on this domain by $ \overline{\mathcal P} f = {\mathcal P} f $ in the sense of distributions generates a strongly continuous semigroup $ (e^{-\frac{t}{m} \overline{\mathcal P}})_{t\geq0} $ on $ L^p(\mathbb R^{2n}) $.
\end{theo}

In the following, in order to alleviate the writing, we will keep denoting by $(\ope,D(\ope))$ the closure $(\overline{\ope},D(\overline{\ope}))$ of the realization of the operator~\eqref{e:def-FP-operator} on $L^p(\mathbb R^{2n})$. 
Notice that in this notation, the dependence of the operator $\ope$ with respect to the Lebesgue exponent $p\in(1,\infty)$ is implicit. When useful, we will occasionally restore the explicit dependence in $p$.
 
The purpose of this paper is to construct (and use) an approximation  of the semigroup $ (e^{-\frac{t}{m} \ope})_{t\geq0} $ as $ t \rightarrow 0^+$
under the additional assumption that the potential $V$ is  smooth and at most quadratic in the following sense:
\begin{equation}\label{hypothesepotentiel}
	V \in C^\infty(\R^n; \R) \quad \text{ and }\quad 
	\partial_x^{\alpha} V \in L^{\infty} ({\mathbb R}^n) \quad \text{ for all }\alpha \in \N^n \text{ such that }|\alpha| \geq 2. 
\end{equation} 
More precisely, we approximate $ e^{-\frac{t}{m} {\mathcal P}}u$, with $u\in L^p ({\mathbb R}^{2n}) $, by mean of  a Fourier integral operator ${\mathcal E}_A(t)$ with complex phase defined by
\begin{equation}\label{parametrixexplicitee}
	({\mathcal E}_{A} (t)u)(x,v)  = \iint_{\mathbb R^{2n}} K_A(t,x,v,y,w) u(y,w)\, dydw,
\end{equation} 
and associated with the following kernel
\begin{equation}\label{formedunoyaustandard}
	K_A(t,x,v,y,w) =  
	(2 \pi)^{- 2n } \iint_{\mathbb R^{2n}} e^{i \psi (t,x,v,\xi,\eta) - i y \cdot \xi - i w \cdot \eta} A (t,x,v,\xi,\eta)\, d \xi d \eta.
\end{equation}

This means that we look for a phase function $ \psi $ such that $ \Im(\psi) \geq 0 $ and an amplitude $ A $ such that, for any $ u\in L^p(\mathbb R^{2n}) $, the operator $ m \partial_t + {\mathcal P} $ applied to (\ref{parametrixexplicitee}) is small as $t\to 0^+$ in a sense to be specified (basically $\mathcal O (t^{\infty}) $ in  $ L^p(\mathbb R^{2n}) $). It turns out that the following explicit phase $\psi$ will allow to fulfill our program: we let
\begin{equation}\label{defImpsi}
	\Im(\psi)(t,x,v,\xi,\eta) = \frac{\gamma}{\beta m} t \left( \left| \eta - \frac{\xi_t}{2} \right|^2 + \frac{1}{3} \left| \frac{\xi_t}{2} \right|^2 \right) + \frac{\gamma m \beta}{4} t \left( \left| v + \frac{w_t}{2} \right|^2 + \frac{1}{3} \left| \frac{w_t}{2} \right|^2 \right),
\end{equation} 
and
\begin{equation}\label{defRepsi}
	\Re(\psi)(t,x,v,\xi,\eta) = x_t \cdot \xi + ( v + w_t ) \cdot \eta, 
\end{equation}
where
\begin{equation}\label{wtdef}
	x_t = x - t v  - \frac{t^2}{2m} (\nabla V)(x) , \qquad w_t = \frac{t}{m} \nabla V (x), \qquad \xi_t = t \xi.
\end{equation}
Obviously, $ \Im(\psi) $ is nonnegative for $ t \geq 0 $ and the exponential decay of $ e^{- \Im(\psi)} $ suggests that within the integral (\ref{parametrixexplicitee}), we should have a fast decay with respect to the quantities
\begin{equation}\label{quantitespertinentes}
	t^{\frac{1}{2}} \eta, \qquad  t^{\frac{3}{2}} \xi, \qquad t^{\frac{1}{2}} v , \qquad t^{\frac{3}{2}} \nabla V (x).
\end{equation}
This last observation is a motivation to introduce the following classes of functions in which we shall pick the amplitude $A$.
 
\begin{defi} \label{vraievaluation} Given $ \delta \geq 0 $ and $ \nu \in {\mathbb R} $, we let $ {\mathcal V}^{\nu,\delta} $ be the space of functions  $ a = a (t,x,v) \in C^0((0,1)\times \R^{2n})$ such that, for all $t \in (0,1)$, $a(t, \cdot) \in C^\infty(\R^{2n})$ and for all $\alpha,\beta \in \N^n$, there is $C_{\alpha\beta}>0$ such that 
$$ \big| \partial_x^{\alpha} \partial_v^{\beta} a (t,x,v) \big| \leq C_{\alpha \beta} t^{\nu + \frac{|\beta|}{2}}  \Big(1 + \big|t^{\frac{1}{2}} v \big| + \big|t^{\frac{3}{2}} (\nabla V) (x)\big|\Big)^{\delta} , $$
for all
\[ 
	(t,x,v) \in (0,1) \times {\mathbb R}^{2n}.
\]
When $ \nu > 0 $, $a$ can be extended by continuity up to $ t= 0 $ by setting $ a (0,x,v) = 0 $.
\end{defi}

\begin{defi} \label{definitionvaluation} Given $ \nu \in {\mathbb R} $ and $ d \geq 0 $, we let $ {\mathcal A}^{\nu,d} $ be  the space of functions $A$ defined on $ (0,1)_t \times {\mathbb R}^{2n}_{x,v} \times {\mathbb R}^{2n}_{\xi,\eta} $ which are finite sums of the form
\begin{equation}
	A (t,x,v,\xi,\eta) = \sum a_{\alpha \beta} (t,x,v) \big( t^{\frac{3}{2}} \xi \big)^{\alpha}  \big(t^{\frac{1}{2}} \eta \big)^{\beta} , \label{expressionoftheform}
\end{equation}
with
\[
	a_{\alpha,\beta} \in {\mathcal V}^{\nu,\delta}, \qquad |\alpha| + |\beta| + \delta \leq d .
\]
Any function $ A \in {\mathcal A}^{\nu,d} $ for some $d$ will be said of valuation $ \nu $. When $ \nu > 0 $, $ A $ can be considered as continuous on $ [0,1)\times {\mathbb R}^{4n} $ with $ A|_{t=0} = 0 $. 
\end{defi}

Notice that the classes $\mathcal A^{\nu,d}$ satisfy ${\mathcal A}^{\nu,d} = t^\nu {\mathcal A}^{0,d}$ and that  $\nu \geq \nu'$ implies ${\mathcal A}^{\nu,d} \subset {\mathcal A}^{\nu',d} $ (and even  $\nu \geq \nu'$ and $d'\geq d$ imply ${\mathcal A}^{\nu,d} \subset {\mathcal A}^{\nu',d'}$).

The following theorem, which is the main result of the article, states that the evolution operators $ e^{-\frac{t}{m} \ope} $ are well approximated by such Fourier integral operators.

\begin{theo}[Parametrix for the Fokker--Planck semigroup] \label{maintechnicalresulttheorem} 
Assume that $V$ satisfies~\eqref{hypothesepotentiel}. Let $p\in(1,\infty)$. With $\psi$ defined in~\eqref{defImpsi}--\eqref{defRepsi}--\eqref{wtdef}, there exist $t_0\in (0,1)$ and a sequence of amplitudes $(A_k)_{k\in\mathbb N}$, with $ A_0 \equiv 1 $ and $A_k \in {\mathcal A}^{k,2k} $ when $k\geq1$, such that for every $t\in[0,t_0]$ and $N\in\mathbb N^*$,
\begin{equation}\label{eq:approxevolop}
	e^{- \frac{t}{m}\ope} = {\mathcal E}_{1+A_1 + \cdots + A_N} (t) + {\mathcal R}_N (t),
\end{equation}
where 
\begin{equation}\label{eq:remainder}
	{\mathcal R}_N (t) := - \frac{1}{m} \int_0^{t} e^{- \frac{t-s}{m} \ope}  {\mathcal E}_{R_N} (s)\,ds\quad\text{with}\quad R_N\in{\mathcal A}^{N,2N+2}.
\end{equation}
\end{theo}

\begin{rema} 
An intuition on the expression of the phase function $\psi$ and on the classes of which the amplitudes $A$ belong is given in Section \ref{subsec:guess}. We only mention here that this function can be guessed by studying the Taylor expansion in $t$ (at $t=0$) of an ideal phase satisfying (\ref{equationsurpsi}), using the condition that $ \psi |_{t=0} = x \cdot \xi + v \cdot \eta $.
\end{rema}

\begin{rema} \label{remarquefonda} For the reader wishing to get directly a feeling of what Theorem \ref{maintechnicalresulttheorem} tells on the  fundamental solution of $ e^{-\frac{t}{m}{\mathcal P}} $, we display here the explicit form of the Schwartz kernel $ K_1 $ of the leading order term $ {\mathcal E}_1 (t) $,
\[
	K_1(t,x,v,y,w)
	= \frac{ 3^{\frac{n}{2}}}{(2 \pi)^n} \bigg( \frac{\beta m}{\gamma t^2} \bigg)^n e^{- \frac{3 \beta m}{\gamma t^3} | x - y - t \frac{v+w}{2} |^2} e^{- \frac{\beta m}{4 \gamma t} | v - w + \frac{t}{m} \nabla V |^2 }  
	e^{- \frac{\gamma m \beta}{4} t  ( | v + \frac{t}{2m} \nabla V |^2 + \frac{1}{3} | \frac{t}{2m} \nabla V |^2 )} . 
\]
See Lemma \ref{lemma-e:kernel-first-term} for higher order terms. We  point out that, although Theorem \ref{maintechnicalresulttheorem} allows to derive an expansion in $t$ for the kernel of $ e^{- \frac{t}{m}{\mathcal P}} $ in physical space (\textit{i.e.} in terms of $ (x,v,y,w) $ and without using the dual variables $ (\xi,\eta) $), the FIO point of view turns out to be very useful to estimate the various operator norms we consider in our applications.
\end{rema}

\begin{rema}
We are not aware of many results proving asymptotic properties of the fundamental solution of the Fokker--Planck semigroup (when $t$ is small). We recall that the elliptic selfadjoint case is well understood; for the Laplace-Beltrami operator, parametrices can be derived either by approximating directly the heat kernel on physical space using Hadamard's strategy \cite{MiPl}, \cite[Section E.III]{bergergauduchonmazet}, \cite[Theorem 2.30]{BerlineGetzlerVergne} (see also~\cite{Bilal} in the physics literature) or by using pseudodifferential calculus see e.g.  \cite{Seeley}, \cite[Section 1.7]{gilkeybook} or~\cite[Chapter~7, Section~13]{Taylor:II};  for more general elliptic operators, one can use semiclassical pseudodifferential calculus \cite{Robert,DSbis}. In the hypoelliptic selfadjoint case, a parametrix for the heat semigroup by mean of a FIO with complex phase was given in \cite{menikoffsjostrand}. For type II hypoelliptic equations, the literature seems to be more sparse. The fundamental solution of the Kolmogorov operator $ \partial_t - v  \partial_x - \partial_v^2 $ can be found in \cite{Kolmogorov} and \cite[Section VII.6]{Hormander1}, and for the free Fokker--Planck (\ref{VlasovFokkerPlanck}) (\textit{i.e.} with $ V \equiv 0 $), a formula can be found in \cite[Section II.1]{bouchut}. More generally, formulas can be obtained for quadratic potentials, see \cite{AlphonseBernier1, HormanderMehler, PSMehler}, but not much seems to be available for general potentials $ V $. For geometric Fokker--Planck operators, we note however that a parametrix in physical space has been computed by Lebeau in \cite[Theorem 4.1]{Lebeau1}. The closest approach to our work is perhaps  the recent work  \cite{Smith2}, where Smith constructs  a FIO parametrix for a family of semiclassical kinetic equations modelled on the Kolmogorov operator (there is no harmonic potential $ |v|^2 $ in his case), which he uses mostly for different applications from ours (only the bound (6.2) in his Theorem 6.1 is close to (\ref{estimationprincipale4}) in our Theorem~\ref{theoremeLL}). 
\end{rema}

\subsection{Approximate polar decomposition}

We next complement~\eqref{eq:approxevolop} by stating an approximate polar decomposition for the operators $\mathcal E_A(t)$,  in which we factor out the transport part of the semigroup.
To that end, we introduce the map $ F_t : {\mathbb R}^{2n} \rightarrow {\mathbb R}^{2n} $ defined for every $(x,v)\in\mathbb R^{2n}$ by
\begin{equation}\label{eq:functionFt} 
	F_t (x,v) := \big(\Re (\partial_{\xi} \psi) , \Re(\partial_{\eta} \psi) \big)(x,v)
	= \bigg( x - t v - \frac{t^2}{2m} \nabla V (x) , v + \frac{t}{m} \nabla V (x)  \bigg) ,
\end{equation}
which is an appropriate Taylor expansion of the flow of the vector field $-\frac{1}{m}\mathcal{T}$ at $t=0$  (recall the definition of $\mathcal{T}$ in~\eqref{definittransportharmonique}). 
In the next proposition, $ {\mathcal I}_{F_t} $ denotes the composition operator by $ F_t $, given for every $u\in L^p(\mathbb R^{2n})$ by
\begin{equation}\label{eq:compounit}
	{\mathcal I}_{F_t} u = u \circ F_t,
\end{equation}
and $\mathrm{Op}$ denotes the standard quantization on $\R^{2n}$, defined in the usual way by 
\begin{equation}\label{e:classical-quantization}
	\big(\Op(a)u \big)(x,v)  = (2\pi)^{-2n} \iiiint_{\mathbb R^{4n}}  e^{ i(x- y) \cdot \xi + i (v- w) \cdot \eta} a (x,v,\xi,\eta) u(y,w)  \, dy  dw d \xi d \eta .
\end{equation}

\begin{theo}[Approximate polar decomposition]\label{thm:polardecompoleger}
Assume that $V$ satisfies~\eqref{hypothesepotentiel}. There exists $t_0\in(0,1)$ such that for every $t\in(0,t_0)$ and $ A \in {\mathcal A}^{\nu,d} $,
\begin{equation}\label{eq:polardecompo}
	\mathcal E_A (t) = {\mathcal I}_{F_t} \Op (a_t),
\end{equation} 
where $(a_t)_{t\in(0,t_0)}$ is a family of symbols in $C^{\infty}(\mathbb R^{4n})$ satisfying: for every $N>0$ and $ \alpha,\beta,\gamma,\delta\in\mathbb N^n$, there exists a positive constant $c>0$ such that for every $t\in(0,t_0)$ and $(x,v,\xi,\eta) \in {\mathbb R}^{4n}$,
\[
	\big| \partial_x^{\alpha} \partial_v^{\beta} \partial_{\xi}^{\gamma} \partial_{\eta}^{\delta} a_t (x,v,\xi,\eta) \big| 
	\le ct^{\nu + \frac{|\beta|+|\delta|+3 |\gamma|}{2}  } \Big(1 + \big| t^{\frac{3}{2}} \nabla V (x) \big| + \big| t^{\frac{1}{2}} v \big| + \big| t^{\frac{3}{2}} \xi \big| + \big| t^{\frac{1}{2}} \eta \big| \Big)^{-N}.
\]
\end{theo}

\begin{rema} \label{rema:polar} (1) Note that the scales in Theorem \ref{thm:polardecompoleger} are consistent with~\eqref{quantitespertinentes}.

(2) In the specific case $p=2$, let us justify the name ``approximate polar decomposition'' for the formula~\eqref{eq:polardecompo}. Recall that given $t\in(0,t_0)$ and $A \in {\mathcal A}^{\nu,d}$, the evolution operator $\mathcal E_A(t)$, being bounded on $L^2(\mathbb R^{2n})$, it admits the following unique polar decomposition
\[
	\mathcal E_A(t) = U_tS_t,
\]
where $U_t$ is a partial isometry on $L^2(\mathbb R^{2n})$ and $S_t$ is a bounded selfadjoint operator on $L^2(\mathbb R^{2n})$, satisfying $\Ker U_t = \Ker S_t$, see e.g.~\cite[Theorem 4.4.3]{NR}. From the fact that $F_t$ has Jacobian $ 1 +\mathcal O (t)$ by Proposition~\ref{lemmedediffeo} below, we deduce that $ {\mathcal I}_{F_t} $ is asymptotically unitary on $ L^2(\mathbb R^{2n}) $. Moreover, from Theorem~\ref{propositionpolaire}, the symbol $a_t$ is asymptotically real and $\Op(a_t)$ is selfadjoint modulo $\mathcal O(t)$. Hence, the expression~\eqref{eq:polardecompo} actually provides an approximate polar decomposition for the operators $\mathcal E_A(t)$ when $t\to 0^+$. Moreover, combined with Theorem~\ref{maintechnicalresulttheorem}, the decomposition~\eqref{eq:polardecompo} implies that for $t\in(0,t_0]$,
\begin{equation}\label{e:decomposition-polaire}
	e^{- \frac{t}{m}\ope} = {\mathcal I}_{F_t} \Op (a_t) + {\mathcal R}_N (t).
\end{equation} 
The decomposition~\eqref{e:decomposition-polaire}, which allows in some sense to pass from a non-selfadjoint problem to a selfadjoint one, is key in the study of the evolution operators $e^{- \frac{t}{m}\ope}$ and the operator $\mathcal P$ itself, and more specifically in the study of the smoothing-localization phenomena and the spectral applications that are presented in the rest of the introduction.
\end{rema}

\begin{rema} 
Let us fix anew $p=2$ and further discuss about the approximate polar decomposition~\eqref{e:decomposition-polaire} by mentioning that in the quadratic case where $V(x) = \frac{\cq}2\vert x\vert^2$ for some real number $\cq\in\mathbb R$, a precise description of the polar decomposition of the evolution operators $e^{-t\ope}$ has been made in the works~\cite{AlphonseBernier2, AlphonseBernier1} (where all the physical constants $\gamma,\beta$ and $m$ are taken equal to $1$). Precisely,~\cite[Theorem 1.1, Theorem 1.2]{AlphonseBernier1} state that there exist a family $(a_t)_{t\geq0}$ of non-negative quadratic forms $a_t:\mathbb R^{2n}\rightarrow\mathbb R_+$ and a family $(U_t)_{t\geq0}$ of metaplectic operators (being in particular unitary on $L^2(\mathbb R^{2n})$) such that for every $t\geq0$,
\begin{equation}\label{eq:polarquadratic}
    e^{-t\ope} = e^{\frac{nt}{2}}U_te^{-ta_t^w}.
\end{equation}
Furthermore, there exist $t_0>0$ and a family $(b_t)_{t\in[0,t_0)}$ of real-valued quadratic forms $b_t:\mathbb R^{2n}\rightarrow\mathbb R$ such that for every $t\in[0,t_0)$, the metaplectic operator $U_t$ takes the form $U_t = e^{-itb^w_t}$. Moreover, the quadratic forms $a_t$ satisfy the following coercivity estimate for short-times $t\in[0,t_0]$ and for every $(x,v,\xi,\eta)\in\mathbb R^{4n}$,
\begin{equation}\label{eq:coercivity}
    a_t(x,v,\xi,\eta)\gtrsim \vert\eta\vert^2 + \vert v\vert^2 + t\xi\cdot\eta + \frac{t^2}{3}\vert\xi\vert^2.
\end{equation}
Finally, in the particular case $\cq=0$, we even have explicit formulas~\cite[Example 2.3]{AlphonseBernier1} given for every $t\geq0$ by
\[
    a_t^w(x,v,D_x,D_v) = -\Delta_v + \vert v\vert^2 + (\tanh t)\nabla_x\cdot\nabla_v - \bigg(\frac{t-\tanh t}{4t} + \frac{(\tanh t)^2}{4}\bigg)\Delta_x,
\]
and
\[
    U_t = e^{-(\tanh t)v\cdot\nabla_x}.
\]
To the best of our knowledge, there are very few classes of operators for which such an explicit formula is known. In addition to the Fokker--Planck operator with no external force presented above, we get explicit formulas for the Ornstein--Uhlenbeck operators~\cite[Remark 2.5]{AlphonseBernier1}, including the Kolmogorov operator for which the formula writes for every $t\geq0$,
\[
    e^{t(\Delta_v - v\cdot\nabla_x)} = e^{-tv\cdot\nabla_x}e^{t\Delta_v - t^2\nabla_x\cdot\nabla_x + \frac{t^3}{3}\Delta_x}.
\]
The  operator $-\Delta_v + v\cdot\nabla_x$ was introduced in~\cite{Kolmogorov} by Kolmogorov, who exhibited an explicit fundamental solution (see also~\cite[pp. 210--211]{Hormander1} for a derivation).
Let us mention that the study in~\cite{AlphonseBernier1} deals with a class of more general non-selfadjoint differential operators, inclusing the Fokker--Planck operator associated with a quadratic external force. In particular, the description~\eqref{eq:polarquadratic} and the short-time coercivity estimate of the form~\eqref{eq:coercivity} (adapted to the operator at play) hold for this general class of operators. In the case of the Fokker--Planck operator~\eqref{e:def-FP-operator}--\eqref{definittransportharmonique} associated with a potential $V$ satisfying~\eqref{hypothesepotentiel}, it would be interesting to also have an exact polar decomposition of the form~\eqref{eq:polarquadratic} completed with an estimate like~\eqref{eq:coercivity}, but it seems out of reach with our technics for the moment.
\end{rema}

\subsection{Smoothing and localization estimates}

Let us now turn to the regularization and localization estimates that may be deduced from the parametrix of Theorem~\ref{maintechnicalresulttheorem}. The latter states that the short-time behavior of the evolution operators $ e^{-\frac{t}{m} \ope} $ is captured by $ {\mathcal E}_{1+A_1+\cdots + A_N} (t) $. We will see that this parametrix  decays microlocally with respect to the $t$-dependent relevant quantities~\eqref{quantitespertinentes} (this can be guessed from Theorem \ref{thm:polardecompoleger}) and leads to the following result.

\begin{theo}[Smoothing and localization estimates]\label{theoremeLL} Assume that $V$ satisfies~\eqref{hypothesepotentiel}. Let $1 < p \leq q < \infty$. Then for each $t> 0$, the operator $e^{- \frac{t}{m} \ope} $ maps $ L^p(\mathbb R^{2n}) $ into $L^q(\mathbb R^{2n}) \cap C^{\infty}(\mathbb R^{2n})$. Moreover, for any $ T > 0 $ and $ \alpha ,\beta,\delta,\gamma\in \N^n$, there exists $ C_{\alpha,\beta,\delta,\gamma} > 0 $ such that for all $ t \in (0,T] $ and $ u \in L^p({\mathbb R}^{2n}) $, 
\begin{equation}\label{estimationprincipale4} 
	\big\Vert\nabla V (x)^{\alpha}v^{\beta}\partial_x^{\gamma}\partial_v^{\delta}(e^{- \frac{t}{m} \ope} u) \big\Vert_{L^q} 
	\leq\frac{C_{\alpha,\beta,\delta,\gamma}}{t^{\frac{3\vert\alpha\vert}{2}+\frac{\vert\beta\vert}{2}+\frac{3\vert\gamma\vert}{2}+\frac{\vert\delta\vert}{2}+2n (\frac{1}{p} - \frac{1}{q} )}}\| u \|_{L^p}.
\end{equation}
In particular, when $q=p$, the operators
\begin{equation} \label{Miklin} 
	t \Delta_v e^{-\frac{t}{m}\ope} , \qquad t^3 \Delta_x e^{- \frac{t}{m}\ope} , \qquad 
	t |D_x|^{\frac{2}{3}} e^{-\frac{t}{m}\ope} , \qquad t |v|^2 e^{-\frac{t}{m}\ope} , \qquad  t^3 |\nabla V(x)|^2 e^{- \frac{t}{m}\ope},
\end{equation}
are bounded on $ L^p(\mathbb R^{2n}) $ uniformly with respect to $ t \in [0,T] $. 
\end{theo}

\begin{rema} (1) Note that when $\alpha=\beta=\delta=\gamma=0$, the corresponding estimate is the same as that of a second order uniformly elliptic operator on $ {\mathbb R}^{2n} $, e.g. the same $ L^p(\mathbb R^{2n}) \rightarrow L^q(\mathbb R^{2n}) $ bound as that of the heat semigroup for $ - \Delta_x - \Delta_v $. 

(2) The uniform boundedness of the operators $ t \Delta_v e^{-\frac{t}{m}\ope} $ and $t |v|^2 e^{-\frac{t}{m}\ope}$ is fairly clear from the ellipticity of the operator $\ope$ with respect to the velocity variable $v$; that of $t^3 \Delta_x e^{- \frac{t}{m}\ope}$, $ t|D_x|^{2/3} e^{- \frac{t}{m} \ope} $ and $t^3|\nabla V(x)|^2 e^{- \frac{t}{m}\ope}$ is a manifestation of the hypoellipticity of $\ope$. The need for the powers $3$ and $ 2/3 $ of $t$ is consistent with the relevant scales in~\eqref{quantitespertinentes}, and follows from the exponential decay of $ e^{- \Im(\psi)} $ with respect to $ t^{\frac{3}{2}} \xi $ and $t^\frac32 \nabla V(x)$ respectively.
\end{rema}

\begin{rema}\label{rk:stratsmoothing} The precise strategy of proof of Theorem \ref{theoremeLL} is first to establish, using Theorem \ref{thm:polardecompoleger}, the following smoothing and localization estimates for the parametrix $\mathcal E_A(t)$ associated with an amplitude $A\in{\mathcal A}^{\nu,d}$, which hold for every  $t\in(0,t_0]$ and $u\in \mathcal S(\mathbb R^{2d})$,
\begin{equation}\label{eq:smoothlocpara}
	\big\Vert W_t {\mathcal E}_A (t) W_t^{\prime} u \big\Vert_{L^p}   \lesssim t^{\nu}\Vert u\Vert_{L^p},
\end{equation}
where the differential operators $W_t$ and $W_t$ are defined by
\begin{align*}
	& W_t = (t^{\frac{1}{2}} v)^{\alpha} (t^{\frac{3}{2}}\nabla V (x))^{\beta} (t^{\frac{1}{2}} \partial_v)^{\gamma} (t^{\frac{3}{2}} \partial_x)^{\delta} ,  \label{formedeWt4}  \\
	& W^{\prime }_t = (t^{\frac{1}{2}} \partial_v)^{\gamma^{\prime}} (t^{\frac{3}{2}} \partial_x)^{\delta^{\prime}}  (t^{\frac{1}{2}} v)^{\alpha^{\prime}} (t^{\frac{3}{2}}\nabla V (x))^{\beta^{\prime}},
\end{align*}
with $ \alpha,\alpha^{\prime},\beta,\beta^{\prime},\gamma,\gamma^{\prime},\delta,\delta^{\prime}\in\mathbb N^n $. Notice once again that the operators $W_t$ and $W_t'$ are defined according to the relevant quantities~\eqref{quantitespertinentes}. Thanks to these estimates, and using the definition~\eqref{eq:remainder} of the remainder $\mathcal R_N(t)$, we then establish that for all multiindices $ \alpha , \alpha^{\prime},\beta,\beta^{\prime},\gamma,\gamma^{\prime},\delta,\delta^{\prime} \in \N^n$ and all $ M > 0 $, there exists $N\in\mathbb N^*$ such that for every $ t \in [0,t_0] $ and  $ u \in \mathcal S(\mathbb R^{2n}) $,
\[
	\big\|  (\nabla V (x))^{\alpha} v^{\beta} \partial_x^{\gamma}  \partial_v^{\delta}  {\mathcal R}_N (t) \partial_v^{\delta^{\prime}} \partial_x^{\gamma^{\prime}} (\nabla V(x))^{\alpha^{\prime}} v^{\beta^{\prime}}  u\big\|_{L^p } 
	\lesssim t^M \| u \|_{L^p}.
\]
Notice that in this estimate, the differential operators in front of $ {\mathcal R}_N (t) $ do not depend on the time $t$, unlike those in~\eqref{eq:smoothlocpara}. All this analysis then leads to the short-time estimates 
\begin{equation}\label{eq:sharpsmooth}
	\big\| W_t e^{-t \ope} W_t^{\prime}u\big\|_{L^q} \lesssim C t^{-2n (\frac{1}{p} - \frac{1}{q} )}\Vert u\Vert_{L^p},
\end{equation}
the estimates~\eqref{estimationprincipale4} in Theorem \ref{theoremeLL} being a particular case.
\end{rema}

\begin{rema} Smoothing and localization estimates for the Fokker--Planck semigroup have already been derived in the literature. Such short-time estimates have been first obtained in~\cite[Theorem~1.1]{Herau07} in the case where the potential $V\in C^{\infty}(\mathbb R^n)$ is smooth and also confining (the assumptions on $V$ are actually more restrictive), with technics based on the construction of a Lyapunov functional. In particular, the estimates~\eqref{estimationprincipale4} extend those in~\cite[Theorem 1.1]{Herau07} since they hold to any order of derivation and localization, and they are not restricted to $L^2(\mathbb R^{2n})$.

In the quadratic case where $V(x) = \frac{\cq}2\vert x\vert^2$ for some $\cq\in\mathbb R^*$, very precise short-time estimates have been obtained in the papers~\cite{AlphonseBernier2, AlphonseBernier1} (which actually consider more general nonselfadjoint quadratic operators, and generalize previous results from~\cite{HitrikPravdaStarov, HPSV17, HPSV18}). Precisely, in the case $p=q=2$ (the paper~\cite{AlphonseBernier2} considers smoothing estimates in the case $1\le p\le q\le\infty$, but we restrict the discussion for simplicity), and when $\cq\in \mathbb R^*$, the authors obtain estimates of the form~\eqref{estimationprincipale4}, with a constant $C_{\alpha,\beta,\delta,\gamma}>0$ being explicitly given by
\begin{equation}\label{eq:formcst}
	C_{\alpha,\beta,\delta,\gamma} = C_{n,\cq}\,\sqrt{\alpha!}\,\sqrt{\beta!}\,\sqrt{\delta!}\,\sqrt{\gamma!},
\end{equation}
where $C_{n,\cq}>0$ depends only on the dimension $n$ and $\cq$. The strategy is to use the precise description~\eqref{eq:polarquadratic}--\eqref{eq:coercivity} of the exact polar decomposition of the evolution operators at play. The particular form~\eqref{eq:formcst} of the constant $C_{\alpha,\beta,\delta,\gamma}$ encodes the fact that the evolution operators $e^{-\frac{t}{m}\ope}$ enjoy ultra-analytic smoothing and localization properties, \textit{i.e.} exponential decrease, in the relevant variables~\eqref{quantitespertinentes}, which is a more precise information than the one given by~\eqref{estimationprincipale4} (which, however, hold in the more general case where the potential $V$ satisfies~\eqref{hypothesepotentiel}). It would be interesting to see if the estimates~\eqref{estimationprincipale4} with $C_{\alpha,\beta,\delta,\gamma}$ given by~\eqref{eq:formcst} can be obtained for more general potentials.
\end{rema}

\subsection{Resolvent estimates and spectral asymptotics}

Let us now present spectral applications of the parametrix of Theorem~\ref{maintechnicalresulttheorem}. We first prove the following pseudospectral estimate.

\begin{theo}[Spectrum and resolvent estimates] \label{theoremespectreresolvante}  
Assume that $V$ satisfies~\eqref{hypothesepotentiel}. Let $ p \in (1,\infty) $. There exists $c, C > 0  $ such that the $ L^p $ spectrum of $\ope $ is contained in the set
\[
	\Big\{ z \in {\mathbb C} \ \big| \ \Re(z) \geq - c,\, | \Im(z) | \leq C  \big( \Re(z)^2 + 1 \big) \Big\},
\]
and for $z$ outside this set, we have the resolvent estimates
\[
	\big\| (\ope - z )^{-1} \big\|_{L^p \rightarrow L^p} \leq C \big(1 + |\Re(z)| + |\Im(z)|^{\frac{1}{2}} \big)^{-1} .
\]
\end{theo}

\begin{rema} In the case $p=2$, this result is already known e.g. from~\cite[Theorem 4.1]{HerauNier} (see also~\cite[Figure 6.1]{HelfferNier}). We also refer to~\cite{Lebeau1, Lebeau0, Lebeau2, Nier:18, NSW24A} in a geometric context and to~\cite[Theorem 2.1]{OPPS12} where general quadratic operators are considered, including the operator $\ope$ when $V(x) = a\vert x\vert^2$ for some $a\in\mathbb R^*$. This result is obtained as a consequence of the construction of a parametrix for the resolvent $(\ope-z)^{-1}$ at high-frequency, which we deduce from that of $e^{-\frac{t}{m}\ope}$. For readability, we defer the statements concerning the parametrix of the resolvent to Section~\ref{s:parametrix-of-resolvent} below. 
\end{rema}

So far, there is no confining assumption on $ V $ (we use only (\ref{hypothesepotentiel})) and the spectrum of $ \ope $ might not be discrete. Note also that Theorem \ref{theoremespectreresolvante} gives no information on the ``low energy'' spectrum, in particular for $z $ close to $ 0 $; note just that in the $ L^2 $ case, we have $ \Re(\ope) = \mathcal{H} \geq 0 $ (using the skew-adjointness of $\mathcal{T}$ together with the Heisenberg uncertainty principle) hence the spectrum is contained in $ \{ \Re(z) \geq 0 \} $. Studying the spectrum near $ 0 $ is important for long-time issues as the return to equilibrium. Here, our high frequency approach gives only information on the spectrum in the large asymptotics.

As a last illustration of the applications of our parametrix, we give counting estimates on the number of eigenvalues of $ \ope$.
Given an operator $(\opeg, D(\opeg))$ on a Banach space, and $z_j \in \Sp(\opeg)$ an isolated eigenvalue of $\opeg$, we recall that the Riesz projector onto the spectral subspace associated to $z_j$ is defined by
\begin{equation}\label{eq:projector}
	\Pi(z_j) = \Pi(\opeg,z_j) :=\frac{i}{2\pi}\oint_{|z-z_j| = \varepsilon} (  A -z )^{-1} dz ,
\end{equation}
where $ \varepsilon >0$ is 
small enough so that 
\[
	\big\{z \in \C\ \vert\ |z-z_j| \leq \varepsilon \big\} \cap \Sp(\opeg) = \{z_j\},
\]
and the circle $|z-z_j| = \varepsilon$ is oriented counterclockwise, see e.g.~\cite[Chapter~III, paragraph~5, p.180]{Kato}. We call the (algebraic) multiplicity of the  isolated eigenvalue $z_j\in\Sp(\opeg)$ the positive integer 
\begin{equation}\label{eq:multiplicite}
    \mult (z_j):= \rank (\Pi(z_j)) = \tr (\Pi(z_j)) \in \N^*\cup \{+\infty\}, 
\end{equation}
and say that the multiplicity of $z_j$ is finite if $\mult (z_j) <+\infty$.
Moreover, we say that the spectrum of $ \opeg $ is  discrete if it consists only in isolated eigenvalues with finite algebraic multiplicity. In the sequel, we shall consider operators $ \opeg $, for instance $ {\mathcal P} $ (as a consequence of Theorem~\ref{theoremespectreresolvante}), with discrete spectrum and the additional property that each half plane $\{ z\ \vert \ \Re (z) \leq \lambda \}$ contains finitely many eigenvalues; for such operators,  we define the following spectral counting function
\begin{equation}\label{e:def-counting-function}
    \mathcal{N}_\opeg(\lambda) := \sum_{z_j \in \Sp(\opeg), \Re(z_j) \leq \lambda} \mult(z_j), \quad \lambda \in \R .
\end{equation}
This function counts the number of eigenvalues of the operator $\opeg$ (repeated with multiplicity) having real-part less than $\lambda$. It is sometimes written as
\[
	\mathcal{N}_\opeg(\lambda) =	\# \big\{ j \in {\mathbb N}\ \vert\ \Re( z_j ) \leq \lambda \big\} ,
\]
where $ (z_j)_{j \in {\mathbb N}} $ denotes the spectrum of the operator $\opeg$, where each eigenvalue is repeated according to its algebraic multiplicity.

\begin{theo}[Distribution of eigenvalues] \label{eigenvaluestheorem}  Assume that $V$ satisfies~\eqref{hypothesepotentiel} and, in addition, that there exist $ \sigma \in (0,1] $ and $c,C>0$ such that for every $x\in\mathbb R^n$,
\begin{equation}\label{conditiondHilbertSchmidt} 
	|\nabla V (x)| \geq c |x|^{\sigma} -C.
\end{equation}
Then, the following statements hold.
\begin{enumerate}
	\item The $ L^p $ spectrum of $ \ope $ is discrete and does not depend on $p$, that is to say $\Sp_{L^p}(\ope) = \Sp_{L^2}(\ope)$ for all $p\in (1,\infty)$, and the algebraic multiplicity of each eigenvalue does not depend on $p$ either. 
	\item
	There exist $C>0$ and $t_0>0$ such that for every $t \in (0,t_0)$,
	\begin{equation}\label{eq:upperboundcount}
		\mathcal{N}_\mathcal{P}(t^{-1})\leq e \big\| e^{-t \ope} \big\|_{\tr} \leq Ct^{- \frac{3n}{2 \sigma}- \frac{5n}{2}},
	\end{equation}
	where $ \| \cdot \|_{\tr} $ is the trace class norm over $ L^2(\mathbb R^{2n})$.
	\item If in addition $V$ satisfies that there exists $C>0$ such that for every $x\in\mathbb R^n$,
	\begin{align}\label{e:lower-sigma}  
		|\nabla V (x)| \leq C \langle x\rangle^{\sigma},
	\end{align} 
	where $\sigma\in(0,1]$ is the same as in the assumption~\eqref{conditiondHilbertSchmidt}, then there are $c>0$ and $t_0>0$ such that for every $t \in (0,t_0)$,
	\begin{equation}\label{eq:lowerboundcount}
		\mathcal{N}_\mathcal{P}(t^{-1})\geq c \bigg(\frac{t^{-1}}{\log(t^{-1})}\bigg)^{ \frac{n}{2 \sigma} +\frac{3n}{2}}.
	\end{equation}
\end{enumerate}
\end{theo}

\begin{rema}\label{rk:stratspe}
(1) Notice that given $p\in(1,\infty)$, when~\eqref{conditiondHilbertSchmidt} holds, it follows from Proposition \ref{prop:compactness} that each evolution operator $e^{-\frac{t}{m}\ope}$ is compact ($t>0$), and Lemma \ref{e:compact-semigpe-resolvent} then implies that the operator $\ope$ has compact resolvent and therefore discrete $L^p$ spectrum.

(2) The left inequality in~\eqref{eq:upperboundcount}, \textit{i.e.} the estimate of the counting function by the trace class norm of the semigroup, is a consequence of the Weyl inequalities (see~\eqref{eq:weylin}); the parametrix of Theorem~\ref{maintechnicalresulttheorem} and the approximate polar decomposition from Theorem \ref{thm:polardecompoleger} allow to provide with an upper bound (under Assumption~\eqref{conditiondHilbertSchmidt}) on this trace norm. We precisely prove that under the assumption~\eqref{conditiondHilbertSchmidt} only, the trace norm is of order $\mathcal O(t^{- \frac{3n}{2 \sigma}- \frac{5n}{2}})$ as $t\to 0^+$.

(3) The lower estimate~\eqref{eq:lowerboundcount} of the counting function is obtained via a Tauberian argument (motivated by the Karamata Tauberian theorem) which roughly implies that when the following two estimates hold for some $\alpha,\beta>0$ and for every $t\in(0,t_0)$,
\[
	t^{-\alpha} \lesssim \big| \tr(e^{-\frac{t}{m}\ope}) \big|  \quad \text{ and }\quad \big\Vert e^{-\frac{t}{m}\opec}\big\Vert_{\tr} \lesssim t^{-\beta},
\]
then we have
\[
	\bigg(\frac{t^{-1}}{\log(t^{-1})}\bigg)^{\alpha}\lesssim \mathcal{N}_\mathcal{P}(t^{-1})
	\lesssim t^{-\beta},
\]
the $\log$ factor being removed when $\alpha = \beta$, see Corollary \ref{c:asympt-spectrales} for a more general and precise statement. 

Under the two assumptions~\eqref{conditiondHilbertSchmidt} and~\eqref{e:lower-sigma}, we prove actually that 
\[
    t^{- \frac{5n}{2} - \frac{3n}{2 \sigma} } \lesssim  \big\| e^{-\frac{t}{m} \ope} \big\|_{\tr}\lesssim 	t^{- \frac{5n}{2} - \frac{3n}{2 \sigma} },
    \quad \text{ and }
    \quad t^{- \frac{3n}{2} - \frac{n}{2 \sigma} } \lesssim  | \tr( e^{-\frac{t}{m} \ope} ) | \lesssim t^{- \frac{3n}{2} - \frac{n}{2 \sigma} },
\]
as $t \to 0^+$, 
again using Theorems \ref{maintechnicalresulttheorem} and \ref{thm:polardecompoleger}. We can even get an asymptotic of the trace when $V$ satisfies some homogeneity condition at infinity (see Theorem \ref{t:trace-asymptotics}).

This part of the proof is performed in a general functional analytic framework. We introduce a class of operators, which we call \schtroumpf operators, for which we prove among other things that trace and trace norm estimates for $e^{-t\opeg}$ yield asymptotic bounds on the spectral counting function $\mathcal{N}_{\opeg}$. 
Roughly speaking, we say that $\opeg$ is a \schtroumpf operator of order $\delta \in (0,1]$ if $-\opeg$ generates a bounded semigroup and if $\|(\opeg- i y)^{-1}\| \leq C\langle y \rangle^{-\delta}$ for $y \in \R$ large enough. Theorem~\ref{theoremespectreresolvante} proves that, up to an additive constant, the Fokker--Planck operator $\ope$ is a \schtroumpf operator of order $\frac12$ on $L^p(\mathbb R^{2n})$.

(4)
Notice that the situation here is very different from the selfadjoint case: if $P$ is a selfadjoint operator on a Hilbert space such that $e^{-tP}$ is trace class, then $\tr(e^{-tP})=\|e^{-tP}\|_{\tr}$. For instance if $P=-\Delta_g $ is the Laplace-Beltrami operator on a compact manifold of dimension $n$, $\tr(e^{t\Delta_g})=\|e^{t\Delta_g}\|_{\tr} \sim c t^{- \frac{n}{2}} $ as $t\to 0^+$. 
Combined with the Karamata Tauberian theorem, this can be used to prove the Weyl asymptotics on the counting function $\mathcal{N}(t^{-1})\sim c_0 \tr(e^{t\Delta_g})=c_0\|e^{t\Delta_g}\|_{\tr} \sim c_0c t^{- \frac{n}{2}}$, see e.g.~\cite[Chapter~8, Section~3]{Taylor:II}. This approach has been recently extended in~\cite{CHT:22} to prove Weyl asymptotics for type I H\"ormander hypoelliptic operators from associated heat kernel asymptotics obtained in~\cite{CHT:21}. 

(5) Beyond selfadjoint operators, Weyl asymptotics are known for perturbations of selfadjoint operators, see e.g.~\cite{MM:79,Markus:88,Sjostrand:00,AR:20}.
In the non-perturbative setting, some results (and counterexamples) are known for elliptic operators~\cite{AM:89}. We refer to~\cite{Sjostrand:19} for a recent textbook presentation of these topics. 
We stress again the fact that the Fokker--Planck operator $\ope$ is neither elliptic nor a small perturbation of a selfadjoint one.

\end{rema}

\begin{rema} In the quadratic case, {\it i.e.} when $ V (x) = \frac{\cq}2 |x|^2 $ for some real number $ \cq \ne 0 $, the spectrum can be explicitly computed.
Although this is rather classical, see e.g.~\cite[Section 5.1.1]{HelfferNier} and~\cite[Section 13]{HerauSjostrandStolk05}, we provide in Appendix \ref{a:quadratic} with a detailed computation of the spectrum in the particular case $n=1$, including the multiplicities of the eigenvalues (which, interestingly, may vary with $ \cq $). 
As a consequence, we prove in Theorem \ref{thm:spequad} that for $\cq\in(0,+\infty)\setminus\{1/4\}$ the following non-selfadjoint Weyl law holds
\begin{equation}\label{eq:weylfp}
	\mathcal{N}_{\ope}(t^{-1}) \underset{t\to 0^+}{\sim} 2\max\bigg(1,\frac1{4\cq}\bigg)t^{-2},
\end{equation}
while the upper (resp. lower) bound in Theorem \ref{eigenvaluestheorem} with $ \sigma = 1 $ (and $n=1$) is $\mathcal O (t^{-4})$ (resp. $\geq c(\frac{t^{-1}}{\log(t^{-1})})^2$). However, very little seems to be known about the distribution of eigenvalues in the framework of Theorem \ref{eigenvaluestheorem} and the above results shows how to derive  a rough bound with relatively simple arguments. We also emphasize that
the lack of sharpness of Theorem \ref{eigenvaluestheorem} (for quadratic potentials at least) is due to the use of Weyl's inequalities, not to the parametrix. 
We thus hope that our parametrix might be used to prove sharper bounds on the distributions of eigenvalues but this would require a dedicated analysis which is beyond the scope of this paper.
\end{rema}

\begin{rema}
Note that on $L^2(\mathbb R^{2n})$, the operator $\mathcal{T}$ is skew-adjoint and $\mathcal{H}$ is selfadjoint. 
Note also that $\ope$ having real coefficients, it commutes with complex conjugation and hence the $L^2$ spectrum of $\ope$ is symmetric with respect to the real axis.
\end{rema}

\begin{rema} 
(1) Compactness of the resolvent of Fokker--Planck operators in the case $p=2$ has been extensively studied, particularly in connection with the Helffer--Nier conjecture~\cite[Conjecture 1.2]{HelfferNier}, which makes the link with the compactness of the resolvent of the Witten Laplacian. The very first result in this topic~\cite[Corollary 5.10]{HelfferNier} (see also~\cite[Theorem 2.1]{HerauNier}) paved the way for a series of papers proving subelliptic or hypoelliptic estimates and deducing sufficient criterion on the potential $V$ so that the operator $\ope$ has a compact resolvent. A first series is the one by Li~\cite{Li12, Li18}, where the potential is only assumed to be in $C^2(\mathbb R^n)$. Precisely, ~\cite[Corollary 1.3]{Li12} states that the Fokker--Planck operator $\ope$ has a compact resolvent when the potential $V$ satisfies 
\[
	\vert\nabla V(x)\vert\rightarrow+\infty\quad\text{as $\vert x\vert\rightarrow+\infty$},
\]
together with
\begin{equation}\label{eq:growthv}
	\forall\vert\alpha\vert = 2,\exists c_{\alpha}>0,\forall x\in\mathbb R^n,\quad \vert\partial^{\alpha}V(x)\vert\le c_{\alpha}\langle\nabla V(x)\rangle^s,
\end{equation}
for some constant $0<s<4/3$. Notice that this pair of conditions is weaker than the assumptions~\eqref{hypothesepotentiel}--\eqref{conditiondHilbertSchmidt} considered in the present work. As shown then in~\cite[Theorem 1.2]{Li18}, in the case $s=4/3$, the second-order derivatives in the left-hand side of~\eqref{eq:growthv} can actually be replaced by the sum of the positive eigenvalues of the Hessian matrix $(\partial^2_{x_i,x_j}V(x))_{1\le i,j\le n}$. 
A series of papers by Ben Said and Ben Said--Nier--Viola extended these results in the works~\cite{BenSaid19, BenSaid22A, BenSaid22B, BenSaidNierViola}.

(2) The literature also contains numerous papers in which a semiclassical study of the Fokker--Planck operator is performed. Let us first mention the paper~\cite{HerauSjostrandStolk05} in which the potential $V\in C^{\infty}(\mathbb R^n)$ is a Morse function whose gradient is bounded from below outside a compact region and whose derivatives of order larger than $2$ are bounded (the latter corresponds to the assumption~\eqref{hypothesepotentiel}).  Theorem~1.1 of~\cite{HerauSjostrandStolk05} states resolvent estimates and a description of the spectrum, including semiclassical expansion, is presented in Theorem 1.3. Still in the semiclassical setting, let us finally mention the series of works~\cite{HerauHitrikSjostrand08A, HerauHitrikSjostrand08B, HerauHitrikSjostrand11} on tunneling effects and applications.
\end{rema}

\medskip

\noindent\textit{Outline of the article.} Section \ref{s:parametrix-section} is devoted to the construction of the parametrix for the evolution operators $e^{-\frac{t}{m}\ope}$. Precisely, we give an intuition of the choice of the phase function $\psi$ and of the classes $\mathcal A^{\nu,d}$ in which we pick the amplitudes of the parametrix in Theorem \ref{maintechnicalresulttheorem}, which is proven in this section. We also explicitly compute the kernels of the operators $\mathcal E_A(t)$ there. In Section \ref{sectionLpboundedness}, we prove the approximate polar decomposition for the operators $\mathcal E_A(t)$, presented in Theorem \ref{thm:polardecompoleger}, from which we derive smoothing and localization estimates for the semigroup $(e^{-\frac{t}{m}\ope})_{t\geq0}$, namely the ones stated in Theorem \ref{theoremeLL}, following the strategy presented in Remark \ref{rk:stratsmoothing}. The construction of parametrix for the semigroup $(e^{-\frac{t}{m}\ope})_{t\geq0}$ allows to derive a parametrix for the resolvent of the operator $\ope$ in Section \ref{s:parametrix-of-resolvent}, which allows to prove the spectrum localization and the resolvent estimates from Theorem \ref{theoremespectreresolvante}. In Section~\ref{section:fafp}, we introduce the class of \schtroumpf operators, to which the Fokker--Planck operator $\ope$ belongs, and we describe some of their functional analytic and spectral properties. In particular, we develop there the Tauberian (Karamata like) arguments presented in Remark \ref{rk:stratspe}. The latter then allows to get the spectral counting results from Theorem \ref{eigenvaluestheorem} thanks to the estimates on the trace and the trace norms of the evolution operators $e^{-\frac{t}{m}\ope}$ derived in Section \ref{sec:trace}. We also provide three appendices that contain important technical tools for our purpose, but which we collect separately to clarify the presentation of the main sections.
In Appendix \ref{s:functional-analysis}, we present standard functional analysis results for the Fokker--Planck operator $\ope$, which allows to prove Theorem~\ref{theoreme-semigroup-intro}, and in particular that $\ope$ generates a semigroup $(e^{-\frac{t}{m}\ope})_{t\geq0}$ on $L^p(\mathbb R^{2n})$ for every $p\in(1,\infty)$. Appendix \ref{section:technical} contains some technical results used all along the paper. Finally, in Appendix \ref{a:quadratic}, we present a complete study of the spectral properties of the Fokker--Planck operator $\ope$ in the particular case where the potential $V$ is quadratic, of the form $V = \frac{\cq}2\vert x\vert^2$ with $\cq\in(0,+\infty)\setminus\{1/4\}$. Precisely, we compute the spectrum of $\ope$ in this setting and then compute the multiplicity of the eigenvalues, which allow to derive the Weyl law \eqref{eq:weylfp}.

\medskip

\noindent\textit{Acknowledgements.} 
We wish to warmly thank Jean-Fran{\c c}ois Babadjian for discussions on Lebesgue spaces and Luc Hillairet and Alexandre Bailleul for discussions on the Karamata Tauberian theorem. We are also very grateful to Xue-Ping Wang for suggesting the spectral counting estimates in Theorem \ref{eigenvaluestheorem}, and J\'er\'emy Faupin for his explanations on the quadratic spectrum, which we follow in Appendix~\ref{a:quadratic}. This work was supported by the ANR LabEx CIMI (under grant ANR-11-LABX-0040) within the French State Programme ``Investissements d'Avenir''.

\section{Parametrix of the semigroup} \label{s:parametrix-section}

This section is devoted to the construction of the parametrix of Theorem~\ref{maintechnicalresulttheorem}. First, we explain the construction of the phase and the valuation spaces. Then, we check that they satisfy the required properties. The  ``explanation/construction'' section, namely Section \ref{subsec:guess}, can be skipped by the reader wanting to enter directly the proofs, but is kept here for pedagogical reasons.

\subsection{Guessing the phase function and the valuation spaces}\label{subsec:guess}
This short section is devoted to the intuition leading to the choice of phase function $\psi$, and, as a consequence, the deduction of the relevant scales~\eqref{quantitespertinentes} of the problem. The latter eventually lead to the definition of the valuation spaces in Definitions~\ref{vraievaluation}--\ref{definitionvaluation}. We stress that this section contains no proof. As far as demonstrations are concerned, this section can thus be skipped by the reader.

\paragraph{Eikonal and transport equations.} To guess suitable conditions to be satisfied by $ \psi $ and $ A $, using 
\[
	e^{-i \psi} \Delta_v (e^{i \psi} A ) =(-|\nabla_v \psi|^2+i \Delta_v\psi)A+2i \nabla_v\psi \cdot \nabla_vA+\Delta_v A,
\]
we have
\[
	{\mathcal H} \big( e^{i \psi} A \big) = \frac{\gamma}{\beta} e^{i \psi} \Big( (\nabla_v \psi)^2 + \frac{m^2 \beta^2}{4} |v|^2 - i \Delta_v \psi - \frac{m \beta n}{2} \Big) A
	+ \frac{\gamma}{\beta} e^{i \psi} ( - 2 i \nabla_v \psi \cdot \nabla_v A - \Delta_v A ).
\] 
Using that $\mathcal{T}(e^{i \psi} A ) = e^{i \psi} (i \mathcal{T}(\psi) A+\mathcal{T}(A) )$, we also have
\begin{multline}\label{Eik0}
	( m \partial_t + {\mathcal P} ) (e^{i \psi} A ) = e^{i \psi} \Big(  i m \partial_t \psi  + i {\mathcal T} \psi    + \frac{\gamma}{\beta} (\nabla_v \psi)^2  - i \frac{\gamma}{\beta} \Delta_v \psi   + \frac{\gamma m^2 \beta }{4} |v|^2 - \frac{m \gamma n}{2} \Big) A  \\
	+ e^{i \psi} \Big( m \partial_t A + {\mathcal T} A   - 2 i  \frac{\gamma}{\beta} \nabla_v \psi \cdot \nabla_v A -   \frac{\gamma}{\beta} \Delta_v A \Big) .
\end{multline}
This suggests that $ \psi $ should ideally satisfy the ``eikonal'' equation
\begin{equation}
	i m \partial_t \psi  + i {\mathcal T} \psi    + \frac{\gamma}{\beta} (\nabla_v \psi)^2  - i \frac{\gamma}{\beta} \Delta_v \psi   + \frac{\gamma m^2 \beta }{4} |v|^2 - \frac{m \gamma n}{2}  = 0,  \label{equationsurpsi}
\end{equation}
and the amplitude $A$ should ideally solve the ``transport'' equation
\[
	m \partial_t A + {\mathcal T} A   - 2 i  \frac{\gamma}{\beta} \nabla_v \psi \cdot \nabla_v A -   \frac{\gamma}{\beta} \Delta_v A = 0 .
\]
These equations should be complemented by the initial data
\begin{equation}
	\psi(0,x,v,\xi,\eta) =  x \cdot \xi + v \cdot \eta\quad  \text{ and } \quad A(0,x,v,\xi,\eta)= 1, \label{conditionsinitiales}
\end{equation}
so that the Fourier inversion formula yields $\mathcal{E}_A(0)u = u$ in~\eqref{parametrixexplicitee}.

As for the usual parametrix constructions, see e.g.~\cite{Robert}, the equation for the phase is closed, so has to be solved first, the equation on the amplitude depending on the phase.
Equation~\eqref{equationsurpsi} is a Hamilton-Jacobi equation with complex coefficients and complex unknown so, even by dropping the diffusive term $ - \Delta_v \psi $, it cannot be solved by the standard methods of characteristics. On the other hand, trying to take advantage of $ \Delta_v \psi $ and  to solve the full equation as a nonlinear parabolic one seems to be critical for the class of functions in which $ \psi $ has to be seeked.

\paragraph{The phase.}
 
It turns out that a  more naive and elementary approach allows to find a suitable phase (which, however, won't solve (\ref{equationsurpsi})). We proceed as follows. We first rewrite~\eqref{equationsurpsi}  equivalently as 
\begin{equation}
 	\partial_t \psi  =   \Big( \frac{1}{m} (\nabla_x V) \cdot \partial_v  - v \cdot \partial_x  \Big)  \psi  + i \frac{ \gamma}{\beta m} (\nabla_v \psi)^2 + \frac{\gamma}{\beta m} \Delta_v \psi + i \frac{\gamma m \beta}{4} |v|^2 - i \frac{\gamma n }{2} . \label{aderiver2fois}
 \end{equation}
Second, we use that requiring (\ref{parametrixexplicitee}) to coincide with $u$ at $ t = 0 $ leads to the initial conditions~\eqref{conditionsinitiales}, and then use it to guess the Taylor expansion around $ t = 0 $ of a function which would satisfy  (\ref{aderiver2fois}).
So we assume the existence of some $ \psi $ satisfying (\ref{aderiver2fois}) and then infer  from (\ref{conditionsinitiales}) that
\begin{equation}
	\partial_t \psi |_{t=0}  =  \frac{1}{m} (\nabla_x V) \cdot \eta - v \cdot \xi + i \frac{\gamma}{\beta m} |\eta|^2 + i \frac{\gamma m \beta}{4} |v|^2 - i \frac{\gamma n}{2}  . \label{conditioninitiale1}
\end{equation}
Then, by differentiating (\ref{aderiver2fois}) with respect to $t$, we obtain
\begin{equation}
 	\partial_t^2 \psi  =   \Big( \frac{1}{m} (\nabla_x V ) \cdot \partial_v  - v \cdot \partial_x  \Big) \partial_t \psi   + i \frac{2 \gamma}{\beta m} (\nabla_v \psi) \cdot (\nabla_v \partial_t \psi) +  \frac{\gamma}{\beta m} \Delta_v \partial_t \psi,   \label{aderiver1fois}
 \end{equation}
so, using  (\ref{conditionsinitiales}) and (\ref{conditioninitiale1}), we get
\begin{equation}\label{conditioninitale2}
	\partial_t^2 \psi|_{t=0}  =  \frac{1}{m} (\nabla_x V) \cdot \Big(  i \frac{\gamma m \beta}{2} v  -  \xi  \Big) - \frac{1}{m} (\partial_x^2 V) (v,\eta)  - i \frac{2 \gamma}{\beta m} \eta \cdot \xi - \gamma^2 v \cdot \eta + i \frac{\gamma^2 n}{2}.
\end{equation}
Finally,  differentiating  (\ref{aderiver1fois}) with respect to $t$, we find
\[
	\partial_t^3 \psi = \frac{1}{m} (\nabla_x V) \cdot \partial_v \partial_t^2 \psi   - v \cdot \partial_x \partial_t^2 \psi + i \frac{2 \gamma}{\beta m} (\nabla_v \partial_t \psi)^2  + i \frac{2 \gamma}{\beta m} \nabla_v \psi \cdot \nabla_v \partial_t^2 \psi + \frac{\gamma}{\beta m} \Delta_v \partial_t^2 \psi,
\]
which, together with (\ref{conditionsinitiales}), (\ref{conditioninitiale1}) and (\ref{conditioninitale2}),  gives                     
\begin{align}
	\partial_t^3 \psi|_{t=0}  & =    i \frac{\gamma \beta}{2 m} |\nabla_x V |^2 + i \frac{2 \gamma}{\beta m} \Big( \xi - i \frac{\gamma m \beta}{2} v \Big)^2 - \frac{1}{m^2} (\partial_x^2 V) (\eta, \nabla_x V)  - \frac{\gamma^2}{m} \nabla_x V \cdot \eta  \nonumber \\
	& \quad   + \frac{1}{m} (\partial_x^2 V) \Big(v , \xi - i \frac{\gamma m \beta}{2} v  \Big)  + i  \Big( \frac{2 \gamma}{\beta m} \eta \Big) \cdot \Big(  i \frac{\gamma \beta}{2} \nabla_x V - \frac{1}{m} \partial_x^2 V \eta  - \gamma^2 \eta \Big)  \nonumber \\
	& \quad  + \frac{1}{m} (\partial_x^3 V) (v,v,\eta)   . \label{signefact}
\end{align}
The complexity of the expressions of $ \partial_t^k \psi |_{t=0} $ is naturally increasing with $k$ so to guess a tractable formula for an effective phase, let us do a rough analysis of
\[
	\exp \left\{ -  \mbox{Im} \bigg( t \partial_t \psi|_{t=0} + \frac{t^2}{2} \partial_t^2 \psi|_{ t=0} + \frac{t^3}{3!} \partial_{t}^3 \psi |_{t= 0} \bigg)  \right\} .
\]
First, (\ref{conditioninitiale1}) gives
\begin{equation}
	\exp \left\{ -  \mbox{Im} \big( t \partial_t \psi|_{t=0} \big)   \right\} = \exp \left\{ - \frac{\gamma}{\beta m} t  |\eta|^2 - \frac{\gamma m \beta}{4}  t |v|^2 + t \frac{\gamma n}{2} \right\}  \label{premierfacteurexp}
\end{equation}   
and  suggests  that we should expect a fast (here Gaussian) decay with respect to the variables $ t^{\frac{1}{2}} v $ and $ t^{\frac{1}{2}} \eta $ in (\ref{parametrixexplicitee}), which corresponds to the natural scales for the Harmonic oscillator $-\partial_v^2 +\frac{|v|^2}{4}$. In particular, this observation tends to justify that the last (cubic) term in (\ref{signefact}), which is a contribution of the real part of $ \psi$, should actually be of lower order in the sense that
\begin{equation}
	\exp \left\{ i \frac{t^3}{6 m} (\partial_x^3 V)(v,v,\eta)  \right\}  =    1 + t^{\frac{3}{2}} \mathcal O (|t^{1/2} v|^2 |t^{\frac{1}{2}} \eta| ),   \label{expansioncubique}
\end{equation}   
where all the powers of $ t^{\frac{1}{2}}v $ and $ t^{\frac{1}{2}} \eta $  are controlled by (\ref{premierfacteurexp}) (\textit{i.e.} their products with (\ref{premierfacteurexp}) is still a rapidly decaying function of $ (t^{\frac{1}{2}} v , t^{\frac{1}{2}} \eta ) $). Note that here we use (\ref{hypothesepotentiel}) with $ |\alpha| = 3 $. Then,  (\ref{conditioninitale2}) yields
\begin{equation}\label{deuxiemefacteurexp}
	\exp \left\{ -  \mbox{Im} \bigg( \frac{t^2}{2} \partial_t^2 \psi|_{t=0} \bigg)   \right\} = \exp \left\{ - \frac{\gamma \beta}{4} t^2 \nabla_x V \cdot v + \frac{\gamma}{\beta m} t^2 \eta \cdot \xi   - t^2 \frac{\gamma^2 n}{4} \right\},  
\end{equation}    
but this does not give any intuition on localization or decay, since   the exponent has no definite sign. We thus consider  the  last factor, obtained from (\ref{signefact}), namely
\begin{multline}\label{troisiemefacteurexp}
	\exp \left\{  -  \mbox{Im} \bigg( \frac{t^3}{3!} \partial_t^3 \psi|_{t=0} \bigg) \right\}  =  \exp \left\{ - \frac{\gamma \beta t^3}{12m}  |\nabla_x V|^2
	- \frac{\gamma t^3}{3 \beta m}  |\xi|^2
	+ \frac{\gamma^3 m \beta}{12} t^3 |v|^2 + \frac{\gamma^3 t^3}{3 \beta m} |\eta|^2 \right. \\ 
	+ \left. \frac{\gamma \beta t^3}{12}  (\partial_x^2 V) (v,v)   + \frac{\gamma t^3}{3 \beta m^2} (\partial_x^2 V ) (\eta,\eta)   \right\}.
\end{multline}  
In the right-hand side of (\ref{troisiemefacteurexp}), all the terms $\mathcal O (t^3 |v|^2) $ and $\mathcal O (t^3 |\eta|^2) $ (\textit{i.e.} all but the first two terms) will be controlled by the terms $\mathcal O (t |v|^2 + t |\eta|^2) $ of (\ref{premierfacteurexp}). Thus, the formula (\ref{troisiemefacteurexp}) mainly suggests that we should have a fast decay with respect to the quantities $ t^{\frac{3}{2}} \xi $ and $ t^{\frac{3}{2}} \nabla_x V $ in (\ref{parametrixexplicitee}), which is where the subelliptic estimates will eventually come from. Of course, this will be true provided the decays of (\ref{premierfacteurexp}) and (\ref{troisiemefacteurexp}) are strong enough to control the growth of (\ref{deuxiemefacteurexp}). It is actually the case since, by reorganizing appropriately the terms, we find that
\[
	\mbox{Im} \bigg( t \partial_t \psi|_{t=0} + \frac{t^2}{2} \partial_t^2 \psi|_{ t=0} + \frac{t^3}{3!} \partial_{t}^3 \psi |_{t= 0} \bigg)
\]
is equal to
\begin{multline}\label{motivedefinitionphase}
	\frac{\gamma }{\beta m} t \bigg( \bigg| \eta - \frac{t}{2}\xi \bigg|^2 + \frac{1}{3} \bigg| \frac{t}{2} \xi \bigg|^2 \bigg) + \frac{\gamma m \beta }{4} t \bigg( \bigg|  v + \frac{t}{2} \frac{\nabla_x V}{m}\bigg|^2 + \frac{1}{3} \bigg| \frac{t}{2} 	\frac{\nabla_x V}{m} \bigg|^2 \bigg) \\
	+ \frac{t^2 \gamma^2 n}{4} - \frac{t \gamma n}{2}  - \frac{\gamma^3 m \beta}{12} t^3 |v|^2 - \frac{\gamma \beta t^3}{12} (\partial_x^2 V )(v,v) - \frac{\gamma^3 }{3 \beta m} t^3 |\eta|^2 - \frac{\gamma}{3 \beta m^2} t^3 (\partial_x^2 V) (\eta,\eta), 
\end{multline} 
where, after exponentiation, all terms in the second line of (\ref{motivedefinitionphase}) will have lower order contribution. This justifies the choice of $\Im(\psi)$ in~\eqref{defImpsi}.

\medskip

Guessing a good candidate for $ \Re(\psi) $ can be done along the same lines (after the determination of $ \Im(\psi) $ which allowed to guess the relevant variables (\ref{quantitespertinentes})). Rather than repeating the above arguments, we prefer emphasizing the connection between our choice of $ \Re(\psi) $ in~\eqref{defRepsi}--\eqref{wtdef} and the flow of $ {\mathcal T} $.
Let $(\varphi_t)_{t\in\mathbb R}$ be the flow of $ \frac{1}{m}\mathcal{T}$, which solves 
\[
	\frac{d}{dt}\varphi_t(x,v) = \frac{1}{m}\mathcal{T}(\varphi_t(x,v)),\quad \varphi_0(x,v)=(x,v)  .                	
\]
Considering the reverse flow, that is to say 
$\varphi_{-t}(x,v) =: (x(t),v(t))$ with $\dot x(t) = - v(t)$, $\dot v(t)=\frac{1}{m}\nabla V(x(t))$ and $x(0)=x$, $v(0)=v$, the Taylor expansion at $t=0$ reads
\[
	x(t) = x - t v  - \frac{t^2}{2m} (\nabla V)(x) -\frac{t^3}{3!}\partial_x^2 V(x)v +\cdots , \qquad v(t) =v +  \frac{t}{m} \nabla V (x)  - \frac{t^2}{2m} \partial_x^2 V(x)v +\cdots
\]
and we can notice that 
\begin{equation} \label{encoredestermesaflinguer}
	\frac{t^3}{3!}\partial_x^2 V(x)v \cdot \xi  = t \,\mathcal O\big( \big|t^\frac12 v\big| \big|t^\frac32 \xi\big| \big) , \quad  
	\frac{t^2}{2m} \partial_x^2 V(x)v \cdot \eta  =  t \,\mathcal O\big( \big|t^\frac12 v\big| \big|t^\frac12 \eta\big| \big) , 
\end{equation}
should have small contribution when $t \rightarrow 0$, similarly to (\ref{expansioncubique}). It turns out that (\ref{defRepsi}) is exactly what we get by discarding (\ref{encoredestermesaflinguer}) and higher order terms within
\begin{align*}
	 x(t)\cdot \xi + v(t)\cdot \eta & = \\
	& \! \! \bigg( x - t v  - \frac{t^2}{2m} (\nabla V)(x) -\frac{t^3}{3!}\partial_x^2 V(x)v +\cdots\bigg) \cdot \xi + \bigg( v +  \frac{t}{m} \nabla V (x)  - \frac{t^2}{2m} \partial_x^2 V(x)v +\cdots \bigg)\cdot \eta.
\end{align*} 
 
\paragraph{The amplitude and the valuation spaces.}
To motivate the introduction of the class of symbols for the amplitude $A$ in Definitions~\ref{vraievaluation}--\ref{definitionvaluation}, we push the above discussion a little further. Let us consider (\ref{expansioncubique}) in which we  Taylor expand  the exponential to arbitrary order
\[
	\exp \bigg\{ i \frac{t^3}{6 m} (\partial_x^3 V)(v,v,\eta)  \bigg\} 
	= \sum_{k=0}^{N-1} \frac{t^{\frac{3k}{2}}}{k!}  \bigg( i \frac{\partial_x^3 V(t^{\frac{1}{2}}v,t^{\frac{1}{2}}v,t^{\frac{1}{2}}\eta )}{6 m} \bigg)^k + \mathcal O \big( t^{\frac{3N}{2}} |t^{\frac{1}{2}}v|^{2N}|t^{\frac{1}{2}}\eta|^{2N} \big) .
\]
Each of its terms (or even the remainder) multiplied by the exponential of the first line of (\ref{motivedefinitionphase}) will be of the form
\begin{equation}
	f \big( t^{\frac{1}{2}}v , t^{\frac{1}{2}}\eta , t^{\frac{3}{2}} \nabla_x V , t^{\frac{3}{2}} \xi \big) \times t^{\frac{3k}{2}} P_k (x,t^{\frac{1}{2}}v,t^{\frac{1}{2}}\eta), \label{produitreste}
\end{equation} 
where $f$ is a Schwartz function (here a Gaussian) and $ P_k $ is a polynomial in the variables $ t^{\frac{1}{2}} v , t^{\frac{1}{2}} \eta $ with bounded coefficients depending on $x$ (by assumption (\ref{hypothesepotentiel})). Thus,
\[
	(\ref{produitreste}) = \mathcal O\big(  t^{\frac{3 k}{2}} \big\langle t^{\frac{1}{2}} v , t^{\frac{1}{2}} \eta , t^{\frac{3}{2}} \nabla_x V , t^{\frac{3}{2}} \xi \big\rangle^{-\infty} \big)
\]
and justifies that, as $ t \rightarrow 0^+ $, the terms of the form (\ref{produitreste}) are smaller and smaller as $ k $ increases. This suggests that the terms of the Taylor expansion of $ \exp (it^3 (\partial_x^3 V)(v,v,\eta)/6m) $ are a good prototype for terms of the expansion of the amplitude $A $ in (\ref{parametrixexplicitee}). Taking into account the fast decay with  respect to $ t^{\frac{3}{2}} \nabla_x V $ and $ t^{\frac{3}{2}} \xi $ supplied by the exponential of the first line of (\ref{motivedefinitionphase}), we would get the same estimate on (\ref{produitreste}) if $ P_k $  depended polynomially as well on $ t^{\frac{3}{2}} \nabla_x V $ and $ t^{\frac{3}{2}} \xi $. This eventually motivates the valuation spaces introduced in Definitions~\ref{vraievaluation}--\ref{definitionvaluation}. 

Note that in Definition~\ref{vraievaluation}, we assume a symbolic type behaviour with respect to $ t^{\frac{1}{2}} v $ which is more general than a pure polynomial behavior (as is the case in $ \Im(\psi) $). This will be  convenient to study the composition of such functions with the Hamiltonian flow of $  \frac{|v|^2}{2} + \frac{1}{m}V(x) $ (as we shall do in Section \ref{sectionflot}) which does not preserve the polynomial structure.
  
 \subsection{The phase} 

The purpose of this short section is to prove that the function $\psi$ defined in~\eqref{defImpsi}--\eqref{wtdef}--\eqref{defRepsi} is an appropriate phase function. More precisely, we check the following result.
\begin{prop}[Approximate eikonal solution] \label{eikonalreelleimaginaire} With $x_t , w_t ,\xi_t$ defined in~\eqref{wtdef} and $ \psi $ defined by (\ref{defImpsi})--(\ref{defRepsi}), let
\[
	\Eik:=  \partial_t \psi + v \cdot \partial_x \psi - \frac{1}{m} \nabla_x V \cdot \partial_v \psi - i \frac{\gamma}{\beta m} (\nabla_v \psi)^2 - \frac{\gamma}{\beta m} \Delta_v \psi - i \frac{\gamma m \beta}{4} |v|^2 .
\]
Then
\begin{equation}\label{eikonalreelle}
	\Re(\Eik)=  \frac{t}{m} (\partial_x^2V) \Big( v , \eta - \frac{\xi_t}{2} \Big)     + \gamma^2 t (\eta- \xi_t) \cdot \left( v + \frac{w_t}{2} \right) \in {\mathcal A}^{0,2} 
\end{equation}
and 
\begin{equation}\label{eikonaleimaginaire}
	\Im(\Eik) = \frac{\gamma \beta}{4} t^2 (\partial_x^2 V) \Big(v,v+ \frac{2}{3}w_t \Big)  + \frac{\gamma^3 m \beta}{4} t^2 \left| v + \frac{w_t}{2} \right|^2 - \frac{\gamma^2 t    n}{2} \in {\mathcal A}^{1,2}  . 
\end{equation}
\end{prop}
 
The exact expressions of (\ref{eikonalreelle}) and (\ref{eikonaleimaginaire}) are not very important, what matters are their valuations. Indeed, what this proposition mainly says is that we have found a phase $ \psi $ such that the parenthesis in the first line of the right-hand side of (\ref{Eik0}) has valuation $ 0 $. This is of course not as good as satisfying the eikonal equation (\ref{equationsurpsi}) but, as we shall see in Section \ref{sectionflot}, this is sufficient for our purpose.
 
 The proof of Proposition \ref{eikonalreelleimaginaire} is elementary. Nevertheless, we supply intermediate computations below to check the details.
 
\begin{prop}[Computation of $\partial_t \psi$] \label{prop0}
With  $\psi$ defined in~\eqref{defImpsi}--\eqref{wtdef}--\eqref{defRepsi}, we have
\[
	\partial_t \Re(\psi) =  - v \cdot \xi + \frac{1}{m} \nabla_x V \cdot \eta - w_t  \cdot \xi
\]
and
\[
	\partial_t \Im(\psi) = \frac{\gamma}{\beta m} |\eta - \xi_t|^2 + \frac{\gamma m \beta}{4} |v+w_t|^2.
\]
\end{prop}
 
\begin{proof} The case of $ \partial_t \Re(\psi) $ is straightforward from (\ref{wtdef}) and
\begin{align*}
	\partial_t \Im(\psi) & = \frac{\Im(\psi)}{t}  + \frac{\gamma}{\beta m} t \left( - \partial_t\xi_t \cdot \left( \eta - \frac{\xi_t}{2} \right) + \frac{1}{3} \partial_t \xi_t \cdot  \frac{\xi_t}{2}  \right) \\
	& \qquad  + \frac{\gamma m \beta}{4} t \left( \partial_t w_t \cdot \left( v + \frac{w_t}{2} \right) + \frac{1}{3} \partial_t w_t \cdot  \frac{w_t}{2}  \right)  \\
	& = \frac{\gamma}{\beta m}  \left(  \left| \eta - \frac{\xi_t}{2} \right|^2 + \frac{1}{3} \left| \frac{\xi_t}{2} \right|^2 - \xi_t \cdot \left( \eta - \frac{\xi_t}{2} \right) + \frac{1}{3} \xi_t \cdot  \frac{\xi_t}{2}  \right) \\
	& \qquad  + \frac{\gamma m \beta}{4}  \left(  \left| v + \frac{w_t}{2} \right|^2 + \frac{1}{3} \left| \frac{w_t}{2} \right|^2 + w_t \cdot \left( v + \frac{w_t}{2} \right) + \frac{1}{3} w_t \cdot  \frac{w_t}{2}  \right)  \\
	& = \frac{\gamma}{\beta m} |\eta-\xi_t|^2 + \frac{\gamma m \beta}{4} |v + w_t|^2  
\end{align*}
completes the proof.
\end{proof}
 
\begin{prop}[Computation of $ v \cdot \partial_x \psi $] \label{prop1}
With  $\psi$ defined in~\eqref{defImpsi}--\eqref{wtdef}--\eqref{defRepsi}, we have
\[
	v \cdot \partial_x \Re( \psi ) = v \cdot \xi  + \frac{t}{m} (\partial_x^2 V ) \Big(v, \eta -  \frac{\xi_t}{2} \Big)
\]
and
\[
	v \cdot \partial_x \Im(\psi) = \frac{\gamma \beta}{4} t^2 (\partial_x^2 V) \Big(v,v+ \frac{2}{3} w_t \Big)
\]
\end{prop}
 
\begin{proof} Follows from  $ v \cdot \partial_x x_t = v - \frac{t^2}{2m} \partial_x^2 Vv  $ and $ v \cdot \partial_x w_t = \frac{t}{m} \partial_x^2 V v $.
\end{proof}
 
\begin{prop}[Computation of $ - \frac{1}{m} \nabla_x V \cdot \partial_v \psi $] \label{prop2}
With  $\psi$ defined in~\eqref{defImpsi}--\eqref{wtdef}--\eqref{defRepsi}, we have
\[
	- \frac{1}{m} \nabla_x V \cdot \partial_v \Re(\psi) = - \frac{1}{m} \nabla_x V \cdot \eta + w_t \cdot \xi
\]
and
\[
	- \frac{1}{m} \nabla_x V \cdot \partial_v \Im(\psi) = - \frac{ \gamma m \beta}{2}  w_t\cdot \left( v + \frac{w_t}{2} \right) .
\]
\end{prop}
 
\begin{proof} Follows from 
\begin{equation}
	\partial_v \Re(\psi) = - t \xi + \eta, \qquad \partial_v \Im(\psi)  =   \frac{\gamma m \beta}{2} t \left( v+ \frac{w_t}{2} \right) . \label{quadratique}
\end{equation}
\end{proof}
      
\begin{prop}[Computation of $- i \frac{\gamma}{\beta m} (\nabla_v \psi)^2$] \label{prop3}
With $\psi$ defined in~\eqref{defImpsi}--\eqref{wtdef}--\eqref{defRepsi}, we have
\[
	\Re \left[ - i \frac{\gamma}{\beta m} (\nabla_v \psi)^2 \right] = \gamma^2 t (\eta - \xi_t) \cdot \left( v + \frac{w_t}{2} \right),
\]
\[
   \Im \left[ - i \frac{\gamma}{\beta m} (\nabla_v \psi)^2 \right] = - \frac{\gamma}{\beta m} |\eta-\xi_t|^2 + \frac{\gamma^3 m \beta}{4} t^2 \left| v + \frac{w_t}{2} \right|^2 .
\]
\end{prop}    
      
\begin{proof} Follows again from (\ref{quadratique}).
\end{proof}    
 
\begin{prop}[Computation of $ - \frac{\gamma}{\beta m} \Delta_v \psi $] \label{prop4}
With  $\psi$ defined in~\eqref{defImpsi}--\eqref{wtdef}--\eqref{defRepsi}, we have
\[
	- \frac{\gamma}{\beta m} \Delta_v \Re( \psi ) = 0, \qquad  - \frac{\gamma}{\beta m} \Delta_v \Im( \psi ) = - \frac{\gamma^2 t n}{2} .
\]
\end{prop}
 
\begin{proof} The function $ \psi $ is a polynomial of degree $2$ in $v$ and the quadratic part follows from the second term of $ \Im(\psi) $.
\end{proof}
 
\begin{proof}[Proof of Proposition \ref{eikonalreelleimaginaire}]
Follows by summing the contributions of the terms computed in Propositions~\ref{prop0}, \ref{prop1}, \ref{prop2}, \ref{prop3} and \ref{prop4}, as well as using that
\[
	\frac{\gamma m \beta}{4} |v+w_t|^2 - \frac{\gamma m \beta}{2} w_t \cdot \Big( v + \frac{w_t}{2} \Big) - \frac{\gamma m \beta}{4} |v|^2 = 0
\]
for the imaginary part. 
\end{proof}
  
\subsection{The amplitude} \label{sectionflot}

Using the phase $ \psi $ defined in~\eqref{defImpsi}--\eqref{wtdef}--\eqref{defRepsi}, we now wish to find an amplitude $A$ such that (\ref{Eik0}) is small. With the notation $ \Eik $ introduced in Proposition \ref{eikonalreelleimaginaire}, we can rewrite  (\ref{Eik0}) as
\begin{equation}\label{iterationA}
	e^{-i \psi}( m \partial_t + {\mathcal P} ) (e^{i \psi} A ) =  m \partial_t A + {\mathcal T} A   - 2 i  \frac{\gamma}{\beta} \nabla_v \psi \cdot \nabla_v A -   \frac{\gamma}{\beta} \Delta_v A  + m \Big( i \Eik - \frac{\gamma n}{2} \Big) A .
\end{equation}
Our strategy is to look, given any fixed $N \in {\mathbb N}^*$, for an amplitude $A$ of the form
\begin{equation}
	A = 1 + A_1 + A_2 + \cdots + A_{N}, \qquad A_k \ \mbox{of valuation} \ k  ,  \label{Azero0}
\end{equation} 
\textit{i.e.} as a sum of terms of increasing valuation (according to Definition \ref{definitionvaluation}). To explain the intuition of the iterative scheme determining the $A_k $, we record the next two propositions.

\begin{prop} \label{debutvaluation} Let $ \nu  , \nu^{\prime}  \in {\mathbb R}$ and $ d , d^{\prime} \geq 0 $. 
\begin{enumerate}
	\item{If $ B \in {\mathcal A}^{\nu,d} $ and $ C \in {\mathcal A}^{\nu^{\prime},d^{\prime}} $ then $  BC \in {\mathcal A}^{\nu + \nu^{\prime},d+d^{\prime}} $.}
	\item{If $ B \in {\mathcal A}^{\nu,d} $ then $ \Delta_v B \in {\mathcal A}^{\nu + 1,d} $.}
	\item{If $ B \in {\mathcal A}^{\nu,d} $ then $  \nabla_v \psi \cdot \nabla_v B  \in {\mathcal A}^{\nu,d+1}$.}
\end{enumerate}
\end{prop}
 
\begin{proof} The items 1 and 2 are straightforward from Definitions \ref{vraievaluation} and \ref{definitionvaluation} so we focus on item 3. By (\ref{quadratique}),  $ \nabla_v \Re(\psi) \in {\mathcal A}^{  - \frac{1}{2},1} $ and  $ \nabla_v \Im(\psi) \in {\mathcal A}^{\frac{1}{2},1} $ so  $ \nabla_v \psi $ belongs to $ {\mathcal A}^{-\frac{1}{2},1} $. Since on the other hand $ \nabla_v B $ belongs to $ {\mathcal A}^{\nu + \frac{1}{2},d} $, again by  Definitions \ref{vraievaluation} and \ref{definitionvaluation}, the inner product against $ \nabla_v \psi $ belongs to $ {\mathcal A}^{\nu,d+1} $ by item 1.
\end{proof}
 
Let $ \phi_{ t} (x,v) = (\overline{x}_{t},\overline{v}_{t}) $ be the flow of $ v \cdot \partial_x - \frac{1}{m} (\nabla V)(x) \cdot \partial_v $, namely, the solution to
\begin{equation}
	\partial_t \overline{x}_{t} = \overline{v}_{t}, \qquad \partial_t \overline{v}_{t} = - \frac{1}{m} (\nabla_x V)(\overline{x}_{t})  , \label{Hamiltonequation}
\end{equation} 
with initial condition $ (\overline{x}_{0},\overline{v}_{0}) = (x,v) $ at $ t = 0$. This is of course the Hamiltonian flow of $ \frac{1}{2} |v|^2 + \frac{1}{m}V(x) $. It is complete by assumption (\ref{hypothesepotentiel}) with $ |\alpha| = 2 $. This allows to define the operator $ {\mathcal S} $ as follows: for a given function $ B$, we let
\begin{equation}\label{integrationnonsinguliere}
	( {\mathcal S} B )(t,x,v,\xi,\eta) := \frac{1}{m} \int_0^t B (s,\overline{x}_{s-t},\overline{v}_{s-t},\xi,\eta)\,ds,
\end{equation}
which is the solution to the inhomogeneous transport equation with zero initial condition
\[
 	\big( m \partial_t  + {\mathcal T} \big) ({\mathcal S}B) = B, \qquad ({\mathcal S}B)|_{t=0} = 0 .
\]
Recall that $ {\mathcal T} $ is defined in (\ref{definittransportharmonique}). In practice, we shall consider $ B $ in some $ {\mathcal A}^{\nu,d} $ with $ \nu \geq 0 $ so the integration in (\ref{integrationnonsinguliere}) won't cause any problem.
 
\begin{prop} \label{transportvaluation} For any  $ \nu \geq 0 $ and $  d \geq 0$, $ {\mathcal S} $ maps $ {\mathcal A}^{\nu,d} $ into $ {\mathcal A}^{\nu+1,d} $.
\end{prop}
 
Intuitively, $ {\mathcal S} $ acts similarly to primitivation in time hence, as such, improves the valuation by $1$. Assume for a while this proposition has been proved and let us show how to use it to determine iteratively the terms of (\ref{Azero0}).
 
The first (`$ A_0$') term $1$ in (\ref{Azero0}) is determined by the condition that $ {\mathcal E}_A (0) = I $. If we let $A=1$ in (\ref{iterationA}), we get $ m ( i \Eik - \frac{\gamma n}{2} ) $ which has valuation $ 0 $ by Proposition \ref{eikonalreelleimaginaire}. Next, we look for $ A_1 $ such that
\[
	( m \partial_t + {\mathcal T}) A_1 = -  m \Big( i \Eik - \frac{\gamma n}{2} \Big) , \qquad A_1 |_{t=0} = 0 ,
\]
namely we pick
\[
	A_1 = - m {\mathcal S} \Big( i \Eik - \frac{\gamma n}{2} \Big),
\]
which has valuation $1$ according to Proposition \ref{transportvaluation}. Then, if we consider the right-hand side of (\ref{iterationA}) applied to $ 1 + A_1 $, we get
\[
	- 2 i  \frac{\gamma}{\beta} \nabla_v \psi \cdot \nabla_v A_1 -   \frac{\gamma}{\beta} \Delta_v A_1  + m \Big( i \Eik - \frac{\gamma n}{2} \Big) A_1 =: B_2,
\]
which has also valuation $1$  according to Proposition \ref{debutvaluation}. This allows to look for $ A_2 $ such that
\[
	( m \partial_t + {\mathcal T}) A_2  = - B_2 , \qquad A_2 |_{t=0} = 0,
\]
namely $ A_2 = - {\mathcal S} B_2 $, which has now valuation $2$ by Proposition \ref{transportvaluation} and one can then iterate this process to find $ A_3 $, etc. To describe the formalized scheme, it is convenient to denote
\[
	{\mathcal L}_{\psi} = - 2 i \frac{\gamma}{\beta} \nabla_v \psi \cdot \nabla_v - \frac{\gamma}{\beta} \Delta_v, \qquad {\mathcal V}_{\psi} = m \Big( \mathrm{Eik}  - \frac{\gamma n}{2} \Big). 
\]
We define
\begin{equation}\label{initialisationampli}
	A_0= 1, \qquad  A_1 = - {\mathcal S} \big( {\mathcal V}_{\psi} 1 \big),
\end{equation}  
and then, for $ k \geq 2 $,
\begin{equation}\label{recursionampli}
	A_k = - {\mathcal S} \big( {\mathcal L}_{\psi} A_{k-1} + {\mathcal V}_{\psi} A_{k-1} \big).
\end{equation}  
With this choice, we find that for any given $N$,
\begin{align*}
	\big( m \partial_t + {\mathcal T} + {\mathcal L}_{\psi} + {\mathcal V}_{\psi} \big) (1 + A_1 + \cdots + A_N) & = ( m \partial_t + {\mathcal T}) A_1 + {\mathcal V}_{\psi}1 +  {\mathcal L}_{\psi} A_N + {\mathcal V}_{\psi} A_N  \\
	& \qquad+ \sum_{k=2}^{N} (m \partial_t  + {\mathcal T} ) A_k + {\mathcal L}_{\psi} A_{k-1} + {\mathcal V}_{\psi} A_{k-1} \\
	& = {\mathcal L}_{\psi} A_N + {\mathcal V}_{\psi} A_N .
\end{align*}
Along this process, it is easy to check from Propositions \ref{eikonalreelleimaginaire} and \ref{debutvaluation} that $ A_k $ belongs to ${\mathcal A}^{k,2k}$. We have thus proved 

\begin{prop}[Determination of the amplitude]  \label{pourlamplitudedutheoreme}  With the phase $ \psi $ defined by (\ref{defImpsi}) and (\ref{defRepsi}), the sequence $ (A_k)_{k \geq 0} $ defined by (\ref{initialisationampli}) and (\ref{recursionampli}) satisfies\footnote{using the notation (\ref{definittransportharmonique})}
\[
	\big( m \partial_t  + {\mathcal T} + {\mathcal H} \big) \big( e^{i\psi} (1+ A_1 + \cdots + A_N) \big) = e^{i \psi}R_N,
\]
where each $A_k$ belongs to $ {\mathcal A}^{k,2k} $ and $ R_N := {\mathcal L}_{\psi} A_N + {\mathcal V}_{\psi} A_N $ belongs to $ {\mathcal A}^{N,2N+2} $.
\end{prop}

We are now in position to end the construction of the parametrix.

\begin{proof}[Proof of Theorem \ref{maintechnicalresulttheorem}] The identity~\eqref{eq:approxevolop} follows from the fact that by Proposition \ref{pourlamplitudedutheoreme},
\[
	(m \partial_t + {\mathcal P}) {\mathcal E}_{1+A_1+\cdots + A_N} (t) = {\mathcal E}_{R_N} (t),
\]
together with the uniqueness part of Proposition \ref{p:existence-uniqueness}.
\end{proof}
 
The rest of the section is devoted to the proof of Proposition \ref{transportvaluation}.

\begin{lemm} \label{lemmeflotborne}
For all multiindices such that $ |\alpha + \beta| \geq 1 $, there is $ C_{\alpha, \beta} > 0 $ such that for all $ (x,v) \in {\mathbb R}^{2n} $ and $ |t| \leq 1 $,
\begin{equation}
	|\partial_x^{\alpha} \partial_{v}^{\beta} \overline{x}_{t} |  \leq C_{\alpha, \beta} |t|^{|\beta|}, \qquad |\partial_x^{\alpha} \partial_v^{\beta}\overline{v}_{t} | \leq C_{\alpha, \beta} |t|^{\max(|\beta|-1,0)},  \label{bonneestimationflot}
\end{equation}
\end{lemm}
 
\begin{proof} By differentiating (\ref{Hamiltonequation})  in $x$ and $v$, we find by an elementary induction on $ |\alpha + \beta| $ that
\begin{equation}
	\partial_t (\partial_x^{\alpha} \partial_v^{\beta} \overline{v}_{t}) = - \frac{1}{m} (\partial_x^2 V)(\overline{x}_{s,t}) (\partial_x^{\alpha} \partial_{v}^{\beta} \overline{x}_{t}) + R_{\alpha, \beta}(t),  \label{FaaDiBruno}
\end{equation} 
where $ R_{\alpha, \beta} (t) = 0 $ if $ |\alpha + \beta| = 1 $ and is otherwise a linear combination with universal coefficients of
\[
	(\partial_x^{1+k} V ) (\overline{x}_{t}) \big( \partial_x^{\alpha_1} \partial_{v}^{\beta_1} \overline{x}_{t} , \ldots ,  \partial_x^{\alpha_k} \partial_{v}^{\beta_k} \overline{x}_{t}  \big),
\]
with
\[
	k \geq 2, \qquad  (\alpha_1 , \beta_1 ) + \cdots + (\alpha_k, \beta_k) = (\alpha,\beta), \qquad |\alpha_j + \beta_j|  \geq 1 \ \ \mbox{for all} \ 1 \leq j \leq k .
\]
When $ |\alpha + \beta | = 1 $, an application of the Gr\"onwall inequality to
\[
	\left( \begin{matrix} \partial_{x}^{\alpha} \partial_v^{\beta} \overline{x}_{t}  \\  \partial_{x}^{\alpha} \partial_v^{\beta} \overline{v}_{t}  \end{matrix} \right) 
	=   \left( \begin{matrix} \partial_{x}^{\alpha} \partial_v^{\beta}  x \\  \partial_{x}^{\alpha} \partial_v^{\beta} v \end{matrix} \right) 
	+ \int_0^t  \left( \begin{matrix} 0 & I \\ -\frac{1}{m} (\partial_x^2 V)(\overline{x}_{s}) & 0 \end{matrix} \right)  \left( \begin{matrix} \partial_{x}^{\alpha} \partial_v^{\beta}  \overline{x}_{s} \\  \partial_{x}^{\alpha} \partial_v^{\beta}  \overline{v}_{s} \end{matrix} \right)\,ds,
\]
and the boundedness of $ \partial_x^2 V $ leads to (\ref{bonneestimationflot}); more precisely, if $ |\beta | = 1 $, $ \partial_{v}^{\beta} \overline{v}_{t} = \mathcal O (1) $ while if $ |\alpha |= 1 $, $ \partial_x^{\alpha} \overline{v}_{t} = \mathcal O (t) $, so both are $\mathcal O (1)  = \mathcal O (|t|^{\max(|\beta|-1,0)})$. When $ |\alpha + \beta| \geq 2 $, we get (\ref{bonneestimationflot}) by induction using (\ref{FaaDiBruno}) and
\[
	\left( \begin{matrix} \partial_{x}^{\alpha} \partial_v^{\beta} \overline{x}_{t}  \\  \partial_{x}^{\alpha} \partial_v^{\beta} \overline{v}_{t}  \end{matrix} \right) 
	=  \int_0^t  \left( \begin{matrix} 0 & I \\ -\frac{1}{m} (\partial_x^2 V)(\overline{x}_{s}) & 0 \end{matrix} \right)  \left( \begin{matrix} \partial_{x}^{\alpha} \partial_v^{\beta}  \overline{x}_{s} \\  \partial_{x}^{\alpha} \partial_v^{\beta}  \overline{v}_{s} \end{matrix} \right)\,ds 
	+  \int_0^t  \left( \begin{matrix} 0  \\ R_{\alpha, \beta} (s) \end{matrix} \right)\,ds,
\]
where, by the induction assumption and (\ref{hypothesepotentiel}), $ R_{\alpha, \beta} (s) = \mathcal O (|s|^{|\beta|}) $. The Gr\"onwall inequality implies that 
\[
	\partial_x^{\alpha} \partial_{v}^{\beta} \overline{x}_{t} = \mathcal O (|t|^{|\beta|+1}) , \qquad\partial_x^{\alpha} \partial_v^{\beta} \overline{v}_{t} = \mathcal O (|t|^{|\beta|+1}),
\]
which is even stronger than (\ref{bonneestimationflot}) (but (\ref{bonneestimationflot}) holds regardless $ |\alpha + \beta| \geq 2$ or $ |\alpha + \beta| = 1 $). 
\end{proof}

\begin{lemm} \label{variationlente}
There exists $ C > 0 $ such that for all $ (x,v) \in {\mathbb R}^{2n} $ and $ |t| \leq 1 $,
\[
	|\overline{v}_{t} - v| + \big|(\nabla_x V ) (\overline{x}_{t}) - (\nabla_x V)(x) \big|  \leq C |t| \big( 1 + |v| + \big|(\nabla_x V)(x) \big| \big) .
\]
\end{lemm}

\begin{proof} The equation $ \partial_t^2 \overline{v}_{t} = - m^{-1}  (\partial_x^2 V)(\overline{x}_{t}) \overline{v}_{t}  $, written in terms of the vectors
\[
	\left( \begin{matrix} X_t \\ Y_t \end{matrix} \right) := \left( \begin{matrix} \overline{v}_{t} - v \\ \partial_t \overline{v}_{t}  + \frac{(\nabla_ x V)(x)}{m} \end{matrix} \right), \qquad 
	\left( \begin{matrix} Z_t \\ T_t \end{matrix} \right) := \int_0^t  \left( \begin{matrix} 0 & I \\
	- \frac{\partial_x^2 V (\overline{x}_{s})}{m} & 0 \end{matrix} \right) \left( \begin{matrix} v \\\ - \frac{(\nabla_x V)(x)}{m} \end{matrix} \right)\,ds,
\]
reads
\[
	\left( \begin{matrix} X_t \\ Y_t \end{matrix} \right) = \int_0^t \left( \begin{matrix} 0 & I \\
	- \frac{\partial_x^2 V (\overline{x}_{s})}{m} & 0 \end{matrix} \right) \left( \begin{matrix} X_s \\Y_s \end{matrix} \right)\,ds + \left( \begin{matrix} Z_t \\ T_t \end{matrix} \right),
\]
so the conclusion follows from the Gr\"onwall inequality and (\ref{hypothesepotentiel}) with $ |\alpha| = 2 $.
\end{proof}

\begin{proof}[Proof of Proposition \ref{transportvaluation}] 
Let $ a \in {\mathcal V}^{\nu,\delta} $.  By the Fa\`{a} di Bruno formula, we get that $ \partial_x^{\alpha} \partial_v^{\beta}  \big(a(s,\overline{x}_{s-t},\overline{v}_{s-t}) \big) $ is a linear combination of terms of the form
\[
	(\partial_x^{\Theta} \partial_v^{\Xi} a)(s,\overline{x}_{s-t}, \overline{v}_{s-t}) \big( \partial_x^{\alpha_1} \partial_{v}^{\beta_1} \overline{x}_{s-t} \big)  \cdots  
	\big( \partial_x^{\alpha_{|\Theta|}} \partial_{v}^{\beta_{|\Theta|}} \overline{x}_{s-t} \big)   \big( \partial_x^{\alpha^{\prime}_1} \partial_{v}^{\beta^{\prime}_1} \overline{v}_{s-t} \big)  \cdots  
	\big( \partial_x^{\alpha^{\prime}_{|\Xi|}} \partial_{v}^{\beta^{\prime}_{|\Xi|}} \overline{v}_{s-t} \big),
\]
with $ \beta_1 + \cdots + \beta_{|\Theta|} + \beta^{\prime}_1 + \cdots + \beta^{\prime}_{|\Xi|} = \beta $. By (\ref{bonneestimationflot}) and the fact that $ a $ belongs to $ {\mathcal V}^{\nu,\delta} $, this is bounded by a constant times
\begin{equation}
	|s|^{\nu + \frac{|\Xi|}{2}}  |t-s|^{|\beta_1| + \cdots + |\beta_{|\Theta|}|} |t-s|^{\max(|\beta_1^{\prime}| - 1 , 0 ) + \cdots + \max (|\beta_{|\Xi|}^{\prime}|-1,0)} \label{exposantbienchoisi}
\end{equation}
times
\begin{equation}
	\left(1 + \big|s^{\frac{1}{2}} \overline{v}_{s-t} \big| + \big|s^{\frac{3}{2}} (\nabla_x V) (\overline{x}_{s-t})\big|\right)^{\delta} . \label{poidslent}
\end{equation}
In the exponents of (\ref{exposantbienchoisi}), observe that
\[
	\frac{|\Xi|}{2} + \max(|\beta_1^{\prime}| - 1 , 0 ) + \cdots + \max (|\beta_{|\Xi|}^{\prime}|-1,0) \geq \frac{|\beta_1^{\prime}|  + \cdots + |\beta_{|\Xi|}^{\prime}|}{2},
\]
since the left-hand side is not smaller than both $ \frac{|\Xi|}{2} $ and $ |\beta_1^{\prime}|  + \cdots + |\beta_{|\Xi|}^{\prime}| - \frac{|\Xi|}{2}  $ hence is not smaller than their mean which is the right-hand side. Since $s$ is between $0$ and $t$, it follows that
\[
	(\ref{exposantbienchoisi}) \leq |t|^{\nu + \frac{ |\beta_1^{\prime}|  + \cdots + |\beta_{|\Xi|}^{\prime}|}{2} +  |\beta_1|  + \cdots + |\beta_{|\Theta|}|} \leq |t|^{\nu + \frac{|\beta|}{2}} ,
\]
and, from Lemma \ref{variationlente}, that
\[
	(\ref{poidslent}) \lesssim  \Big(1 + \big|t^{\frac{1}{2}} v \big| + \big|t|^{\frac{3}{2}} (\nabla_x V) (x)\big|\Big)^{\delta} .
\]
We conclude that
\[
	\big| \partial_x^{\alpha} \partial_v^{\beta} ( a (s,\overline{x}_{s-t},\overline{v}_{s-t}) ) \big| \lesssim |t|^{\nu + \frac{|\beta|}{2}}  \Big(1 + \big|t^{\frac{1}{2}} v \big| + \big|t^{\frac{3}{2}} (\nabla_x V) (x)\big|\Big)^{\delta},
\]
so integration in time over $ [0,t]  $ yields the missing additional power of $ t $ and shows that $ \int_0^t   a (s,\overline{x}_{s-t},\overline{v}_{s-t})\,ds$ belongs to $  {\mathcal V}^{\nu+1,\delta} $. The result follows since applying $ {\mathcal S} $ to an expression of the form (\ref{expressionoftheform}) does not affect the variables $ \xi, \eta $.  
\end{proof}

\subsection{Kernel of the parametrix}

To conclude this section, we compute explicitly the kernel of each term in the parametrix of Theorem~\ref{maintechnicalresulttheorem}.

\begin{lemm} \label{lemma-e:kernel-first-term} Assume that $V$ satisfies~\eqref{hypothesepotentiel}. Let $\nu\in \R$, $d\geq 0$ and consider $A \in {\mathcal A}^{\nu,d}$ (see Definition~\ref{definitionvaluation}) written as 
\begin{equation}\label{e:expansion-Balphabeta}
	A = \sum_{|\alpha| + |\beta| + \delta \leq d} A_{\alpha \beta} \quad \text{ with }\quad  
	A_{\alpha \beta} (t,x,v,\xi,\eta) =  a_{\alpha \beta} (t,x,v) \big( t^{\frac{3}{2}} \xi \big)^{\alpha}  \big(t^{\frac{1}{2}} \eta \big)^{\beta} ,  \quad b_{\alpha \beta} \in {\mathcal V}^{\nu,\delta}.
\end{equation}
Then, the Schwartz kernel $K_A(t)$ of $\mathcal{E}_A(t)$ is given by 
\begin{equation}\label{e:kernel-B1}
	K_A(t,x,v,y,w) =  \sum_{|\alpha| + |\beta| + \delta \leq d} a_{\alpha \beta} (t,x,v)
	e^{-\frac{\gamma m \beta}{4} t ( | v + \frac{w_t}{2} |^2 + \frac{1}{3} | \frac{w_t}{2} |^2 )}K^{\alpha\beta}_t(x,v,y,w),
\end{equation}
with $w_t = \frac{t}{m} \nabla V (x)$, the functions $K^{\alpha\beta}_t$ being defined by
\begin{equation}\label{e:kernel-B2}
	K^{\alpha\beta}_t(x,v,y,w)  := \sum_{\kappa \leq \beta}\binom{\beta}{\kappa}
	\frac{1}{2^{|\kappa|}t^{\frac{3n}{2}}} \tilde{G}_{\alpha+\kappa}\bigg( \frac{x-y-\frac{t}{2}(v+w)}{t^\frac32} \bigg)
	\frac{1}{t^{\frac{n}{2}}}G_{\beta-\kappa}\bigg(\frac{v-w+w_t}{t^\frac12} \bigg),
\end{equation}
where
\begin{equation}\label{e:def-Galpha}
	G_\alpha(x) : = \frac{1}{(2\pi)^n}\int_{\mathbb R^n} e^{ix\cdot \xi} \xi^\alpha e^{-\frac{\gamma}{\beta m}|\xi|^2}\, d\xi  \quad \text{ and } \quad 
	\tilde{G}_\alpha(x) :  = (12)^{\frac{n+|\alpha|}{2}}G_\alpha(\sqrt{12}x), \quad \alpha \in \N^n  .
\end{equation}
In particular, if $A\equiv1$, then the Schwartz kernel $K_1(t)$ of $ {\mathcal E}_1 (t) $ is
\begin{equation}\label{e:kernel-first-term}
	K_1(t,x,v,y,w)
	= \frac{ 3^{\frac{n}{2}}}{(2 \pi)^n} \bigg( \frac{\beta m}{\gamma t^2} \bigg)^n e^{- \frac{3 \beta m}{\gamma t^3} | x - y - t \frac{v+w}{2} |^2} e^{- \frac{\beta m}{4 \gamma t} | v - w + \frac{t}{m} \nabla V |^2 }  
	e^{- \frac{\gamma m \beta}{4} t  ( | v + \frac{t}{2m} \nabla V |^2 + \frac{1}{3} | \frac{t}{2m} \nabla V |^2 )} . 
\end{equation}
\end{lemm}

Note that $G_\alpha$ and $\tilde{G}_\alpha$ are explicitly computable Gaussians times Hermite functions. 

\begin{proof}
We recall the definition of $\mathcal{E}_A(t)$ in~\eqref{parametrixexplicitee},
with the phase $\psi$ defined in~\eqref{defImpsi}--\eqref{defRepsi}--\eqref{wtdef}.
This can be rewritten as
\[
	\mathcal{E}_A(t) = \mathcal{F}^{-1} e^{i\check\psi} A \mathcal{F} , 
\]
with $w_t = \frac{t}{m} \nabla V (x)$, $\mbox{Im}(\check\psi) =   \mbox{Im}(\psi)$ in~\eqref{defImpsi} and 
\[
	\mbox{Re}(\check\psi) = \bigg(-tv-\frac{t^2}{2m} (\nabla V)(x)\bigg) \cdot \xi + w_t \cdot \eta  
	= \bigg(- t v  - \frac{t}{2} w_t(x)\bigg) \cdot \xi + w_t \cdot \eta  .
\]
As a consequence,
\begin{align}\label{e:comput-kernel-1}
	\big(\mathcal{E}_A(t)u\big)(x,v)
	& = \mathcal{F}^{-1}_{(\xi,\eta)\to(x,v)}\big(  e^{i (- t v  - \frac{t}{2} w_t) \cdot \xi +i w_t \cdot \eta} \nonumber \\
	& \qquad \times e^{-\frac{\gamma}{\beta m} t ( | \eta - \frac{t \xi}{2} |^2 + \frac{1}{3} | \frac{t \xi}{2} |^2 ) - \frac{\gamma m \beta}{4} t ( | v + \frac{w_t}{2} |^2 + \frac{1}{3} | \frac{w_t}{2} |^2 )} A(\mathcal{F}u)(\xi,\eta)\big) \nonumber\\
	&  = e^{ - \frac{\gamma m \beta}{4} t ( | v + \frac{w_t}{2} |^2 + \frac{1}{3} | \frac{w_t}{2} |^2 )}\nonumber \\
	& \qquad\times\mathcal{F}^{-1}_{(\xi,\eta)\to(x,v)}\big(  e^{i (- t v  -\frac{t}{2} w_t) \cdot \xi +i w_t \cdot \eta}
	e^{-\frac{\gamma}{\beta m} t ( | \eta - \frac{t \xi}{2} |^2 + \frac{1}{3} | \frac{t \xi}{2} |^2 )} A(\mathcal{F}u)(\xi,\eta)\big).
\end{align}
Recalling the expression of $A$ in~\eqref{e:expansion-Balphabeta}, and using linearity, we have  $\mathcal{E}_A(t) = \sum_{|\alpha| + |\beta| + \delta \leq d} \mathcal{E}_{A_{\alpha \beta}}(t)$ and we are left to computing the kernel of $\mathcal{E}_{A_{\alpha \beta}}(t)$ in~\eqref{e:expansion-Balphabeta}.
The computation~\eqref{e:comput-kernel-1} now yields
\begin{equation}\label{e:decomposition-linearite}
	\big(\mathcal{E}_{A_{\alpha \beta}}(t)u\big)(x,v)
	= e^{ - \frac{\gamma m \beta}{4} t ( | v + \frac{w_t}{2} |^2 + \frac{1}{3} | \frac{w_t}{2} |^2 )}a_{\alpha \beta} (t,x,v)\big( T^{\alpha\beta}(t)u\big)(x,v),
\end{equation}
where
\[
	\big( T^{\alpha\beta}(t)u \big)(x,v)  :=  \mathcal{F}^{-1}_{(\xi,\eta)\to(x,v)}\big(  e^{i (- t v  -\frac{t}{2} w_t) \cdot \xi + i w_t \cdot \eta}
	\big( t^{\frac{3}{2}} \xi \big)^{\alpha}  \big(t^{\frac{1}{2}} \eta \big)^{\beta}
	e^{-\frac{\gamma}{\beta m} t ( | \eta - \frac{t \xi}{2} |^2 + \frac{1}{3} | \frac{t \xi}{2} |^2 )} \widehat{u}(\xi,\eta)\big).
\]
The kernel $K^{\alpha\beta}_t$ of $T^{\alpha\beta}(t)$ is given by
\[
	K^{\alpha\beta}_t(x,v,y,w) = \frac{1}{(2\pi)^{2n} }\iint_{\mathbb R^{2n}} e^{i (x-y- t v  -\frac{t}{2} w_t) \cdot \xi + i(v-w +w_t) \cdot \eta}
	\big( t^{\frac{3}{2}} \xi \big)^{\alpha}  \big(t^{\frac{1}{2}} \eta \big)^{\beta}
	e^{-\frac{\gamma}{\beta m} t ( | \eta - \frac{t \xi}{2} |^2 + \frac{1}{3} | \frac{t \xi}{2} |^2 )}\, d\xi d\eta.
\]
Setting $(\xi,\zeta):=(\xi,\eta-\frac{t\xi}{2})$ and changing variables, we obtain 
\begin{multline*}
	K^{\alpha\beta}_t(x,v,y,w) = \frac{1}{(2\pi)^{2n}}\iint_{\mathbb R^{2n}} 
	e^{i (x-y- t v  -\frac{t}{2} w_t) \cdot \xi + i(v-w +w_t) \cdot (\zeta+\frac{t\xi}{2})} \\ 
	\times\big( t^{\frac{3}{2}} \xi \big)^{\alpha}  \bigg( t^{\frac{1}{2}}\zeta+\frac{t^{\frac{3}{2}}\xi}{2}\bigg)^{\beta}
	e^{-\frac{\gamma}{\beta m} t ( | \zeta |^2 + \frac{1}{3} | \frac{t \xi}{2} |^2 )}\, d\xi d\zeta. 
\end{multline*}
Now, we first notice that
\[
	\bigg(x-y- t v  -\frac{t}{2} w_t\bigg) \cdot \xi + (v-w +w_t) \cdot \bigg(\zeta+\frac{t\xi}{2}\bigg)
	= \bigg(x-y -\frac{t}{2}(v+w)\bigg) \cdot \xi +(v-w +w_t) \cdot \zeta 
\]
Secondly, expanding 
\[
	\bigg( t^{\frac{1}{2}}\zeta+\frac{t^{\frac{3}{2}}\xi}{2}\bigg)^{\beta} = \sum_{\kappa \leq \beta}\binom{\beta}{\kappa}\big(t^{\frac{1}{2}}\zeta\big)^{\beta-\kappa}2^{-|\kappa|}\big(t^{\frac{3}{2}}\xi\big)^{\kappa} ,
\]
we finally obtain 
\begin{align*}
	&\ (2\pi)^{2n}   K^{\alpha\beta}_t(x,v,y,w)  \\
	= &\ \sum_{\kappa \leq \beta} \binom{\beta}{\kappa}2^{-|\kappa|}\iint_{\mathbb R^{2n}}e^{i ( x-y- \frac{t}{2}( v+w )) \cdot \xi + i(v-w +w_t) \cdot \zeta}    \big( t^{\frac{3}{2}} \xi \big)^{\alpha+\kappa} 
	\big(t^{\frac{1}{2}}\zeta\big)^{\beta-\kappa} e^{-\frac{\gamma}{\beta m} t ( | \zeta |^2 + \frac{1}{3} | \frac{t \xi}{2} |^2 )}\,d\xi d\zeta \\
	= &\ \sum_{\kappa \leq \beta} \binom{\beta}{\kappa}2^{-|\kappa|}\bigg(\int_{\mathbb R^n}e^{i ( x-y- \frac{t}{2}( v+w ) )\cdot \xi }\big( t^{\frac{3}{2}} \xi \big)^{\alpha+\kappa}
	e^{-\frac{\gamma}{\beta m} \frac{t}{3} | \frac{t \xi}{2} |^2}\,d\xi\bigg)
	\bigg(\int_{\mathbb R^n} e^{i(v-w +w_t) \cdot \zeta} \big(t^{\frac{1}{2}}\zeta\big)^{\beta-\kappa}e^{-\frac{\gamma}{\beta m} t | \zeta |^2}\,d\zeta\bigg).
\end{align*}
Recalling the expression of $G_\alpha ,\tilde{G}_\alpha$  in~\eqref{e:def-Galpha} and noticing that 
\[
	\tilde{G}_\alpha(x) = (12)^{\frac{n+|\alpha|}{2}}G_\alpha(\sqrt{12}x) = \frac{1}{(2\pi)^n}\int_{\mathbb R^n} e^{ix\cdot \xi} \xi^\alpha e^{-\frac{\gamma}{\beta m}\frac{1}{12}|\xi|^2}\, d\xi  ,
\]
 we have now obtained (after rescaling changes of variables)
\[
	K^{\alpha\beta}_t(x,v,y,w)   = \sum_{\kappa \leq \beta} \binom{\beta}{\kappa}2^{-|\kappa|}
	\frac{1}{t^{\frac{3n}{2}}} \tilde{G}_{\alpha+\kappa}\bigg( \frac{x-y-\frac{t}{2}(v+w)}{t^\frac32} \bigg)
	\frac{1}{t^{\frac{n}{2}}}G_{\beta-\kappa}\bigg(\frac{v-w+w_t}{t^\frac12} \bigg).
\]
Coming back to~\eqref{e:decomposition-linearite}, this implies~\eqref{e:kernel-B1}--\eqref{e:kernel-B2}.

Finally, in case $A\equiv1$, the expression~\eqref{e:kernel-first-term} is a consequence of the fact that 
\[
	G_0(x) = \frac{1}{(2\pi)^{\frac{n}{2}}}\bigg( \frac{\beta m}{2\gamma}\bigg)^{\frac{n}{2}} e^{-\frac{\beta m}{4\gamma}|x|^2} \quad \text{ and } \quad 
	\tilde{G}_0(x) = \frac{3^\frac{n}{2}2^n}{(2\pi)^{\frac{n}{2}}}\bigg( \frac{\beta m}{2\gamma}\bigg)^{\frac{n}{2}} e^{-3\frac{\beta m}{\gamma}|x|^2}.
\]
\end{proof}

\section{Approximate polar decomposition and \texorpdfstring{$L^p$}{} boundedness} \label{sectionLpboundedness}

In this section, we study general operators under the form~\eqref{parametrixexplicitee}--\eqref{formedunoyaustandard}. We give an approximate polar decomposition of such operators and deduce $ L^p $ operator bounds, which are eventually used to prove the smoothing and localization estimates from Theorem \ref{theoremeLL}.

\subsection{Approximate polar decomposition}

Let us begin by proving Theorem \ref{thm:polardecompoleger} on the approximate polar decomposition of the operators $\mathcal E_A(t)$ (recall Remark \ref{rema:polar}). To that end, we recall the map $ F_t : {\mathbb R}^{2n} \rightarrow {\mathbb R}^{2n} $ defined by~\eqref{eq:functionFt} is given for every $(x,v)\in\mathbb R^{2n}$ by
\begin{equation}\label{e:def-Ft} 
	F_t (x,v) : =  \bigg( x - t v - \frac{t^2}{2m} \nabla V (x) , v + \frac{t}{m} \nabla V (x)  \bigg).
\end{equation}
We will prove in Proposition~\ref{lemmedediffeo} below that, under the assumption that  $V$ satisfies~\eqref{hypothesepotentiel}, the map $F_t$ is a diffeomorphism for $ |t| \leq t_0  $ small enough with Jacobian $ 1 +\mathcal O (t) $. Recall also that the composition operator ${\mathcal I}_{F_t} $ is defined by~\eqref{eq:compounit} and that the standard quantization $\Op$ is the one in~\eqref{e:classical-quantization}. Precisely, we aim at proving the following result, which contains Theorem \ref{thm:polardecompoleger} and is slightly more precise.

\begin{theo} \label{propositionpolaire}
Assume that $V$ satisfies~\eqref{hypothesepotentiel}. There exists $t_0\in(0,1)$ such that for every $t\in(0,t_0)$ and $ A \in {\mathcal A}^{\nu,d} $,
\begin{equation}\label{approx0}
	{\mathcal E}_A (t) = {\mathcal I}_{F_t} \Op (a_t),
\end{equation} 
where $ \Op (a_t) = a_t (x,v,D_x,D_v) $ is the pseudodifferential operator with symbol
\begin{equation}\label{formedelacompositionOIF} 
	a_t (x,v,\xi,\eta) = e^{- \Im\psi (t,\tilde{x}_t,\tilde{v}_t,\xi,\eta)} A (t,\tilde{x}_t,\tilde{v}_t,\xi,\eta),
\end{equation} 
where $ (\tilde{x}_t , \tilde{v}_t ) = F_t^{-1} (x,v) $. This symbol satisfies, for each $ N > 0 $ and multiindices $ \alpha,\beta,\gamma,\delta\in\mathbb N^n$, 
\begin{align}
	\big| \partial_x^{\alpha} \partial_v^{\beta} \partial_{\xi}^{\gamma} \partial_{\eta}^{\delta} a_t (x,v,\xi,\eta) \big| 
	& \lesssim t^{\nu + \frac{|\beta|+|\delta|+3 |\gamma|}{2}  } \left(1 + \big| t^{\frac{3}{2}} \nabla V (\tilde{x}_t) \big| + \big| t^{\frac{1}{2}} \tilde{v}_t \big| + \big| t^{\frac{3}{2}} \xi \big| + \big| t^{\frac{1}{2}} \eta \big| \right)^{-N}   \label{avecomposition}  \\
	& \lesssim t^{\nu + \frac{|\beta|+|\delta|+3 |\gamma|}{2}  } \left(1 + \big| t^{\frac{3}{2}} \nabla V (x) \big| + \big| t^{\frac{1}{2}} v \big| + \big| t^{\frac{3}{2}} \xi \big| + \big| t^{\frac{1}{2}} \eta \big| \right)^{-N}  \label{sanscomposition}             
\end{align} 
for all $ t \in (0,t_0) $ and $ (x,v,\xi,\eta) \in {\mathbb R}^{4n} $.
\end{theo}

Before proving Theorem \ref{propositionpolaire}, we need to establish preliminary estimates on the map $F_t$.

\begin{prop} \label{lemmedediffeo}
Assume that $V$ satisfies~\eqref{hypothesepotentiel}. There exist $C_0,t_0>0$ such that for all $t \in [-t_0,t_0]$, the map $F_t:\R^{2n}\to\R^{2n}$ defined by~\eqref{e:def-Ft} is a diffeomorphism from $ {\mathbb R}^{2n} $ onto itself with Jacobian satisfying 
\[
	\| \mathrm{det}(dF_t) -1 \|_{L^\infty} \leq C_0 |t|,\quad t\in[-t_0,t_0].
\]
In addition, if we denote anew by $ (\tilde{x}_t,\tilde{v}_t) := F_t^{-1} (x,v) $ for $ |t| \leq t_0 $, then given $ |\alpha + \beta| \geq 1 $, there is $C_{\alpha, \beta}>0$ such that  for $ (x,v) \in {\mathbb R}^{2n} $ and $ |t| \leq t_0 $,
\begin{equation}\label{estimeemax}
	\big| \partial_x^{\alpha} \partial_v^{\beta} \big( \tilde{x}_t - x , \tilde{v}_t - v \big) \big| \leq C_{\alpha, \beta} t^{\max(1,|\beta|)}.
\end{equation}
Moreover, there exists $ C>0 $ such that for all $ (x,v) \in {\mathbb R}^{2n} $ and $ 0 < t \leq t_0 $,
\begin{equation}\label{equivalencedenorme0}
	\frac{1}{C}  \big( \big| t^{\frac{1}{2}} \tilde{v}_t \big|  + \big|  t^{\frac{3}{2}} \nabla V (\tilde{x}_t) \big| \big)  
	\leq  \big|t^{\frac{1}{2}} v\big| +  \big|t^{\frac{3}{2}} \nabla V (x) \big|  
	\leq C \big( \big| t^{\frac{1}{2}} \tilde{v}_t \big|  + \big| t^{\frac{3}{2}} \nabla V (\tilde{x}_t)  \big| \big).
\end{equation} 
\end{prop}

\begin{proof} The Jacobian matrix of $F_t$ is  
\begin{equation}\label{formeJacobienneF} 
	dF_t (x,v)= \left( \begin{matrix} I - \frac{t^2}{2m} \partial_x^2 V(x) & - t I \\  \frac{t}{m} \partial_x^2 V(x)  & I  \end{matrix} \right),
\end{equation}             
and satisfies $\|dF_t  -  I_{{\mathbb R}^{2n}}  \|_{L^\infty} \leq C |t| \|\partial_x^2 V \|_{L^\infty} $. Thus, from~\eqref{hypothesepotentiel}, $ F_t - I $ is  a contraction on $ {\mathbb R}^{2n} $ if $ t $ is small enough so it is injective and one sees that $ F_t ({\mathbb R}^{2n}) = {\mathbb R}^{2n} $ by solving, for any $ (y,w) \in {\mathbb R}^{2n} $, the equation $ F_t (x,v) = (y,w) $ as the fixed point problem
\[
	(x,v) =  (y,w) - ( F_t - I ) (x,v).
\]
To estimate the derivatives of $ F_t^{-1} $, we proceed by induction on the order of differentiation, by differentiating the identity $ F_t \circ F_t^{-1} = I $. At order one, $ dF_t^{-1} = (d F_t )^{-1}|_{F_t^{-1}} = I + \mathcal O (t)$  yields (\ref{estimeemax}) when $ |\alpha + \beta| = 1 $. To study   higher order derivatives, let us denote by $ (X_t,V_t) $ the components of $ F_t $. Then, for $ |\alpha + \beta| \geq 2 $,
\[
	\begin{pmatrix}  0 \\ 0 \end{pmatrix} 
	= \begin{pmatrix} \partial_x X_t & \partial_v X_t \\ \partial_x V_t & \partial_v V_t \end{pmatrix}\bigg\vert_{(\tilde{x}_t,\tilde{v}_t)}  \left( \begin{matrix} \partial_x^{\alpha} \partial_v^{\beta} \tilde{x}_t \\ \partial_x^{\alpha} \partial_v^{\beta} \tilde{v}_t \end{matrix} \right) 
	+ \left( \begin{matrix} \hat{x}_{\alpha \beta} \\ \hat{v}_{\alpha \beta} \end{matrix} \right),
\]
where $ \hat{x}_{\alpha \beta} $ is a linear combination of
\begin{equation}
	\big( d_x^j d_v^k X_t \big) \big( \partial_{x}^{\hat{\alpha}_1} \partial_v^{\hat{\beta}_1} \tilde{x}_t, \cdots ,  \partial_{x}^{\hat{\alpha}_j} \partial_v^{\hat{\beta}_j} \tilde{x}_t ,  \partial_x^{\check{\alpha}_1} \partial_v^{\check{\beta}_1} {v}_t   , \cdots ,\partial_x^{\check{\alpha}_k} \partial_v^{\check{\beta}_k} {v}_t \big) \label{hatx}
\end{equation} 
and $ \hat{v}_{\alpha \beta} $ of
\begin{equation}
	\big( d_x^j d_v^k V_t \big) \big( \partial_{x}^{\overline{\alpha}_1} \partial_v^{\overline{\beta}_1} \tilde{x}_t, \cdots ,  \partial_{x}^{\overline{\alpha}_j} \partial_v^{\overline{\beta}_j} \tilde{x}_t ,  \partial_x^{\tilde{\alpha}_1} \partial_v^{\tilde{\beta}_1} {v}_t   , \cdots ,\partial_x^{\tilde{\alpha}_k} \partial_v^{\tilde{\beta}_k} {v}_t \big), \label{hatv}
\end{equation} 
where $ j + k \geq 2 $, so that in particular
\[
	d_x^j d_v^k X_t = \mathcal O (t^{\max(1,k)}), \qquad d_x^j d_v^k V_t = \mathcal O (t^{\max(1,k)}),
\]
and where
\begin{align*}
	(\hat{\alpha}_1 , \hat{\beta}_1 ) +  \cdots  + (\hat{\alpha}_j ,\hat{\beta}_j ) + ( \check{\alpha}_1 , \check{\beta}_1) +  \cdots + ( \check{\alpha}_k , \check{\beta}_k ) & = (\alpha , \beta), \\
	(\overline{\alpha}_1 , \overline{\beta}_1 ) + \cdots + (\overline{\alpha}_j , \overline{\beta}_j ) +  ( \tilde{\alpha}_1 , \tilde{\beta}_1) + \cdots + ( \tilde{\alpha}_k, \tilde{\beta}_k ) & = (\alpha , \beta),
\end{align*}
all pairs having length at least $1$ so all of them (to the left) having length strictly less than $ |\alpha + \beta| $. By the induction assumption, we have the bounds
\[
	\big| \partial_{x}^{\hat{\alpha}} \partial_v^{\hat{\beta}} \tilde{x}_t \big| \lesssim |t|^{|\hat{\beta}|}, \qquad  
	\big| \partial_{x}^{\hat{\alpha}} \partial_v^{\hat{\beta}} \tilde{v}_t  \big| \lesssim |t|^{\max(1,|\hat{\beta}|)-1} \lesssim |t|^{|\hat{\beta}|-1},
\]
for $ 1 \leq |\hat{\alpha}+\hat{\beta}| < |\alpha + \beta| $. All this implies on the one hand that both (\ref{hatx}) and  (\ref{hatv}) are $\mathcal O (t) $
and on the other hand that
\begin{align*}
	|(\ref{hatx})| & \lesssim |t|^{k+ |\hat{\beta}_1| + \cdots + |\hat{\beta}_j| + (|\check{\beta}_1|-1) + \cdots + (|\check{\beta}_k| - 1)} = t^{|\beta|}, \\
	|(\ref{hatv})| & \lesssim |t|^{k+ |\overline{\beta}_1| + \cdots + |\overline{\beta}_j| + (|\tilde{\beta}_1|-1) + \cdots + (|\tilde{\beta}_k| - 
 1)} = t^{|\beta|},
\end{align*}
which yields the result.   To prove (\ref{equivalencedenorme0}), it suffices to observe that $ F_t (\tilde{x}_t, \tilde{v}_t) = (x,v)$ yields
\[
	x = \tilde{x}_t - t \tilde{v}_t - \frac{t^2}{2m} \nabla V (\tilde{x}_t), \qquad t^{\frac{1}{2}} v = t^{\frac{1}{2}} \tilde{v}_t + \frac{t^{\frac{3}{2}}}{m} \nabla V (\tilde{x}_t),
\]
where the first equality, together with the boundedness of $ \nabla^2 V $ (see (\ref{hypothesepotentiel})) and the Taylor formula, shows that
\begin{align*}
	\nabla V (x) & = \nabla V (\tilde{x}_t) + \mathcal O\left(t^{\frac{1}{2}} (|t^{\frac{1}{2}} \tilde{v}_t| + |t^{\frac{3}{2}} \nabla V(\tilde{x}_t)|) \right) \\
	& = \nabla V (\tilde{x}_t) + \mathcal O\left(t^{\frac{1}{2}} \left( \Big|t^{\frac{1}{2}} \tilde{v}_t + \frac{t^{\frac{3}{2}}}{m}\nabla V (\tilde{x}_t) \Big| + |t^{\frac{3}{2}} \nabla V(\tilde{x}_t)| \right) \right) .
\end{align*}
From this we infer, for $t$ small enough, the analogue of (\ref{equivalencedenorme0}) where $ t^{\frac{1}{2}} \tilde{v}_t $ is replaced by $ t^{\frac{1}{2}} \tilde{v}_t + \frac{t^{\frac{3}{2}}}{m}\nabla V (\tilde{x}_t) $, which itself implies  (\ref{equivalencedenorme0}).
\end{proof}

\begin{proof}[Proof of Theorem \ref{propositionpolaire}] The formula (\ref{approx0}) is a direct consequence of the fact that (recalling the quantization in~\eqref{e:classical-quantization}) the Schwartz kernel of $\Op(a_t) $ is
\[
	( 2 \pi)^{-2n} \iint_{\mathbb R^{2n}} e^{  i(x- y ) \cdot \xi + i (v - w )\cdot \eta} a_t (x,v,\xi,\eta)\, d \xi d \eta
\]
and that $ \Re(\psi) = F_t (x,v) \cdot (\xi,\eta) $ (for the definition of $ \psi $, see (\ref{defImpsi}) and (\ref{defRepsi})). To prove the estimates on the symbol, we observe that $  \partial_x^{\alpha} \partial_v^{\beta} \partial_{\xi}^{\gamma} \partial_{\eta}^{\delta} a_t $ is a linear combination of products of $ e^{-\Im(\tilde{x}_t,\tilde{v}_t,\xi,\eta)} $ and the following factors
\begin{align}
	\prod_{j=1}^{J} \big(\partial_x^{\Omega^j} \partial_v^{\Lambda^j} \partial_{\xi}^{\gamma^{\prime}_j} \partial_{\eta}^{\delta^{\prime}_j}  \Im (\psi) \big) (\tilde{x}_t,\tilde{v}_t,\xi,\eta) \prod_{m=1}^n \prod_k (\partial_x^{\alpha^{\prime}_{k,\Omega^j_m}} \partial_v^{\beta^{\prime}_{k,\Omega^j_m}}  \tilde{x}_{t,m} )  \prod_{\ell} ( \partial_x^{\alpha^{\prime}_{\ell,\Lambda^j_m}} \partial_v^{\beta^{\prime}_{\ell,\Lambda^j_m}} \tilde{v}_{t,m} ),  \label{1erfacteursymbole}  \\
	(\partial_x^{\Omega} \partial_v^{\Lambda} \partial_{\xi}^{\gamma^{\prime \prime}} \partial_{\eta}^{\delta^{\prime \prime}} A)(\tilde{x}_t,\tilde{v}_t,\xi,\eta) \prod_{m=1}^n \prod_K (\partial_x^{\alpha^{\prime \prime}_{K,\Omega_m}} \partial_v^{\beta^{\prime \prime}_{K,\Omega_m}}  \tilde{x}_{t,m} )  \prod_L ( \partial_x^{\alpha^{\prime \prime}_{L,\Lambda_m}} \partial_v^{\beta^{\prime \prime}_{L,\Lambda_m}} \tilde{v}_{t,m} ), \label{2emefacteur}                     
\end{align} 
where
\begin{align*}
	0 \leq J \leq |\alpha + \beta + \gamma + \delta| , \qquad
	1 \leq \big| \Omega^1 + \Lambda^1 + \cdots + \Omega^J + \Lambda^J + \Omega + \Lambda \big| \leq |\alpha + \beta|,          
\end{align*}
\[
	1 \leq k \leq \Omega_m^j , \qquad 1 \leq \ell \leq \Lambda_m^j , \qquad 1 \leq K \leq \Omega_m , \qquad 1 \leq L \leq \Lambda_m,
\]
(everywhere, if $ \kappa $ is an index, $ \prod_{\kappa = 1}^0 $ is equal to $1$) and
\begin{align*}
	& \alpha = \sum_{j=1}^J \sum_{m=1}^n \left( \sum_{k=1}^{\Omega_m^j}  \alpha^{\prime}_{k,\Omega_m^j} + \sum_{\ell = 1}^{\Lambda_m ^j} \alpha^{\prime}_{\ell, \Lambda_m^j} + \sum_{K=1}^{\Omega_m} \alpha^{\prime \prime}_{K,\Omega_m} + \sum_{L=1}^{\Lambda_m} \alpha^{\prime \prime}_{L,\Lambda_m}  \right), \\
	& \beta = \sum_{j=1}^J \sum_{m=1}^n \left( \sum_{k=1}^{\Omega_m^j}  \beta^{\prime}_{k,\Omega_m^j} + \sum_{\ell = 1}^{\Lambda_m ^j} \beta^{\prime}_{\ell, \Lambda_m^j} + \sum_{K=1}^{\Omega_m} \beta^{\prime \prime}_{K,\Omega_m} + \sum_{L=1}^{\Lambda_m} \beta^{\prime \prime}_{L,\Lambda_m}  \right), \\
	& \gamma = \gamma^{\prime \prime} + \sum_{j=1}^J \gamma_j^{\prime}   , \qquad \delta = \delta^{\prime \prime}  + \sum_{j=1}^J \delta^{\prime}_j .
\end{align*} 
By Definition \ref{definitionvaluation}, we have the estimate
\[
	\big| \partial_x^{\Omega} \partial_v^{\Lambda} \partial_{\xi}^{\gamma^{\prime \prime}} \partial_{\eta}^{\delta^{\prime \prime}} A (t,x,v,\xi,\eta) \big| 
	\lesssim t^{\nu + \frac{|\Lambda|+ 3 |\gamma^{\prime \prime}| + |\delta^{\prime \prime}|}{2}}  \big(1 + |t^{\frac{1}{2}}v| + |t^{\frac{1}{2}} \eta| + |t^{\frac{3}{2}} \xi| + |t^{\frac{3}{2}} \nabla V (x)| \big)^d,
\]
and similarly, from (\ref{defImpsi}),
\[
	\big| \partial_x^{\Omega^j} \partial_v^{\Lambda^j} \partial_{\xi}^{\gamma_j^{\prime}} \partial_{\eta}^{\delta_j^{\prime}} \Im \psi(t,x,v,\xi,\eta) \big| 
	\lesssim  t^{ \frac{3 \min(1,|\Omega^j|)+ |\Lambda^j|+ 3 |\gamma^{\prime}_j| + |\delta^{\prime }_j|}{2}} \big(1 + |t^{\frac{1}{2}}v| + |t^{\frac{1}{2}} \eta| + |t^{\frac{3}{2}} \xi| + |t^{\frac{3}{2}} \nabla V (x)| \big)^2.
\]
By Proposition \ref{lemmedediffeo}, if $ |\alpha + \beta| \geq 1 $, we have
\[
	|\partial_x^{\alpha} \partial_v^{\beta} \tilde{x}_t| \lesssim t^{|\beta|} \lesssim t^{\frac{|\beta|}{2}}, \qquad  |\partial_x^{\alpha} \partial_v^{\beta} \tilde{v}_t| \lesssim t^{|\beta|} \lesssim t^{\frac{|\beta|-1}{2}} .
\]
All this shows that $ (\ref{1erfacteursymbole}) \times (\ref{2emefacteur}) $ is bounded by
\[
	\big( 1 + |t^{\frac{1}{2}} \tilde{v}_t| + |t^{\frac{1}{2}} \eta| + |t^{\frac{3}{2}}\xi| + |t^{\frac{3}{2}}\nabla V (\tilde{x}_t)| \big)^{2 J + d}
\]
times
\[
	t^{\nu + \frac{3 |\gamma|+ |\delta|}{2} + \frac{ |\Lambda| + \sum_j |\Lambda^j|}{2} + \frac{1}{2} \left( \sum |\beta^{\prime}_{k,\Omega_j}| + \sum |\beta^{\prime \prime}_{K,\Omega_m}| + \sum( |\beta^{\prime}_{\ell,\Lambda_m^j}|-1 ) + \sum (|\beta^{\prime \prime}_{L,\Lambda_m}|-1) \right) } 
	= \mathcal O (t^{\nu + \frac{|\delta|+ 3|\gamma|+ |\beta|}{2}}).
\]
Using the exponential decay of $ \exp (- \Im(\psi)(t,\tilde{x}_t,\tilde{v}_t,\xi,\eta) ) $, we get (\ref{avecomposition}).  Finally, using (\ref{equivalencedenorme0}), we infer  (\ref{sanscomposition}) from (\ref{avecomposition}). This completes the proof.
\end{proof}

\begin{rema} Consider the (semi)norms on  $ A \in {\mathcal A}^{\nu,d} $
\[
	\mathrm{N}_{\nu,d,N} (A) = \sup t^{-  3 \frac{|\gamma|}{2} - \frac{|\delta|}{2}  - \frac{|\beta|}{2}  - \nu} \Big( 1 + |t^{\frac{3}{2}} \nabla V (x)| + |t^{\frac{1}{2}}v| + |t^{\frac{3}{2}} \xi| + |t^{\frac{1}{2}} \eta| \Big)^{-d} \big| \partial_x^{\alpha} \partial_{v}^{\beta} \partial_{\xi}^{\gamma} \partial_{\eta}^{\delta} A \big|,
\]
the sup being taken over $ (t,x,v,\xi,\eta) \in (0,t_0] \times {\mathbb R}^{4n} $ and multiindices $ |\alpha + \beta + \gamma + \delta | \leq N $.
Similarly, we define seminorms on $t$ dependent symbols $a_t$ satisfying (\ref{sanscomposition}) by
\[
	\opnorm{a}_{\nu,N} = \sup t^{-  3 \frac{|\gamma|}{2} - \frac{ |\delta|}{2} - \frac{ |\beta|}{2}  - \nu} \Big( 1 + |t^{\frac{3}{2}} \nabla V (x)| + |t^{\frac{1}{2}}v| + |t^{\frac{3}{2}} \xi| + |t^{\frac{1}{2}} \eta| \Big)^{N} \big| \partial_x^{\alpha} \partial_{v}^{\beta} \partial_{\xi}^{\gamma} \partial_{\eta}^{\delta} a_t \big|,
\]
the sup being taken over $ (t,x,v,\xi,\eta) \in (0,t_0] \times {\mathbb R}^{2n} $ and multiindices $ |\alpha + \beta + \gamma + \delta | \leq N $. Then, the proof of Theorem \ref{propositionpolaire} shows actually that given $d \geq 0$, $ \nu \in {\mathbb R} $ and $ N \in {\mathbb N} $, there exists $ C > 0 $ and $ N^{\prime} \in {\mathbb N} $ such that
\begin{equation}
	\opnorm{a}_{\nu,N} \leq C {\mathrm N}_{\nu,d,N^{\prime}} (A) , \label{pourphasenonstationnaire4}              
\end{equation}
for all $ A \in {\mathcal A}^{\nu,d} $ and where $ a = (a_t)_{t \in (0,t_0]} $ is defined by (\ref{formedelacompositionOIF}).
\end{rema}

For further use, we emphasize the following simple (and standard) lemma on which all of our $ L^p $ estimates rely.

\begin{lemm} \label{lelemmequiutiliseSchur} Let $ \nu > 0 $. There exists $ C > 0 $ such that for all family of symbols  $ (a_t)_{t \in (0,t_0]} $ satisfying (\ref{sanscomposition}),    for all $ t \in (0,t_0] $ and all $ p \in [1,\infty] $,
\[
	\big\Vert\Op(a_t) \big\Vert_{L^p\rightarrow L^p} \leq C t^{\nu}\opnorm{a}_{\nu,2n+1}.
\]
\end{lemm}

Note in particular that the constant does not depend on $p$, though we shall not use it in the paper. 

\begin{proof}[Proof of Lemma  \ref{lelemmequiutiliseSchur}] Using only that
\[
	|\partial_{\xi}^{\alpha} \partial_{\eta}^{\beta} a_t (x,v,\xi,\eta)| \lesssim_{\alpha,\beta} t^{\nu + \frac{3 |\alpha|+|\beta|}{2}}\Big(1+|t^{\frac{3}{2}}\xi|+|t^{\frac{1}{2}} \eta|\Big)^{-2n-1} ,
\]
the lemma is a consequence of the Schur test under the form of Proposition \ref{referencesection} with $ L (\xi,\eta) = ( t^{\frac{3}{2}} \xi , t^{\frac{1}{2}} \eta ) $ to get uniform control of the estimates with respect to $t$.
\end{proof}

\subsection{Quantitative study of the FIO}

In the next proposition, we check  that the Schwartz space is  stable by operators of the form $ {\mathcal E}_A (t) $ (as is customary for FIO) but also give bounds that are uniform in $t$. This can be useful to show that certain integrals in $t$ of such operators also preserve the Schwartz space, like the parametrix in Section \ref{s:parametrix-of-resolvent}.
 
 \begin{prop} \label{stabiliteSchwartz} Let  $ A \in {\mathcal A}^{0,d} $ for some $ d \geq 0 $. Then, there exists $ t_0 > 0 $ such that for all seminorm $ {\mathcal N} $ of the Schwartz space $ {\mathcal S}({\mathbb R}^{2n}) $, there exists $ C > 0 $ and another seminorm $ {\mathcal N}^{\prime} $ such that  for all $ t \in [0,t_0] $ and all $ u \in {\mathcal S}({\mathbb R}^{2n}) $,
\begin{equation} \label{seminormesfonctions}
	 {\mathcal N} \big( {\mathcal E}_A (t) u \big) \leq C {\mathcal N}^{\prime} (u).
 \end{equation}
 \end{prop}
 
\begin{proof} We write $ {\mathcal E}_A (t) $ as (\ref{approx0}). Elementary calculations using the chain rule and basic commutation formulas of pseudodifferential operators with $ \partial_x $ and $ \partial_v $ show that
\begin{align*}
	\partial_{v_j} {\mathcal I}_{F_t} \Op (a_t) & = {\mathcal I}_{F_t} \big( \partial_{v_j} - t \partial_{x_j} \big) \Op (a_t) \\
	& = {\mathcal I}_{F_t} \Op \big( (\partial_{v_j} - t \partial_{x_j}) a_t \big) + {\mathcal I}_{F_t} \Op (a_t) (\partial_{v_j} - t \partial_{x_j}) .
\end{align*}
Likewise,
\begin{align*}   
	\partial_{x_j}   {\mathcal I}_{F_t} \Op (a_t)  & = {\mathcal I}_{F_t} \left( \partial_{x_j} +  \frac{t^{\frac{1}{2}}}{m} (\partial_j \nabla V) (\tilde{x}_t) \cdot \left( t^{\frac{1}{2}} \partial_v - \frac{1}{2} t^{\frac{3}{2}} \partial_x \right) \right) \Op (a_t) \\
	& = {\mathcal I}_{F_t} \Op (\partial_{x_j} a_t) + {\mathcal I}_{F_t} \Op (a_t) \partial_{x_j} \\
	& \qquad + \frac{t^{\frac{1}{2}}}{m} {\mathcal I}_{F_t}  \Op \left(  (\partial_j \nabla V)(\tilde{x}_t) \cdot \left( t^{\frac{1}{2}} \partial_v a_t - \frac{1}{2} t^{\frac{3}{2}} \partial_x a_t \right) \right) \\
	& \qquad + \frac{t^{\frac{1}{2}}}{m} {\mathcal I}_{F_t} \Op \big( a_t (\partial_j \nabla V)(\tilde{x}_t) \big) \cdot   \bigg( t^{\frac{1}{2}} \partial_v - \frac{1}{2} t^{\frac{3}{2}} \partial_x \bigg),
\end{align*}
where $ \partial_x^{\alpha} \partial_{v}^{\beta}( \partial_j \nabla V) (\tilde{x}_t) = \mathcal O (t^{|\beta|}) $ by (\ref{hypothesepotentiel}) and (\ref{estimeemax}). This implies that for any multiindices $ \alpha , \beta $,
\[
	\partial_x^{\alpha} \partial_{v}^{\beta} {\mathcal I}_{F_t} \Op (a_t) = \sum_{|\hat{\alpha} + \hat{\beta} | \leq |\alpha + \beta|} {\mathcal I}_{F_t} \Op (a_{\hat{\alpha} \hat{\beta},t}) \partial_x^{\hat{\alpha}} \partial_v^{\hat{\beta}}, 
\]
with symbols $ a_{\hat{\alpha} \hat{\beta},t} $ satisfying estimates similar to those on  $a_t$ in (\ref{sanscomposition}). Regarding the multiplications by (multi)powers of $ (x,v) $ we use that
\begin{align*}
	x_j {\mathcal I}_{F_t} \Op (a_t) & = {\mathcal I}_{F_t} \tilde{x}_{t,j}  \Op (a_t) \\
	& = {\mathcal I}_{F_t} \frac{\tilde{x}_{t,j}}{1+|x|^2 + |v|^2} (1 + |x|^2 + |v|^2)  \Op (a_t),
\end{align*}
where $ \tilde{x}_{j,t} / (1+|x|^2+|v|^2) $ is bounded according to (\ref{estimeemax}), and where $ (1 + |x|^2 + |v|^2)  \Op (a_t)  $ reads
\[
	\Op (a_t) (1+|x|^2 + |v|^2) + 2 i \Op (\nabla_{\xi,\eta} a) \cdot (x,v) - \Op (\Delta_{\xi,\eta} a_t).
\]
The interest is both that $ \tilde{x}_{j,t}/ (1+|x|^2 + |v|^2) $ is bounded (by Proposition \ref{lemmedediffeo}) and that elementary properties of pseudodifferential operators allow to compute the commutator with $ |x|^2 $ and $ |v|^2 $ (commuting $ \tilde{x}_{j,t} $ would require symbolic calculus which we want to avoid here). A similar formula holds for $ v_j {\mathcal I}_{F_t} = {\mathcal I}_{F_t} \tilde{v}_{t,j} $ and all this allows to prove that
\[
	x^{\rho} v^{\delta} \partial_x^{\alpha} \partial_{v}^{\beta} {\mathcal I}_{F_t} \Op (a_t) 
	= \sum_{|\hat{\alpha} + \hat{\beta} | \leq |\alpha + \beta|,\,
|\hat{\rho}+ \hat{\delta}| \leq 2 |\rho + \delta|} {\mathcal I}_{F_t} b_{\hat{\rho}, \hat{\delta},t}(x,v) \Op (a_{\hat{\alpha}, \hat{\beta}, \hat{\rho}, \hat{\delta},t})  x^{\hat{\rho}} v^{\hat{\gamma}} \partial_x^{\hat{\alpha}} \partial_v^{\hat{\beta}},
\]
with bounded functions $ b_{\hat{\rho},\hat{\delta},t} $ (uniformly in $t$) and symbols $ a_{\hat{\alpha}, \hat{\beta}, \hat{\rho}, \hat{\delta},t} $ satisfying estimates like (\ref{sanscomposition}). Since $ \Op (a_{\hat{\alpha}, \hat{\beta}, \hat{\rho}, \hat{\delta},t}) $ is bounded on $ L^{\infty}(\mathbb R^{2n}) $, according to Lemma \ref{lelemmequiutiliseSchur}, as well as $ {\mathcal I}_{F_t} $, we infer directly the estimate (\ref{seminormesfonctions}).
\end{proof}
 
\begin{prop} \label{Lpboundedness} Given multiindices $ \alpha,\alpha^{\prime},\beta,\beta^{\prime},\gamma,\gamma^{\prime},\delta,\delta^{\prime}\in\mathbb N^n $, denote
\begin{align}
	& W_t = (t^{\frac{1}{2}} v)^{\alpha} (t^{\frac{3}{2}}\nabla V (x))^{\beta} (t^{\frac{1}{2}} \partial_v)^{\gamma} (t^{\frac{3}{2}} \partial_x)^{\delta} ,  \label{formedeWt4}  \\
	& W^{\prime }_t = (t^{\frac{1}{2}} \partial_v)^{\gamma^{\prime}} (t^{\frac{3}{2}} \partial_x)^{\delta^{\prime}}  (t^{\frac{1}{2}} v)^{\alpha^{\prime}} (t^{\frac{3}{2}}\nabla V (x))^{\beta^{\prime}} .  \nonumber
\end{align}
For all $ A \in {\mathcal A}^{\nu,d} $, there exists $ C > 0 $ such that for every $ p \in [1,\infty] $, $t\in(0,t_0] $ and $ u\in {\mathcal S}({\mathbb R}^{2n}) $,
\[
	\big\Vert W_t {\mathcal E}_A (t) W_t^{\prime} u \big\Vert_{L^p}   \leq C t^{\nu}\Vert u\Vert_{L^p}.
\]
\end{prop}

\begin{proof} We record first that, by elementary computations,
\[
\begin{array}{rcl}
	{\mathcal I}_{F_t}^{-1} \big(  t^{\frac{3}{2}} \partial_{x_j}  \big) {\mathcal I}_{F_t}\hspace{-7pt} & = &\hspace{-7pt} t^{\frac{3}{2}}\partial_{x_j} + \dfrac{t^2}{m} (\partial_j \nabla V)(\tilde{x}_t) \cdot \left( t^{\frac{1}{2}} \partial_v - \dfrac{t^{\frac{3}{2}}}{2} \partial_x \right) , \\
	{\mathcal I}_{F_t}^{-1} \big(  t^{\frac{1}{2}} \partial_{v_j}  \big) {\mathcal I}_{F_t}\hspace{-7pt}  & = &\hspace{-7pt} t^{\frac{1}{2}} \partial_{v_j} - t^{\frac{3}{2}} \partial_{x_j}  , \\[5pt]
	{\mathcal I}_{F_t}^{-1} \big(t^{\frac{1}{2}}  v_{j}  \big) {\mathcal I}_{F_t}\hspace{-7pt} & = &\hspace{-7pt} t^{\frac{1}{2}} \tilde{v}_{t,j}, \\[5pt]
	{\mathcal I}_{F_t}^{-1} \big(t^{\frac{3}{2}} \partial_j V (x) \big) {\mathcal I}_{F_t}\hspace{-7pt} & = &\hspace{-7pt} t^{\frac{3}{2}} (\partial_j V) (\tilde{x}_t) .
\end{array}
\]
This is similar to what we did in the proof of Proposition \ref{stabiliteSchwartz} up to the powers of $t$.
Now, using (\ref{approx0}), the above identities and the assumption (\ref{hypothesepotentiel}), one can write
\[
	W_t {\mathcal E}_A (t) W_t^{\prime}  = {\mathcal I}_{F_t} \widetilde{W}_t \Op (a_t) W^{\prime}_t
\]
where $ \widetilde{W}_t $ is a finite sum of differential operators of the form
\[
	b_t (x,v)  \big(t^{\frac{1}{2}} \tilde{v}_t \big)^{\tilde{\alpha}} \big( t^{\frac{3}{2}} \nabla V (\tilde{x}_t) \big)^{\tilde{\beta}} \big(t^{\frac{1}{2}} \partial_v \big)^{\tilde{\gamma}} \big( t^{\frac{3}{2}} \partial_x \big)^{\tilde{\delta}},
\]
where $ | b_t (x,v) | \lesssim 1 $ for $(x,v) \in {\mathbb R}^{2n}$ and $t \in [0,t_0]$. Since $ {\mathcal I}_{F_t} $ is bounded on each $ L^p ({\mathbb R}^{2n}) $ by Proposition~\ref{lemmedediffeo}, it suffices
to estimate  the operator norm of $ \widetilde{W}_t \Op (a_t) W_t^{\prime} $. To simplify the expression of this pseudodifferential operator without resorting to any non trivial composition theorem, we use the following elementary formulas
 \begin{align*}
	& \Op (b) t^{\frac{1}{2}} \partial_{v_j} =  i \Op \big(b t^{\frac{1}{2} }\eta_j \big), \qquad  t^{\frac{1}{2}} \partial_{v_j} \Op (b)  =  i \Op \big(b t^{\frac{1}{2}} \eta_j \big) + \Op (t^{\frac{1}{2}}\partial_{v_j} b), \\
	& \Op (b) t^{\frac{3}{2}} \partial_{x_j} = i \Op \big(b t^{\frac{3}{2} } \xi_j \big), \qquad  t^{\frac{3}{2}} \partial_{x_j} \Op (b)  =  i \Op \big(b t^{\frac{3}{2}} \xi_j \big) + \Op (t^{\frac{3}{2}}\partial_{x_j} b), \\
	& t^{\frac{1}{2}} \tilde{v}_{t,j} \Op (b) = i \Op \big(b t^{\frac{1}{2} } \tilde{v}_{t,j} \big), \qquad \Op (b)  t^{\frac{1}{2}} v_j  =   \Op \big(b t^{\frac{1}{2}} v_j \big) - i \Op (t^{\frac{1}{2}}\partial_{\eta_j} b) .
 \end{align*}
We shall also use that, by writing $ (\nabla V (y))^{\beta^{\prime}} = (\nabla V (x) + \nabla V (y) - \nabla V (x))^{\beta^{\prime}} $ and  
\[
	\nabla V (y) - \nabla V (x) = \int_0^1 \nabla^2 V (x+ s (y-x))\cdot (y-x)\,ds,
\]
we have 
\[
	\big( t^{\frac{3}{2}} \nabla V (y) \big)^{\beta^{\prime}} = \mbox{linear comb. of} \  \ B_{\hat{\beta},\check{\beta}}(x,y) \big( t^{\frac{3}{2}} \nabla V (x) \big)^{\hat{\beta}} \big(t^{\frac{3}{2}}(x-y) \big)^{\check{\beta}}
\]
with
\[
	|\beta^{\prime}| = |\hat{\beta}| + |\check{\beta}|, \qquad B_{\hat{\beta},\check{\beta}} \ \mbox{smooth and bounded together with all its derivatives}.
\]
We thus infer that the Schwartz kernel of  $ \widetilde{W}_t \Op (a_t) W_t^{\prime} $ is a finite sum of integrals of the form
\[
	\iint_{\mathbb R^{2n}} e^{i (x-y)\cdot \xi + i (v-w) \cdot \eta} c_t (x,v,\xi,\eta)\, d \xi d\eta
\]
multiplied by a bounded function of $ t,x,v,y,w $, and with symbol $ c_t $ satisfying the same type of estimates as $ a_t $ in Theorem \ref{propositionpolaire}.  $ L^p $ operator norm estimates follow again from Proposition \ref{referencesection}.
\end{proof}

\subsection{Estimating the remainder}

After studying the smoothing and localization properties of the operators $\mathcal E_A(t)$ is the previous section, we now tackle the same study for the remainders $\mathcal R_N(t)$. Precisely, the aim of this section is to prove the following result.

\begin{prop}\label{prop:smoothremainder}
There exists $t_0\in(0,1)$ such that for all multiindices $ \alpha , \alpha^{\prime},\beta,\beta^{\prime},\gamma,\gamma^{\prime},\delta,\delta^{\prime} \in \N^n$ and all $ M > 0 $, there exist $N\in\mathbb N^*$ and $C>0$ such that for every $ t \in [0,t_0] $ and  $ u \in \mathcal S(\mathbb R^{2n}) $,
\[
	\big\|  (\nabla V (x))^{\alpha} v^{\beta} \partial_x^{\gamma}  \partial_v^{\delta}  {\mathcal R}_N (t) \partial_v^{\delta^{\prime}} \partial_x^{\gamma^{\prime}} (\nabla V(x))^{\alpha^{\prime}} v^{\beta^{\prime}}  u\big\|_{L^p } 
	\leq C t^M \| u \|_{L^p},
\]
where $\mathcal R_N(t)$ stands for the remainder~\eqref{eq:remainder}.
\end{prop}

Note that the differential operators in front of $ {\mathcal R}_N (t) $ in Proposition \ref{prop:smoothremainder} do not depend on $t$, unlike those in Proposition \ref{Lpboundedness} above. We first start by proving the following estimates.

\begin{prop} \label{sousparametrix} 
Let $p\in (1,\infty)$. Given $ N \geq 1 $, consider the remainder $ \mathcal R_N $ defined by~\eqref{eq:remainder}. With the notation of Proposition \ref{Lpboundedness}, assuming that $ N \geq \frac{|\alpha^{\prime}| + 3 |\beta^{\prime}| + |\gamma^{\prime}| + 3 |\delta^{\prime}|}{2} $, we have for every $0 \leq t \leq t_0$ and $u\in {\mathcal S}({\mathbb R}^{2n})$,
\begin{equation}
	\big\Vert{\mathcal R}_N (t) W_t^{\prime}u\big\Vert_{L^p } \lesssim t^{N+1} \Vert u \Vert_{L^p}. \label{regularisationadroitesemigroupe}
\end{equation}
\end{prop}
 
\begin{proof} To prove the estimate (\ref{regularisationadroitesemigroupe}), we use that
\[
	{\mathcal E}_{R_N (s)} W_t^{\prime} 
	= t^{\frac{|\alpha^{\prime}| + 3 |\beta^{\prime}| + |\gamma^{\prime}| + 3 |\delta^{\prime}|}{2}} s^{N -\frac{|\alpha^{\prime}| + 3 |\beta^{\prime}| + |\gamma^{\prime}| + 3 |\delta^{\prime}|}{2} } 
	\big( s^{-N}{\mathcal E}_{R_N}(s) W_s^{\prime} \big),
\]
where $ s^{-N}{\mathcal E}_{R_N}(s) W_s^{\prime} $ is uniformly bounded on $ L^p(\mathbb R^{2n}) $ according to Proposition \ref{Lpboundedness}. The condition on $N$ allows to estimate
\[
	\int_0^t s^{N -\frac{|\alpha^{\prime}| + 3 |\beta^{\prime}| + |\gamma^{\prime}| + 3 |\delta^{\prime}|}{2} }\,ds \lesssim t^{N+1 -\frac{|\alpha^{\prime}| + 3 |\beta^{\prime}| + |\gamma^{\prime}| + 3 |\delta^{\prime}|}{2} },
\]
and the result follows. 
\end{proof}

\begin{prop} \label{commutationpseudodiff4} Let $(a_t)_{t \in (0,t_0]} $ be symbols satisfying (\ref{sanscomposition}). Then,
\[
	\Op (a_t) \partial_x^{\alpha} \partial_v^{\beta} v^{\gamma} (\nabla V)^{\delta}  
	= \sum \partial_x^{\alpha^{\prime}} \partial_v^{\beta^{\prime}} v^{\gamma^{\prime}} (\nabla V)^{\delta^{\prime}} {\mathcal A}_{\alpha^{\prime \prime} \beta^{\prime \prime} \gamma^{\prime \prime} \delta^{\prime \prime}}(t),
\]
where the sum is taken over multiindices such that
\[
	\alpha^{\prime} + \alpha^{\prime \prime} 
	= \alpha, \ \ \beta^{\prime} + \beta^{\prime \prime} = \beta , \ \ \gamma^{\prime} + \gamma^{\prime \prime} = \gamma, \ \ \delta^{\prime} \leq \delta, \ \ |\delta^{\prime}|+|\delta^{\prime \prime}| = |\delta|,
\]
and where, if $ {\mathcal K}_t $ and $ k_t $ are the respective kernels of $   {\mathcal A}_{\alpha^{\prime \prime} \beta^{\prime \prime} \gamma^{\prime \prime} \delta^{\prime \prime}}(t)  $ and $ \Op \big( \partial_x^{\alpha^{\prime \prime}}  \partial_v^{\beta^{\prime \prime}} \partial_{\eta}^{\gamma^{\prime \prime}} \partial_{\xi}^{\delta^{\prime \prime}} a_t  \big),  $
\begin{equation}
	{\mathcal K}_t (x,v,y,w) = k_t (x,v,y,w) b (x,y), \qquad \mbox{for some} \ b \in C^{\infty}_b ({\mathbb R}^{2n}). \label{noyaufonctionbornee}
\end{equation}
\end{prop}

\begin{proof} This is obtained along the same lines as the proof of Proposition \ref{Lpboundedness}.
\end{proof}

\begin{prop} \label{commutationflotapproche4} The operator $ {\mathcal I}_{F_t}  \partial_x^{\alpha} \partial_v^{\beta} v^{\gamma} (\nabla V)^{\delta} {\mathcal I}_{F_t}^{-1} $ is a sum of operators of the form
\[
	\partial_x^{\alpha^{\prime}} (\partial_v + t \partial_x)^{\beta^{\prime}} (\nabla V)^{\delta^{\prime}} \Big( v + \frac{t}{m} \nabla V \Big)^{\gamma^{\prime}} ( t \partial_v )^{\alpha^{\prime \prime}+\beta^{\prime \prime}} 
	( t^2 \partial_x )^{\alpha^{\prime \prime \prime}+\beta^{\prime \prime \prime}} ( t v )^{\delta^{\prime \prime}}  (t^2 \nabla_x V )^{\delta^{\prime \prime \prime}} b (t,x,v)
\]
with multiindices such that 
\[
	\alpha^{\prime} + \alpha^{\prime \prime} + \alpha^{\prime \prime \prime} \leq \alpha, \qquad
	\beta^{\prime} + \beta^{\prime \prime} + \beta^{\prime \prime \prime} \leq \beta, \qquad
	\gamma^{\prime} \leq \gamma, \qquad
	\delta^{\prime} + \delta^{\prime \prime} + \delta^{\prime \prime \prime} \leq \delta, 
\]
and with functions $b$ such that
\begin{equation}\label{fonctionsrobustes}
	\big| \partial_x^{\rho} \partial_v^{\theta}  b (t,x,v) \big| \leq C_{ \theta, \rho} t^{|\theta|}.
\end{equation} 
\end{prop}

\begin{proof} Note first that
\begin{align*}
	{\mathcal I}_{F_t} \partial_{x_j} {\mathcal I}_{F_t}^{-1} & = (\partial_{x_j} \tilde{x}_t \big)(F_t(x,v)) \cdot \partial_x + (\partial_{x_j} \tilde{v}_t \big) (F_t(x,v)) \cdot \partial_v , \\
	{\mathcal I}_{F_t} \partial_{v_j} {\mathcal I}_{F_t}^{-1} & = (\partial_{v_j} \tilde{x}_t \big)(F_t(x,v)) \cdot \partial_x + (\partial_{v_j} \tilde{v}_t \big) (F_t (x,v)) \cdot \partial_v.
\end{align*}
Using the expression (\ref{formeJacobienneF}) of the Jacobian matrix of $ F_t $, we find
\[
	\left(\begin{matrix} (\partial_x \tilde{x}_t) (F_t(x,v)) & (\partial_v \tilde{x}_t) (F_t(x,v)) \\ (\partial_x \tilde{v}_t) (F_t(x,v)) & (\partial_v \tilde{v}_t ) (F_t(x,v)) \end{matrix} \right) = \left( \begin{matrix} I - \frac{t^2}{2m} \partial^2 V & - t I \\ 
	\frac{t}{m} \partial^2 V & I \end{matrix} \right)^{-1} 
\]
whose right-hand side reads, after a simple calculation,
\[
	\begin{pmatrix} 
		( I + \frac{t^2}{2m} \partial^2 V )^{-1} & t ( I + \frac{t^2}{2m} \partial^2 V )^{-1} \\ 
		- \frac{t}{m} \partial^2 V ( I + \frac{t^2}{2m} \partial^2 V )^{-1} & I - \frac{t^2}{m} \partial^2 V ( I + \frac{t^2}{2m} \partial^2 V )^{-1} 
	\end{pmatrix},
\]
that is
\[
	\left( \begin{matrix} I & t I \\ - \frac{t}{m} \partial^2 V & I \end{matrix} \right) + t^2 \left( \begin{matrix} B_{11}(t,x) & B_{12}(t,x) \\ B_{21}(t,x) & B_{22} (t,x) \end{matrix}  \right),
\]
where the matrices $ B_{jk} $ have coefficients in $ C_b^{\infty} ({\mathbb R}^n_x) $, uniformly with respect to $ t \in [0,t_0] $. We thus infer
\begin{align*}
	{\mathcal I}_{F_t} \partial_x  {\mathcal I}_{F_t}^{-1} & = \partial_x   + t^2  B_{11}(t,x)^T \partial_x 
	+ \left( - \frac{t}{m} \partial^2 V  + t^2 B_{21}(t,x)^T \right) \partial_v,   \\
	{\mathcal I}_{F_t} \partial_v {\mathcal I}_{F_t}^{-1}  & = \big( \partial_v + t \partial_x \big)   + t^2 \big( B_{12}(t,x)^T \partial_x + B_{22}(t,x)^T \partial_v  \big)  .
\end{align*}  
Similarly (and more directly), using the Taylor formula,
\begin{align*}
	{\mathcal I}_{F_t} \nabla V {\mathcal I}_{F_t}^{-1} & = (\nabla V ) \bigg(x - t v - \frac{t^2}{2m} \nabla V  \bigg)  \\
	& = (\nabla V ) (x) + t B (t,x,v) v  + t^2 \tilde{B} (t,x,v) \nabla V (x),
\end{align*}
where $ B (t,x,v) $ and $ \tilde{B}(t,x,v) $ are matrices with entries satisfying (\ref{fonctionsrobustes}) (here we use (\ref{hypothesepotentiel})). Finally,
\[
	{\mathcal I}_{F_t} v {\mathcal I}_{F_t}^{-1} = \Big(v + \frac{t}{m} \nabla V (x)\Big).
\]
The result follows by raising the above expressions to $ \alpha,\beta,\gamma,\delta $ respectively and using the Leibniz rule to commute functions and differentiations.
\end{proof}
 
\begin{prop}[Propagation of regularity] \label{theoremepropagationregularite4}  Let $ p \in (1,\infty) $. There exists $C > 0 $ such that, for all $ t \in [0,t_0] $ and all $u\in\mathcal S(\mathbb R^{2n})$,
\[
	\big\Vert v^{\gamma} (\nabla V)^{\delta} \partial_x^{\alpha} \partial_v^{\beta} e^{-\frac{t}{m}\ope} u\big\Vert_{L^p} 
	\leq C \sum \Big\Vert\Big(v - \frac{t}{m} \nabla V \Big)^{\gamma^{\prime}} (\nabla V)^{\delta^{\prime}} \partial_x^{\alpha^{\prime}} (\partial_v - t \partial_x )^{\beta^{\prime}} u\Big\Vert_{L^p},
\]
the sum being taken over multiindices such that
\begin{equation}
	\alpha^{\prime} \leq \alpha, \qquad \beta^{\prime} \leq \beta, \qquad \gamma^{\prime} \leq \gamma, \qquad \delta^{\prime} \leq \delta . \label{comptagedesmultiindices4}
\end{equation}
\end{prop}
 
\begin{proof} Let $u_0,u_1 $ be in the Schwartz space. Denote by $ \tilde{\mathcal P} $ the realization of $ {\mathcal P} $ on $ L^{p^{\prime}}(\mathbb R^{2n}) $. Then,
\begin{align}
	\big\langle v^{\gamma} (\nabla V)^{\delta} \partial_x^{\alpha} \partial_v^{\beta} e^{-\frac{t}{m}\ope} u_0 , u_1 \big\rangle_{{\mathcal S}^{\prime},{\mathcal S}} 
	& = (-1)^{|\alpha + \beta|} \big\langle u_0 , \big( e^{-\frac{t}{m}\ope} \big)^*  \partial_x^{\alpha} \partial_v^{\beta} v^{\gamma} (\nabla V)^{\delta} u_1 \big\rangle_{L^p,L^{p^{\prime}}} \nonumber \\
	& = (-1)^{|\alpha + \beta|} \big\langle U u_0 , e^{-\frac{t}{m}\tilde{\mathcal P}} U  \partial_x^{\alpha} \partial_v^{\beta} v^{\gamma} (\nabla V)^{\delta} u_1 \big\rangle_{L^p,L^{p^{\prime}}}  \nonumber \\
	& = (-1)^{|\alpha| + 2 | \beta| + |\gamma| } \big\langle U u_0 , e^{-\frac{t}{m}\tilde{\mathcal P}}   \partial_x^{\alpha} \partial_v^{\beta} v^{\gamma} (\nabla V)^{\delta} U u_1 \big\rangle_{L^p,L^{p^{\prime}}}  \label{fonctionnementpardualite4}              
\end{align}
using (\ref{e:adjointsansadjoint}) to get the second line. Using Proposition \ref{sousparametrix} for $ \tilde{\mathcal P} $ and with $N$ large enough
\[
	e^{-\frac{t}{m}\tilde{\mathcal P}}   \partial_x^{\alpha} \partial_v^{\beta} v^{\gamma} (\nabla V)^{\delta} 
	= {\mathcal E}_{1+\cdots + A_N} (t)   \partial_x^{\alpha} \partial_v^{\beta} v^{\gamma} (\nabla V)^{\delta} + {\mathcal B}_N (t),
\]
with
\[
	\Vert{\mathcal B}_N (t)\Vert_{L^{p^{\prime}} \rightarrow L^{p^{\prime}} } \lesssim 1, \qquad t \in (0,t_0] .
\]
On the other hand, using the approximate polar decomposition of $ {\mathcal E}_{1+\cdots + A_N}(t) $ together with Propositions~\ref{commutationpseudodiff4} and \ref{commutationflotapproche4}, $ {\mathcal E}_{1+\cdots + A_N} (t)   \partial_x^{\alpha} \partial_v^{\beta} v^{\gamma} (\nabla V)^{\delta}$ is a sum of operators of the form
\[
	\partial_x^{\alpha^{\prime}} (\partial_v + t \partial_x)^{\beta^{\prime}} (\nabla V)^{\delta^{\prime}} \Big( v + \frac{t}{m} \nabla V \Big)^{\gamma^{\prime}} b_t W_t  {\mathcal I}_{F_t} {\mathcal A} (t),
\]
with multiindices as in (\ref{comptagedesmultiindices4}), $b_t$ a bounded function (with bound uniform in $t$),  $ W_t $ an operator of the form (\ref{formedeWt4}) (it comes from an operator where $ v , \partial_v $ carry a factor $t$ rather than $ t^{1/2} $ and $ \nabla V , \partial_x $ carry a  factor $t^2$ rather than $ t^{3/2} $) and $ {\mathcal A}(t) $ an operator with Schwartz kernel of the form (\ref{noyaufonctionbornee}). Therefore (\ref{fonctionnementpardualite4}) is bounded by a finite sum of the form
\begin{equation}\label{commutationU4}
	\sum  \big\vert \big\langle  v^{\gamma^{\prime}} (\nabla V)^{\delta^{\prime}}  \partial_x^{\alpha^{\prime}} \partial_v^{\beta^{\prime}} U u_0 , b_t W_t {\mathcal I}_{F_t} {\mathcal A} (t)   U u_1 \big\rangle_{L^p,L^{p^{\prime}}}\big\vert 
	+ \Vert u_0\Vert_{L^p}\Vert u_1\Vert_{L^{p^{\prime}}},
\end{equation}
in which each operator $  W_t  {\mathcal I}_{F_t} {\mathcal A} (t)  $ is  uniformly bounded on $ L^{p^{\prime}}(\mathbb R^{2n}) $ with respect to $ t \in [0,t_0] $: if there were no $b$ in (\ref{noyaufonctionbornee}), this would be a direct consequence of Proposition \ref{Lpboundedness} (with $ \nu \geq 0 $), however the additional factor $b$ is completely harmless for $W_t$ applied to $ {\mathcal A}(t) $ is an operator whose kernel is that of a sum of pseudodifferential operators with symbols satisfying (\ref{sanscomposition}) multiplied by derivatives of $b$ hence is bounded on $ L^{p^{\prime}}(\mathbb R^{2n})$ by the Schur test. 
Finally, commutation with $U$ in the right-hand side of (\ref{commutationU4}) changes $v$ into $-v$ and $ \partial_v $ into  $ - \partial_v  $, so using the H\"older inequality, we get
\[
	\big\vert\big\langle v^{\gamma} (\nabla V)^{\delta} \partial_x^{\alpha} \partial_v^{\beta} e^{-\frac{t}{m}\ope} u_0 , u_1 \big\rangle_{{\mathcal S}^{\prime},{\mathcal S}}\big\vert
	\lesssim  \sum \Big\Vert\Big(v - \frac{t}{m} \nabla V \Big)^{\gamma^{\prime}} (\nabla V)^{\delta^{\prime}} \partial_x^{\alpha^{\prime}} (\partial_v - t \partial_x )^{\beta^{\prime}}  u_0\Big\Vert _{L^p}\Vert u_1\Vert_{L^{p^{\prime}}},
\]
from which the result follows.
\end{proof}

\begin{proof}[Proof of Proposition \ref{prop:smoothremainder}] 
Let $u$ be in the Schwartz space $\mathcal S(\mathbb R^{2n})$. We start from the Duhamel formula~\eqref{eq:approxevolop} and we estimate $ {\mathcal R}_N (t) u $ as follows. By Proposition \ref{theoremepropagationregularite4}, the $L^p$ norm of
\begin{equation}
	(\nabla V (x))^{\alpha} v^{\beta} \partial_x^{\gamma}  \partial_v^{\delta}  e^{-\frac{t-s}{m}\ope} 
	{\mathcal E}_{R_N} (s) \partial_v^{\delta^{\prime}} \partial_x^{\gamma^{\prime}} (\nabla V(x))^{\alpha^{\prime}} v^{\beta^{\prime}}  u \label{quantiteaestimerf4}
 \end{equation} 
is bounded by a constant times a sum of terms of the form
\[
	\Big\Vert(\nabla V (x))^{\alpha^{\prime \prime}} \Big(v - \frac{t-s}{m} \nabla V \Big)^{\beta^{\prime \prime}} \partial_x^{\gamma^{\prime \prime}}  \big(\partial_v  - (t-s) \partial_x \big)^{\delta^{\prime \prime}}   {\mathcal E}_{R_N} (s) \partial_v^{\delta^{\prime}} \partial_x^{\gamma^{\prime}} (\nabla V(x))^{\alpha^{\prime}} v^{\beta^{\prime}}  u \Big\Vert_{L^p},
\]
with $ \alpha^{\prime \prime} \leq \alpha $, $ \beta^{\prime \prime} \leq \beta $, $ \gamma^{\prime \prime} \leq \gamma $ and $ \delta^{\prime \prime} \leq \delta $.  Using that $ R_N $ is of valuation $N$ (see Proposition \ref{pourlamplitudedutheoreme}) and picking $N$ large enough, we write
\[
	s^{\frac{N-M}{2}} \big( s^{M-N} {\mathcal E}_{R_N} (s) \big) s^{\frac{N-M}{2}},
\]
where the parenthesis in the middle has valuation $ M $, and the operators
\[
	s^{\frac{N-M}{2}} (\nabla V (x))^{\alpha^{\prime \prime}} \Big(v - \frac{t-s}{m} \nabla V \Big)^{\beta^{\prime \prime}} \partial_x^{\gamma^{\prime \prime}}  \big(\partial_v  - (t-s) \partial_x \big)^{\delta^{\prime \prime}}, 
	\qquad s^{\frac{N-M}{2}} \partial_v^{\delta^{\prime}} \partial_x^{\gamma^{\prime}} (\nabla V(x))^{\alpha^{\prime}} v^{\beta^{\prime}},
\]
are respectively  sums of operators of the form $ W_s $ and $ W_s^{\prime} $ (see Proposition \ref{Lpboundedness}), with harmless bounded factors that are (nonnegative) powers of $(t-s)$  and $s$. Basically, we have to pick $N$ satisfying
\[
	N-M \geq  3 \max \big( |\alpha| +  |\beta | +  |\gamma| + |\delta| , |\alpha^{\prime}| + |\beta^{\prime}| + |\gamma^{\prime}| + |\delta^{\prime}| \big) .
\]
With this choice, we get from Proposition \ref{Lpboundedness} that the $ L^p $ norm of (\ref{quantiteaestimerf4}) is bounded by $ C s^M\Vert u\Vert_{L^p} $ with a constant independent of $s$ and the result follows after integration over $ [0,t] $. 
\end{proof}

\subsection{Smoothing and localization phenomena for the semigroup}

We are now almost in position to prove Theorem \ref{theoremeLL}. In this subsection, we temporarily restore the dependence in $p\in(1,\infty)$ and denote by $\ope_p$ the Fokker--Planck operator acting on $L^p(\R^{2n})$, defined by Theorem~\ref{theoreme-semigroup-intro}. In the next result, we prove that the action of the evolution operators $e^{-\frac{t}{m}\ope_p}$ on the Schwartz space $\mathcal S(\mathbb R^{2n})$ actually does not depend on $p$.

\begin{prop}\label{prop:depschartzspace} For every $p\in(1,\infty)$ and $t\geq0$, the evolution operator $e^{-\frac{t}{m}\ope_p}$ stabilizes the Schwartz space $\mathcal S(\mathbb R^{2n})$, and more precisely is bounded on $ {\mathcal S}({\mathbb R}^{2n}) $ locally uniformly in time\footnote{meaning that for each seminorm $ {\mathcal N} $ of $ {\mathcal S}({\mathbb R}^{2n}) $ there exist $ C>0$ and another seminorm $ {\mathcal N}^{\prime} $ such that $ {\mathcal N} (e^{-\frac{t}{m}{\mathcal P}_p} u) \leq C {\mathcal N}^{\prime} (u) $ for all $ t \in [0,1] $ and all $ u \in {\mathcal S}({\mathbb R}^{2n}) $  }. We also have that for every $q\in(1,\infty)$ and $u\in {\mathcal S}(\mathbb R^{2n})$,
\begin{equation}\label{eq:noinfschwartz}
	e^{-\frac{t}{m}\ope_p}u = e^{-\frac{t}{m}\ope_q}u.
\end{equation}
\end{prop}

Before proving Proposition \ref{prop:depschartzspace}, we need two preparation lemmas.

\begin{lemm} \label{densitySchwartzclef} Given $ p \in (1,\infty) $ and $ \bar{N} = (N_1,\ldots, N_5) \in {\mathbb N}^5 $,  let $ {\mathcal B}_{\bar{N}}^p $ be the space of functions $u \in L^p(\mathbb R^{2n}) $ satisfying
\[ 
    x^{\sigma} v^{\gamma} (\nabla V)^{\delta} \partial_x^{\alpha} \partial_v^{\beta} u \in L^p(\mathbb R^{2n})
\]
for all multiindices $\sigma,\gamma,\delta,\alpha,\beta\in\mathbb N^n$ such that
\begin{equation} \label{conditionsindicesBanach}
    |\sigma| \leq N_1, \qquad |\gamma| \leq N_2, \qquad |\delta| \leq N_3, \qquad |\alpha| \leq N_4, \qquad |\beta | \leq N_5  .  
\end{equation}
Then, equipped with the norm
\[
    \Vert u \Vert_{p,\bar{N}} :=
    \sum_{(\ref{conditionsindicesBanach}) } \Vert x^{\sigma} v^{\gamma} (\nabla V)^{\delta} \partial_x^{\alpha} \partial_v^{\beta} u\Vert_{L^p} ,
\]
the set $  {\mathcal B}^p_{\bar{N}} $ is a Banach space in which the Schwartz space is dense.
\end{lemm}

\begin{proof} We only sketch the main steps  since they follow from fairly standard arguments. If $ (u_j)_j $ is a Cauchy sequence in $  {\mathcal B}^p_{\bar{N}} $ then it is a Cauchy sequence in $ L^p(\mathbb R^{2n}) $ hence converges, in $ L^p(\mathbb R^{2n}) $, to some $u\in L^p(\mathbb R^{2n})$. If $ \varphi \in C_c^{\infty}(\mathbb R^{2n}) $ and the multiindices $\sigma, \gamma, \delta,\alpha,\beta$ satisfy (\ref{conditionsindicesBanach}), we use
\[ 
    \big| \big\langle u , (x^{\sigma} v^{\gamma} (\nabla V)^{\delta} \partial_x^{\alpha} \partial_v^{\beta} )^T \varphi \big\rangle_{{\mathcal D}^{\prime},{\mathcal D}} \big| 
    = \big|\lim_{j \rightarrow +\infty} \big\langle x^{\sigma} v^{\gamma} (\nabla V)^{\delta} \partial_x^{\alpha} \partial_v^{\beta}  u_j , \varphi) \big\rangle_{{\mathcal D}^{\prime},{\mathcal D}} \big| 
    \leq C \Vert \varphi \Vert_{L^{p^{\prime}}}
\]
(since $ (x^{\sigma} v^{\gamma} (\nabla V)^{\delta} \partial_x^{\alpha} \partial_v^{\beta} u_j)_j $ is a Cauchy, hence bounded, sequence in $ L^p(\mathbb R^{2n}) $) to infer that $u$ belongs to $ {\mathcal B}^p_{\bar{N}} $. This implies in turn easily that $ (u_j)_j $ converges to $u$ in $  {\mathcal B}^p_{\bar{N}} $. To prove the density of the Schwartz space, we observe first that compactly supported functions are dense: indeed if $ \chi \in C_c^{\infty}({\mathbb R}^{2n}) $ is equal to $1$ near the origin, then it is straightforward to check that $ \chi (\varepsilon x,\varepsilon v) u \rightarrow u $ in $ {\mathcal B}^p_{\bar{N}} $ as $ \varepsilon \rightarrow 0 $. To complete the proof, it suffices to show that any compactly supported function $u \in {\mathcal B}^p_{\bar{N}}$ can be approached by Schwartz functions. Pick such a compactly supported function $u$ and the same $ \chi $ as above. For multiindices $\sigma, \gamma, \delta,\alpha,\beta\in\mathbb N^n$ satisfying (\ref{conditionsindicesBanach}) and $ N>0$ to be chosen, consider
\begin{equation} \label{pseudodiffsemiclassique}
    x^{\sigma} v^{\gamma} (\nabla V)^{\delta} \partial_x^{\alpha} \partial_v^{\beta} \big( u - \chi (\varepsilon D_x , \varepsilon D_v)u \big) 
    = \big( x^{\sigma} v^{\gamma} (\nabla V)^{\delta} \langle x,v \rangle^{-N} \big) A_{\varepsilon}  \big( \langle x , v \rangle^N \partial_x^{\alpha} \partial_v^{\beta} u \big)
\end{equation}
where $ \langle x , v \rangle = (1+|x|^2 + |v|^2)^{1/2} $ is the usual Japanese bracket and
\[ 
    A_{\varepsilon} = I- \langle x,v \rangle^{N} \chi (\varepsilon D_x , \varepsilon D_x) \langle x , v \rangle^{-N}  
    = I -  \chi (\varepsilon D_x, \varepsilon D_v)  + \mathcal O (\varepsilon),
\]
where $ \mathcal O (\varepsilon) $ is to be taken for the operator norm on $ L^p(\mathbb R^{2n}) $. This is a simple consequence of standard semiclassical pseudodifferential calculus saying that $ \langle x,v \rangle^{N} \chi (\varepsilon D_x , \varepsilon D_x) \langle x , v \rangle^{-N} - \chi (\varepsilon D_x , \varepsilon D_v)  $ has a symbol with seminorms of size $ \varepsilon $ in class of functions bounded in $ (x,v) $ and Schwartz in $ (\xi,\eta) $. In particular, for any fixed $ N $, $ (A_{\varepsilon})_{\varepsilon}$ converges strongly to zero on $ L^p(\mathbb R^{2n}) $ since $ (I-\chi (\varepsilon D_x , \varepsilon D_v))_{\varepsilon} $ does. Now, by picking $ N$ large enough, the first parentheses in the right-hand side of (\ref{pseudodiffsemiclassique}) is a bounded function (here we use that $ \nabla V $ grows at most polynomially according to (\ref{hypothesepotentiel})), while $ \langle x ,v \rangle^{N} \partial_x^{\alpha} \partial_v^{\beta} u $ belongs to $ L^p(\mathbb R^{2n})$ since $u$ has compact support (and $ \partial_x^{\alpha} \partial_v^{\beta} u \in L^p(\mathbb R^{2n}) $). As a consequence, we obtain that the family of Schwartz functions $ (\chi (\varepsilon D_x , \varepsilon D_v) u)_{\varepsilon} $ converges to $ u $ in $ {\mathcal B}^p_{\bar{N}} $ as $ \varepsilon \rightarrow 0 $. This completes the proof.
\end{proof}

\begin{lemm} \label{lemmeSchwartz4} Let $ p \in (1,\infty) $. For all $ \bar{N} = (N_1,\ldots, N_5) \in {\mathbb N}^5 $ there exist $ \bar{M} = (M_1,\ldots, M_5) \in {\mathbb N}^5 $ and a constant $ C> 0 $ such that $ M_1 = N_1 $ and for all $ t \in [0,1] $ and all $ u \in {\mathcal B}^p_{\bar{M}} $,
\begin{equation} \label{independanceexposantSchwartz}
    \big\Vert e^{- \frac{t}{m} \ope_p} u \big\Vert_{p,\bar{N}} \leq C \Vert u \Vert_{p,\bar{M}}.
\end{equation}
\end{lemm}

\begin{proof} We  proceed by induction on $ N_1 $. By Proposition \ref{theoremepropagationregularite4} and the density of the Schwartz space obtained in Lemma \ref{densitySchwartzclef}, we see that the result holds when $ N_1 = 0 $.
Assume next that the result is true for some $N_1$ and let us show it holds at rank $ N_1 + 1 $. Pick $ \sigma $ of length $ N_1 + 1 $ and write $ x^{\sigma} = x_j x^{\sigma^{\prime}} $ with $ |\sigma^{\prime}| = N_1 $. Then,
\[
    x^{\sigma} v^{\gamma} (\nabla V)^{\delta} \partial_x^{\alpha} \partial_v^{\beta} e^{- \frac{t}{m} \ope_p} u  
    = x^{\sigma^{\prime}} v^{\gamma} (\nabla V)^{\delta} \partial_x^{\alpha} \partial_v^{\beta} \big( x_j e^{-\frac{t}{m}\ope_p} u \big) 
    + \alpha_j  x^{\sigma^{\prime}} v^{\gamma} (\nabla V)^{\delta} \partial_x^{\alpha^{\prime}} \partial_v^{\beta}  e^{-\frac{t}{m}\ope_p} u,
\]
where $ \alpha^{\prime} $ is such that $ \partial_{x_j} \partial_x^{\alpha^{\prime}}= \partial_x^{\alpha}$ (if ever $ \alpha_j \ne 0 $ otherwise the second term in the right-hand side above vanishes). The second term in the right-hand side can be estimated thanks to the induction assumption. On the other hand,
using
\[ 
    x_j e^{- \frac{t}{m}\ope_p} u 
    = e^{- \frac{t}{m}\ope_p} x_j u + \int_0^t e^{- \frac{t-s}{m}\ope_p} v_j e^{- \frac{s}{m}\ope_p} u\, ds 
\]
and the induction assumption, we  estimate 
\[
    \big\Vert x^{\sigma^{\prime}} v^{\gamma} (\nabla V)^{\delta} \partial_x^{\alpha} \partial_v^{\beta} e^{- \frac{t}{m}\ope_p} u \big\Vert_{L^p}  
    \leq C \Vert x_j u \Vert_{p,\bar{M}^{\prime}} \leq C^{\prime} \Vert u \Vert_{p,\bar{M}} 
\]
for some $ \bar{M}^{\prime} $ of the form $ ( N_1 , M_2^{\prime} , \ldots , M_5^{\prime} ) $ and thus some $ \bar{M} = (N_1 + 1, M_2,\ldots, M_5) $; similarly, with the same notation for $ \bar{M}^{\prime} $,
\[
    \big\Vert x^{\sigma^{\prime}} v^{\gamma} (\nabla V)^{\delta} \partial_x^{\alpha} \partial_v^{\beta} e^{- \frac{t-s}{m} \ope_p} v_j e^{- \frac{s}{m}\ope_p} \big\Vert_{L^p} 
    \leq C \big\Vert v_j e^{-\frac{s}{m} \ope_p} u \big\Vert_{p,\bar{M}^{\prime}} 
    \leq C^{\prime} \Vert u \Vert_{p,\bar{M}^{\prime \prime}}
\]
with $ \bar{M}^{\prime \prime} $ of the form $ ( N_1 , M_2^{\prime \prime} , \ldots , M_5^{\prime \prime} )   $. The result follows.
\end{proof}

\begin{proof}[Proof of Proposition \ref{prop:depschartzspace}] Observe first that
\begin{equation} \label{Schwartzathand}
    {\mathcal S}({\mathbb R}^{2n}) = \bigcap_{\bar{N} \in {\mathbb N}^5} {\mathcal B}^p_{\bar{N}} . 
\end{equation}
Indeed, the Schwartz space is clearly contained in the intersection; conversely, any $u$ in the intersection is such that $ x^{\sigma} v^{\gamma} \partial_x^{\alpha} \partial_v^{\beta} u \in L^p(\mathbb R^{2n}) $ for all multiindices; this implies in turn that $ (1-\Delta_{x,v})^{2n} \big( x^{\sigma} v^{\gamma} \partial_x^{\alpha} \partial_v^{\beta} u \big) \in L^p(\mathbb R^{2n})  $ hence that $ x^{\sigma} v^{\gamma} \partial_x^{\alpha} \partial_{v}^{\beta} u \in L^{\infty}(\mathbb R^{2n}) $ by Sobolev embedding. Now with (\ref{Schwartzathand}) at hand, it follows from Lemma \ref{lemmeSchwartz4} that the Schwartz space is stable by $ e^{-\frac{t}{m}\ope_p} $ for $ t \in [0,1]$. By the semigroup property, this is then true for all $ t \geq 0 $. The boundedness on the Schwartz space locally uniformly in time follows mainly from (\ref{independanceexposantSchwartz}) along with the fact that each seminorm in the Schwartz space is controlled by some $ {\mathcal B}^p_{\bar{N}} $ norm (using a Sobolev embedding as above) and conversely that any such norm is controlled by a seminorm in the Schwartz space since $ \nabla V $ grows at most linearly. Let us now prove  (\ref{eq:noinfschwartz}) for every $u \in {\mathcal S}({\mathbb R}^{2n})$, say for $q < p$. Using the characterization in Proposition \ref{p:existence-uniqueness}, it suffices to prove that $ t\mapsto e^{-\frac{t}{m}\ope_p} u $ is a $ L^q $-valued continuous function.
This rests on the following consequence of Lemma \ref{lemmeSchwartz4}: for any $ \bar{N} $, there exists $ \bar{M} $ such that $ e^{-\frac{t}{m}\ope_p}$ is bounded from $ {\mathcal B}_{\bar{M}}^p $ to $ {\mathcal B}^p_{\bar{N}} $ locally uniformly in $t$. In particular, given any $t_0 \geq 0$, with suitable $ \bar{N} $ and $\bar{M}$, we can estimate for $t$ close to $t_0$, 
\begin{align*}
    \big\Vert e^{-\frac{t}{m}{\ope}_p} u - e^{-\frac{t_0}{m}{\ope}_p} u \big\Vert_{L^q}  
    & \lesssim  \big\Vert (1+|x|^2 + |v|^2)^{2n} \big( e^{-\frac{t}{m}\ope_p} u - e^{-\frac{t_0}{m}\ope_p} u \big) \big\Vert_{L^p} \\
    & \lesssim  \bigg\Vert \int_{t}^{t_0} (1+|x|^2+|v|^2)^{2n} \ope_p e^{- \frac{s}{m} \ope_p} u\, ds \bigg\Vert_{L^p} \\
    & \lesssim  \int_{[t_0,t]} \big\Vert e^{- \frac{s}{m} \ope_p} u \big\Vert_{p, \bar{N}}\, ds \\
    & \lesssim  |t-t_0| \Vert u \Vert_{p,\bar{M}} 
\end{align*}
using the H\"older inequality in the first line. This shows that $ t\mapsto e^{-\frac{t}{m}\ope_p} \in L^q(\mathbb R^{2n})$ is continuous at $t_0$ hence on $ [0,+\infty ) $ since $t_0$ is arbitrary. This completes the proof.
\end{proof}

As a byproduct of the continuity of the semigroup $ ({e^{-\frac{t}{m}{\mathcal P}}})_{t\geq0}$ on the Schwartz space, we can derive the following natural property.

\begin{theo} For all $ p \in (1,\infty) $ and $ t \geq 0 $,  $ e^{-\frac{t}{m}\ope} $ is positivity preserving on $ L^p(\mathbb R^{2n}) $, {\it i.e.} if $ u \in L^p(\mathbb R^{2n})$ satisfies $u\geq 0$ a.e. on $\R^{2n}$, then $ e^{-\frac{t}{m}\ope} u \geq 0 $ a.e. on $\R^{2n}$.
\end{theo}

\begin{proof} It suffices to prove the result when $ t \in [0,1] $.  By density, more precisely using that any nonnegative $u \in L^p(\mathbb R^{2n}) $ can be approached in $ L^p(\mathbb R^{2n}) $ by a sequence of nonnegative $ C_c^{\infty}(\mathbb R^{2n}) $ functions, we may assume that $ u $ is a Schwartz function. According to Proposition \ref{prop:depschartzspace}, we can restrict ourselves to the case $ p = 2 $. We recall~\eqref{e:def-FP-operator} that ${\mathcal P} = {\mathcal T} + {\mathcal H}$ and
	let $ (e^{-\frac{t}{m}{\mathcal H}})_{t\geq0} $ be the semigroup on $ L^2(\mathbb R^{2n}) $ associated with the harmonic oscillator, and $ (e^{- \frac{t}{m}{\mathcal T}})_{t\in\mathbb R} $ be the unitary group on $ L^2(\mathbb R^{2n}) $ of composition with the flow of $ {\mathcal T} $ at time $ - \frac{t}{m} $, that is the composition with $ (\bar{x}_{-t},\bar{v}_{-t}) $ using the notation (\ref{Hamiltonequation}). Both are positivity preserving (this is obvious for the composition operator and follows for instance from the Mehler formula for the harmonic oscillator). Moreover, both preserve the Schwartz space with uniform bounds locally in time; for the harmonic oscillator, this follows from the fact that the semigroup is a pseudodifferential operator\footnote{alternatively, one may check it using the exact commutation formulas $ v_j e^{-\frac{t}{m}{\mathcal H}} = e^{-\frac{t}{m}{\mathcal H}} \big(\cosh (\gamma t)  v_j - \frac{2}{m \beta}  \sinh (\gamma t) \partial_{v_j}  \big) $ and $ \partial_{v_j} e^{-\frac{t}{m}{\mathcal H}} = e^{-\frac{t}{m}{\mathcal H}} \big(\cosh (\gamma t) \partial_{v_j} - \frac{m \beta}{2} \sinh (\gamma t) v_j  \big) $} and for $ e^{-\frac{t}{m}{\mathcal T}} $, it is a consequence of Lemma \ref{lemmeflotborne}: this lemma allows to check that for any $ N $ there exists a constant independent of $t\in[0,1]$ such that
\[
    \langle x \rangle^N \langle v \rangle^N \leq C \langle \bar{x}_{-t} \rangle^{N} \langle \bar{v}_{-t} \rangle^N
\]
since $ x = \bar{x}_{t} (\bar{x}_{-t}, \bar{v}_{-t})$ and $ v = \bar{v}_t (\bar{x}_{-t},\bar{v}_{-t}) $ and derivatives of $ \bar{x}_t,\bar{v}_t$ are bounded according to (\ref{bonneestimationflot}). These estimates also show that, given a Schwartz function $ u $, $ \partial_x^{\alpha} \partial_v^{\beta} ( u (\bar{x}_{-t},\bar{v}_{-t}) ) $ is a linear combination of $ (\partial^{\delta} u ) (\bar{x}_{-t},\bar{v}_{-t}) $ times bounded functions. Thus, 
\[ 
    \big| \langle x \rangle^N \langle v \rangle^N  \partial_x^{\alpha} \partial_v^{\beta} ( u (\bar{x}_{-t},\bar{v}_{-t}) ) \big| 
    \leq C \sum_{|\delta| \leq |\alpha + \beta|} \langle \bar{x}_{-t} \rangle^{N} \langle \bar{v}_{-t} \rangle^N \big|(\partial^{\delta} u )(\bar{x}_{-t},\bar{v}_{-t})\big| 
\]
proves that Schwartz seminorms of $ e^{-\frac{t}{m}{\mathcal T}} u $ are bounded by Schwartz seminorms of $ u $, locally uniformly in time. With these preliminaries at hand, given $u\in\mathcal S(\mathbb R^{2n})$, we shall prove the result by adapting the classical proof of the Trotter product formula
\begin{equation}\label{Trotterlimite}
    e^{-\frac{t}{m}\ope} u  = \lim_{N \rightarrow \infty} \big( e^{-\frac{t}{mN} {\mathcal T}} e^{-\frac{t}{mN}{\mathcal H}} \big)^N  u    \quad \text{ in } L^2(\R^{2n}), 
\end{equation}
the principle of which can be found in \cite[Theorem VIII.30]{ReedSimon1}. We shall overcome the domain issues by using Schwartz functions $u \in \mathcal{S}(\R^{2n})$. Using that $ e^{- \frac{t}{m}\ope} = (e^{- \frac{t}{mN}\ope})^N $   together with the algebraic fact $A^N-B^N=\sum_{k=0}^{N-1}A^k(A-B)B^{N-1-k}$ (expanding the RHS yields a telescopic sum), we have on the one hand
\[  
    e^{-\frac{t}{m}\ope} u -  \big(e^{-\frac{t}{mN} {\mathcal T}} e^{-\frac{t}{mN}{\mathcal H}} \big)^N u
    = \sum_{k=0}^{N-1} \big( e^{- \frac{t}{mN} {\mathcal T}} e^{-\frac{t}{m N} {\mathcal H}  } \big)^k \big( e^{- \frac{t}{mN}\ope } - e^{- \frac{t}{mN} {\mathcal T}} e^{-\frac{t}{mN} {\mathcal H}} \big) e^{- \frac{t(N-1-k)}{mN}\ope } u . 
\]
On the other hand, if $ {\mathcal A} $ denotes any of the operators $ {\mathcal T}, {\mathcal H}, \ope $, we have 
\begin{align}
    e^{-\frac{t}{mN} {\mathcal A}} u 
    & = u  - \frac{t}{m N} {\mathcal A} u + \frac{1}{(mN)^2} {\mathcal A}^2 \int_0^{t}  (t-s) e^{- \frac{s}{mN} {\mathcal A}} u ds \nonumber \\
    & = \bigg( I - \frac{t}{m N} {\mathcal A} + \frac{1}{N^2} {\mathcal R}_{N,{\mathcal A}}(t) \bigg) u\label{DLdiminf}
\end{align} 
where $ {\mathcal R}_{N,A}(t) $ is bounded on the Schwartz space (by Proposition \ref{prop:depschartzspace} if $ {\mathcal A} = \ope $), uniformly with respect to $N$ and to $ t \in [0,1] $.
From this and the boundedness of $ {\mathcal T}, {\mathcal H} $ and $ \ope $ on the Schwartz space, we can expand $ e^{-\frac{t}{mN}\ope} - e^{-\frac{t}{mN}{\mathcal T}} e^{-\frac{t}{mN}{\mathcal H}} $ using (\ref{DLdiminf}) and, using ${\mathcal P} = {\mathcal T} + {\mathcal H}$ to cancel the terms of order $N^{-1}$,  observe that only terms of order $ \mathcal O (N^{-2}) $ remain. We thus infer that, for some constant $ C > 0 $ and some seminorm $ {\mathcal N}^{\prime} $ on $\mathcal{S}(\R^{2n})$, 
\[ 
    \big\Vert \big(  e^{-\frac{t}{mN}\ope} - e^{-\frac{t}{mN}{\mathcal T}} e^{-\frac{t}{mN}{\mathcal H}}  \big) e^{- \frac{t(N-1-k)}{mN}\ope} u \big\Vert_{L^2} 
    \leq \frac{C}{N^2} {\mathcal N}^{\prime} (u)  
\]
for all $t \in [0,1] $ and all $ N \geq 1 $. Thus,
\[ 
    \big\Vert e^{-\frac{t}{m}\ope} u -  (e^{-\frac{t}{mN} {\mathcal T}} e^{-\frac{t}{mN}{\mathcal H}} )^N u \big\Vert_{L^2} 
    \leq  \frac{C}{N} {\mathcal N}^{\prime} (u)
\]
yields (\ref{Trotterlimite}). The result of the theorem finally follows from (\ref{Trotterlimite}) combined with the fact that both $e^{-\sigma\mathcal{T}}$ and $e^{-\sigma\mathcal{H}}$ preserve positivity for all $\sigma\geq0$.
\end{proof}

We also need to derive an $L^p-L^q$ estimate for the evolution operators $e^{-\frac{t}{m}\ope_p}$.

\begin{prop}\label{prop:lplq}
Let $ 1 < p \leq q < \infty $ and $ T > 0 $. There exists $ C > 0 $ such that  for all $ t \in (0,T] $ and $ u \in L^p ({\mathbb R}^{2n}) $,
\[
	\big\Vert e^{-\frac{t}{m}\ope_p} u \big\Vert_{L^q} \leq C t^{-2n ( \frac{1}{p} - \frac{1}{q} )}  \Vert u \Vert_{L^p}.
\]
\end{prop}

\begin{proof} It suffices to prove the result when $ T  $ is small enough. Otherwise, using that $ (1 - \Delta_v - \Delta_x )^N e^{-\frac{t}{m} {\mathcal P_p}} $ is bounded on $ L^p(\mathbb R^{2n}) $ for each $N$ as long as $t $ is at positive distance from $ 0 $, we infer that $e^{-t {\ope}} $ is bounded from $ L^p(\mathbb R^{2n}) $ to $ L^q(\mathbb R^{2n}) $ using Sobolev embeddings. So we assume that $ t \leq t_0 $. Let $ (a_t)_{t \in (0,t_0]} $ be a family of symbols satisfying (\ref{sanscomposition}). Then, if we define the symbols $b_t$ by  $ b_t (x,v,\xi,\eta) = a_t (x,v,t^{-\frac{3}{2}}\xi, t^{-\frac{1}{2}}\eta) $, we have
\begin{equation}
	\big| \partial_{\xi}^{\gamma} \partial_{\eta}^{\delta} b_t (x,v,\xi,\eta) \big| \lesssim_{N,\gamma,\delta} t^{\nu} (1+|\xi|+|\eta|)^{-N} . \label{pourL1-0}
\end{equation} 
This implies that the Schwartz kernel $ K_t (x,v,y,w) $ of $ \Op (a_t) $ reads
\[
	(2 \pi)^{-2n} \iint_{\mathbb R^{2n}} e^{i (x-y)\cdot \xi + i(v-w)\cdot \eta} b_t \big(x,v,t^{\frac{3}{2}}\xi,t^{\frac{1}{2}}\eta\big)\, d\xi d\eta = (2 \pi t)^{-2n} \hat{b}_t \Big( x,v, \frac{y-x}{t^{\frac{3}{2}}} , \frac{w-v}{t^{\frac{1}{2}}} \Big),
\]
where $ \hat{b}_t $ stands for the Fourier transform of $b_t$ with respect to $ (\xi,\eta) $ and is bounded with respect to all of its arguments according to~\eqref{pourL1-0} with $ N > 2n $. Thus
\[
	\sup_{ (x,v) \in {\mathbb R}^{2n}} \Vert K_t (x,v,\cdot,\cdot)\Vert_{L^{\infty}_{y,w}} \lesssim t^{\nu - 2n}, \qquad t\in(0,t_0].
\]
On the other hand, using the fast decay of $ \hat{b}_t $ with respect to its last $ 2 n $ variables (see (\ref{conditionSchursymbole})) we also have
\[
	\sup_{ (x,v) \in {\mathbb R}^{2n}}\Vert K_t (x,v,\cdot,\cdot)\Vert_{L^1_{y,w}} \lesssim t^{\nu },
\]
so by interpolation, for any $ r \in [1,\infty] $,
\[
	\sup_{ (x,v) \in {\mathbb R}^{2n}}\Vert K_t (x,v,\cdot,\cdot)\Vert_{L^r_{y,w}} \lesssim t^{\nu - \frac{2n}{r^{\prime}}}.
\]
This implies that
\[
	\big\Vert \Op (a_t) u \big\Vert_{L^{\infty}} \leq \sup_{ (x,v) \in {\mathbb R}^{2n}}  \Vert  K_t (x,v,\cdot,\cdot)\Vert_{L^{p^{\prime}}}  \Vert u \Vert_{L^p}
	\lesssim t^{\nu - \frac{2n}{p}}  \Vert u \Vert_{L^p},
\]
and then, using the boundedness of $ {\mathcal I}_{F_t} $ on $ L^{\infty}(\mathbb R^{2n}) $, that
\[
	\big\Vert  {\mathcal I}_{F_t} \Op (a_t) u \big\Vert_{L^{\infty}} \lesssim t^{\nu - \frac{2n}{p}}  \Vert u \Vert_{L^p} .
\]
By interpolation with the $\mathcal O (t^{\nu}) $ estimate on the  $ L^p \rightarrow L^p $ norm  of $ {\mathcal I}_{F_t} \Op (a_t) $ from Theorem \ref{propositionpolaire}, we infer
\[
	\big\Vert {\mathcal I}_{F_t} \Op (a_t) u \big\Vert_{L^q} \lesssim t^{\nu - 2n ( \frac{1}{p} - \frac{1}{q} )}  \Vert u \Vert_{L^p}.
\]
This proves that  
\[
	\big\Vert {\mathcal E}_{1+\cdots+A_N}(t) u \big\Vert_{L^{q}} \lesssim t^{ - 2n ( \frac{1}{p} - \frac{1}{q} )}  \Vert u \Vert_{L^p} .
\]
The remainder $ {\mathcal R}_N (t) $ is bounded from $ L^p(\mathbb R^{2n}) $ to $ L^q(\mathbb R^{2n}) $ using again Sobolev embeddings. The result follows.
\end{proof}

\begin{proof}[Proof of Theorem \ref{theoremeLL}] More precisely, we prove the estimates~\eqref{eq:sharpsmooth}. We begin by splitting the evolution operators as follows
\[
	e^{-\frac{t}{m} \ope_p} = e^{- \frac{t}{3m} \ope_p} e^{- \frac{t}{3m} \ope_p} e^{- \frac{t}{3m} \ope_p}.
\]
Since, by Proposition \ref{prop:depschartzspace}, the evolution operators generated by the Fokker--Planck operator stabilize the Schwartz space $\mathcal S(\mathbb R^{2n})$, and that their action on $\mathcal S(\mathbb R^{2n})$ is actually independent on the Lebesgue space on which we consider the realization of $\ope$, we have for every $u\in\mathcal S(\mathbb R^{2n})$,
\[
	\big\Vert W_t e^{-\frac{t}{m} \ope_p} W_t^{\prime}u\big\Vert_{L^q} \leq  \big\Vert W_t e^{-\frac{t}{3m} \ope_q} \big\Vert_{L^q \rightarrow L^q} 
	\big\Vert e^{- \frac{t}{3m} \ope_p}  \big\Vert_{L^p \rightarrow L^q} \big\Vert e^{-\frac{t}{3m} \ope_p} W_t^{\prime} u \big\Vert_{L^p}.
\]
The first and the last term are treated thanks to Theorem \ref{maintechnicalresulttheorem}, Proposition \ref{Lpboundedness} and Proposition \ref{prop:smoothremainder}, and the second one can be bounded by using Proposition \ref{prop:lplq}.

We now consider the boundedness of the operators (\ref{Miklin}). For $ t \Delta_v  $ and $ t^3 \Delta_x $, this is directly a consequence of (\ref{estimationprincipale4}). For $ t |D_x|^{2/3} $, we pick $ \varrho \in C_c^{\infty} ({\mathbb R}) $ equal to $1$ near the origin and write
\[
	t |D_x|^{\frac{2}{3}} e^{-\frac{t}{m} \ope_p} =  \Big( t |D_x|^{\frac{2}{3}} \varrho (t^{\frac{3}{2}}D_x) \Big) e^{-\frac{t}{m} \ope_p} 
	+  \Big( (1-\varrho)(t^{\frac{3}{2}}D_x) t |D_x|^{\frac{2}{3}} (1-t^3\Delta_x)^{-1} \Big) (1 - t^3 \Delta_x) e^{- \frac{t}{m} \ope_p}.
\]
The result will then follow from (\ref{estimationprincipale4}) and the uniform boundedness on $ L^p(\mathbb R^{2n}) $ (with respect to $t$) of the parentheses in the right-hand side. In both parentheses, the symbols $ \varrho(\xi)|\xi|^{2/3} $ and  $(1-\varrho(\xi))|\xi|^{2/3} (1+|\xi|^2)^{-1} $ satisfy the assumptions of the H\"ormander-Mikhlin theorem so the associated Fourier multipliers are bounded on $ L^p(\mathbb R^{2n}) $. Their uniform boundedness in $t$ is obtained by rescaling. The proof is complete.
\end{proof}

We conclude this section with a proof of the following natural result.

\begin{prop}[A compactness criterion]\label{prop:compactness}
Assume that $V$ satisfies \eqref{hypothesepotentiel} and that $ |\nabla V (x) | \rightarrow + \infty $ as $ \vert x\vert \rightarrow +\infty $. Then, for each $t > 0$ and $ p \in (1,\infty) $, $ e^{-\frac{t}{m}{\mathcal P_p}} $ is a compact operator on $ L^p ({\mathbb R}^{2n}) $.
\end{prop}

\begin{proof} Let $ \chi \in C_c^{\infty}({\mathbb R}^{2n}) $ be equal to $1$ near the origin. Let $ t > 0 $. Let us write 
\[
    (1-\chi(\varepsilon x, \varepsilon v)) e^{-\frac{t}{m}\ope_p} 
    = \big( (1-\chi(\varepsilon x, \varepsilon v)) (1+|v|^2 + |\nabla V|^2)^{-1} \big) \big( (1+|v|^2 + |\nabla V|^2) e^{-\frac{t}{m}\ope_p} \big) 
\]
so that using Theorem \ref{theoremeLL}, we can bound 
\[
    \big\Vert (1-\chi(\varepsilon x , \varepsilon v)) e^{-\frac{t}{m}\ope_p} \big\Vert_{L^p \rightarrow L^p} 
    \lesssim_t   \sup_{(x,v)\in\mathbb R^{2n}} \big| (1-\chi(\varepsilon x, \varepsilon v)) (1+|v|^2 + |\nabla V|^2)^{-1} \big| \rightarrow 0 \quad \text{as $\varepsilon \rightarrow 0$},
\]
since $ (1+|v|^2 + |\nabla V|^2) \rightarrow + \infty $ as $ (x,v) \rightarrow \infty $.
It is thus sufficient to prove that $ \chi (\varepsilon x , \varepsilon v) e^{-\frac{t}{m}\ope_p} $ is compact on $ L^p(\mathbb R^{2n}) $ for any $ \varepsilon > 0 $. Writing
\[  
    \chi (\varepsilon x , \varepsilon v) e^{-\frac{t}{m}\ope_p}  
    = \big( \chi (\varepsilon x , \varepsilon v) (1-\Delta_x - \Delta_v)^{-N} \big) \big( (1-\Delta_x - \Delta_v)^N e^{-\frac{t}{m}\ope_p}  \big)  
\]
and using that the second parentheses in the right-hand side is bounded on $ L^p(\mathbb R^{2n}) $ by Theorem \ref{theoremeLL}, the result follows from the compactness of $ K = \chi (\varepsilon x , \varepsilon v) (1-\Delta_x - \Delta_v)^{-N}  $ on $ L^p(\mathbb R^{2n}) $. To justify the latter compactness with fairly simple and standard arguments, we pick $N$ large enough and consider the operator $ Q= (1+|x|^2+|v|^2)^{N/3} (1-\Delta_x - \Delta_v)^{N/3} $ so that
\[ 
	K = Q^{-1} \big( Q K Q \big) Q^{-1} 
\]
is compact on $ L^p(\mathbb R^{2n}) $ since on the one hand, $ Q^{-1} $ maps continuously both $ L^p(\mathbb R^{2n}) $ to $ L^2(\mathbb R^{2n}) $ and $ L^2(\mathbb R^{2n}) $ to $ L^p(\mathbb R^{2n}) $ and on the other hand, since $ Q K Q$ is compact on $ L^2(\mathbb R^{2n}) $. This completes the proof.
\end{proof}
 
\section{Parametrix of the resolvent} \label{s:parametrix-of-resolvent}
 
In this part, we use the parametrix of $ e^{- \frac{t}{m}\ope} $ to construct a parametrix of the resolvent of $ \ope$. As we shall see, a fairly simple non stationary phase argument using this parametrix allows to get Theorem~\ref{theoremespectreresolvante}. Since our parametrix is a high frequency one, it is convenient to look at the following semiclassical normalization of the resolvent
\[
	\big( \lambda^{-1} \ope - \lambda^{-1} z \big)^{-1} = \lambda (\ope-z)^{-1},
\]
where $ \lambda > 0 $ will essentially be a large parameter; in pratice we shall in particular consider the case $ \lambda \approx   |\Im(z)|^{1/2}  $. Given $ N \in {\mathbb N}^* $, we define
\[
	{\mathcal Q}_{\lambda} (z) := \frac{1}{m} \int_0^{+\infty} {\mathcal E}_{1+A_1+\cdots + A_N} ( t / \lambda) e^{  \frac{t}{m} \frac{z}{\lambda}} \varrho (t)\,dt,
\]
with $ \varrho \in C_c^{\infty} ({\mathbb R}) $ being equal to $ 1 $ near $ t = 0 $. We expect this operator to be an approximation of the resolvent since, if the parametrix $ {\mathcal E}_{1+A_1+\cdots + A_N} (t/\lambda) $ was replaced by the full semigroup $  e^{-\frac{t}{\lambda m} \ope}$ and if $ \varrho $ was replaced by $1$ we would get
\begin{equation}\label{pourlagauchedelensembleresolvant}
	\frac{1}{m} \int_0^{+\infty} e^{- \frac{t}{\lambda m} \ope}  e^{\frac{t}{ m} \frac{z}{\lambda}}\, dt = \lambda (\ope-z)^{-1},
\end{equation}  
at least when $ \Re(z)<  - C $ with $ C $ large enough. Besides, one of the interests of $ {\mathcal Q}_{\lambda}(z) $ is to be defined for all $ z \in {\mathbb C}$; this will allow us to prove the invertibility of $ \ope - z $ for certain $z$ with positive real part.
 
We also note that, according to Proposition \ref{stabiliteSchwartz}, $ {\mathcal Q}_{\lambda}(z) $ preserves the Schwartz space, so the identity in the next lemma can be interpreted as an equality as operators acting on the Schwartz space.
 
 \begin{lemm} With $R_N$ from Proposition \ref{pourlamplitudedutheoreme}, we have the identity
\[
	\lambda^{-1} \big(\ope - z \big)  {\mathcal Q}_{\lambda} (z) 
	= I + \frac{1}{\lambda m} \int_0^{+\infty} e^{\frac{t}{m} \frac{z}{\lambda}} {\mathcal E}_{R_N} (t/\lambda) \varrho (t)\,dt 
	+ \int_0^{+\infty} e^{\frac{t}{m} \frac{z}{\lambda}} {\mathcal E}_{1+\cdots + A_N} (t/\lambda) \varrho^{\prime}(t)\,dt .                         
\]
 \end{lemm}
 
\begin{proof} Using Proposition \ref{pourlamplitudedutheoreme}, we have
\begin{align*}
	\frac{1}{\lambda m} {\mathcal P} {\mathcal E}_{1+\cdots + A_N} (t / \lambda) & = \left[ \frac{1}{\lambda m} (\lambda m \partial_t + {\mathcal  P}) \right] {\mathcal E}_{1+\cdots + A_N} (t/\lambda) - \partial_t \big( {\mathcal E}_{1+\cdots+A_N} (t/ \lambda) \big)  \\
	& = \frac{1}{\lambda m} {\mathcal E}_{R_N} (t/\lambda)  - \partial_t \big( {\mathcal E}_{1+\cdots + A_N} (t/ \lambda) \big) .
\end{align*}  
Then, the result follows from 
\[
	- \frac{1}{\lambda m} z e^{ \frac{t}{m} \frac{z}{\lambda}} = - \partial_t \big( e^{\frac{t}{m} \frac{z}{\lambda}} \big) ,
\]
an integration by part in $t$ and that $ {\mathcal E}_{1+A_1+\cdots + A_N} (0) = I $.
\end{proof}

\begin{prop} \label{restevaluationnonzero} Let $ c > 0 $ be fixed. One can choose $ \varrho $ with small enough support such that for each $ p \in [1,\infty] $, we have
\[
	\left\Vert\frac{1}{\lambda m} \int_0^{+\infty} e^{\frac{t}{m} \frac{z}{\lambda}} {\mathcal E}_{R_N} (t/\lambda) \varrho (t)\,dt  \right\Vert_{L^p \rightarrow L^p} \leq \frac{1}{4}
\]
for all $ \lambda , z $ satisfying
\begin{equation}\label{zexponentielle}
	\lambda m \geq c , \qquad  \Re(z) \leq \lambda m. 
\end{equation}
\end{prop}

\begin{proof} Pick $ \varrho $ of the form $ \varrho (t) = \varrho_0 (t/\varepsilon) $ with $ \varrho_0 $ compactly  supported and equal to $1$ near $ 0 $.  It follows from Propositions \ref{pourlamplitudedutheoreme} and \ref{Lpboundedness}, along with  the fact that $ |e^{\frac{t}{m} \frac{z}{\lambda}}| \leq e^{t} $ according to (\ref{zexponentielle}), that  
\[
	\left\Vert\frac{1}{\lambda m} \int_0^{+\infty} e^{\frac{t}{m} \frac{z}{\lambda}} {\mathcal E}_{R_N} (t/\lambda) \varrho (t)\,dt  \right\Vert_{L^p \rightarrow L^p} 
	\lesssim \int_0^{+\infty} ( t / \lambda )^N \varrho_0 (t / \varepsilon)\,dt,
\]
which is less than $ 1/4 $ if $ \varepsilon $ is small enough.
\end{proof}

\begin{prop}  \label{pourRezadroite} Let $c>0$.   One can choose $ M > 0 $ large enough such that, for all $ \lambda , z $ satisfying
\begin{equation}
	\lambda m \geq  c , \qquad  \Re(z) \leq \lambda m , \qquad |\Im(z)| \geq M \lambda^2 , \label{zexponentiellebis}
\end{equation}
we have, for each $ p \in [1,\infty] $,
\[
	\left\Vert\int_0^{+\infty}  e^{\frac{t}{m} \frac{z}{\lambda}} {\mathcal E}_{1+A_1+\cdots+A_N} (t/\lambda) \varrho^{\prime}(t)\,dt  \right\Vert_{L^p \rightarrow L^p} \leq \frac{1}{4}.
\]
\end{prop}

\begin{proof} We start by observing that we can consider the contribution of $ A_0 = 1 $ only since higher order terms $ A_1, A_2,\ldots $ have positive valuations and can be seen to have small contribution as in Proposition~\ref{restevaluationnonzero}. We use the decomposition (\ref{approx0}) for $ {\mathcal E}_1 (t / \lambda) $, namely
\[
	{\mathcal E}_1 (t / \lambda) = {\mathcal I}_{F_{t/\lambda}} \Op (a_{t/\lambda}),
\]
with  a symbol defined by mean of (\ref{defImpsi}) as
\[
	a_{t/\lambda}(x,v,\xi,\eta) = \exp \left[ - \Im(\psi) \left( \frac{t}{\lambda} , \tilde{x}_{\frac{t}{\lambda}} , \tilde{v}_{\frac{t}{\lambda}} , \xi , \eta \right)\right],
\]
and which satisfies (\ref{avecomposition}) and (\ref{sanscomposition}) with $ \nu = 0 $ and $ t / \lambda $ instead of $t$. Picking $ \chi \in C_c^{\infty} ({\mathbb R}^{2n}) $ being equal to $1$ near $ 0 $, we define, for $ R \gg 1 $ to be chosen,
\[
	\chi_{R,\lambda} (x,v,\xi , \eta) = \chi \left( \frac{|v|^2 + |\eta|^{2}}{\lambda R} + \frac{|\nabla V (x)|^2 + |\xi|^2}{\lambda^3 R^3} \right).
\]
On the support of $ \varrho^{\prime}  $, $t$ is bounded from above and below by positive constants so that in the estimates (\ref{avecomposition}) and (\ref{sanscomposition}) on $ a_{t/\lambda} $, we may replace $ t/\lambda $ by $ 1 / \lambda $. It then follows from the Schur test that
\[
	\big\Vert\Op \big( a_{t /\lambda} (1 - \chi_{R,\lambda}) \big) \big\Vert_{L^p \rightarrow L^p} \rightarrow 0, \qquad R \rightarrow \infty,
\]
uniformly with respect to $ t $ in the support of $ \varrho^{\prime} $ and to $ \lambda m \geq 1  $. Thus, under the conditions (\ref{zexponentielle}) without any further condition on $ \Im(z) $, we have
\[
	\bigg\Vert\int_0^{+\infty}  e^{\frac{t}{m} \frac{z}{\lambda}} {\mathcal I}_{F_{t/\lambda}} \Op \big( a_{t/\lambda} (1 - \chi_{R,\lambda}) \big) \varrho^{\prime}(t)\,dt   \bigg\Vert_{L^p\rightarrow L^p} \leq \frac{1}{8},
\]
if $R$ is large enough. It remains to consider the contribution of $ a_{t / \lambda} \chi_{R,\lambda} $ with $R$ fixed. We shall proceed by non stationary phase estimates by writing the integral kernel of
\[
	\int_0^{+\infty}  e^{\frac{t}{m} \frac{z}{\lambda}} {\mathcal I}_{F_{t/\lambda}} \Op \big( a_{t/\lambda}  \chi_{R,\lambda} \big) \varrho^{\prime}(t)\,dt
\]
as
\begin{equation}\label{nonstationaryzero}
	(2 \pi)^{-2n} \int_0^{+\infty}  \iint_{\mathbb R^{2n}} e^{\frac{t}{m} \frac{{\rm Re }(z)}{\lambda}} e^{i \varphi } e^{- \Im(\psi)(t/\lambda,x,v,\xi,\eta)}  \chi_{R,\lambda}  (x,v,\xi,\eta ) \varrho^{\prime}(t)\, dt d \xi d \eta,   
\end{equation}  
with
\[
	\varphi = \frac{t}{m} \frac{\Im(z)}{\lambda} +  (x-y) \cdot \xi + (v-x) \cdot \eta - \frac{t}{\lambda}  v \cdot \xi  - \frac{1}{2m} \frac{t^2}{\lambda^2} (\nabla V(x)) \cdot \xi + \frac{t}{\lambda m}  (\nabla V (x)) \cdot \eta.
\]
This implies clearly that
\[
	\partial_t \varphi = \frac{\Im(z)}{\lambda m} - \frac{1}{\lambda} v \cdot \xi  + \frac{1}{m\lambda} (\nabla V )(x) \cdot \eta - \frac{t}{m} \frac{1}{\lambda^2} (\nabla V ) (x) \cdot \xi .
\]
Now, the key observation is that on the support of $ \chi_{\lambda,R} $, and with $t$ in the support of $ \varrho^{\prime} $, we have
\[
	|v| + |\eta| \lesssim_R \lambda^{\frac{1}{2}} , \qquad  |\nabla V (x)| + |\xi| \lesssim_R \lambda^{\frac{3}{2}},
\]
so that
\[
	\partial_t \varphi  = \frac{\Im(z)}{\lambda m} + \mathcal O_R (\lambda),
\]
and if $ M $ is large enough in (\ref{zexponentiellebis}),
\begin{equation}\label{clefnonstationnaire4} 
	\partial_t \varphi \geq \frac{M \lambda}{2}.
\end{equation}
This allows to integrate by part  in (\ref{nonstationaryzero}) using that $ (i \partial_t \varphi)^{-1} \partial_t $ leaves $ e^{i \varphi} $ invariant. More precisely,
\[
	(\ref{nonstationaryzero}) =(2\pi)^{-2n}\int_0^{+\infty} \iint_{\mathbb R^{2n}} e^{\frac{t}{m} \frac{{\rm Re }(z)}{\lambda}} e^{i \varphi } e^{- \Im(\psi)(t/\lambda,x,v,\xi,\eta)} 
	\frac{1}{\partial_t \varphi} \big(  A  \varrho^{\prime}(t) + B \varrho^{\prime \prime} (t) \big)\, dt d \xi d \eta,
\]
with
\begin{align*}
	A & = \left( \frac{i}{m} \frac{\Re(z)}{\lambda} - \frac{i \partial_t^2 \varphi}{\partial_t \varphi} - \frac{i}{\lambda} \big( \partial_t \Im(\psi) \big) (t/\lambda,x,v,\xi,\eta) \right)  \chi_{R,\lambda} (x,v,\xi,\eta) , \\
	B & = i \chi_{R,\lambda} (x,v,\xi,\eta) .
\end{align*}
In these expressions, it is important to note that any derivative of $ \partial_t \varphi $ does not depend on $ \Im(z) $ and it is not hard to check, using  (\ref{clefnonstationnaire4}), that for any $ \alpha,\beta,\gamma,\delta $ and on the support of $ \chi_{\lambda,R} $,
\[
	\bigg| \partial_x^{\alpha} \partial_v^{\beta} \partial_{\xi}^{\gamma} \partial_{\eta}^{\delta} \frac{1}{\partial_t \varphi} \bigg| \leq \frac{C_{\alpha, \beta, \gamma, \delta, R}}{M \lambda},
\]
where the important point is that the constant in the right-hand side is independent on $ M  $ and $ \lambda $. Note also that $ e^{t \Re(z)/(m\lambda)} {\mbox Re}(z)/\lambda $ is bounded uniformly with respect to $ t $ in the support of $ \varrho^{\prime}$ and to $ \Re(z)/ \lambda \lesssim 1 $. With these observations at hand, it follows from Theorem \ref{propositionpolaire} together with (\ref{pourphasenonstationnaire4}) and Lemma \ref{lelemmequiutiliseSchur}, that
\[
	\bigg\Vert\int_0^{+\infty}  e^{\frac{t}{m} \frac{z}{\lambda}} {\mathcal I}_{F_{t/\lambda}} \Op ( a_{t/\lambda}  \chi_{R,\lambda} ) \varrho^{\prime}(t)\,dt \bigg\Vert_{L^p\rightarrow L^p} \lesssim_R \frac{1}{M \lambda},
\]
hence less than $ 1/8 $ if $ M$ is large enough. The result follows.
\end{proof}

\begin{proof}[Proof of Theorem \ref{theoremespectreresolvante}] Using Propositions \ref{restevaluationnonzero} and \ref{pourRezadroite}, we have
\[
	\big\Vert \big( \lambda^{-1} (\ope - z) {\mathcal Q}_{\lambda}(z) - I \big) u\big\Vert_{L^p} \leq \frac{1}{2} \Vert u\Vert_{L^p}
\]
for all $u$ in the Schwartz space and parameters satisfying (\ref{zexponentiellebis}). In other words,
\[
	\lambda^{-1} (\ope - z) {\mathcal Q}_{\lambda}(z)u = (I + B_{\lambda}(z) )u,
\]
with $\Vert B_{\lambda}(z)\Vert_{L^p \rightarrow L^p} \leq 1/2 $ so that $ I + B_{\lambda}(z) $ can be inverted by Neumann series and we get the invertibility of $ \ope-z $ with
\[
	(\ope -z)^{-1} = \lambda^{-1} {\mathcal Q}_{\lambda} (z) (I + B_{\lambda}(z))^{-1}.
\]
Under the conditions (\ref{zexponentiellebis}) (actually here the condition on $ \Im(z) $ is not used), $ {\mathcal Q}_{\lambda}(z) $ is uniformly bounded on $ L^p(\mathbb R^{2n})$ according to Proposition \ref{Lpboundedness}. Now, if
\[
	|\mathrm{Im}(z)| \geq \frac{M}{m^2} ( \Re(z)^2 + c^2 ), \qquad \lambda = \frac{|\Im(z)|^{\frac{1}{2}}}{M^{1/2}},
\]
the conditions (\ref{zexponentiellebis}) are satisfied. This proves that  the set $ \{z \ | \ |\Im(z)| \geq C (\Re(z)^2 + 1) \} $, with $C$ large enough, is contained in the resolvent set of $ \ope $ and that, on this set, we have the resolvent estimate
\[
	\big\Vert(\ope -z)^{-1}\big\Vert_{L^p \rightarrow L^p} \lesssim \lambda^{-1} \lesssim |\Im(z)|^{-1/2} \lesssim \big(1+|\Re(z)|+|\Im(z)|^{1/2} \big)^{-1}.
\]
From Proposition \ref{p:generation-sgp}, we also know that $( e^{-\frac{t}{m} (\ope+c_0)} )_{t\geq0}$ is a semigroup of contractions for some $c_0 > 0$. Thus, there is no spectrum within $ \{z \ | \ \Re(z) < - c_0 \} $, and say for $ \Re(z) < - 2 c_0 $ it follows from (\ref{pourlagauchedelensembleresolvant}) (with $ \lambda = 1 $) that
\[
	\big\Vert (\ope-z)^{-1} \big\Vert_{L^p \rightarrow L^p}  \leq \frac{1}{m} \int_0^{+\infty} e^{\frac{t}{m}(\Re(z)+c_0)}\, dt \lesssim (1+|\Re(z)|)^{-1}.
\]
If $ |\Im(z)| \lesssim (\Re(z)^2 + 1 ) $, we may replace $ 1+ |\Re(z)| $ by $ 1 + | \Re(z)| + |\Im(z)|^{1/2} $ in the above estimate. The case where $ |\Im(z)| > C (\Re(z)^2+1) $ with $C$ large enough has been considered previously so the proof is complete. 
\end{proof}

\section{Functional analysis and spectral properties of \schtroumpf operators}\label{section:fafp}

This section on functional analysis and spectral theory collects several results concerning the spectral properties of a certain class of operators $\opeg$, which we later call \schtroumpf operators, and to which the Fokker--Planck operator belongs. The ultimate goals of the section are:
\begin{itemize}
	\item  to prove Proposition \ref{propositionprojecteurs} relating the Riesz projectors of $\opeg$ and $e^{-t \opeg}$ for \schtroumpf operators, 
	\item to deduce in Propositions~\ref{p:upper-bound-trace} and Corollary~\ref{c:asympt-spectrales} estimates of the number of eigenvalues of $\opeg$ from the asymptotics of $\tr(e^{-t\opeg})$ and $\|e^{-t\opeg}\|_{\tr}$ as $t\to 0^+$.
\end{itemize}
Examples of operators $\opeg$ belonging to the class of \schtroumpf operators (see Definition \ref{d:schtroumpf-op}) include all sectorial operators~\cite[Definition 4.1, p.96]{EngelNagel}, the Fokker--Planck operator $\opeg=\ope$ (see Proposition \ref{prop:fpsubsecto}), geometric Kramers--Fokker--Planck operators~\cite{Lebeau1, Lebeau2}, a general class of hypoelliptic quadratic differential operators including the nonselfadjoint harmonic oscillator $\opeg= -\Delta + i \vert x\vert^2$~\cite{OPPS12}, generalizations of the latter of the form $\opeg = -i\Delta  + V(x)$ with $c\langle x \rangle^{\gamma} \leq V (x)\leq C\langle x \rangle^{\gamma}$ and $\gamma,c>0$~\cite{LeautaudLerner}.

\subsection{Preliminaries: from compact semigroup to compact resolvent}

\begin{lemm}\label{e:compact-semigpe-resolvent}
Let $\opeg$ be an operator that generates a semigroup $(e^{-t\opeg})_{t\geq 0}$ on a Banach space $\mathcal{B}$. Assume that $e^{-t\opeg}$ is a compact operator for all $t>0$. Then, $\opeg$ has compact resolvent and, in particular, discrete spectrum.
\end{lemm}

Note that it is actually only needed $e^{-t\opeg}$ to be a compact operator for some $t>0$, but all applications we have are for all $t>0$ (and this simplifies slightly the proof).

\begin{proof}
There exist $\omega,C>0$ such that $\|e^{-t\opeg}\| \leq Ce^{\omega t}$ (see e.g.~\cite[Corollary~1.4, p.3]{MR710486}). For $\lambda >\omega$, we may write for $v \in \mathcal{B}$,
\[
	(\opeg +  \lambda )^{-1}v  = \int_0^{+\infty} e^{ - t (\opeg + \lambda)} v \, dt = \lim_{N \rightarrow +\infty} \int_{N^{-1}}^{N}  e^{ - t (\opeg + \lambda)}  v \, dt ,
\]
where the operator $v \mapsto  \int_{N^{-1}}^{N}  e^{ - t (\opeg + \lambda)}  v \, dt$ is  compact for each $N >0$ (as for instance the limit of its Riemann sums). Moreover, the last convergence is actually in operator norm since 
\begin{align*}
	\bigg\| (\opeg +  \lambda )^{-1}v - \int_{N^{-1}}^{N}  e^{ - t (\opeg + \lambda)}  v \, dt \bigg\| & \leq \int_{[0,N^{-1}]\cup[N,\infty)}  \| e^{ - t (\opeg + \lambda)} \|  \|v \| \, dt  \\
	& \leq  \int_{[0,N^{-1}]\cup[N,\infty)} C e^{ t (\omega - \lambda)} dt\,  \|v\| \\
	& \leq \frac{C}{\lambda-\omega} \|v \| \big( e^{-N^{-1}(\lambda-\omega)}-1 + e^{-N(\lambda-\omega)}\big) .
\end{align*}
Hence, the operator $(\opeg +  \lambda )^{-1}$ is compact, which concludes the proof of the lemma.
\end{proof}

\subsection{Preliminary resolvent estimates}
All along, we denote by $ \rho(\opeg) \subset  \C $ the resolvent set of the operator $\opeg$, and by $\Sp(\opeg):= \C\setminus \rho(\opeg)$ its spectrum.
We recall the following classical lemma, with a short proof.

\begin{lemm}\label{l:dist-spect0}
Let $\mathcal{B}$ be a Banach space and $(\opeg,D(\opeg))$ be a closed operator on $\mathcal{B}$. Then  
\begin{equation}\label{e:equiv-lem-dist-spec}
	z_0 \in \rho(\opeg)\quad  \implies \quad  \big\{ z \in \C\ \vert\  |z-z_0| < \|(\opeg-z_0)^{-1}\|^{-1} \big\} \subset \rho(\opeg) ,
\end {equation}
and, for all $z \in \C$ such that $|z-z_0| < \|(\opeg-z_0)^{-1}\|^{-1}$, 
\begin{equation}\label{e:estim-resolvent}
	\| (\opeg-z)^{-1} \| \leq \frac{\| (\opeg-z_0)^{-1}\|}{1- |z-z_0| \| (\opeg-z_0)^{-1}\| } = \frac{1}{\| (\opeg-z_0)^{-1}\|^{-1}- |z-z_0|}.
\end{equation}
\end{lemm}

\begin{proof}
To prove property~\eqref{e:equiv-lem-dist-spec}, we write for $z_0 \in \rho(\opeg)$,
\begin{equation}\label{e:perturb-againagain}
	\opeg-z = ( 1 + (z_0-z)(\opeg-z_0)^{-1} ) (\opeg-z_0) .
\end{equation}
Therefore, if $|z-z_0| < \|(\opeg-z_0)^{-1}\|^{-1}$, we deduce that $\| (z_0-z)(\opeg-z_0)^{-1}  \| <1$ and Neumann's series argument implies that $1 + (z_0-z)(\opeg-z_0)^{-1}:\mathcal{B} \to \mathcal{B}$ is invertible, with
\[
	\big(1 + (z_0-z)(\opeg-z_0)^{-1}\big)^{-1} = \sum_{k\in \N} \big( (z-z_0)(\opeg-z_0)^{-1}\big)^k,
\]
and in particular
\[
	\big\| \big(1 + (z_0-z)(\opeg-z_0)^{-1}\big)^{-1} \big\| \leq  \sum_{k\in \N} \big( |z-z_0| \| (\opeg-z_0)^{-1}\| \big)^k = \frac{1}{1- |z-z_0| \| (\opeg-z_0)^{-1}\| } .
\]
Thus,~\eqref{e:perturb-againagain} implies that $\opeg-z$ is bijective from $D(\opeg)\to \mathcal{B}$ and therefore $z \in \rho(\opeg)$ (since $\opeg$ is closed), hence~\eqref{e:equiv-lem-dist-spec}.
Coming back to~\eqref{e:perturb-againagain}, this yields 
\[
	(\opeg-z)^{-1} = (\opeg-z_0)^{-1}( 1 + (z_0-z)(\opeg-z_0)^{-1})^{-1},
\]
together with the estimate~\eqref{e:estim-resolvent}.
\end{proof}

A very important assumption is the following: there exist $y_0,C_0,\delta>0$ such that
\begin{equation}\label{e:imaginary-axis-bound}
	i \R_{y_0} := \big\{ iy \ \vert\ y\in\mathbb R,\,|y| \geq y_0\big\} \subset \rho(\opeg) \quad \text{ and } \quad \|(\opeg- z)^{-1}\|\leq C_0 \langle z\rangle^{-\delta}\quad\text{ for } z\in i \R_{y_0} .
\end{equation}

As a consequence of Lemma~\ref{l:dist-spect0}, we have the following direct corollary.

\begin{coro}\label{l:dist-spect}
Let $\mathcal{B}$ be a Banach space and $(\opeg,D(\opeg))$ be a closed operator on $\mathcal{B}$. Assume that there are $y_0,C_0,\delta>0$ such that~\eqref{e:imaginary-axis-bound} holds. Then, we have 
\begin{equation}\label{e:schtroumpf-complex}
	\big\{z\in \C\ \vert\ |\Im(z)|\geq y_0\quad \text{and}\quad |\Re(z)| < C_0^{-1} \langle \Im(z)\rangle^{\delta}\big\} \subset \rho(\opeg),  \\
\end{equation}
and also for every $z\in\mathbb C$ such that $|\Im(z)|\geq y_0$ and $|\Re(z)| \leq (2C_0)^{-1} \langle \Im(z)\rangle^{\delta}$,
\begin{equation}\label{e:schtroumpf-estimate} 
	\|(\opeg- z)^{-1}\|\leq 2C_0 \langle \Im(z) \rangle^{-\delta}.
\end {equation}
\end{coro}

\begin{proof}
We apply Lemma~\ref{l:dist-spect0} to $z=x+iy$ and $z_0=iy$ to obtain
\[
	\big\{x+iy \in \C\ \vert\ |y| \geq y_0, |x|< C_0^{-1} \langle y\rangle^{\delta} \big\} \subset  \rho(\opeg) , \\
\]
and on this set,
\[
	\| (\opeg-(x+iy))^{-1}\| \leq \frac{1}{C_0^{-1} \langle y\rangle^{\delta}-|x|}.
\]
The result follows when restricting this to $|x|\leq(2C_0)^{-1} \langle y\rangle^{\delta}$.
\end{proof}

\subsection{Eigenvalues of \texorpdfstring{$\opeg$}{} and \texorpdfstring{$e^{-t\opeg}$}{}: spectral mapping}

In this section, we furnish an elementary proof that the spectra of  $\opeg$ and $e^{-t\opeg}$ are related naturally. The proof that the associated Riesz projectors coincide for a restricted class of operators (namely, the \schtroumpf operators introduced below in Definition \ref{d:schtroumpf-op}) is achieved in Section~\ref{proofmultiplicite}.

\begin{lemm} \label{prop1spec}
Assume that $\opeg$ generates a semigroup $(e^{-t\opeg})_{t\geq 0}$ on a Banach space $\mathcal{B}$. Assume that for some $ z \in {\mathbb C} $ and some $ u  $ in $ D(\opeg) $ we have $ (\opeg-z) u = 0 $. Then, for each $ t \geq 0 $,
\[
	\big( e^{-t \opeg} - e^{-tz} \big) u = 0.
\]
In particular, if $z$ is an eigenvalue of $ \opeg $, then $ e^{-tz} $ is an eigenvalue of $ e^{- t\opeg} $.
\end{lemm}

\begin{proof} This follows by applying the identity, which holds for every $v \in D(\opeg)$,
\[
	\big( e^{-t \opeg}  - e^{-t z} \big)v = e^{-tz}  \bigg( - \int_0^t e^{-s \opeg} e^{zs} (\opeg-z)v \, ds \bigg),
\]
to $v=u \in D(\opeg)$.
\end{proof}

\begin{lemm} \label{prop2spec} 
Assume that $\opeg$ generates a strongly continuous semigroup $(e^{-t\opeg})_{t\geq 0}$ on a Banach space $\mathcal{B}$. Assume that, for some $ \zeta \in {\mathbb C} \setminus \{ 0 \} $, $ t_0 > 0 $ and $ u \ne 0$ in $\mathcal{B}$, 
\begin{equation}\label{e:eig-sgp-op}
	e^{-t_0 \opeg} u = \zeta u .
\end{equation}
Then, there exists $ z \in {\mathbb C} $ such that
\[
	\zeta = e^{-t _0 z} \quad \text{and} \quad \ z \ \text{is an eigenvalue of} \ \opeg .
\]
\end{lemm}

\begin{proof} Up to changing $ \opeg $ into $ \frac{t_0}{2 \pi} \opeg $, we may assume that $ t_0 = 2 \pi $. Write
\[
	\zeta = e^{2\pi ( x+iy)},
\]
with $ 2 \pi  x = \ln |\zeta| $ and some $ y \in {\mathbb R} $. Consider the continuous function of $t \in [0, 2 \pi ] $
\[
	u (t) = e^{ - t x - i t y} e^{-t \opeg} u,
\]
so that, using~\eqref{e:eig-sgp-op},
\[
	u (2 \pi) = e^{-2\pi x-i2\pi y}e^{-2\pi\opeg} u = u = u (0).
\]
Since $ u (0) = u \ne 0 $, and $u(\cdot) \in C^0(\R_+;\mathcal{B})$, $u(\cdot)$ it is not identically $ 0 $. Thus, one of its (vector-valued)  Fourier coefficients
\[
	u_{n_0} := \frac{1}{2 \pi} \int_0^{2 \pi} e^{- in_0 t} u (t)\, dt =  \frac{1}{2 \pi} \int_0^{2 \pi} e^{- in_0 t}  e^{ - t x - i t y} e^{-t \opeg} u \, dt 
\] 
is not zero: otherwise, if we denote by $u^*\in \mathcal{B}'$ a vector such that $\langle u^* ,u \rangle_{\mathcal{B}',\mathcal{B}} = \Vert u\Vert^2_{\mathcal{B}}$ (such a $u^*$ always exists thanks to the Hahn-Banach theorem), the Fourier coefficients of the scalar valued function $ \langle u^* ,u(t)\rangle_{\mathcal{B}',\mathcal{B}} $ should all vanish, which is impossible since this continuous function goes to $\Vert u\Vert^2_{\mathcal{B}} >0$ at $ 0 $). Notice moreover that, as an integral of the semigroup, $u_{n_0} \in D(\opeg)$ (see e.g.~\cite[Proposition~3.1.16, p.118]{ABHN}). Now, compute
\begin{align*}
	\big( \opeg - (-x-iy-in_0) \big) u_{n_0} & = \frac{1}{2 \pi} \int_0^{2\pi} \big( \opeg - (-x-iy-in_0) \big) e^{-t (\opeg +x + i y + i n_0)} u\, dt \\
	& = - \frac{1}{2 \pi}  \Big[ e^{-t (\opeg  + x + i y + i n_0  )} u \Big]_{t=0}^{t=2 \pi} \\
	& = - \frac{1}{2 \pi} \big( u (2 \pi) - u (0) \big)  = 0.
\end{align*}
Thus, $z = - x - i y -  i n_0 $ is an eigenvalue of $ \opeg $ with eigenvector $ u_{n_0} $ and $ e^{-2 \pi z} = e^{2\pi (x+iy)} = \zeta $, which concludes the proof of the lemma.
\end{proof}

Lemmas \ref{prop1spec} and \ref{prop2spec} imply in particular the following result.

\begin{prop} \label{prop0spec}
Assume that $\opeg$ generates a strongly continuous semigroup $(e^{-t\opeg})_{t\geq 0}$ on a Banach space $\mathcal{B}$, such that $e^{-t\opeg}$ is a compact operator for all $t>0$. 
Write the spectrum of $\opeg$ as the countable set $\{  z_j \ | \ j \in {\mathbb N} \} $. Then, for any $ t > 0 $,
\[
	\mathrm{Sp} \big(e^{-t \opeg}\big) = \{0\} \cup \big\{ e^{- t z_j} \ | \ j \in {\mathbb N} \big\}.
\]
\end{prop}

Proposition~\ref{prop0spec}  is known in greater generality see~\cite[Theorem~3.7, p.277]{EngelNagel} (and we do not use it here as such). 
However, it does not provide information on the multiplicities of eigenvalues, that is to say, dimensions of the associated Riesz projectors. Note that~\cite[Corollary~3.8, p.277]{EngelNagel} describes dimensions of the eigenspaces, which is not sufficient for our needs. 
We will see in Proposition \ref{propositionprojecteurs} below that, for \schtroumpf operators (such as the Fokker--Planck operators), the multiplicities behave naturally.
Before this, we need to introduce \schtroumpf operators and provide with different changes of contours for the associated functional calculus.

\subsection{\Schtroumpf operators and contours of integration} 

We now introduce the class of \schtroumpf\footnote{Note that another possible name would be hyposectorial. However, ``sectorial'' comes from the Latin as well as ``sub'' (whereas ``hypo'' comes from the greek).} operators as follows.

\begin{defi}[\Schtroumpf operators]\label{d:schtroumpf-op}
Let $\mathcal{B}$ be a Banach space and $(\opeg,D(\opeg))$ be an operator on $\mathcal{B}$. We say that $\opeg$ is \schtroumpf  of order $\delta>0$  if $-\opeg$ generates a {\em bounded} semigroup $(e^{-t\opeg})_{t\geq0}$ and if there are $y_0,C_0>0$ such that $\opeg$ satisfies~\eqref{e:imaginary-axis-bound}.  
 We simply say that $\opeg$ is \schtroumpf if it is \schtroumpf of order $\delta$ for some $\delta>0$.
\end{defi}

\begin{rema}
We could define a slightly larger class of operators by assuming that  $-\opeg$  generates a general strongly continuous semigroup (instead of a bounded semigroup) and replacing~\eqref{e:imaginary-axis-bound} by: there exists $x_0 \in \R$ such that
\[
	x_0 + i \R_{y_0} := x_0+ i \big\{ y \in \R\ \vert\ |y| \geq y_0\big\} \subset \rho(\opeg) \quad \text{ and }\quad \|(\opeg- z )^{-1}\|\leq C_0 \langle \mathrm{Im}(z)\rangle^{-\delta}\quad\text{ for } z \in x_0+i \R_{y_0}.
\]
Assuming that $\opeg$ satisfies the latter two assumptions, we notice that the operator $\opeg+\omega$ is a \schtroumpf operator in the sense of Definition~\ref{d:schtroumpf-op}, where $\omega$ is a growth bound of the semigroup generated by $-\opeg$ (\textit{i.e.} such that there is $C>0$ such that $\| e^{-t\opeg}\| \leq C e^{\omega t}$ for all $t\geq0$), as a consequence of Lemma~\ref{l:dist-spect} (and up to changing the values of $C_0, y_0$, but not $\delta>0$). As a consequence, all results proved below for \schtroumpf operators also hold for operators satisfying these (slightly weaker) assumptions.
\end{rema}

\begin{rema} Note that the class of \schtroumpf operators $\opeg$ (defined in Definition~\ref{d:schtroumpf-op}) is empty if $ \delta > 1 $, unless $  {\mathcal B}  = \{ 0 \} $. Indeed, if we let $ y = \varepsilon^{-1} $ with $ \varepsilon > 0 $, then the estimate in  ~\eqref{e:imaginary-axis-bound} implies that
\[
	( i \varepsilon \opeg + 1)^{-1} =  - iy (\opeg - i y)^{-1} \rightarrow 0 \quad \text{ as $\varepsilon=y^{-1} \rightarrow 0^+$},
\]
the convergence holding in operator norm. Thus, if $ u $ belongs to the domain of $ \opeg$,
\[
	u = ( i \varepsilon \opeg + 1)^{-1} ( i \varepsilon \opeg + 1) u \rightarrow 0, \quad \text{as $\varepsilon \rightarrow 0^+$},
\]
which shows that  $ D (\opeg) = \{ 0 \}  $. Since $\opeg$ generates a  strongly continuous semigroup, this domain is dense hence $ {\mathcal B} = \{ 0 \} $.
\end{rema}

\begin{rema}
As a first example (which however is not used in the present paper),  a sectorial operator (of angle $\alpha \in [0,\pi/2)$) is \schtroumpf of order $\delta=1$ (up to a change of sign if, as in~\cite[Definition~4.1, p.93]{EngelNagel}, operators are ``negative'' whereas we consider here ``positive'' operators).

Conversely, if $\opeg$ is a \schtroumpf operator of order $\delta = 1$, then Corollary~\ref{l:dist-spect} implies that there is $c_0\geq0$ (depending only on $y_0$ and $C_0$) such that $\opeg+c_0$ is sectorial of angle $\alpha \in [0,\pi/2)$ (depending only on $C_0$ and $c_0$, hence only on $C_0$ and $y_0$).
Definition~\ref{d:schtroumpf-op} of \schtroumpf operators is thus somehow a generalization of sectorial operators of angle $\alpha \in [0,\pi/2)$.
\end{rema}

As already mentioned, an important example of subsectorial operator is given by the Fokker--Planck operator we are interested in.

\begin{prop}\label{prop:fpsubsecto} Let $p\in(1,\infty)$. Assume that $V$ satisfies~\eqref{hypothesepotentiel}. Then, setting $c_0 := \frac{m\gamma n}{2}$, the Fokker--Planck operator $\ope+c_0$ is   subsectorial of order $\frac12$ on $L^p(\mathbb R^{2d})$.
\end{prop}

\begin{proof} First of all, recall from Proposition \ref{p:generation-sgp} that the strongly continuous semigroup $(e^{-t\ope})_{t\geq0}$ generated by the operator $\ope$ on $L^p(\mathbb R^{2n})$ satisfies for every $t\geq0$ and $u_0\in L^p(\mathbb R^{2n})$,
\[
    \big\Vert e^{-t\ope}u_0\big\Vert_{L^p}\le e^{c_0t}\Vert u_0\Vert_{L^p}.
\]
The semigroup generated by $\ope+c_0$ is therefore a bounded semigroup on $L^p(\mathbb R^{2n})$. Moreover, as a consequence of Theorem \ref{theoremespectreresolvante}, there exists $y_0>0$ such that $i\mathbb R_{y_0}\subset\rho(\ope+c_0)$, where the notation $i\mathbb R_{y_0}$ is introduced in~\eqref{e:imaginary-axis-bound}. The very same result also implies the following resolvent estimate for every $y\in\mathbb R$ such that $\vert y\vert\geq y_0$,
\[
    \big\Vert(\ope + c_0 -iy)^{-1}\big\Vert_{L^p\rightarrow L^p}\lesssim\big(1+\vert y\vert^{\frac12}\big)^{-1}.
\]
By using the Minkowski inequality $(1+y^2)^{1/4}\le(1+\vert y\vert^{1/2})$, we therefore deduce that for every $y\in\mathbb R$ satisfying $\vert y\vert\geq y_0$,
\[
    \big\Vert(\ope + c_0 -iy)^{-1}\big\Vert_{L^p\rightarrow L^p}\lesssim\langle y\rangle^{-\frac12},
\]
which shows that the operator $\ope+c_0$ satisfies the assumption~\eqref{e:imaginary-axis-bound} with $\delta = 1/2$. According to Definition~\ref{d:schtroumpf-op}, we have proven that $\ope+c_0$ is a subsectorial operator on $L^p(\mathbb R^{2n})$, as expected.
\end{proof}

Corollary~\ref{l:dist-spect} shows that a \schtroumpf operator $\opeg$ satisfies~\eqref{e:schtroumpf-complex} (spectrum-free region) and~\eqref{e:schtroumpf-estimate} (resolvent estimate in this region).  The condition that $ e^{-t\opeg} $ is bounded in Definition~\eqref{d:schtroumpf-op} implies that $(\opeg-z)^{-1}=\int_{\R_+}e^{-t\opeg} e^{zt}\, dt$ for $\Re(z)<0$ and 
\begin{equation}\label{e:growth-sgpe-resolvent}
	\big\{z \in \C \ \vert\ \Re(z) <0 \big\} \subset \rho(\opeg)  \quad \text{ and } \quad \| (\opeg-z)^{-1} \| \leq \frac{C}{\Re(-z)} \quad \text{ for } \Re(z) < 0 .
\end{equation}
The fact that the spectrum of a \schtroumpf operator $\opeg$ is contained in the right half plane $\{ \Re(z) \geq 0 \} $ will be used extensively below.

In what follows, we express the semigroup as a contour integral and discuss contour deformations. We start from the approach of~\cite{HerauNier}, but supply more details on the contours and their deformations since we shall not only give an integral representation of the semigroup but also study the Riesz projectors to compare the multiplicities of the eigenvalues of $ \opeg $ and $ e^{-t\opeg} $. Similar considerations for the multiplicities  are presented in~\cite[pp. 52-53]{BismutLebeau} for the geometric Fokker--Planck operator on the cotangent bundle of a compact manifold, but we provide with proofs in our abstract framework for the sake of selfcontainedness.

For $ C_2 > 1 $ and $ N >  1 $ to be fixed below, we define the contour
\[
	\Gamma = \bigg\{  z \in {\mathbb C} \ | \  \Re(z) \geq - 1 + C_2^{- \frac{1}{N-1}},  \ \ |z + 1| = C_2 (\Re(z)+1)^N  \bigg\} ,
\]
the lower bound on $ \Re(z) $ ensuring that we pick only one connected component of the set $ \{ |z+1| = C_2 |\Re(z+1)|^N \} $ (see Figure~\ref{f:figure-Gamma}). 
We also define the open set
\[
	{\mathcal S}_{\Gamma} := \bigg\{ z \in {\mathbb C} \ | \ \Re(z) > -1 + C_2^{- \frac{1}{N-1}}, \ \  |z+1| < C_2 (\Re(z)+1)^N  \bigg\},
\] 
which is located to the right of $ \Gamma $ in the complex plane.

\begin{center}
\includegraphics[scale=0.25]{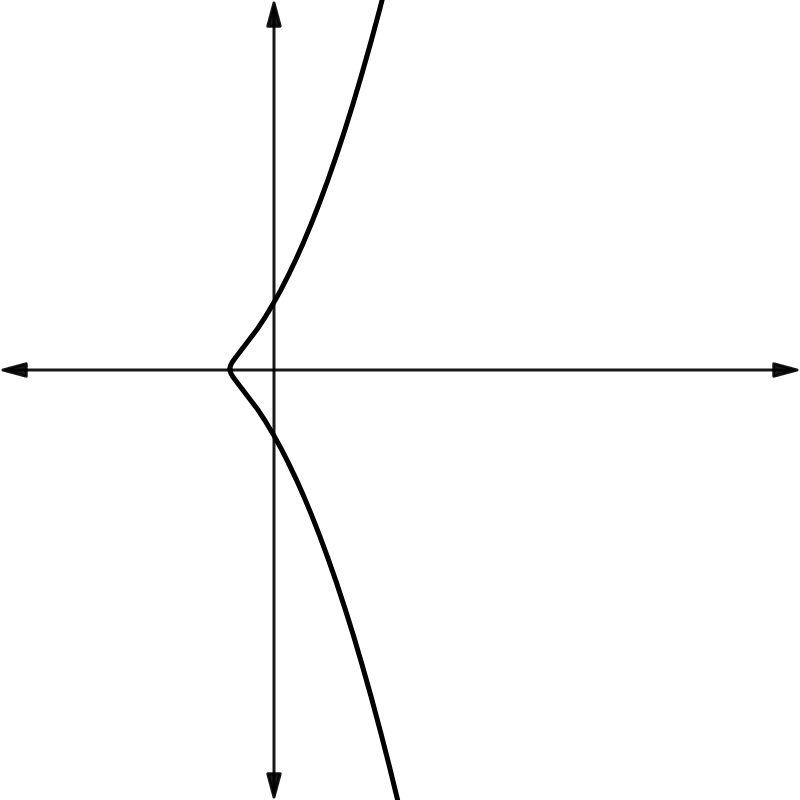}
\captionof{figure}{The contour $\Gamma$}
\label{f:figure-Gamma}
\end{center}

We next recall that if $\opeg$ is a \schtroumpf operator with constants $C_0,\delta>0$ in the sense of Definition~\ref{d:schtroumpf-op}, then it satisfies~\eqref{e:schtroumpf-complex} (spectrum-free region) and~\eqref{e:schtroumpf-estimate}  together with~\eqref{e:growth-sgpe-resolvent} (spectrum in the half plane $\{ \Re(z) \geq 0 \}$). 
Therefore, recalling that $0<\delta\leq 1$ and fixing $N$ such that 
\begin{equation}\label{varepsilonetN}
	\frac{1}{N} <\delta \leq 1, 
\end{equation}
there exists $C_2>0$ large enough (depending on $\delta$, $C_0$ and $N$) such that 
\begin{equation}\label{estimeeresolvanteHerauNier}
	\Sp(\opeg) \subset {\mathcal S}_{\Gamma} \qquad\text{and} \qquad  
	\| (\opeg - z)^{-1} \| \leq C \langle z \rangle^{- \delta}   \quad \mbox{for} \  z  \in \left\{ \Re(z) \geq - 1 \right\} \setminus {\mathcal S}_{\Gamma}. 
\end{equation}
Following~\cite{HerauNier}, we introduce for $t>0$
\begin{equation} \label{defE(t)}
	E (t) = \frac{i}{2  \pi} \int_{\Gamma} e^{-tz} (\opeg - z)^{-1}\, dz ,
\end{equation}
where $ \Gamma $ is oriented from top to bottom, {\it i.e.} in the sense of decreasing $ \Im(z) $.

\begin{lemm} \label{chaleurdroite} 
Assume that $\opeg$ is a \schtroumpf operator on a Banach space $\mathcal{B}$. Then, the integral in~\eqref{defE(t)} converges for any $t>0$ in the operator topology.
Moreover, for any $u \in D (\opeg)$, we have $t \mapsto E(t)u \in C^1(\mathbb R^*_+ ; \mathcal{B})$ and also that for every $t>0$,
\[
	\partial_t E (t) u + E (t) \opeg u = 0.
\]
\end{lemm}

\begin{proof} For this result, the orientation of $ \Gamma $ does not matter. We record first that $ \Gamma  \cap \{ \Re(z) > 0 \} $ can be parametrized as $  z(s) = x (s) + i y (s) $ with $ s \in {\mathbb R}^* $ and
\begin{equation}\label{parametrizationGamma}
	x (s) = |s|,  \qquad  y (s) = \mbox{sign}(s) \sqrt{ C^2_2 (|s|+1)^{2N}  - (|s|+1)^2 },
\end{equation} 
so that, when $s$ is large, we have
\[
	\langle z (s)\rangle \approx \langle s \rangle^N , \qquad \bigg| \frac{dz}{ds}(s) \bigg| \lesssim \langle s \rangle^{N-1} ,
\]
where $ \approx $ means that the quotient of the two sides is bounded from above and below by positive constants.
In particular, for any $t > 0$, (\ref{estimeeresolvanteHerauNier}) implies that for every $s \in {\mathbb R}^*$,
\[ 
	\bigg\Vert e^{-tz(s)} (\opeg - z(s))^{-1} \frac{dz}{ds}(s) \bigg\Vert 
	\lesssim_t e^{-t|s|} \langle s \rangle^{-\delta N} \langle s \rangle^{N-1}
	= e^{-t|s|} \langle s \rangle^{(1-\delta) N-1},
\] 
along the parametrization (\ref{parametrizationGamma}) so that the integral in (\ref{defE(t)}) is convergent. The proof that $ E (t) $ can be differentiated in $t$ under the integral sign follows from the above estimates too. If $u \in D (\opeg)$, we then have
\[
	\partial_t E (t) u + E (t) \opeg u =  \bigg(  \frac{i}{2  \pi} \int_{\Gamma} e^{-tz}\,  dz \bigg) u,
\]
where the integral vanishes; to see this, it suffices to close the curve $ \Gamma_R := \Gamma \cap \{ |\Im(z)| \leq R \} $ by an arc of circle centered at $0$ and located in $\{ \Re(z) \geq 0 \}$ (with radius going to infinity as $R$ does) to check that $ \int_{\Gamma_R} e^{-tz}\, dz  $ goes to zero as $ R \rightarrow +\infty $ since this is the case for the integral on the arc of circle. This completes the proof.
\end{proof}

We next need to check that $ E (t) $ is nontrivial; it satisfies the same differential equation as $ e^{-t \opeg} $ but we cannot exclude so far that $ E (t) =0 $ for all $t>0$. Proposition \ref{fort=0} below shows this does not happen.

Let $ \gamma $ be the deformation of $ \Gamma $ obtained by deforming   $ \Gamma \cap \{ \Re(z) \geq 0 \} $ to the imaginary  axis and leaving $ \Gamma \cap \{ \Re(z) \leq 0 \} $ unchanged (here again $ \gamma $ is oriented from $ i \infty $ to $ - i \infty $):

\begin{center}
\includegraphics[scale=0.25]{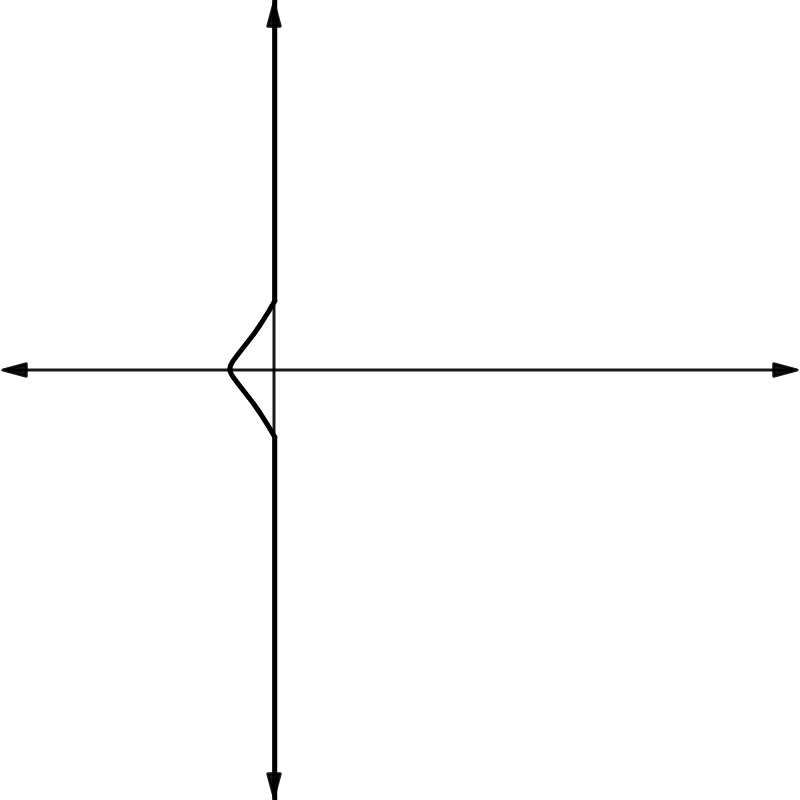}
\captionof{figure}{The contour $ \gamma  $  }
\label{f:figure-gamma}
\end{center}

\begin{prop} \label{fort=0} Assume that $\opeg$ is a \schtroumpf operator on a Banach space $\mathcal{B}$ and let $ \varepsilon > 0 $. Then, with $E (t)$ defined in~\eqref{defE(t)} and $\gamma$ on Figure~\ref{f:figure-gamma}, we have on the one hand for every $t>0$,
\begin{equation}\label{fort>0}
	E (t) (\varepsilon \opeg + 1)^{-1} = \frac{i}{2  \pi} \int_{\gamma} e^{-tz} ( \opeg-z)^{-1} \frac{1}{\varepsilon z + 1}\, dz,
 \end{equation}
and on the other hand, for every $t \geq 0$,
\begin{equation}\label{fort>00}
	e^{-t \opeg} (\varepsilon \opeg + 1)^{-1} = \frac{i}{2  \pi} \int_{\gamma} e^{- t z} (\opeg - z)^{-1} \frac{1}{\varepsilon z + 1}\, dz.
\end{equation} 
\end{prop}

The original proof of this result is given in~\cite{HerauNier} but we supply some details to record some useful estimates to be used further in the text. Note also that $ (\varepsilon \opeg + 1 )^{-1}  $ is well defined since the spectrum of 
$   \opeg $ is contained in the right half plane.

\begin{lemm} \label{regularisationHerauNier} For all $ t > 0 $ and $ \varepsilon > 0 $,
\[
	E (t) (\varepsilon \opeg + 1)^{-1} = \frac{i}{2  \pi} \int_{\gamma} e^{-tz} ( \opeg-z)^{-1} \frac{1}{\varepsilon z + 1}\, dz .
\]
\end{lemm}

\begin{proof} For $ R  $ large enough, we denote by $ \Gamma_R $ and $ \gamma_R $ the curves 
\[
	\Gamma_R = \Gamma \cap \big\{ |\Im(z)| \leq R \big\}  \quad\text{and}\quad  \gamma_R = \gamma \cap \big\{ |\Im(z)| \leq R \big\} .
\]

\begin{center}
\includegraphics[scale=0.25]{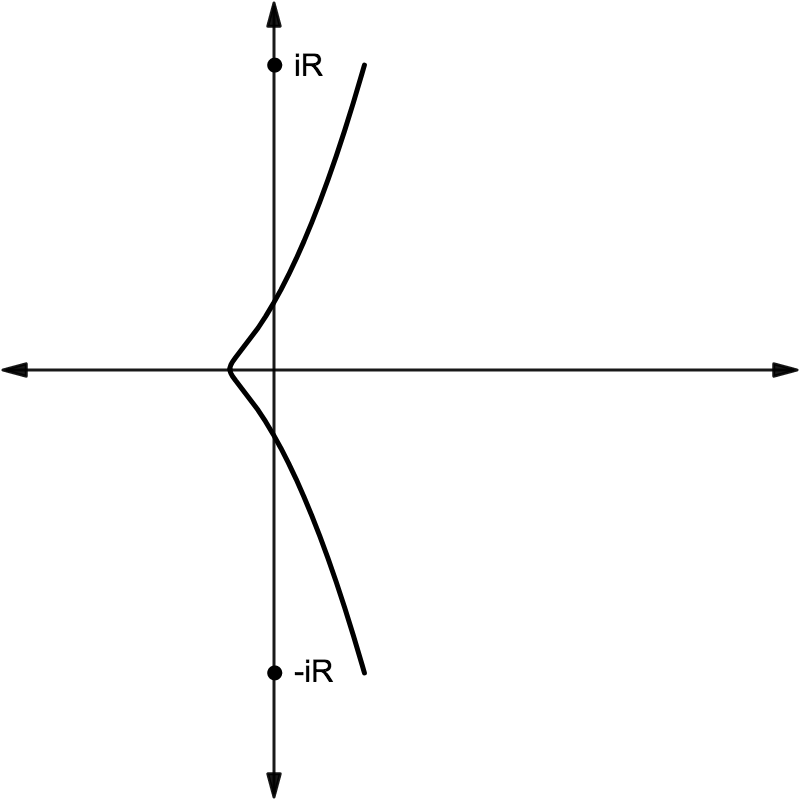}
\includegraphics[scale=0.25]{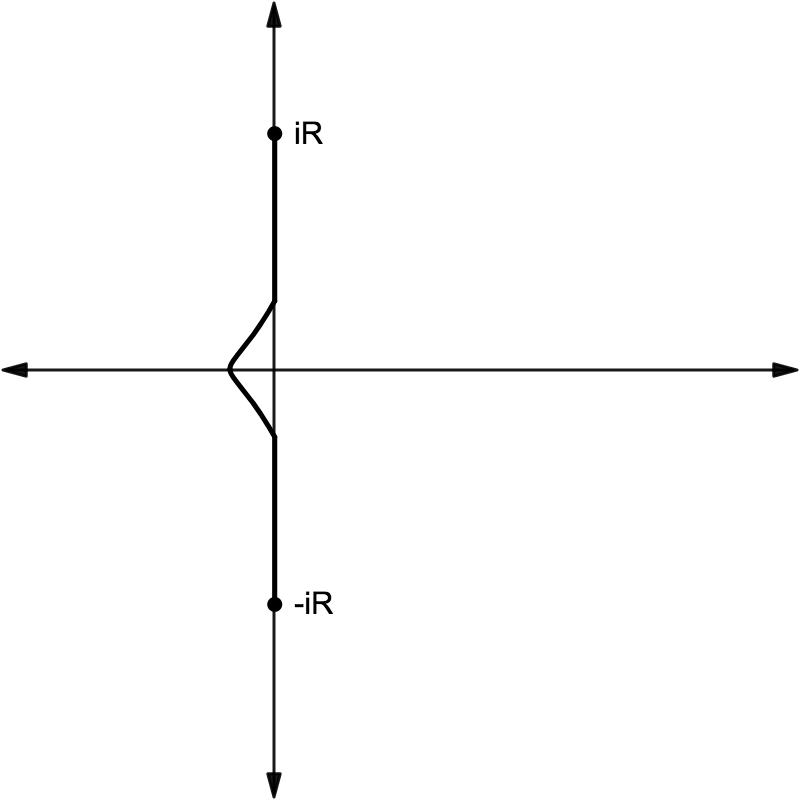}
\captionof{figure}{$ \Gamma_R $ (left) and $ \gamma_R $ (right)}
\label{f:figuregamma0}
\end{center}

\noindent Write first 
\begin{equation}\label{pouremplacerGammagamma}
	E (t) (\varepsilon \opeg + 1)^{-1} = \lim_{R \rightarrow + \infty} \frac{i}{2  \pi} \int_{\Gamma_R} e^{-tz} (\opeg-z)^{-1} (\varepsilon \opeg + 1)^{-1}\, dz . 
\end{equation}
Let $ H_{ R} $ (resp. $ H_{-R} $) be the horizontal segment connecting the top (resp. bottom) of $ \Gamma_R $ to the top (resp. bottom) of $ \gamma_R $. They correspond to points $ z = x \pm i R  $ with  $ x \geq 0 $ such that
\[
	R^2 \geq C^2 (x+1)^{2N} - (x+1)^2,
\]
according to the equation of $ \Gamma $. In particular, along $ H_{\pm R} $, we have $ x \lesssim R^{\frac{1}{N}} $ and
\[
	\bigg\Vert \int_{H_{\pm R}} e^{-t z} (\opeg-z)^{-1} (\varepsilon \opeg + 1)^{-1}\, dz \bigg\Vert 
	\lesssim \langle R \rangle^{-\delta} R^{\frac{1}{N}} \rightarrow 0 \quad \text{as $R \rightarrow + \infty$},
\]
using~\eqref{varepsilonetN} and  (\ref{estimeeresolvanteHerauNier}). This allows to replace $ \Gamma_R $ by $ \gamma_R $ within (\ref{pouremplacerGammagamma}). Using the resolvent identity,
\[
	(\opeg-z)^{-1} (\varepsilon \opeg + 1)^{-1} = \frac{1}{\varepsilon z + 1} ( ( \opeg-z)^{-1} -  \varepsilon (\varepsilon \opeg + 1)^{-1}),
\]
and the fact that, since $ t > 0 $,
\[
	\int_{\gamma_R} e^{-tz} (\varepsilon z + 1)^{-1}\, dz \rightarrow 0 \quad \text{as $R \rightarrow + \infty$},
\]
from a routine argument by replacing $ \gamma_R $ by the half circle $ \{ |z| = R \} \cap \{ \Re(z) \geq 0 \} $, we get
\[
	E (t) (\varepsilon \opeg + 1)^{-1} = \lim_{R \rightarrow + \infty} \frac{i}{2  \pi} \int_{\gamma_R} e^{-tz} (\opeg-z)^{-1} \frac{1}{\varepsilon z + 1}\, dz. 
\]
For $z = i y $  on the vertical half-lines of $ \gamma$, we have 
\begin{equation}\label{constanteimplictejusquazero}
	\Vert(\opeg - z)^{-1} (\varepsilon z + 1)^{-1}\Vert \lesssim \langle y \rangle^{-\delta- 1},
\end{equation}
so the integral along $ \gamma $ is convergent. This completes the proof.
\end{proof}

The operator $ (\varepsilon \opeg + 1)^{-1} $ is a usual approximation of identity that allows several integrals to converge (as the above on $ \gamma $).
Before proving Proposition \ref{fort=0}, we justify this approximation.

\begin{lemm} 
Let $\mathcal{B}$ be a Banach space and $\opeg$ be an operator with dense domain satisfying~\eqref{e:growth-sgpe-resolvent}. Then for all $\varepsilon>0$, the operator $(\varepsilon \opeg+ 1)^{-1}:\mathcal B\rightarrow\mathcal B$ is bounded and for any $ u\in \mathcal{B}$, $ (\varepsilon \opeg+ 1)^{-1} u \rightarrow u $ in $\mathcal{B}$ as $ \varepsilon \rightarrow 0^+$.
\end{lemm}

\begin{proof} First, using~\eqref{e:growth-sgpe-resolvent}, we have $\Vert(\varepsilon \opeg+1)^{-1}\Vert \leq C$,
uniformly for $\eps^{-1}>0$.
It is thus sufficient to prove the result when $ u\in D(\opeg) $ which is dense in $\mathcal{B}$. For  $ u\in D(\opeg) $,  observe that
\[
	(\varepsilon\opeg + 1)^{-1} u - u = - (\varepsilon \opeg+ 1)^{-1} \varepsilon \opeg u\rightarrow0 \quad\text{as $\varepsilon \rightarrow 0^+$}.
\]
This completes the proof.
\end{proof}

\begin{proof}[Proof of Proposition~\ref{fort=0}] The identity~\eqref{fort>0} follows from Lemma \ref{regularisationHerauNier}. Since both sides of~\eqref{fort>00} are cancelled by $ \partial_t + \opeg $ on $ \{ t > 0 \} $ according to~\eqref{fort>0} and Lemma \ref{chaleurdroite} for the right-hand side, it suffices to check that they coincide at $t=0$ to get the result. In other words, to prove~\eqref{fort>00}, it suffices to show 
\[
	\frac{i}{2  \pi} \int_{\gamma}  (\opeg - z)^{-1} \frac{1}{\varepsilon z + 1}\, dz =  (\varepsilon \opeg +1)^{-1} .
\]
Write the integral as the limit of the same integral taken on $ \gamma_R $ as $ R \rightarrow +\infty $ (see Figure \ref{f:figuregamma0}). Then, by closing $ \gamma_R $ by an arc of circle $ A_R $ centered at $0$ and with radius going to infinity with $R$,  and more importantly located in the left plane $ \{ \Re(z) \leq 0 \} $, the Cauchy formula yields
\[
	\frac{i}{2  \pi} \int_{\gamma_R}  (\opeg - z)^{-1} \frac{1}{\varepsilon z + 1}\, dz = (\varepsilon \opeg +1)^{-1}  +  \frac{i}{2  \pi} \int_{A_R}  (\opeg - z)^{-1} \frac{1}{\varepsilon z + 1}\, dz
\]
if $R$ is large enough since the pole $ z = - 1 / \varepsilon $ lies inside the region limited by the contour $ \gamma_R \cup A_R $. Note here that  we use the orientation inherited from those of $ \gamma $ and $ \Gamma $ so that $  \gamma_R \cup C^-_R $ is oriented clockwise. We can then let $ R \rightarrow +\infty $ so that the integral over $ A_R $ goes to zero, using~\eqref{estimeeresolvanteHerauNier} to estimate the resolvent. This completes the proof.
\end{proof}

For the later (and main) purpose of studying Riesz projectors  of $ \opeg $ and $ e^{-t\opeg} $, the interest of Proposition~\ref{fort=0} is the following result.

\begin{prop} \label{0continuation} 
Assume that $\opeg$ is a \schtroumpf operator on a Banach space $\mathcal{B}$. Let $\gamma$ as on Figure~\ref{f:figure-gamma}, and $ t > 0 $. Then, there is $M_t>1$ such that for all complex number $ \zeta $ such that $ |\zeta| > M_t$ and all $ \varepsilon > 0 $, one has
\begin{equation}\label{avantcontinuationanalytique}
	( e^{-t \opeg} - \zeta )^{-1} (\varepsilon \opeg+1)^{-1}  = \frac{i}{2  \pi} \int_{\gamma}  \frac{1}{ e^{-tz} - \zeta } (\opeg - z)^{-1} \frac{1}{\varepsilon z + 1}\, dz. 
\end{equation}
\end{prop}

\begin{proof} Starting with $ |\zeta| > \Vert e^{- t \opeg}\Vert $, we can write 
\[
	(e^{-t \opeg}-\zeta)^{-1} = - \frac{1}{\zeta} \big(1 - \zeta^{-1}e^{-t \opeg}\big)^{-1} = -  \sum_{k \geq 0} \frac{1}{\zeta^{k+1}} e^{- tk \opeg} .
\]
From (\ref{fort=0}), we have 
\[
	e^{- kt \opeg} (\varepsilon \opeg + 1)^{-1} = \frac{i}{2  \pi} \int_{\gamma} e^{- k t z} (\opeg - z)^{-1} \frac{1}{\varepsilon z + 1}\, dz ,  
\]
and we infer from the last two lines that
\begin{equation}\label{permutationsommeintegrale}
	( e^{-t \opeg} - \zeta )^{-1}  (\varepsilon \opeg + 1)^{-1} =  - \sum_{k \geq 0 }\frac{1}{2 i \pi} \int_{\gamma}  \frac{1}{\zeta^{k+1}} e^{- t kz} (\opeg - z)^{-1} \frac{1}{\varepsilon z + 1}\, dz . 
\end{equation} 
To conclude, it suffices to examine how to  swap $ \sum_k $ and $ \int_{\gamma} $ since $ - \sum_{k} \zeta^{-k-1} e^{-tkz} $ is equal to $ (e^{-tz}-\zeta)^{-1} $. Indeed, on the  vertical half lines of $ \gamma $, we have
\[
	\sum_k \bigg| \frac{1}{\zeta^{k+1}} e^{-tk z} \bigg| =  \sum_k \bigg| \frac{1}{\zeta^{k+1}} \bigg|  = \frac{1}{|\zeta| - 1}
\]
assuming, as we may, that $ |\zeta| > 1 $, while on $ \gamma \cap \{ \Re(z) \leq 0 \} $, we have
\[
	\sum_k \bigg| \frac{1}{\zeta^{k+1}} e^{-tk z} \bigg| \leq  \sum_k  \frac{e^{tk}}{|\zeta|^{k+1}}   = \frac{1}{|\zeta| - e^t }
\]
if we strengthen the assumption on $ \zeta $ by considering $ |\zeta| > e^t $. Together with (\ref{constanteimplictejusquazero}) the last two estimates allow to use the dominated convergence theorem to swap the sum and the integral in (\ref{permutationsommeintegrale}). This concludes the proof of the proposition.
\end{proof}

From now on, we will in addition assume that the \schtroumpf operator $ \opeg $ has compact resolvent, and hence discrete spectrum.
The next step is to get another expression of the right-hand side of (\ref{avantcontinuationanalytique}) allowing to get a formula for $ (e^{-t \opeg}-\zeta)^{-1} $ that is also valid for $ \zeta $ in a small circle around an arbitrary  eigenvalue of $ e^{-t \opeg} $. The idea is to make the following contour change: given any real number $ a > 0 $ such that none of the real parts of the eigenvalues of $ \opeg $ are equal to $a$, and given $ r > 0 $ small enough, we let $ \Lambda_{a,r} $ be the union of the line $ \{ \Re(z) = a \} $ and of the (finitely many) circles of radius $ r $ centered at the eigenvalues of $ \opeg $ in the region $ \{\Re(z)< a \} $, as in Figure~\ref{f:figure-lambda}. As before for $ \gamma $, the line $ \{ \Re(z) = a \} $ is oriented from top to bottom, and to get the expected contour deformation  the small circles of radius $r$ are oriented counterclockwise.

\begin{center}
\includegraphics[scale=0.25]{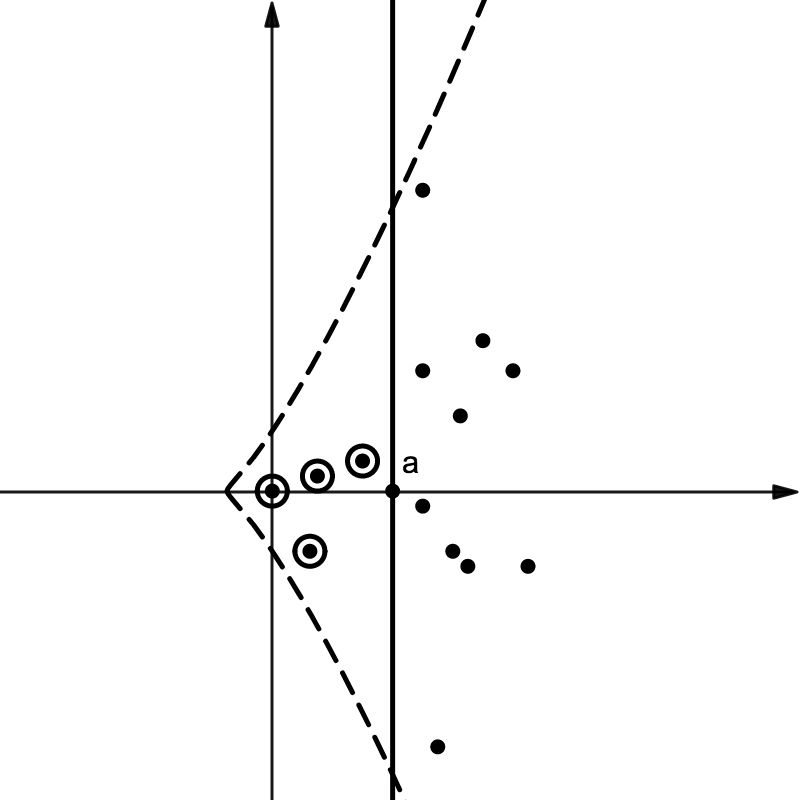}
\captionof{figure}{Contour $ \Lambda_{a,r} $ (union of $ $ circles of radius $r$ and the line $ \Re(z) = a $)}
\label{f:figure-lambda}
\end{center}

\begin{prop} \label{1continuation}  
Assume that $\opeg$ is a \schtroumpf operator on a Banach space $\mathcal{B}$, with compact resolvent. Let $\gamma$ be as on Figure~\ref{f:figure-gamma}.
Let $t > 0 $ and $ a > 0 $ and $\Lambda_{a,r}$ as in Figure~\ref{f:figure-lambda}.  Then, there is $r_t> 0$ small enough such that for all $ r  \in (0,r_t) $, for all $ \zeta $ satisfying $ |\zeta| > 2 $ and all $ \varepsilon > 0 $, we have
\[
	\frac{i}{2  \pi} \int_{\gamma}  \frac{1}{ e^{-tz} - \zeta } (\opeg - z)^{-1} \frac{1}{\varepsilon z + 1}\, dz  =  \frac{i}{2  \pi} \int_{\Lambda_{a,r}}  \frac{1}{ e^{-tz} - \zeta } (\opeg - z)^{-1} \frac{1}{\varepsilon z + 1}\, dz .
\]
\end{prop}

\begin{proof} We write the integral over $ \gamma $ as the limit of the integral over $ \gamma_R $  (recall Figure 3). We close $ \gamma_R $ with the horizontal segments $ [0,a] \pm i R $ and the vertical one $ a + i [-R,R] $. The integrals over the horizontal segments go to $ 0 $ as $ R \rightarrow +\infty $  using that
\[
	\Vert(\opeg - x \pm i R)^{-1} (\varepsilon (x \pm i R) + 1)^{-1} \Vert \lesssim \langle R \rangle^{- \delta - 1}
\] 
uniformly with respect to $ x \in [0,a] $, using (\ref{estimeeresolvanteHerauNier}). The integral over $ a + i [-R,R] $ has a limit as $ R \rightarrow +\infty $ thanks to the analogue of (\ref{constanteimplictejusquazero}) on  the vertical line $ \{ \Re(z) = a \} $ (we use the resolvent estimate (\ref{estimeeresolvanteHerauNier}) near infinity and the continuity of  the resolvent on compact subsets of the line which does not meet the spectrum). Everywhere, the factor $  ( e^{-tz} - \zeta )^{-1}$ is harmless since $ |e^{-tz}| \leq 1 $ while $ |\zeta| > 2 $. Now the integral of this closed contour is equal to the integral over the circles of $ \Lambda_{a,r} $ (they are all inside the contour by choosing $r$ small enough). Furthermore, $r_t$ may be chosen small enough  in a way that given any circle centered at  $z_j$ with $ \Re(z_j) \geq 0 $ and with radius $r  \leq r_t $, we have $ | e^{-tz} | \leq e^{t r} \leq 3/2 $ for $z$ in this circle so that again $ (e^{-tz} - \zeta )^{-1} $ has no singularity in $z$ over the domain of integration. The result follows by letting $ R \rightarrow +\infty $.
\end{proof}

\subsection{Riesz projectors of \texorpdfstring{$\opeg$}{} and \texorpdfstring{$e^{-t\opeg}$}{} for \schtroumpf operators: mutliplicities}
\label{proofmultiplicite}

We now come back to the relation between the spectra of $\opeg$ and $e^{-t\opeg}$, stated in Proposition~\ref{prop0spec}. We prove that for \schtroumpf operators such that $e^{-t\opeg}$ is compact, the associated multiplicities behave naturally.
Note that compactness of $e^{-t\opeg}$ implies compactness of the resolvent of $\opeg$ according to Lemma~\ref{e:compact-semigpe-resolvent}, hence all results above (including Lemma~\ref{1continuation}) apply.
Given any non zero eigenvalue $ \zeta_{\alpha} $ of $ e^{-t \opeg} $ (resp. any eigenvalue $z_j$ of $ \opeg $), we denote the associated Riesz projector, following \eqref{eq:projector}, by
\[
	\Pi (e^{-t \opeg},\zeta_{\alpha}) = \frac{i}{2 \pi} \oint (e^{-t \opeg} - \zeta)^{-1}\, d\zeta \qquad 
	\bigg( \mbox{resp}. \   \Pi (\opeg, z_j) = \frac{i}{2 \pi} \oint (\opeg - z)^{-1}\, dz \bigg),
\]
where the integral is taken over a small circle, oriented counterclockwise, around the eigenvalue. Recall from~\eqref{eq:multiplicite} that the  multiplicities are then defined as
\[
	\mult (\zeta_{\alpha}) = \rank\Pi (e^{-t \opeg}, \zeta_{\alpha}), \qquad \mult (z_j) = \rank\Pi (\opeg,z_j).
\]
A main purpose of this section is to prove the following natural result.

\begin{prop} \label{propositionprojecteurs} 
Assume that $\opeg$ is a \schtroumpf operator on a Banach space $\mathcal{B}$, such that $e^{-t\opeg}$ is a compact operator for all $t>0$. 
If $ \zeta_{\alpha} $ is a non-zero eigenvalue of $ e^{-t \opeg} $ and $\{z_{1,\alpha} , \ldots , z_{J,\alpha}\}$ is the set of all finitely many distincts eigenvalues of $ \opeg $ such that\footnote{recall Proposition \ref{prop0spec}}
\[
	\zeta_{\alpha} = e^{-t z_{j,\alpha}}, \qquad 1 \leq j \leq J,
\]
then
\[
	\Pi (e^{-t \opeg},\zeta_{\alpha}) = \sum_{j=1}^J  \Pi \big(\opeg, z_{j,\alpha} \big).
\]
In particular,
\[
	{\rm mult} (\zeta_{\alpha}) = \sum_{j=1}^J {\rm mult} (z_{j,\alpha}).
\]
\end{prop}

To prove Proposition~\ref{propositionprojecteurs}, we need some elementary (but a little technical) preparation that will be crucial later to evaluate certain integrals. 

\begin{lemm} \label{lemmeconnexite} Under the assumptions of Proposition~\ref{propositionprojecteurs}, let $ a > 0 $, $ t > 0 $ and $ (z_k)_{k \in K} $ be the finite number of (distinct) eigenvalues of $\opeg$ satisfying $ \Re(z_k) < a $. Then, for all $ 0 < \rho \ll 1  $, the open set
\[
	\Omega := \big\{ | \zeta | > e^{-ta} \big\} \cap \bigcap_{k \in K} \big\{ |\zeta - e^{-tz_k}| > \rho \big\}
\]
is connected. 
\end{lemm}

This lemma just means that the complement of finitely many small enough disks within $ \{ |\zeta| > e^{-ta} \} $ is connected; this is intuitively obvious, we supply a proof for completeness.

\begin{proof}[Proof of Lemma \ref{lemmeconnexite}] It suffices to show that any point $ \zeta $ in $ \Omega $ can be connected by a continuous curve contained in $ \Omega $ to a point $ \zeta^{\prime} $ such that $ |\zeta^{\prime}| > 2 $. To do so, it suffices to prove that if the radius $ \rho $ is small enough, any $ \zeta \in \Omega $ can be connected to some $ \zeta^{\prime \prime} \in \Omega $ such that $ |\zeta^{\prime \prime}| \geq |\zeta| + \rho  $ via a continous curve contained in $ \Omega $.  Then, in a finite number of steps, $ \zeta $ can be connected to a point of modulus at greater than $ 2 $. A possible choice of $ \rho $ is as follows. We denote by $ \{ \zeta_{\beta} \ | \ \beta \in B \} = \{  e^{-tz_k} \ | \ k \in K\} $ the finite set of distinct images of the $ z_k $s by $ z \mapsto e^{-tz} $. We choose $ \rho > 0 $ small enough such that $ |\zeta_{\beta} - \zeta_{\beta^{\prime}}|  >  5 \rho $ when $ \beta \ne \beta^{\prime} $.  We may also choose $ \rho $ such that $ |\zeta_{\beta}| > e^{-ta} + 5 \rho $. Thanks to this choice, for each $ \beta \in B$, the annulus $ A_{\rho,\beta} := \{ \zeta \ | \ \rho <  |\zeta - \zeta_{\beta}| < 4 \rho \} $ is contained in $ \Omega $. Pick then $ \zeta \in \Omega $. Then, either the (closed) disk $ \overline{D} (\zeta,\rho) $ is contained in $ \Omega $ or it meets exactly one of the disks $ \overline{D} (\zeta_{\beta},\rho) $ or $ \overline{D} (0,e^{-ta}) $. If $ \overline{D} (z,\rho) $ is contained in $ \Omega $ then $ \zeta^{\prime \prime} := \zeta + \rho \frac{\zeta}{|\zeta|}  $ has modulus $ |\zeta| + \rho $ and can be connected to $ \zeta $ via a path contained in $ \overline{D}(\zeta,\rho) $. If $ \overline{D}(\zeta,\rho) $ meets $ \overline{D} (0,e^{-ta}) $, since $ \overline{D}(0,e^{-ta}) $ is at distance at least $ 3 \rho $ from any disk $ \overline{D}(\zeta_{\beta},\rho) $,  the annulus $ \{ e^{-ta} < |\zeta^{\prime \prime}| < e^{-ta} +3 \rho \} $ is contained in $ \Omega $ and contains $ \zeta $; in this case $ |\zeta| \leq e^{-ta} + \rho $ (otherwise the disks do not intersect) and can be connected  pathwise in the annulus to some $ \zeta^{\prime \prime} $ of modulus $ e^{-ta} + 2 \rho $. The last case is when $ \overline{D}(\zeta,\rho) $ intersects some $ \overline{D}(\zeta_{\beta},\rho) $. Then $ \zeta $ belongs to the annulus $ A_{\rho,\beta} $ around $ \zeta_{\beta} $ and moreover $ |\zeta| \leq |\zeta_{\beta}| + 2 \rho  $ (otherwise the disks do not interesect); $ \zeta $ can then be  connected  within the annulus $ A_{\rho,\beta} $ to some $ \zeta^{\prime \prime} $ such that $ |\zeta^{\prime \prime}| = |\zeta_{\beta}| + 3 \rho $ (pick $ \zeta^{\prime \prime} = \zeta_{\beta} + 3 \rho \zeta_{\beta} / |\zeta_{\beta}| $)
which then satisfies $ |\zeta^{\prime \prime}| \geq |\zeta| + \rho $. This completes the proof.
\end{proof}

This lemma allows to obtain the following analytic continuation.

\begin{lemm} \label{poursimplifier} Let $ t > 0 $ and $ a > 0 $. Pick $ \rho > 0 $ small enough as in Lemma \ref{lemmeconnexite}.
Then, if $ r > 0 $ is small enough, for every $ \varepsilon > 0 $,
\begin{equation}\label{surOmega}
	\zeta \mapsto \sum_{k \in K} \frac{i}{2 i \pi} \int_{|z-z_k| = r} \frac{1}{e^{-tz} - \zeta} (\opeg - z)^{-1} \frac{1}{\varepsilon z + 1}\, dz
\end{equation} 
is holomorphic on $ \Omega $, and
\begin{equation}\label{surcomplementdisque}
	\zeta \mapsto \frac{i}{2  \pi} \int_{\Re(z) = a} \frac{1}{e^{-tz} - \zeta} (\opeg-z)^{-1} \frac{1}{\varepsilon z + 1}\, dz
\end{equation}
is holomorphic on $ \{ |\zeta| > e^{-ta} \} $. As a consequence, by analytic continuation, the equality
\begin{equation}\label{continuationtrace}
	(e^{-t \opeg} - \zeta)^{-1} (\varepsilon \opeg + 1)^{-1} = \frac{i}{2  \pi} \int_{\Lambda_{a,r}}  \frac{1}{e^{-tz} - \zeta} (\opeg-z)^{-1} \frac{1}{\varepsilon z + 1}\, dz 
\end{equation}
holds for all $ \zeta \in \Omega $, and 
\[
	(e^{-t \opeg} - \zeta)^{-1} (\varepsilon \opeg + 1)^{-1} - \sum_{\Re(z_k)<a} \frac{i}{2  \pi} \int_{|z-z_k|=r}  \frac{1}{e^{-tz} - \zeta} (\opeg-z)^{-1} \frac{1}{\varepsilon z + 1}\, dz
\]
is holomorphic on $ \{ |\zeta | > e^{-ta} \} $.
\end{lemm}

\begin{proof} For each $ k \in K $, $ z $ such that $ |z - z_k| = r $ and $ \zeta \in \Omega $, hence such that $ |\zeta - e^{-tz_k}| > \rho $, we have 
\begin{align*}
	| e^{-tz} - \zeta| & = |e^{-tz} - e^{-tz_k} + e^{-t z_k} - \zeta |  \\
	& > \rho - |e^{-tz} - e^{-tz_k}|  \\
	& > \rho/2
\end{align*} 
if $ r $ is small enough so that  $ |e^{-tz} - e^{tz_k}| < \rho/2 $ for $ |z-z_k| \leq r $.  This shows that (\ref{surOmega}) is holomorphic on $ \Omega $.  Similarly, if $ \Re(z) = a $ and $  |\zeta | > e^{-ta} $, we have
\[
	| e^{-tz} - \zeta| > |\zeta| - |e^{-tz}| = |\zeta| - e^{-ta} > 0,
\]
so that (\ref{surcomplementdisque}) is holomorphic on $ \{ |\zeta| > e^{-ta} \} $ and we obtain in passing that the right-hand side of (\ref{continuationtrace})  is holomorphic on $ \Omega $.  Since the resolvent $ (e^{-t \opeg} - \zeta)^{-1} $ is holomorphic on $ {\mathbb C} \setminus ( \{ 0 \} \cup \{ e^{-tz_k} \ | \ k \in {\mathbb N} \} ) $ and since  (\ref{continuationtrace}) holds for $ |\zeta| \gg 1 $
by Proposition \ref{0continuation} and Proposition \ref{1continuation}, we conclude that  (\ref{continuationtrace}) holds on $ \Omega $ and the result follows.
\end{proof}

With all these preliminary results at hand, we are now ready to prove Proposition \ref{propositionprojecteurs}.

\begin{proof}[Proof of Proposition \ref{propositionprojecteurs}]
According to Lemma \ref{poursimplifier} and the choice of $ \rho $, the operator
\[
	\Pi (e^{-t \opeg},\zeta_{\alpha})  (\varepsilon \opeg + 1)^{-1}= \bigg(  \frac{i}{2 \pi} \int_{|\zeta - \zeta_{\alpha}| = 2 \rho}  (e^{-t \opeg} - \zeta)^{-1}\, d \zeta \bigg) (\varepsilon \opeg + 1)^{-1}
\]
is equal to
\[
	\sum_{\Re(z_k) < a } \frac{i}{2   \pi} \int_{|\zeta - \zeta_{\alpha}| = 2 \rho}  \bigg( \frac{i}{2  \pi} \int_{|z-z_k| = r} \frac{1}{e^{-tz} - \zeta} (\opeg - z)^{-1} \frac{1}{\varepsilon z + 1}\, dz \bigg)\, d\zeta,
\]
so letting $ \varepsilon \rightarrow 0^+ $,
\begin{equation}\label{pourFubinicomplexe}
	\Pi (e^{-t \opeg},\zeta_{\alpha})  =  \sum_{\Re(z_k) < a } \frac{i}{2   \pi} \int_{|\zeta - \zeta_{\alpha}|=2 \rho}  \bigg( \frac{i}{2  \pi} \int_{|z-z_k| = r} \frac{1}{e^{-tz} - \zeta} (\opeg - z)^{-1}\, dz \bigg)\, d\zeta . 
\end{equation}
For each $ z_k $ in the right-hand side  such that $ e^{-t z_k} \ne \zeta_{\alpha} $, we have
\begin{align*}
	| e^{-tz} - \zeta |  & \geq |\zeta_{\alpha} - e^{-tz_k}| - |\zeta - \zeta_{\alpha}|  - | e^{-tz} - e^{-tz_k}| \\
	& \geq 5 \rho - |z-\zeta_{\alpha}| - \frac{\rho}{2}
\end{align*}  
for $ |z-z_k| = r $ and $ \zeta \in \Omega $. In particular, the map $ \zeta \mapsto (e^{-tz}-\zeta)^{-1} $ is holomorphic on $ \{ |\zeta| \leq 2 \rho \} $ so by using the Fubini theorem in (\ref{pourFubinicomplexe}) in the integrals associated to those $ z_k $, it remains
\[
	\Pi (e^{-t \opeg},\zeta_{\alpha})  =  \sum_{j=1 }^J \frac{i}{2   \pi} \int_{|\zeta - \zeta_{\alpha}| = 2 \rho}  \bigg( \frac{i}{2  \pi} \int_{|z-z_{j,\alpha}| = r} \frac{1}{e^{-tz} - \zeta} (\opeg - z)^{-1}\, dz \bigg)\, d\zeta.
\]
Now observe that if $ |z-z_{j,\alpha}| = r $ and $r$ is small enough, $ |e^{-tz} - e^{-tz_{j,\alpha}}| = |e^{-tz} - \zeta_{\alpha} | < 2 \rho $ so that $ e^{-tz} $ is inside the disk $ \{ \zeta \ | \  |\zeta - \zeta_{\alpha}| \leq 2 \rho \} $ and
\[
	\frac{i}{2  \pi} \int_{|\zeta - \zeta_{\alpha}| = 2 \rho} \frac{1}{e^{-tz} - \zeta}\, d \zeta =  1.
\]
Thus, by the Fubini Theorem
\[
	\Pi (e^{-t \opeg},\zeta_{\alpha})  = \sum_{j=1 }^J  \frac{i}{2  \pi} \int_{|z-z_{j,\alpha}| = r}  (\opeg - z)^{-1}\, dz
	= \sum_{j=1}^{J} \Pi (\opeg,z_{j,\alpha}),
\]
and this  completes the proof. 
\end{proof}

\subsection{Counting function from small-time asymptotics}\label{subsec:counting}

In this section, on separable Hilbert spaces, we prove the two results used in Section \ref{sec:trace} to estimate the number of eigenvalues of an operator $\opeg$ (assuming that $e^{-t\opeg}$ is trace class, in addition to the fact that it is a \schtroumpf operator) from the asymptotics of $\tr(e^{-t\opeg})$ and $\|e^{-t\opeg}\|_{\tr}$ as $t\to 0^+$.

\begin{lemm}\label{p:upper-bound-trace}
Let $\mathcal{H}$ be a separable Hilbert space and assume that $\opeg$ is a \schtroumpf operator on $\mathcal{H}$, such that $e^{-t \opeg}$ is trace class for all $t>0$. Then, recalling the definition of the counting function $\mathcal{N}_{\opeg}$ in~\eqref{e:def-counting-function}, we have for every $t>0$,
\begin{align}
	& \big| \tr(e^{-t \opeg}) \big|  \leq \sum_{j} e^{-t \Re(z_j(\opeg))} \leq \big\Vert e^{-t \opeg}\big\Vert_{\tr} , \label{e:trace-spectre-norme-trace} \\
	& \mathcal{N}_{\opeg}(t^{-1}) \leq e\big\Vert e^{-t \opeg}\big\Vert_{\tr}. \label{e:estim-sup-norme-trace} 
\end{align} 
\end{lemm}

\begin{proof}
The right-hand side inequality in~\eqref{e:trace-spectre-norme-trace} relies on the following Weyl inequality (see~\cite[Theorem XIII.103, p.318]{ReedSimon4}),
\begin{equation}\label{eq:weylin}
	\sum_{\alpha} |\zeta_{\alpha} (e^{-t \opeg})| \leq \big\Vert e^{- t \opeg}\big\Vert_{\tr},
\end{equation}
where $( \zeta_{\alpha} (e^{-t \opeg}))_{\alpha} $ are the eigenvalues counted with their multiplicities. The left-hand side inequality in~\eqref{e:trace-spectre-norme-trace} relies on the Lidskii Theorem (see~\cite[Corollary p 328]{ReedSimon4}) and the triangle inequality, which imply
\[
	\big| \tr(e^{-t \opeg}) \big| = \Big| \sum_{\alpha} \zeta_{\alpha} (e^{-t \opeg}) \Big|  \leq \sum_{\alpha}  \big|  \zeta_{\alpha} (e^{-t \opeg}) \big|  .
\]
Now, according to Propositions~\ref{prop0spec} and~\ref{propositionprojecteurs}, we have 
\[
	\sum_{\alpha\in \mathbb{N}}  \big|  \zeta_{\alpha} (e^{-t \opeg}) \big| =  \sum_{j}  \big| e^{-t (z_j(\opeg))} \big|  =   \sum_{j} e^{-t \Re(z_j(\opeg))}.
\]
Combining the above three lines concludes the proof of~\eqref{e:trace-spectre-norme-trace}.

Finally, the elementary inequality  $e^{-1}\mathds{1}_{tx\leq 1} \leq e^{-tx}$ yields that for every $t>0$
\[
	\mathcal{N}_{\opeg}(t^{-1}) e^{-1} \leq   \sum_{j} e^{-t \Re(z_j(\opeg))},
\]
which, combined with~\eqref{e:trace-spectre-norme-trace} concludes the proof of~\eqref{e:estim-sup-norme-trace}.
\end{proof}

We next deduce a lower bound on the counting function from the inequalities~\eqref{e:trace-spectre-norme-trace} in case both $\tr(e^{-t\opeg})$ and $\|e^{-t\opeg}\|_{\tr}$ have a polynomial behavior (possibly with two different exponents) as $t\to 0^+$. To this aim, we prove the following proposition.

\begin{prop}\label{l:karamata-ersatz}
Let $\mu$ be a nonnegative Radon measure  on $\R$ such that $\mathrm{supp}(\mu) \subset \R_+$ and set $\mathsf{m}(\lambda) := \int_{[0,\lambda]} d\mu$.
Suppose that there exist $\alpha,\beta,C_\alpha,C_\beta, t_0>0$ such that for every $t\in (0,t_0)$,
\begin{equation}\label{e:karamata-asspt}
	C_\alpha t^{-\alpha} \leq \int_{\R_+} e^{-t x}\, d\mu(x) \leq C_\beta t^{-\beta}.
\end{equation}
Then, $\beta \geq \alpha$ and there are $C, \lambda_0>0$ such that for every $\lambda \geq \lambda_0$,
\[
	C \bigg(\frac{\lambda}{\log(\lambda)}\bigg)^{\alpha}\leq \mathsf{m}(\lambda) \leq e C_\beta \lambda^{\beta}.
\]
If in addition $\beta = \alpha$, then the $\log$ factor can be removed.
\end{prop}

The statement of Proposition~\ref{l:karamata-ersatz} is motivated by that of the Karamata Tauberian theorem, see e.g.~\cite[Theorem~15.3, p.30]{Korevaar}, in which the lower and upper bounds are both equivalent to $c_\alpha t^{-\alpha}$ as $t\to 0^+$. Before proving Proposition~\ref{l:karamata-ersatz}, let us draw the following direct corollary.

\begin{coro}\label{c:asympt-spectrales}
Let $\mathcal{H}$ be a separable Hilbert space and assume that $\opeg$ is a \schtroumpf operator on $\mathcal{H}$, such that $e^{-t \opeg}$ is trace class for all $t>0$. Suppose that there exist $\alpha,\beta,C_\alpha,C_\beta, t_0>0$ such that for every $t\in (0,t_0)$,
\begin{equation}\label{e:karamata-asspt-bis}
	C_\alpha t^{-\alpha} \leq \big| \tr(e^{-t \opeg}) \big|  \quad \text{ and }\quad \big\Vert e^{-t \opeg}\big\Vert_{\tr} \leq C_\beta t^{-\beta}.
\end{equation}
Then, there are $C,\lambda_0>0$ such that, recalling the definition of the counting function $\mathcal{N}_{\opeg}$ in~\eqref{e:def-counting-function}, we have for every $\lambda \geq \lambda_0$,
\[
	C \bigg(\frac{\lambda}{\log(\lambda)}\bigg)^{\alpha}\leq \mathcal{N}_{\opeg}(\lambda) \leq e  C_\beta \lambda^{\beta}.
\]
If in addition $\beta = \alpha$, then the $\log$ factor can be removed. 
\end{coro}

\begin{rema}
Note that the proof of the lower bound in Corollary~\ref{c:asympt-spectrales} (as well as that in Proposition~\ref{l:karamata-ersatz}) for $ \mathcal{N}_{\opeg}(\lambda) $ actually only relies on $t^{-\alpha} \lesssim \big| \tr(e^{-t \opeg}) \big|$ together with $\mathcal{N}_{\opeg}(\lambda) \lesssim \lambda^{\beta}$ (and not directly on the upper bound on the trace norm). 
\end{rema}

\begin{proof}
We define the measure 
\[
	\mu := \sum_j \delta_{\Re(z_j)} \in \mathcal{M}_+(\mathbb{R}),
\]
satisfying $\mbox{supp}(\mu) \subset \mathbb{R}_+$,  denote by $\mathsf{m}$ its primitive which vanishes on $\mathbb{R}_-$, and observe that 
\begin{equation}\label{e:counting-function}
	\mathsf{m}(\lambda)  = \int_{[0,\lambda]} d\mu= \sum_j \int_{[0,\lambda]}  d\delta_{\Re(z_j)} = \sum_{j\,\vert\, \Re(z_j) \leq \lambda} 1 =\mathcal{N}_{\opeg}(\lambda)  .
\end{equation}
From~\eqref{e:trace-spectre-norme-trace} in Lemma~\ref{p:upper-bound-trace} combined with the assumption~\eqref{e:karamata-asspt-bis}, the assumption~\eqref{e:karamata-asspt}  of Proposition~\ref{l:karamata-ersatz} holds and Proposition~\ref{l:karamata-ersatz} with~\eqref{e:counting-function} yield the sought result.
\end{proof}

\begin{proof}[Proof of Proposition~\ref{l:karamata-ersatz}]
Note first that the bound $\mathsf{m}(\lambda) \leq e C_\beta \lambda^{\beta}$ is a straightforward consequence of the inequality $e^{-1}\mathds{1}_{tx\leq 1} \leq e^{-tx}$, as in the proof of Lemma~\ref{p:upper-bound-trace} (and, in particular, does not require the lower bound in Assumption~\eqref{e:karamata-asspt}).

We now prove the lower bound. Notice first that the definition of $\mathsf{m}$ together with the estimate $\mathsf{m}(\lambda) \leq e C_\beta \lambda^{\beta}$ imply that $\mathsf{m} \in BV_{\loc}(\R)\cap \mathcal{S}'(\R)$ is a nondecreasing function such that $\mathrm{supp}(\mathsf{m})\subset \R_+$ and $\mathsf{m}'=\mu$ in $\mathcal{S}'(\R)$. In particular, for any $\phi \in \mathcal{S}(\R)$, we have $\int_{\R}\phi(x)\,d\mu(x)=-\int_\R \phi'(x)\mathsf{m}(x)\,dx$. Using this identity to $\phi(x) := \chi(x)e^{-tx}$ for $\chi \in C^\infty(\R)$ such that $\chi=1$ on $\R_+$ and $\mbox{supp}(\chi) \subset [-1,\infty)$, and recalling that $\mu$ and $\mathsf{m}$ are supported on $\R_+$, we deduce that for every $t>0$,
\[
	\int_{\R_+} e^{-t x}\, d\mu(x) = t \int_{\R_+} e^{-t x} \mathsf{m}(x)\,dx.
\]
We next write, for $\para=\para(t)>0$ to be fixed later on, 
\begin{equation}\label{e:decomposition-m}
	\int_{\R_+} e^{-t x}\, d\mu(x) = t \int_{0}^{\para t^{-1}} e^{-t x} \mathsf{m}(x)\,dx + t \int_{\para t^{-1}}^\infty e^{-t x} \mathsf{m}(x)\,dx .
\end{equation}
On the one hand, since $\mathsf{m}$ is nondecreasing, 
\[
	t \int_{0}^{\para t^{-1}} e^{-t x} \mathsf{m}(x)\,dx  \leq  t\mathsf{m}(\para t^{-1}) \int_{0}^{\para t^{-1}} e^{-t x}\, dx  = t\mathsf{m}(\para t^{-1}) \frac{1-e^{-\para}}{t}\leq \mathsf{m}(\para t^{-1}) .
\]
On the other hand, using the upper bound $\mathsf{m}(\lambda) \leq e C_\beta \lambda^{\beta}$ already obtained, we have with $c_\beta= e C_\beta$,
\[
	t \int_{\para t^{-1}}^\infty e^{-t x} \mathsf{m}(x)\,dx 
	\leq c_\beta t  \int_{\para t^{-1}}^\infty e^{-t x} x^\beta\,  dx 
	= c_\beta  \int_{\para t^{-1}}^\infty e^{-t x} t^{-\beta}(tx)^\beta\, d(tx) = c_\beta t^{-\beta} \Phi(\para),
\]
where $\Phi(\para) := \int_{\para}^\infty e^{-y} y^\beta\, dy$. Integrating by parts, we see that there is $\para_0$ such that $\Phi(\para) \leq 2 e^{-\para} \para^\beta$ for all $\para\geq\para_0$. Combining the left inequality in the assumption~\eqref{e:karamata-asspt}, the decomposition~\eqref{e:decomposition-m} and the last two inequalities, we have obtained that for all $t\in (0,t_0)$ and $\para\geq \para_0$,
\begin{equation}\label{e:presque-alphabeta}
	C_\alpha t^{-\alpha} \leq \mathsf{m}(\para t^{-1}) + 2c_\beta t^{-\beta}  e^{-\para} \para^\beta .
\end{equation}
Note that, in the case $\beta=\alpha$, this concludes the proof when fixing $\para$ sufficiently large (independent of $t$) so that to absorb the last term of the right-hand side in the left-hand side.

In the case $\beta>\alpha$, we choose $\para=M \log(t^{-1})$ for $M>0$ independent of $t$ to be chosen later on. We have 
\[
	t^{-\beta}  e^{-\para} \para^\beta = t^{-\beta}  e^{-M \log(t^{-1})} M^\beta \log(t^{-1})^\beta = M^\beta t^{M-\beta}\log(t^{-1})^\beta.
\]
Hence, fixing $M>\beta-\alpha$, there is $t_1=t_1(M)>0$ such that $2c_\beta M^\beta t^{M-\beta}\log(t^{-1})^\beta < \frac{C_\alpha}{2} t^{-\alpha}$ for all $t\in (0,t_1)$, and~\eqref{e:presque-alphabeta} implies that for every $0< t< \min(t_0,t_1)$,
\[
	\frac{C_\alpha}{2} t^{-\alpha}  \leq \mathsf{m}(M \log(t^{-1})t^{-1}).
\]
Setting $\Psi(x) := M x\log(x)$, we therefore get that for every $x>x_0$,
\begin{equation}\label{e:lower-spec-almost}
	\frac{C_\alpha}{2} x^{\alpha}  \leq \mathsf{m}(\Psi(x)).
\end{equation}
Now $\Psi$ is increasing near infinity so that there is a unique $x(\lambda)$ such that $\Psi(x(\lambda))=\lambda$. Moreover, we have that for large $\lambda$
\[
	\Psi(x(\lambda))=\lambda> \lambda\bigg(1 - \frac{\log(M\log(\lambda))}{\log(\lambda)} \bigg) 
	= M \frac{\lambda}{M\log(\lambda)} \log\bigg(\frac{\lambda}{M\log(\lambda)}  \bigg) 
	= \Psi\bigg(\frac{\lambda}{M\log(\lambda)}\bigg) ,
\]
so that $x(\lambda)>\frac{\lambda}{M\log(\lambda)}$ since $\Psi$ is increasing. Coming back to~\eqref{e:lower-spec-almost}, applied to $x=x(\lambda)$, this yields that for every $\lambda\geq\lambda_0$
\[
	\frac{C_\alpha}{2} \bigg(\frac{\lambda}{M\log(\lambda)}\bigg)^\alpha < \frac{C_\alpha}{2} x(\lambda)^{\alpha}  \leq \mathsf{m}(\lambda),
\]
which concludes the proof of the proposition.
\end{proof}

\section{Trace class estimates and spectral asymptotics}\label{sec:trace}

Throughout this section, we shall assume, in addition to~\eqref{hypothesepotentiel} , that for some $ 0 < \sigma \leq 1 $, the condition~\eqref{conditiondHilbertSchmidt}  is satisfied, that is to say, $|\nabla V (x)|$ grows at least like $|x|^\sigma$ at infinity. Recall that this implies in particular that the $L^p$ spectrum of the operator $\ope$ is discrete for every $p\in(1,\infty)$. Indeed, under this condition, it follows from Proposition \ref{prop:compactness} that each evolution operator $e^{-\frac{t}{m}\ope}$ is compact ($t>0$), and Lemma \ref{e:compact-semigpe-resolvent} then implies that the operator $\ope$ has compact resolvent and therefore discrete spectrum. The aim of this section is to prove Theorem \ref{eigenvaluestheorem} by using the counting estimates derived in Subsection~\ref{subsec:counting}.

\subsection{Asymptotics of the trace norm}

In this subsection, we derive bounds on the trace norm of the evolution operators $e^{-\frac{t}{m}\ope}$.
        
\begin{prop} \label{resteatrace}  
Assume that $V$ satisfies~\eqref{hypothesepotentiel} and~\eqref{conditiondHilbertSchmidt}. For all $ M > 0 $, we can find $ N $ large enough such that $ e^{- \frac{t}{m}\ope} - {\mathcal E}_{1 + A_1 + \cdots + A_N} (t) $ is of trace class for all $ t \in [0,t_0 ] $ and
\[
	\big\Vert  e^{- \frac{t}{m}\ope} - {\mathcal E}_{1 + A_1 + \cdots + A_N} (t)  \big\Vert_{\tr} \lesssim t^M, \qquad t \in [0,t_0 ] .
\]
\end{prop}

\begin{proof} If $  s >  \frac{n}{2} $, $ s^{\prime} > \frac{n}{2 \sigma} $, $ s^{\prime \prime} > n $, the functions
\[
	(|v|^2 + 1)^{-\frac{s}{2}}, \qquad (|\nabla V (x)|^{2} + 1)^{-\frac{s^{\prime}}{2}}, \qquad (1+|\xi|^2 + |\eta|^2)^{-\frac{s^{\prime \prime}}{2}},
\]
belong respectively to $ L^2 ({\mathbb R}^n_v) $, $ L^2 ({\mathbb R}^n_x) $ and $ L^2 ({\mathbb R}^{2n}_{\xi,\eta}) $, using (\ref{conditiondHilbertSchmidt}) for the second one. This implies that the operator
\[
	T :=  (1- \Delta_{x,v} )^{-\frac{s^{\prime \prime}}{2}}  (|v|^2 + 1)^{-\frac{s}{2}} (|\nabla V (x)|^{2} + 1)^{-\frac{s^{\prime}}{2}}
\]
is Hilbert-Schmidt on $ L^2 ({\mathbb R}^{2n}_{x,v}) $. Picking $s,s^{\prime},s^{\prime \prime} $ as even integers, we also consider the differential operator
\[
	W = (|v|^2 + 1)^{\frac{s}{2}} (|\nabla V (x)|^2 + 1 )^{\frac{s^{\prime}}{2}}   (1-\Delta_{x,v})^{\frac{s^{\prime \prime}}{2}}.
\]
Within Theorem \ref{maintechnicalresulttheorem}, we write the remainder $ {\mathcal R}_N (t) $ as
\[
	{\mathcal R}_N (t) = T W  {\mathcal R}_{N} (t)  W^* T^*,
\]
so that, thanks to Proposition \ref{prop:smoothremainder},
\[
	\Vert {\mathcal R}_N (t) \Vert_{\tr} \leq C\Vert W {\mathcal R}_N (t) W^*\Vert_{L^2 \rightarrow L^2} \leq C t^M,
\]
provided $N$ is large enough. This completes the proof.
\end{proof}

\begin{prop} Assume that $V$ satisfies~\eqref{hypothesepotentiel} and~\eqref{conditiondHilbertSchmidt}. There exists $ C > 1 $ such that for all $ A \in {\mathcal A}^{\nu,d} $, if we write
\[
	{\mathcal E}_A (t) = {\mathcal I}_{F_t} \Op (a_t)
\]
according to Theorem \ref{propositionpolaire}, the operators $ {\mathcal E}_A (t) $ and $ \Op (a_t) $ are of trace class for each $ t \in (0,t_0 ] $ and we have
\begin{equation}\label{equivalencenormetrace}
	C^{-1} \Vert \Op (a_t)\Vert_{\tr} \leq \Vert{\mathcal E}_A (t)\Vert_{\tr} \leq C \Vert \Op (a_t)\Vert_{\tr}.
\end{equation}
\end{prop}

\begin{proof} That $ \Op (a_t) $ is trace class follows from the fact that $ a_t $ and all its derivatives are integrable.  This can also be checked in more elementary fashion in a way similar to the proof of Proposition \ref{resteatrace}. Now, since $ {\mathcal I}_{F_t} $ is uniformly bounded on $ L^2(\mathbb R^{2n}) $ for $ t \in [0,t_0] $, $ {\mathcal E}_{A}(t) $ is also of trace class and we have the upper bound in (\ref{equivalencenormetrace}). On the other hand, using Proposition \ref{lemmedediffeo}, it follows that $ {\mathcal I}_{F_t} $ is invertible on $ L^2(\mathbb R^{2n}) $ and $ {\mathcal I}_{F_t}^{-1} $ is also uniformly bounded on $ L^2(\mathbb R^{2n}) $. This yields the lower bound in (\ref{equivalencenormetrace}) by writing $ \Op (a_t) = {\mathcal I}_{F_t}^{-1} {\mathcal E}_A (t) $.
\end{proof}

To justify the sharpness of the trace norm estimates, let us consider the following

\begin{ex}\label{ex:quadratic} Let $ m = \beta = \gamma = 1 $ and $ V (x) = \frac{\cq}2 |x|^2 $ with $\cq\in {\mathbb R} \setminus \{ 0 \} $. 
Then, it is easy to check that, as long as $ 1 + \cq t^2/2 > 0 $,
\[
	(\tilde{x}_t, \tilde{v}_t) = F_t^{-1} (x,v) = \frac{1}{1+ \cq t^2/2 } \big( x+tv ,  \big(1 - \cq t^2/2 \big)v - \cq t  x  \big),
\]
and then that $ \Im(\psi)(t,\tilde{x}_t,\tilde{v}_t,\xi,\eta) $ is given by
\[
	t  \bigg| \eta + \frac{t}{2} \xi \bigg|^2 + \frac{t^3}{12} |\xi|^2 + \frac{1}{\big( 1 + \cq t^2/2 \big)^2} \bigg( \frac{t}{4} | v - \cq t x/2 |^2 + \frac{\cq^2 t^3}{48} | x + t v|^2 \bigg) .
\]
Picking the amplitude $ A \equiv 1 $ so that $ a_t = \exp (- \Im(\psi)(t,\tilde{x}_t,\tilde{v}_t,\xi,\eta)) $,
we can compute
\[
	\tr \big( \Op (a_t) \big) = (2 \pi)^{-2n} \iiiint_{\mathbb R^{4n}} a_t (x,v,\xi,\eta)\, d\xi d \eta d x dv \sim c t^{-4n}, \qquad t \rightarrow 0^+,
\]
for some positive constant $c$. Since the trace norm controls the trace, this implies in particular that
\[
	\Vert\Op (a_t)\Vert_{\tr} \gtrsim t^{-4n}.
\]
On the other hand, by estimating the trace norm of $ \Op (a_t) $ by $ L^1 $ norms of finitely many derivatives of the symbol (see \cite{Robert,DSbis}), we also get an upper bound 
\[
	\Vert\Op (a_t)\Vert_{\tr} \lesssim t^{-4n}.
\]
Thus, in this case, the trace norm of $ {\mathcal E}_1 (t) $ is bounded from above and below by (different) constants times $ t^{-4n} $ when $ t \rightarrow 0^+ $.            
\end{ex}

\begin{prop} \label{estimationnormetraceOIF} Assume that $V$ satisfies~\eqref{hypothesepotentiel} and~\eqref{conditiondHilbertSchmidt}. 
For every $ A \in {\mathcal A}^{\nu,d} $, $ {\mathcal E}_A (t) $ is trace class for all $ t \in (0,t_0] $ and
\[
	\Vert {\mathcal E}_A (t) \Vert_{\tr} \lesssim t^{\nu - \frac{5n}{2} - \frac{3n}{2\sigma} }, \qquad t \in (0,t_0].
\]
\end{prop}

\begin{proof} We proceed similarly to the proof of Proposition \ref{resteatrace}. Consider the (now $t$ dependent) differential operator
\[
	W_t := (1+ t^3 |\nabla V (x)|^2)^s  (1+ t |v|^2)^{s^{\prime}} (1 - t^3 \Delta_x - t \Delta_v )^{s^{\prime}},
\]
with $s,s^{\prime},s^{\prime \prime} $  integers.  Picking $ s $  large enough and using (\ref{conditiondHilbertSchmidt}),  we have
\[
	\big\Vert (1+t^3 |\nabla V |^2)^{-s} \big\Vert_{L^2}^2 \lesssim C + \int_{\mathbb R^n} (1+t^3|x|^{2\sigma})^{-2s}\,dx  \lesssim  t^{- \frac{3n}{2 \sigma}}.
\]
We can similarly estimate the $ L^2 $ norms of $ (1+t |v|^2)^{-s^{\prime}} $ and $ (1 + t^3 |\xi|^2 + t |\eta|^2)^{-s^{\prime \prime}}$ with respect to $t$ (provided $s^{\prime}$ and $ s^{\prime \prime} $ are taken large enough as we may). We thus infer
\[
	\big\Vert W_t^{-1}\big\Vert_{\rm Hilbert-Schmidt}  \lesssim t^{- \frac{n}{4}} t^{- \frac{3n}{4 \sigma}}  t^{- \frac{n}{4} - \frac{3n}{4}}.
\]
The result then follows from Proposition \ref{Lpboundedness} together with
\[
	\Vert {\mathcal E}_A (t) \Vert_{\tr}  =  \big\Vert W_t^{-1} W_t {\mathcal E}_A (t) W_t^* (W_t^{-1})^* \big\Vert_{\tr}  \lesssim  t^{\nu} \big\Vert W_t^{-1}\big\Vert_{\rm Hilbert-Schmidt}^2 .
\]
The proof is complete.
\end{proof}

Proposition \ref{estimationnormetraceOIF} just gives an upper bound on the trace norm; however, Example \ref{ex:quadratic} (in which $\nu = 0$ and $ \sigma = 1 $) shows that  it cannot be improved in general.

\bigskip

As a straightforward consequence of Propositions \ref{resteatrace} and \ref{estimationnormetraceOIF}, we get the following theorem.

\begin{theo}[Trace norm estimate] \label{tracenormsemigroup4} 
Assume that $V$ satisfies~\eqref{hypothesepotentiel} and~\eqref{conditiondHilbertSchmidt}. Then, for each $ t > 0 $, the evolution operator $e^{-t \ope} $ is trace class with norm
\[
	\big\Vert e^{-t \ope}\big\Vert_{\tr} \leq C t^{-\frac{5n}{2} - \frac{3n}{2 \sigma}}, \qquad t \rightarrow 0^+.
\]
\end{theo}

\subsection{Asymptotics of the trace}

Before considering the application of Theorem \ref{tracenormsemigroup4} to the eigenvalues distribution, we make some comments on the trace of the semigroup.

We give first a statement assuming only that~\eqref{hypothesepotentiel} and (\ref{conditiondHilbertSchmidt}) are satisfied. In order that the trace has an asymptotic equivalent as $t\to 0^+$, we will then restrict to the case where $|\nabla V|$ is asymptotically homogeneous of degree $\sigma \in (0,1]$, in the sense that 
\begin{equation}
\label{e:homogeneous-V}
	\text{ there is }  a \in C^0(\S^{n-1}; \R_+^*) \text{ such that } \frac{|\nabla V(r\omega)|}{r^\sigma} \to a(\omega) \text{ as } r \to +\infty, \text{ uniformly in } \omega \in \S^{n-1} .
\end{equation}
Note that Assumption~\eqref{e:homogeneous-V} implies~(\ref{conditiondHilbertSchmidt}), and is for instance satisfied in the nondegenerate quadratic case with $\sigma=1$.

\begin{prop}[Leading order term of the parametrix] \label{leadingorderform4} 
Assume that $V$ satisfies~\eqref{hypothesepotentiel} and \eqref{conditiondHilbertSchmidt}. Then, we have
\begin{equation}
\label{e:expression-trace}
	\tr\big( {\mathcal E}_1 (t) \big) = \frac{1}{2^n} \bigg( \frac{\beta m}{\pi \gamma t^3} \bigg)^{\frac{n}{2}}  \bigg(1 + \frac{(\gamma t)^2}{12} \bigg)^{-\frac{n}{2}}  
	\int_{\mathbb R^n} \exp \left\{ - \bigg( 1 + \frac{(\gamma t)^2}{3} - \frac{(\gamma t)^4}{48 + 4 (\gamma t)^2} \bigg)  \frac{\beta t}{4 m \gamma} |\nabla V|^2 \right\}\, dx.
\end{equation}
In particular,
\[
	\tr\big( {\mathcal E}_1 (t) \big)  = \mathcal O (t^{- \frac{3n}{2}  - \frac{n}{2 \sigma} }), \qquad t \rightarrow 0^+.
\]
Assuming in addition~\eqref{e:homogeneous-V}, we have 
\begin{align}\label{e:trace-equivalent}
    \tr\big( {\mathcal E}_1 (t) \big) \sim   t^{-\frac{3n}{2}-\frac{n}{2\sigma}}
    \frac{1}{2^{n+1}\sigma} \bigg( \frac{\beta m}{\pi \gamma} \bigg)^{\frac{n}{2}} \bigg( \frac{4 m \gamma}{\beta } \bigg)^{\frac{n}{2\sigma}}\Gamma\Big(\frac{n}{2\sigma}\Big)\int_{\S^{n-1}}a(\omega)^{-\frac{n}{\sigma}}\, d\omega, \quad \text{ as } t\to 0^+  ,
\end{align}
where $\Gamma$ is the standard $\Gamma$ function.
\end{prop}

\begin{proof}
We start from the explicit computation of the kernel of ${\mathcal E}_1 (t)$ in~\eqref{e:kernel-first-term}. 
To compute the trace, we restrict the kernel to the diagonal $ (x,v) = (y,w) $ in~\eqref{e:kernel-first-term} and rewrite the exponent
\begin{equation}\label{exposantareutiliser4}
	\frac{3 \beta m}{\gamma t} |v|^2 + \frac{\beta t}{4 \gamma m} |\nabla V|^2 + \frac{\gamma m \beta}{4} t \bigg( \bigg| v + \frac{t}{2m} \nabla V \bigg|^2 + \frac{1}{3}  \bigg| \frac{t}{2m} \nabla V \bigg|^2  \bigg)
\end{equation} 
as
\[
	\bigg( \frac{3 \beta m}{\gamma t} + \frac{\gamma m \beta}{4} t \bigg) \bigg| v + \frac{1}{8} \frac{\gamma \beta t^2}{\frac{3 \beta m}{\gamma t} + \frac{\gamma m \beta}{4} t} \nabla V \bigg|^2 
	+ \bigg( \frac{\beta t}{4 \gamma m} + \frac{\gamma \beta t^3}{12 m} -  \frac{1}{64}\frac{(\gamma \beta t^2)^2}{\frac{3 \beta m}{\gamma t} + \frac{\gamma m \beta}{4}t}  \bigg) |\nabla V|^2.
\]
This allows to simplify the integration in $v$ and obtain the expression~\eqref{e:expression-trace} of the trace. To estimate the trace, we bound the integral in $x$ by
\[
	C + \int_{{\mathbb R}^n} e^{- (1+\mathcal O(t^2)) t |x|^{2\sigma}}\, dx = C + t^{- \frac{n}{2\sigma}} \int_{{\mathbb R}^n} e^{- (1+\mathcal O(t^2))|y|^{2 \sigma}}\, dy.
\]
The proof of the first statement of the lemma is complete.

To prove the second statement, and in view of~\eqref{e:expression-trace}, we first compute, with $ \alpha(t) :=   1 + \frac{(\gamma t)^2}{3} - \frac{(\gamma t)^4}{48 + 4 (\gamma t)^2} = 1+\mathcal O(t^2)$, 
\begin{align*}
    I(t) & : = 
	\int_{\mathbb R^n} \exp \left\{ - \alpha(t)  \frac{\beta t}{4 m \gamma} |\nabla V(x)|^2 \right\}\, dx  
	= \int_{\S^{n-1}}\int_{\R_+} \exp \left\{ - \alpha(t)  \frac{\beta t}{4 m \gamma} |\nabla V(r\omega)|^2 \right\} r^{n-1}\, dr d\omega  \\
	& = \int_{\S^{n-1}}\int_{\R_+} \exp \left\{ - \alpha(t)  \frac{\beta t}{4 m \gamma} \big|\nabla V\big(t^{-\frac{1}{2\sigma}}s\omega\big)\big|^2 \right\} t^{-\frac{n}{2\sigma}}s^{n-1} ds d\omega .
\end{align*}
According to Assumption~\eqref{e:homogeneous-V}, 
$$
    ts^{-2\sigma}\big|\nabla V\big(t^{-\frac{1}{2\sigma}}s\omega\big)\big|^2 
    = \frac{\big|\nabla V\big(t^{-\frac{1}{2\sigma}}s\omega\big)\big|^2}{\big(t^{-\frac{1}{2\sigma}}s\big)^{2\sigma}} \to a^2(\omega) , 
    \quad \text{ as } t \to 0^+ , 
$$
uniformly on compact sets of $\R_+^*\times \S^{n-1}$.
As a consequence, we have pointwise convergence of the integrand as $ t \to 0^+ $:
$$
    \exp \left\{ - \alpha(t)  \frac{\beta t}{4 m \gamma} \big|\nabla V\big(t^{-\frac{1}{2\sigma}}s\omega\big)\big|^2 \right\}  s^{n-1} 
    \to \exp \left\{ -  \frac{\beta }{4 m \gamma} a^2(\omega)s^{2\sigma} \right\}  s^{n-1} .
$$
On the other hand, again from Assumption~\eqref{e:homogeneous-V}, there are $a_0,r_0>0$ such that as long as $ t^{-\frac{1}{2 \sigma}} s \geq r_0 $ and $\omega \in \S^{n-1}$, 
\[
    t\big|\nabla V\big(t^{-\frac{1}{2\sigma}}s\omega\big)\big|^2\geq a_0s^{2\sigma},
\]
hence the integrand satisfies, for some $c>0$ and $ t_0 \in (0,1] $ and for every $s\geq r_0$, $\omega \in \S^{n-1}$ and $t \in (0,t_0 ]$,
$$
    \exp \left\{ - \alpha(t)  \frac{\beta t}{4 m \gamma} \big|\nabla V\big(t^{-\frac{1}{2\sigma}}s\omega\big)\big|^2 \right\}  s^{n-1} 
    \leq \exp \{ - c s^{2\sigma}\}  s^{n-1}. 
$$
The region $ 0 \leq s \leq r_0 $ is harmless since the exponent is nonpositive. 
The dominated convergence theorem thus implies that as $t\rightarrow0^+$, with $c_0(\omega) = \frac{\beta }{4 m \gamma} a^2(\omega)>0$, 
\begin{align*}
    t^{\frac{n}{2\sigma}} I(t) 
    & \to\int_{\S^{n-1}}\int_{\R_+} \exp \{ -c_0(\omega)s^{2\sigma} \}s^{n-1}\, ds d\omega 
    = \int_{\S^{n-1}}\int_{\R_+} e^{-\tau} \bigg(\frac{\tau}{c_0(\omega)}\bigg)^{\frac{n-1}{2\sigma}}c_0(\omega)^{-\frac{1}{2\sigma}}\frac{1}{2\sigma}\tau^{\frac{1}{2\sigma}-1}\, d\tau d\omega \\
    &  = \frac{1}{2\sigma}\bigg(\int_{\S^{n-1}}c_0(\omega)^{-\frac{n}{2\sigma}}\, d\omega\bigg)
    \bigg(\int_{\R_+}e^{-\tau} \tau^{\frac{n}{2\sigma}-1}\, d\tau\bigg) 
    = \frac{1}{2\sigma}\bigg(\frac{4 m \gamma}{\beta}\bigg)^{\frac{n}{2\sigma}}\Gamma\Big(\frac{n}{2\sigma}\Big)\int_{\S^{n-1}}a(\omega)^{-\frac{n}{\sigma}}\, d\omega .
\end{align*}
Recalling the expression~\eqref{e:expression-trace} of the trace, we have obtained~\eqref{e:trace-equivalent}.
\end{proof}

For more general amplitudes, estimating the trace is fairly simple.

\begin{prop} \label{pourestimationtraceparametrix} 
Assume that $V$ satisfies~\eqref{hypothesepotentiel} and~\eqref{conditiondHilbertSchmidt}. Then, for each $ A \in {\mathcal A}^{\nu,d} $,
\[
	\tr\big( {\mathcal E}_A (t) \big) = \mathcal O (t^{\nu - \frac{3n}{2}  - \frac{n}{2 \sigma} }), \qquad t \rightarrow 0^+.
\]
\end{prop}

\begin{proof} We first study the form of the Schwartz kernel of $ {\mathcal E}_A (t) $, using (\ref{formedunoyaustandard}) and the form of the amplitude given in (\ref{expressionoftheform}). Up to the $ (x,v) $ dependent factors in the amplitude, the first step is to compute integrals of the form
\[
	(2 \pi)^{-2n} \iint_{\mathbb R^{2n}} e^{i \psi (t,x,v,\xi,\eta) - i y \cdot \xi - i w \cdot \eta} (t^{\frac{3}{2}}\xi)^{\alpha} (t^{\frac{1}{2}}\eta)^{\beta}\, d \xi d \eta,
\]
that is,
\[
	(i t^{\frac{3}{2}} \partial_y)^{\alpha} (i t^{\frac{1}{2}}\partial_w)^{\beta} (2 \pi)^{-2n}  \iint_{\mathbb R^{2n}} e^{i \psi (t,x,v,\xi,\eta) - i y \cdot \xi - i w \cdot \eta}\, d \xi d \eta,
\]
which is nothing but $  (i t^{\frac{3}{2}} \partial_y)^{\alpha} (i t^{\frac{1}{2}}\partial_w)^{\beta}  $ applied to the Schwartz kernel of $ {\mathcal E}_1 (t) $. Using Proposition \ref{leadingorderform4}, we see that the above integral is a sum of polynomials in $ t^{-\frac{3}{2}} \big(x-y - t \frac{v+w}{2} \big) $ and $ t^{-\frac{1}{2}} \big( v  - w + \frac{t}{m} \nabla V \big) $ times the Schwartz kernel of $ {\mathcal E}_1 (t) $. After restriction to $ (x,v) = (y,w) $, addition of the factors $ a_{\alpha \beta}(t,x,v) $ of (\ref{expressionoftheform}) and integration over $ {\mathbb R}^{2n} $, we see that, for some positive $ N $ and $ C $,
\[
	\big| \tr\big({\mathcal E}_A(t) \big) \big| \leq 
	C t^{\nu - 2n} \iint_{\mathbb R^{2n}} (1 + |t^{\frac{1}{2}}v| + |t^{\frac{3}{2}} \nabla V|)^N \big( 1 + t^{-\frac{1}{2}} |v| + t^{ \frac{1}{2}} |\nabla V| \big)^N e^{- (\ref{exposantareutiliser4})}\, dx dv.
\]
On the other hand,
\[
	(\ref{exposantareutiliser4}) \gtrsim \frac{|v|^2}{t} +  t |\nabla V|^2,
\]
and so, for some $c > 0$, we have
\[
	\big| \tr\big({\mathcal E}_A(t) \big) \big| \leq C t^{\nu - 2n} \iint_{\mathbb R^{2n}}  e^{- c ( \frac{|v|^2}{t} + t |\nabla V|^2 ) }\, dx dv,
\]
which is $\mathcal O (t^{\nu - \frac{3n}{2} - \frac{n}{2 \sigma}} ) $ as in the proof of Proposition \ref{leadingorderform4}. The proof is complete.
\end{proof}
  
\begin{theo}[Trace estimate]\label{t:trace-asymptotics} Assume that $V$ satisfies~\eqref{hypothesepotentiel} and~\eqref{conditiondHilbertSchmidt}. Then, as $ t \rightarrow 0^+$,
\[
	\tr\big( e^{-\frac{t}{m}\ope}  \big) = \mathcal O (t^{- \frac{3n}{2} - \frac{n}{2 \sigma} }).
\]
If in addition $V$ satisfies the upper bound~\eqref{e:lower-sigma}
with the same $\sigma $ as in \eqref{conditiondHilbertSchmidt}, then we have a two sided estimate on the trace : there exists $ C_1 >1 $ such that, for $t$ small enough\footnote{this actually holds also on $ (0,T] $ for any fixed $T>0$ but only the small $t$ regime is relevant here}
\[
	C_1^{-1} t^{- \frac{3n}{2} - \frac{n}{2 \sigma} } \leq  \big| \tr\big( e^{-\frac{t}{m}\ope}  \big) \big| \leq C_1 t^{- \frac{3n}{2} - \frac{n}{2 \sigma} }. 
\]
Assuming in addition~\eqref{e:homogeneous-V} with $\sigma \in (0,1]$, we have 
\begin{align}
\label{e:trace-equivalent-etP}
    \tr\big(  e^{-\frac{t}{m}\ope} \big) \sim   t^{-\frac{3n}{2}-\frac{n}{2\sigma}}
    \frac{1}{2^{n+1}\sigma} \bigg( \frac{\beta m}{\pi \gamma} \bigg)^{\frac{n}{2}} \bigg( \frac{4 m \gamma}{\beta } \bigg)^{\frac{n}{2\sigma}}\Gamma\Big(\frac{n}{2\sigma}\Big)\int_{\S^{n-1}}a(\omega)^{-\frac{n}{\sigma}}\, d\omega, \quad \text{ as } t\to 0^+ .
\end{align}
\end{theo} 

Note again that Assumption~\eqref{e:homogeneous-V} implies~\eqref{conditiondHilbertSchmidt} and~\eqref{e:lower-sigma}, and is for instance satisfied in the nondegenerate quadratic case with $\sigma=1$.
 
\begin{proof} The upper bound is a consequence of Propositions \ref{resteatrace} and \ref{pourestimationtraceparametrix}. If we assume moreover that $ |\nabla V (x)| \leq C \langle x \rangle^{\sigma} $, it follows from Proposition \ref{leadingorderform4} and the fact that $ - \langle x \rangle^{2\sigma} \gtrsim - |x|^{2 \sigma} $ for $|x| \geq 1$ that, for some $C>0$,
\[
	\tr({\mathcal E}_1(t))   \gtrsim   t^{-\frac{3n}{2}} \int_{|x| \geq 1} e^{- C t |x|^{2\sigma}}\, dx   
	\gtrsim t^{-\frac{3n}{2} - \frac{n}{2 \sigma}} \int_{|y| \geq t^{\frac{1}{2 \sigma}}} e^{- C  |y|^{2\sigma}}\, dx  \gtrsim   t^{-\frac{3n}{2} - \frac{n}{2 \sigma}} .
\]
Finally, using that
\[
	\big| \tr(e^{\frac{t}{m} \ope}) - \tr({\mathcal E}_1 (t)) \big| \lesssim  t^{1-\frac{3n}{2} - \frac{n}{2 \sigma}} ,
\]
we obtain that $ | \tr(e^{\frac{t}{m} \ope}) | \gtrsim   t^{-\frac{3n}{2} - \frac{n}{2 \sigma}}  $. This completes the part of the proof. 
 The proof of~\eqref{e:trace-equivalent-etP} follows similarly from~\eqref{e:trace-equivalent} together with Proposition~\ref{pourestimationtraceparametrix}.
\end{proof}

\begin{proof}[Proof of Theorem \ref{eigenvaluestheorem} (counting estimates)]
We are now in position to derive the estimates \eqref{eq:upperboundcount} and \eqref{eq:lowerboundcount} in Theorem \ref{eigenvaluestheorem} for the $L^2$ spectrum of the operator $\ope$. Recall from the beginning of the section that under the condition~\eqref{conditiondHilbertSchmidt}, the operator $ \ope $ has a discrete $ L^2 $ spectrum. Moreover, it follows from Proposition \ref{prop:fpsubsecto} that $\ope + c_0$ is a \schtroumpf operator in the sense of Definition~\ref{d:schtroumpf-op}, with $c_0 := \frac{m\gamma n}{2}$. Since we get the following estimates from Theorem \ref{tracenormsemigroup4} and Theorem \ref{t:trace-asymptotics} in the asymptotics $t\rightarrow0^+$,
\[
	C^{-1} t^{- \frac{3n}{2} - \frac{n}{2 \sigma} } \leq \big| \tr(e^{-\frac{t}{m}\ope}) \big|  \quad \text{ and }\quad \big\Vert e^{-\frac{t}{m}\ope}\big\Vert_{\tr} \leq C t^{-\frac{5n}{2} - \frac{3n}{2 \sigma}},
\]
we directly derive the estimates~\eqref{eq:upperboundcount} and~\eqref{eq:lowerboundcount} from Corollary \ref{c:asympt-spectrales}.
\end{proof}

\subsection{Dependence of the spectrum on the space \texorpdfstring{$L^p(\mathbb R^{2n})$}{}}

In this last subsection, we restore the dependence in $p$ and denote by $\ope_p$ the Fokker--Planck operator acting on $L^p(\R^{2n})$, defined by Theorem~\ref{theoreme-semigroup-intro}, and we investigate the dependence on $p$ of the spectrum and spectral projectors of $\ope_p$. We recall the definition of the Riesz projector of the operator $\ope_p$, associated to the eigenvalue $z_j\in \Sp(\ope_p)$, namely 
\[
	\Pi (\ope_p, z_j) = \frac{i}{2 \pi} \oint (\ope_p - z)^{-1}\, dz,
\]
where the integral is taken over a small circle, oriented counterclockwise, surrounding the eigenvalue $z_j$ but no other part of the spectrum of $\ope_p$ (recall anew from the beginning of the section that the spectrum of $\ope_p$ is discrete). Recall also that $\Pi (\ope_p, z_j) : L^p(\R^{2n})\rightarrow L^p(\mathbb R^{2n})$ is bounded and is a projector.

\begin{prop}\label{p:spectre-indep-p}
Assume that $V$ satisfies~\eqref{hypothesepotentiel} and~\eqref{conditiondHilbertSchmidt}. Let $p,q \in (1,\infty)$. Then, we have $\Sp(\ope_p) = \Sp(\ope_q)$ and $\ran \Pi (\ope_p, z_j) =\ran \Pi (\ope_q, z_j) \subset \mathcal{S}(\R^{2n})$ for all $z_j \in \Sp(\ope_p) = \Sp(\ope_q)$. In particular, the multiplicities, defined by $\mult_p (z_j) = \rank \Pi (\ope_p,z_j)$, satisfy $\mult_p (z_j) = \mult_q (z_j)$ for all $z_j \in \Sp(\ope_p) = \Sp(\ope_q)$.
\end{prop}

\begin{proof}
We first fix $z_j \in \Sp(\ope_p)$ and denote $\Pi_p =   \Pi (\ope_p, z_j)$ for short.
By the discrete spectrum assumption, $V_{p}  := \ran\Pi_p =  \Pi_pL^p(\R^{2n})$ is a finite dimensional subspace of $L^p(\R^{2n})$.
We first claim that $V_p \subset \mathcal{S}(\R^{2n})$ and start by proving it.

Since $\Pi_p$ is a spectral projector, there is a nilpotent linear map $N$ of $V_{p}$ such that $\ope_p \Pi_p u = A \Pi_p u$ for all $u \in L^p(\mathbb R^{2n})$, where $A = z_j \Id_{V_{p}} + N \in \mathcal{L}(V_p)$. As a consequence $e^{-t\ope_p} \Pi_p u = e^{-tA} \Pi_p u$ for all $t\geq0$ and $u \in L^p(\mathbb R^{2n})$.
According to Theorem~\ref{theoremeLL}, we have $e^{-t\ope_p}L^p(\mathbb R^{2n}) \subset \mathcal{S}(\R^{2n})$, and hence $e^{-tA} \Pi_p u= e^{-t\ope_p} \Pi_p u\in V_p \cap \mathcal{S}(\R^{2n})$ for all $t > 0$.
But $e^{-tA}$ is an invertible linear map of $V_p$ and hence $\dim \ran (e^{-tA}) = \dim(V_p)$. As a consequence, the space $V_p \cap \mathcal{S}(\R^{2n})$ satisfies $V_p \cap \mathcal{S}(\R^{2n}) \subset V_p$ and for a given $t>0$, $\ran e^{-tA} \Pi_p \subset V_p \cap \mathcal{S}(\R^{2n})$ whence $\dim(V_p) =\dim \ran( e^{-tA}) = \dim \ran( e^{-tA} \Pi_p) \leq \dim ( V_p \cap \mathcal{S}(\R^{2n})) <+\infty$. We deduce that $V_p =  V_p \cap \mathcal{S}(\R^{2n})\subset \mathcal{S}(\R^{2n})$, which concludes the proof of the claim.

We have obtained that $V_p \subset \mathcal{S}(\R^{2n})$, $\ope_p V_p \subset V_p$ and $\ope_p |_{V_p} =  z_j \Id_{V_{p}} + N  \in \mathcal{L}(V_p)$. Since $\ope_q |_{\mathcal{S}(\R^{2n})} =\ope_p |_{\mathcal{S}(\R^{2n})}$, we deduce that $\ope_q V_p \subset V_p$ and $\ope_q |_{V_p} =  z_j \Id_{V_{p}} + N$. This implies in particular that $z_j \in \Sp(\ope_q)$ and $ \ran\Pi_p =V_{p}  \subset  \ran\Pi_q$, where we have denoted $\Pi_q = \Pi (\ope_q, z_j)$.
 
Since $p$ and $q$ play symmetric roles, this concludes the proof of the proposition.
\end{proof}
 
Notice that Proposition \ref{p:spectre-indep-p} implies Item 1. in Theorem \ref{eigenvaluestheorem} and also that the estimates~\eqref{eq:upperboundcount} and~\eqref{eq:lowerboundcount} also hold for the $L^p$ spectrum of the operator $\ope$.

\appendix

\section{Domains and semigroup for the Fokker--Planck operator}\label{s:functional-analysis}

In this first appendix, we collect various functional analysis results on the Fokker--Planck operator $\mathcal P$, which allow in particular to prove Theorem \ref{theoreme-semigroup-intro}. In the following, we only assume that the potential $V$ is of class $C^2(\mathbb R^n)$, and not of class $C^{\infty}(\mathbb R^n)$ as before, and we will be lead to make the adapted subquadraticity assumption~\eqref{e:subquadratic}.

\subsection{Closability of the Fokker--Planck operator}\label{s:closability}
 
In what follows, we consider the Fokker--Planck operator $\mathcal P$ as the following linear continuous differential operator
\begin{equation}\label{e:fokker-planck-smooth}
	\mathcal{P} = {\mathcal T} + {\mathcal H}  : C^\infty_c({\mathbb R}^{2n}) \to C^\infty_c({\mathbb R}^{2n}),
\end{equation}
where the differential operators ${\mathcal T}$ and ${\mathcal H}$ are defined in~\eqref{definittransportharmonique}.
Also, we fix $p \in [1,\infty]$ (we will quickly restrict the analysis to the case $p\in(1,\infty)$) and define, on the Banach space $L^p({\mathbb R}^{2n}) = L^p({\mathbb R}^{2n};{\mathbb C})$, the operator
\[
	\mathcal{P}_0:= \mathcal{P}, \quad \text{ acting on }L^p({\mathbb R}^{2n})\text{ with domain }\quad D(\mathcal{P}_0) := C^\infty_c({\mathbb R}^{2n}).
\]
Note that in the present Section~\ref{s:closability}, we distinguish between the differential operator $\mathcal{P}: C^\infty_c({\mathbb R}^{2n}) \to C^\infty_c({\mathbb R}^{2n})$ and the operator $\mathcal P_0$ with domain $D(\mathcal P_0)$ on the Banach space $L^p({\mathbb R}^{2n})$. 
 
As usual, let us denote $p' := \frac{p}{p-1}$ for $p>1$, $1'=\infty$ and $\infty'=1$, so that $\frac{1}{p'}+\frac{1}{p}=1$.
We also consider $\mathcal Q$ the linear continuous differential operator defined by
\[
	\mathcal{Q}:= - {\mathcal T} + {\mathcal H}  : C^\infty_c({\mathbb R}^{2n}) \to C^\infty_c({\mathbb R}^{2n}),
\]
which is the formal adjoint of $ {\mathcal P} $. On the Banach space $L^{p'}({\mathbb R}^{2n}) = L^{p'}({\mathbb R}^{2n};{\mathbb C})$, we also define  
\[
	\mathcal{Q}_0:= \mathcal{Q}, \quad \text{ acting on }L^{p'}({\mathbb R}^{2n})\text{ with domain }\quad D(\mathcal{Q}_0) := C^\infty_c({\mathbb R}^{2n}) .
\]
Again, for now, we distinguish between the differential operator $\mathcal{Q}: C^\infty_c({\mathbb R}^{2n}) \to C^\infty_c({\mathbb R}^{2n})$ and the operator $\mathcal Q_0$ with domain $D(\mathcal{Q}_0)$ on the Banach space $L^{p^{\prime}}({\mathbb R}^{2n})$. 

We start with a closability property for the operators $\mathcal{P}_0$ and $\mathcal{Q}_0$.

\begin{lemm}
\label{l:closable}
Let $p \in [1,\infty]$.
\begin{enumerate}
	\item When $V \in W^{1,p^{\prime}}_{\rm loc}({\mathbb R}^n)$, the operator $(\mathcal{P}_0,D(\mathcal{P}_0))$ is closable.
	\item When $V \in W^{1,p}_{\rm loc}({\mathbb R}^n)$, the operator $(\mathcal{Q}_0,D(\mathcal{Q}_0))$ is also closable.
\end{enumerate}
\end{lemm}

In the following, we denote by $\mathcal{P}_{\min}:=\overline{\mathcal{P}_0}$ the closure of the operator $\mathcal P_0$ and by $\mathcal{Q}_{\min}:=\overline{\mathcal{Q}_0}$ the closure of the operator $\mathcal Q_0$.

\bigskip

Before proceeding with the proof, notice that if $u\in L^p({\mathbb R}^{2n})$ and  $\nabla V \in L^{p^{\prime}}_{\rm loc}({\mathbb R}^n)$, then the expression $\mathcal{P}_0u$ makes sense as a well-defined distribution in $\mathcal{D}^{\prime}({\mathbb R}^{2n})$: writing $\mathcal{P}_0 =\mathcal{P}_{\rm sm} - \nabla V (x) \cdot \partial_v$ with $\mathcal{P}_{\rm sm}$ being a second order differential operator with smooth coefficients, the term $- \nabla V (x) \cdot \partial_v$ is defined, for $\varphi \in C^\infty_c({\mathbb R}^{2n})$, by   
\begin{equation}\label{e:grad-V-Dprime}
	\langle \nabla_x V \cdot \partial_v u , \varphi \rangle_{\mathcal{D}^{\prime},\mathcal{D}} := -\int_{{\mathbb R}^{2n}}u(x,v) \nabla V (x) \cdot \partial_v \varphi(x,v)\, dx dv ,
\end{equation}
and satisfies, thanks to the H\"older inequality, with $K_\varphi =\mbox{supp}(\varphi)$, 
\begin{align}\label{e:nablaV-distrib}
	\left|\langle \nabla_x V \cdot \partial_v u , \varphi \rangle_{\mathcal{D}^{\prime},\mathcal{D}} \right| 
	& \leq \|u\|_{L^p} \| \nabla V (x) \cdot \partial_v \varphi(x,v)\|_{L^{p^{\prime}}} \nonumber \\
	& \leq \|u\|_{L^p} \| \nabla V \|_{L^{p^{\prime}}(K_\varphi)} \sup_{x \in {\mathbb R}^n} \| \partial_v \varphi(x,\cdot)\|_{L^{p^{\prime}}_v} .
\end{align}
We denote by $\Gamma(\mathcal{P}_0) := \big\{ (\varphi, \mathcal{P}_0 \varphi)\,\vert\, \varphi \in C^\infty_c ({\mathbb R}^{2n}) \big\} \subset L^p ({\mathbb R}^{2n})^2$ the graph of $\mathcal{P}_0$.
Although it belongs to the folklore, we produce here a short selfcontained proof of Lemma~\ref{l:closable}.

\begin{proof}[Proof of Lemma~\ref{l:closable}]
To prove that $\mathcal{P}_0$ is closable, it suffices to prove that the closure $\overline{\Gamma(\mathcal{P}_0)}$ of the graph of $\mathcal{P}_0$ in $L^p(\mathbb R^{2n})^2$ is the graph of an operator. Since $\overline{\Gamma(\mathcal{P}_0)}$ is a vector subspace of $L^p({\mathbb R}^{2n})\times L^p({\mathbb R}^{2n})$, it suffices  by linearity to prove that it has the ``single-valued property'' 
\begin{equation}\label{e:single-valued-pty}
	(0,g) \in  \overline{\Gamma(\mathcal{P}_0)} \implies g =0 .
\end{equation}
We have $(0,g)\in  \overline{\Gamma(\mathcal{P}_0)}$ if and only if there exists a sequence $(f_n)_n$ in $C^\infty_c({\mathbb R}^{2n})$ such that $f_n \to 0$ in $L^p({\mathbb R}^{2n})$ and $\mathcal{P}_0 f_n \to g \in L^p({\mathbb R}^{2n})$. 
On the one hand, we have 
\[
	\langle \mathcal{P}_0 f_n , \varphi \rangle_{\mathcal{D}^{\prime},\mathcal{D}} \to \langle g , \varphi \rangle_{\mathcal{D}^{\prime},\mathcal{D}} .
\]
On the other hand, writing $\mathcal{P}_0 =\mathcal{P}_{\rm sm} - \nabla V (x) \cdot \partial_v$ with $\mathcal{P}_{\rm sm}$ a second-order differential operator with smooth coefficients, we have, for all $\varphi \in C^\infty_c({\mathbb R}^{2n})$,
\[
	\langle \mathcal{P}_0  f_n , \varphi \rangle_{\mathcal{D}^{\prime},\mathcal{D}} 
	= \langle f_n,  ^t \mathcal{P}_{\rm sm}\varphi \rangle_{\mathcal{D}^{\prime},\mathcal{D}} + \int_{{\mathbb R}^{2n}}f_n(x,v) \nabla V (x) \cdot \partial_v \varphi(x,v)\, dx dv.
\]
The first term in the right-hand side converges to zero since $f_n \to 0$ in $L^p({\mathbb R}^{2n})$. 
Concerning the second term in the right-hand, we use~\eqref{e:nablaV-distrib} with $u=f_n$ to obtain, with $K_\varphi=\mbox{supp}(\varphi)$,
\begin{align*}
	\bigg|\int_{{\mathbb R}^{2n}}f_n(x,v) \nabla V (x) \cdot \partial_v \varphi(x,v)\, dx dv  \bigg| 
	& \leq \|f_n\|_{L^p} \| \nabla V \|_{L^{p^{\prime}}(K_\varphi)} \sup_{x \in {\mathbb R}^n} \| \partial_v \varphi(x,\cdot)\|_{L^{p^{\prime}}_v},
\end{align*}
which converges to zero since $f_n \to 0$ in $L^p({\mathbb R}^{2n})$.
This proves that $g=0$, hence $\overline{\Gamma(\mathcal{P}_0)}$ satisfies the single-valued property~\eqref{e:single-valued-pty} and $\mathcal{P}_0$ is closable. 
The proof for $\mathcal{Q}_0$ follows the same lines.
\end{proof}

\subsection{Accretivity of the Fokker--Planck operator \texorpdfstring{$\mathcal{P}_{\min}$}{} on \texorpdfstring{$L^p(\mathbb R^{2n})$}{}}

In this subsection, we prove that the operator $(\mathcal P_{\min},D(\mathcal P_{\min}))$ is accretive (see Lemma \ref{l:accretif} for a precise statement), which is a key step in the proof the existence of the semigroup $(e^{-t\mathcal P})_{t\geq0}$, performed in Subsection \ref{subsec:gene} . From now on, given some $N\in\mathbb N^*$, we use the duality between $L^p(\mathbb R^N)$ and $L^{p^{\prime}}(\mathbb R^N)$ with the following convention (real duality):
\begin{equation}\label{e:duality-bracket}
	\langle w, u \rangle_{L^{p^{\prime}},L^p} =    \langle u ,w \rangle_{L^p,L^{p^{\prime}}}  = \int_{{\mathbb R}^N}w u \, dx,\quad u \in L^p({\mathbb R}^N), w\in L^{p^{\prime}}({\mathbb R}^N) .
\end{equation}
Lemma~\ref{l:closable} shows that unambiguously, for any fixed $p \in (1,\infty)$,
\begin{align*}
	D(\mathcal{P}_{\rm min}) & := \big\{ f \in L^p(\mathbb R^{2n})  \ | \ \exists g \in L^{p}(\mathbb R^{2n}),\,(f,g) \in\overline{\Gamma(\mathcal{P}_0)}\big\}  \\
	& = \big\{ f \in L^p(\mathbb R^{2n})  \ | \ \underbrace{\exists (f_n)_n \in C^\infty_c(\mathbb R^{2n})^{\mathbb N}, \exists g\in L^p(\mathbb R^{2n}),\,f_n \to f,\, \mathcal{P}f_n\rightarrow g}_{(\star)}\big\},
\end{align*}
(as in Theorem \ref{theoreme-semigroup-intro}) and, for any $ f \in D(\mathcal{P}_{\rm min})$,
\begin{align*} 
	\mathcal{P}_{\rm min} f & = \mbox{the unique} \ g \in L^p(\mathbb R^{2n}) \ \mbox{such that} \ (f,g) \in \overline{\Gamma(\mathcal{P}_0)}  \\
	&= \mbox{the unique} \ g \in L^p(\mathbb R^{2n}) \ \mbox{such that $(\star)$}.
\end{align*}

\begin{lemm}\label{l:accretif}
Assume that  $ p \in (1 ,\infty )$ and $V \in W^{1,p^{\prime}}_{\rm loc}({\mathbb R}^n)$. Let $c_0 := \frac{m\gamma n}{2}$. Then, the operator $\mathcal{P}_{\min}+c_0$ is accretive on $L^p({\mathbb R}^{2n})$, in the sense that
\begin{equation}\label{eq:accr}
	\forall u \in D(\mathcal{P}_{\min}), \exists \tilde{u} \in L^{p^{\prime}}({\mathbb R}^{2n}),\quad
	\begin{cases}
		\langle \tilde{u}, u \rangle_{L^{p^{\prime}},L^p} = \|u\|^2_{L^p}=\| \tilde{u}\|^2_{L^{p^{\prime}}},  \\[5pt]
		\Re \langle \tilde{u}, (\mathcal{P}_{\min} +c_0) u \rangle_{L^{p'},L^p} \geq 0. 
	\end{cases}
\end{equation}
The same statement holds for the operator $\mathcal{Q}_{\min}+c_0$ on $L^{p^{\prime}}({\mathbb R}^{2n})$.
 \end{lemm}
 
Note that this lemma does not use that $\mathcal{P}_{\min}=\mathcal{P}_{\max}$ (see paragraph \ref{subsectionPmax} for the definition of $ {\mathcal P}_{\rm max} $), and in particular does not assume that $V$ satisfies the subquadraticity assumption~\eqref{e:subquadratic}.

 Note also that, for $p \in (1,\infty)$, there is a single such  $\tilde{u}$, given by $\tilde u = 0$ when $u = 0$ and otherwize
 \begin{equation}\label{eq:uetoile}
	\tilde{u} := \mathds{1}_{\{u \neq 0\}} \overline{u} \frac{|u|^{p-2}}{\|u\|_{L^p}^{p-2}}.
\end{equation}
This lemma might be true for $p=1$ but our proof does not work.
 
We rely on Lemma~\ref{e:app-dual-continu}  to reduce the proof of~\eqref{eq:accr} only for functions in $C^{\infty}_c(\mathbb R^{2n})$.
  
\begin{proof}[Proof of Lemma~\ref{l:accretif}]
Notice first that the function $\tilde u$ defined by~\eqref{eq:uetoile} belongs to $L^{p^{\prime}}(\mathbb R^{2n})$ and satisfies 
\[
	\langle \tilde{u}, u \rangle_{L^{p^{\prime}},L^p} = \int_{{\mathbb R}^{2n}}\tilde{u}u\, dx = \|u\|^2_{L^p}=\| \tilde{u}\|^2_{L^{p^{\prime}}}.
\]
For $u=0$, the statement of the lemma is straightforward, and in the rest of the proof, we assume that $u \neq 0$.
We notice that, given $v\in L^p(\mathbb R^{2n})$ such that $v\neq0$, the map $L^p(\mathbb R^{2n}) \to L^{p^{\prime}}(\mathbb R^{2n})$, $u \mapsto \tilde{u}$ is continuous at the point $v$ as a consequence of Lemma~\ref{e:app-dual-continu} together with the continuity and the positivity of the map $u \mapsto \|u\|_{L^p}$ near $v\neq 0$. Hence, by density of $C^\infty_c(\mathbb R^{2n})$ in $D(\mathcal{P}_{\min})$, it suffices to check~\eqref{eq:accr} for all $u \in C^\infty_c({\mathbb R}^{2n})\setminus\{0\}$. Given the expression of $\tilde{u}$, it suffices to check that for every $u \in C^\infty_c ({\mathbb R}^{2n})$,
 \[
	\Re \int_{{\mathbb R}^{2n}}u^* \mathcal{P} u\, dx \geq 0, \quad \text{ with } u^* := \mathds1_{\{u \neq 0\}} \overline{u}  |u|^{p-2} .
\]
We next set, for $\varepsilon>0$, $u_\varepsilon : = (|u|^2+\varepsilon^2)^{\frac12}$ so that 
\[
	u_\varepsilon \geq \varepsilon ,  \quad u_\varepsilon \geq |u| , \quad   u_\varepsilon \in C^\infty({\mathbb R}^{2n}), \quad u_\varepsilon = \varepsilon\quad \text{ outside of }\mbox{supp}(u).
\]
With $\partial = \partial_{x_j}$ or $\partial=\partial_{v_j}$, we have 
\[
	\partial u_\varepsilon = \frac12 (|u|^2+\varepsilon^2)^{-\frac12} (\overline{u} \partial u+u \partial \overline{u} ) =u_\varepsilon^{-1} \Re(\overline{u} \partial u),
\] 
and hence,
\[
	\partial (u_\varepsilon^{p}) = p u_\varepsilon^{p-1} \partial u_\varepsilon =  p u_\varepsilon^{p-1} u_\varepsilon^{-1} \Re(\overline{u} \partial u) =  p u_\varepsilon^{p-2}  \Re(\overline{u} \partial u).
\]
Now, we consider each term in the operator $\mathcal{P}$ defined in~\eqref{e:fokker-planck-smooth} and satisfying
\[
	\mathcal{P}+c_0 =  \mathcal{T} +   \frac{\gamma}{\beta} \bigg(-  \Delta_v + \frac{m^2 \beta^2}{4} |v|^2\bigg), \quad \mathcal{T} = m v\cdot \partial_x - \nabla_xV \cdot \partial_v .
\]
First, $\mathcal{T} =m v\cdot \partial_x - \nabla_xV \cdot \partial_v$ having real coefficients, we have
\[
	\mathcal{T} u_\varepsilon^{p}  =  p u_\varepsilon^{p-2}  \Re(\overline{u} \mathcal{T} u).
\]
Noticing that the vector field $\mathcal{T}$ is divergence free and $u_\varepsilon = \varepsilon$ is constant outside a compact set, an integration by parts implies that 
\[
	0 = \int_{{\mathbb R}^{2n}} \mathcal{T} (u_\varepsilon^{p}-\varepsilon^p)\, dx 
	= \int_{{\mathbb R}^{2n}} \mathcal{T} (u_\varepsilon^{p})\, dx 
	=  p  \int_{{\mathbb R}^{2n}}  u_\varepsilon^{p-2}  \Re(\overline{u} \mathcal{T} u)\, dx 
	=  p \Re  \int_{{\mathbb R}^{2n}}  u_\varepsilon^{p-2}  \overline{u} \mathcal{T} u\, dx .
\]
Finally, notice that $u_\varepsilon^{p-2}  \overline{u}= u_\varepsilon^{p-2}  \overline{u} \mathds{1}_{\{u \neq 0\}} \to  u^*$ pointwise, is compactly supported in $\mbox{supp}(u)$, and satisfies 
\begin{align*}
	& |u_\varepsilon^{p-2}  \overline{u}| \le |u|^{p-2} |u| = |u|^{p-1},\quad \text{when $p\in(1,2)$,} \\
	& |u_\varepsilon^{p-2}  \overline{u}| \le(|u|^2 + 1)^\frac{p-2}2 |u|,\qquad\, \text{when $p\geq2$},
\end{align*}
the dominating functions being integrable. By the dominated convergence theorem, we may therefore let $\varepsilon \to 0^+$, which implies $ \Re  \int_{{\mathbb R}^{2n}}  u^* \mathcal{T} u\, dx=0$.

Secondly, we have  $ \Re  \int_{{\mathbb R}^{2n}}  u^*|v|^2 u\, dx= \int_{{\mathbb R}^{2n}}  u^*|v|^2 u\, dx = \int_{{\mathbb R}^{2n}}  |v|^2 |u|^p\, dx \geq 0$.

Thirdly, we have (as above) 
\begin{equation}\label{e:dv-cv}
	\Re\left( - \int_{{\mathbb R}^{2n}}  \overline{u} u_\varepsilon^{p-2}  \partial_{v_j}^2 u\, dx\right)\to \Re\left( - \int_{{\mathbb R}^{2n}}  u^*\partial_{v_j}^2 u\, dx\right) , \quad \text{ as } \varepsilon \to 0^+ ,
\end{equation}
and, integrating by parts, 
\begin{align*}
	&\Re\left( - \int_{{\mathbb R}^{2n}}  \overline{u} u_\varepsilon^{p-2}  \partial_{v_j}^2 u\, dx\right) 
	= \Re\left( \int_{{\mathbb R}^{2n}}  \partial_{v_j}( \overline{u} u_\varepsilon^{p-2} ) \partial_{v_j} u\, dx\right) \\
	=\ & \Re\left( \int_{{\mathbb R}^{2n}}  |\partial_{v_j} u|^2 u_\varepsilon^{p-2}\, dx\right) + \Re\left( \int_{{\mathbb R}^{2n}}  \partial_{v_j}(u_\varepsilon^{p-2} )   \overline{u} \partial_{v_j} u\, dx\right)\\
	=\ & \Re\left( \int_{{\mathbb R}^{2n}}  |\partial_{v_j} u|^2 u_\varepsilon^{p-2}\, dx\right) +\Re\left( \int_{{\mathbb R}^{2n}}  (p-2) u_\varepsilon^{p-4}  \Re(\overline{u} \partial_{v_j} u)  \overline{u} \partial_{v_j} u\, dx\right)\\ 
	=\ & \int_{{\mathbb R}^{2n}}  |\partial_{v_j} u|^2 u_\varepsilon^{p-2}\, dx  + (p-2)  \Re\left( \int_{{\mathbb R}^{2n}} u_\varepsilon^{p-4}  \Re(\overline{u} \partial_{v_j} u)  \overline{u} \partial_{v_j} u\, dx\right)\\ 
	=\ & \int_{{\mathbb R}^{2n}}  |\partial_{v_j} u|^2 u_\varepsilon^{p-2}\, dx  + (p-2)   \int_{{\mathbb R}^{2n}} u_\varepsilon^{p-4}  \Re(\overline{u} \partial_{v_j} u)^2\, dx \\
	=\ & \int_{{\mathbb R}^{2n}}    u_\varepsilon^{p-4}  \big( |\partial_{v_j} u|^2 u_\varepsilon^{2}    + (p-2)  \Re(\overline{u} \partial_{v_j} u)^2 \big)\, dx .
\end{align*}
Now, we remark that, since $|u| \leq u_\varepsilon$, 
\[
	0\leq  \Re(\overline{u} \partial_{v_j} u)^2 \leq  |\overline{u} \partial_{v_j} u|^2 =  |u|^2 |\partial_{v_j} u|^2 \leq u_\varepsilon^2  |\partial_{v_j} u|^2 .
\]
With the previous lines, we have obtained
\begin{align*}
	\Re\left( - \int_{{\mathbb R}^{2n}}  \overline{u} u_\varepsilon^{p-2}  \partial_{v_j}^2 u\, dx\right) 
	& \geq  \int_{{\mathbb R}^{2n}}    u_\varepsilon^{p-4}  \big( \Re(\overline{u} \partial_{v_j} u)^2   + (p-2)  \Re(\overline{u} \partial_{v_j} u)^2 \big)\, dx \\
	& \geq (p-1) \int_{{\mathbb R}^{2n}}    u_\varepsilon^{p-4}   \Re(\overline{u} \partial_{v_j} u)^2\,  dx  \geq 0 .
\end{align*}
We finally let $\varepsilon\to0^+$ in this inequality, using~\eqref{e:dv-cv}, to obtain $ \Re( - \int_{{\mathbb R}^{2n}}  u^*\partial_{v_j}^2 u\, dx) \geq 0$. This concludes the proof of the lemma.
\end{proof}

Note that when $p\geq 2$, the proof of~\eqref{eq:accr} is simpler and does not require the regularization of $|u|$ via $u_\varepsilon$.

\subsection{Domains for the Fokker--Planck operator: \texorpdfstring{$\mathcal{P}_{\min}=\mathcal{P}_{\max}$}{}} \label{subsectionPmax}

As in Theorem \ref{theoreme-semigroup-intro}, we denote
\[
	D(\mathcal{P}_{\rm max}) := \big\{ f \in L^p(\mathbb R^{2n}) \ | \ \mathcal{P}f \in L^p(\mathbb R^{2n})  \big\},
\]
and, for any $  f \in D ( \mathcal{P}_{\rm max})$,
\[
	\mathcal{P}_{\rm max} f := \mathcal{P} f,
\]
where $\mathcal{P} f$ is taken in the sense of distributions. Similarly, we define  $\mathcal{Q}_{\max}$.
The main result of this section is that, if $V$ is subquadratic, then the operators $\mathcal{P}_{\min}$ and $\mathcal{P}_{\max}$ coincide. This is a well-known fact for elliptic operators with smooth coefficients~\cite[Theorem 14.15]{Wong} or for first order pseudodifferential operators~\cite[Lemma 29]{FaureSjostrand11}.
In the rest of this section, we will use the following subquadraticity assumption on $V$:
\begin{equation}\label{e:subquadratic}
	\partial_x^{\alpha}V \in L^\infty({\mathbb R}^n)\quad\text{ for all }\alpha \in \N^n \text{ such that }|\alpha| = 2.
\end{equation}
The stronger condition (\ref{hypothesepotentiel}) is anyways assumed in the main part of the paper.

\begin{prop} \label{fermetureessentielle} 
Let $p \in[1,+\infty)$. Assume that $V$ satisfies the subquadraticity assumption~\eqref{e:subquadratic}. Then,
\begin{enumerate}
	\item For any $f \in D(\mathcal{P}_{\max})$, there is a sequence $(f_k)_k \in C^\infty_c({\mathbb R}^{2n})^{\mathbb N}$ such that $f_k \to f$ in $L^p({\mathbb R}^{2n})$ and $\mathcal{P}f_k \to \mathcal{P}f$ in $L^p({\mathbb R}^{2n})$. 
	\item $D(\mathcal{P}_{\rm min}) = D ( \mathcal{P}_{\rm max} ) $ and $ \mathcal{P}_{\rm min} = \mathcal{P}_{\rm max}$ is closed and densely defined.
\end{enumerate}
Finally, the same statement holds for the operators $\mathcal{Q}_{\min}$ and $\mathcal Q_{\max}$ on the space $L^{p^{\prime}}({\mathbb R}^{2n})$ for $p^{\prime} \in[1,+\infty)$.
\end{prop}

In other words, the operator $\mathcal{P}_0$ with domain $D(\mathcal{P}_0):= C^\infty_c({\mathbb R}^n)$ admits a unique closed extension, namely $\mathcal{P}_{\min} = \mathcal{P}_{\max}$.
The second item in Proposition~\ref{fermetureessentielle} is a consequence of the first. Indeed, by definition, $\mathcal{P}_{\max}$ is an extension of $\mathcal{P}_{\min}$, namely  $D(\mathcal{P}_{\min}) \subset D(\mathcal{P}_{\max})$ and $\mathcal{P}_{\max}|_{D(\mathcal{P}_{\min})}=\mathcal{P}_{\min}$. The first item in the statement of Proposition~\ref{fermetureessentielle} shows the converse inclusion $D(\mathcal{P}_{\max}) \subset D(\mathcal{P}_{\min})$.

We next consider an approximation of the identity $(\rho_\delta)_{\delta>0}$ in ${\mathbb R}^n$, defined by 
\[
	\rho_\delta(x) : = \frac{1}{\delta^n}\rho\Big(\frac{x}{\delta}\Big),\quad x\in\mathbb R^n,
\]
with $\rho$ satisfying  
\begin{equation}\label{e:asspt-rho}
	\rho \in C^\infty_c({\mathbb R}^n;{\mathbb R}_+), \quad   \mbox{supp}(\rho)\subset B(0,1), \quad \rho(-x) = \rho(x), \quad  \int_{{\mathbb R}^n}\rho(x)\,dx=1.
\end{equation}
Given a function $u\in L^p({\mathbb R}^{2n})$, we set 
\[
	u_\delta := u*\rho_\delta = \hat \rho (\delta D_x) u ,\quad  \hat \rho \in \mathcal{S}({\mathbb R}^n).
\]
Recall, from the Young inequality, that for every $p\in[1,+\infty]$,
\begin{equation}\label{e:mapping-convolution}
	\|u_\delta\|_{L^p} \leq \|u\|_{L^p} , \quad \textit{ i.e. } \| \hat \rho (\delta D_x)\|_{L^p\to L^p}\leq 1,
\end{equation}
and whenever $p\ne+\infty$,
\begin{equation}\label{e:convergence-convolution}
	u_\delta\underset{\delta\rightarrow0^+}{\longrightarrow} u \quad \text{ in $L^p({\mathbb R}^{2n})$}.
\end{equation}
The main technical step in the proof that $D(\mathcal{P}_{\max}) \subset D(\mathcal{P}_{\min})$ on $L^p({\mathbb R}^{2n})$ is the following commutator lemma.

\begin{lemm}\label{e:commutator}
Let $p \in[1,+\infty)$. With $\rho$ satisfying~\eqref{e:asspt-rho} and letting $\chi \in C^\infty_c({\mathbb R}^n;[0,1])$ be such that $\chi=1$ on $B(0,1)$, the family of  operators
\begin{align}\label{e:def-xi-delta}
	\Xi_\delta := \hat \rho (\delta D_x) \chi(\delta x)  \hat \rho (\delta D_v) \chi(\delta v) , \quad \delta \in (0,1] ,
\end{align}
satisfies the following properties:
\begin{enumerate}
	\item  \label{i:smoothing}  The operator $\Xi_\delta$ maps $L^p({\mathbb R}^{2n})$ into $C^\infty_c({\mathbb R}^{2n})$ together with $\|\Xi_\delta\|_{L^p\rightarrow L^p}\leq 1$. 
	\item \label{i:strong-cv} For any $u \in  L^p({\mathbb R}^{2n})$,  $\Xi_\delta u \to u$ in $L^p({\mathbb R}^{2n})$ as $\delta \to 0^+$.
	\item  \label{i:commutator} If $V$ satisfies the subquadraticity assumption~\eqref{e:subquadratic}, the operator $[\mathcal{P} , \Xi_\delta]$ maps $L^p({\mathbb R}^{2n})$  into itself and, for any $u \in  L^p({\mathbb R}^{2n})$, $[\mathcal{P} , \Xi_\delta] u \to 0$ in $L^p({\mathbb R}^{2n})$ as $\delta \to 0^+$.
\end{enumerate}
Finally, the same statements hold for the operator $\mathcal{Q}$ on $L^{p^{\prime}}({\mathbb R}^{2n})$ for $p^{\prime} \in[1,+\infty)$.
\end{lemm}

\begin{proof}[Proof of Proposition~\ref{fermetureessentielle}]
To prove the first item of Proposition~\ref{fermetureessentielle} from Lemma~\ref{e:commutator}, let $f\in D(\mathcal{P}_{\max})$,  $\delta =\delta_k := k^{-1}$ and set $f_k := \Xi_{\delta_k}f$ with $\Xi_{\delta}$ defined in~\eqref{e:def-xi-delta}.
From Item~\ref{i:smoothing} in Lemma~\ref{e:commutator}, we have $f_k \in C^\infty_c({\mathbb R}^{2n})$ for any $k \in {\mathbb N}^*$.
According to Item~\ref{i:strong-cv} in Lemma~\ref{e:commutator}, we have $f_k\to f$ in $L^p({\mathbb R}^{2n})$, and it only remains to prove that $\mathcal{P}f_k\to \mathcal{P}f$ in $L^p({\mathbb R}^{2n})$. To this aim, we decompose
\begin{equation}\label{e:Pfn-xin}
	\mathcal{P}f_k = \mathcal{P}\Xi_{\delta_k}f =\Xi_{\delta_k} \mathcal{P}f +[\mathcal{P},\Xi_{\delta_k}]f .
\end{equation}
On the one hand, $\mathcal{P}f\in L^p({\mathbb R}^{2n})$ since $f\in D(\mathcal{P}_{\max})$, and, together with Item~\ref{i:strong-cv} in Lemma~\ref{e:commutator}, this implies $\Xi_{\delta_k} \mathcal{P}f \to \mathcal{P}f$ in $L^p({\mathbb R}^{2n})$.
On the other hand, using the subquadraticity assumption on $V$, Item~\ref{i:commutator}  of Lemma~\ref{e:commutator} implies that  $[\mathcal{P} , \Xi_{\delta_k}] f \to 0$  in $L^p({\mathbb R}^{2n})$.
Coming back to~\eqref{e:Pfn-xin}, we have finally obtained that $\mathcal{P}f_k \to \mathcal{P}f$  in $L^p({\mathbb R}^{2n})$. This concludes the proof of the first item in Proposition~\ref{fermetureessentielle} from Lemma~\ref{e:commutator}, the proof of the second item being a straightforward consequence of the first.
\end{proof}

The following remark is used repeatedly in the proof of Lemma~\ref{e:commutator}.

\begin{rema}\label{r:tensor}
If $X_\delta : L^p({\mathbb R}^n_x)\rightarrow L^p({\mathbb R}^n_x)$ and $V_\delta : L^p({\mathbb R}^n_v)\rightarrow L^p({\mathbb R}^n_v)$ are linear and bounded, then $X_\delta \otimes V_\delta : L^p({\mathbb R}^{2n}_{x,v})\rightarrow L^p({\mathbb R}^{2n}_{x,v})$ defines a bounded operator. In~\eqref{e:def-xi-delta}, we have made the abuse of notation of writing $X_\delta= X_\delta\otimes \mbox{Id}$ and $V_\delta=  \mbox{Id}\otimes V_\delta$ (with $X_\delta =\hat \rho (\delta D_x) \chi(\delta x)$ and $V_\delta=  \hat \rho (\delta D_v) \chi(\delta v)$ in~\eqref{e:def-xi-delta}). Notice that we have 
$$
	X_\delta \otimes V_\delta =(X_\delta\otimes \mbox{Id} )\circ ( \mbox{Id}\otimes V_\delta) = (\mbox{Id}\otimes V_\delta) \circ (X_\delta\otimes \mbox{Id} )= X_\delta V_\delta =   V_\delta X_\delta, 
$$
where the last two inequalities use the same abuse of notation.
We notice, using the Fubini theorem, that $\|X_\delta \otimes V_\delta\|_{L^p\rightarrow L^p} \leq \|X_\delta\|_{L^p\rightarrow L^p}  \|V_\delta\|_{L^p\rightarrow L^p}$. We shall also use that if $ \|V_\delta\|_{L^p\rightarrow L^p} \leq C$ uniformly, then
\[
	X_\delta u \to 0 \text{ for all } u\in L^p({\mathbb R}^{n}) \implies X_\delta \otimes V_\delta f \to 0 \text{ for all } f \in L^p({\mathbb R}^{2n}). 
\]
Indeed, by density and continuation of convergence, it suffices to check this property for tensor products $f=u \otimes w$ with $u,w \in L^p({\mathbb R}^n)$. And for such a function, we have by the Fubini theorem 
\[
	\| (X_\delta \otimes V_\delta)(u \otimes w) \|_{L^p} 
	= \| X_\delta  u\|_{L^p}\| V_\delta  w \|_{L^p} \\
	\leq \| X_\delta  u\|_{L^p}C \| w \|_{L^p}\to 0.
\]
\end{rema}

\begin{proof}[Proof of Lemma~\ref{e:commutator}]
The norm estimate in Item~\ref{i:smoothing} follows from the fact that each factor in the product defining $\Xi_\delta$ is bounded by $1$, as a consequence of~\eqref{e:mapping-convolution} together with $\|\chi \|_{L^\infty} \leq 1$. That $\Xi_\delta : L^p(\mathbb R^{2n}) \to C^\infty_c(\mathbb R^{2n})$ follows from the fact that the operators $ \hat \rho (\delta D_x)$ and $ \hat \rho (\delta D_v)$ are convolutions with a smooth compactly supported function, according to~\eqref{e:asspt-rho}.
  
The statement of Item~\ref{i:strong-cv}  is a consequence of the fact that, denoting by $Q_\delta$ either term in $\hat \rho (\delta D_x)$, $\chi(\delta x)$, $\hat \rho (\delta D_v)$, $\chi(\delta v)$ and for $u \in L^p({\mathbb R}^{2n})$, we have $Q_\delta u \to u$ in $L^p({\mathbb R}^{2n})$.
This is a consequence of~\eqref{e:convergence-convolution} or dominated convergence.
 
We now turn to the proof of the main statement Item~\ref{i:commutator}, concerning the commutator $[\mathcal{P} , \Xi_\delta]$. Recalling the definition of $\mathcal{P}$ in~\eqref{e:fokker-planck-smooth}--\eqref{definittransportharmonique}, it suffices to prove that each term in 
\begin{equation}\label{e:4terms}
	[\partial_{v_j}^2 , \Xi_\delta]  , \quad [v_j^2 , \Xi_\delta]  , \quad [ v_j \partial_{x_j} , \Xi_\delta] , \quad  [\partial_{x_j} V\partial_{v_j} , \Xi_\delta] ,
\end{equation}
is bounded on $L^p({\mathbb R}^{2n})$ and that each term
\[
	[\partial_{v_j}^2 , \Xi_\delta] u , \quad [v_j^2 , \Xi_\delta] u , \quad [ v_j \partial_{x_j} , \Xi_\delta] u , \quad  [\partial_{x_j} V\partial_{v_j} , \Xi_\delta] u,
\]
converges to zero in $L^p({\mathbb R}^{2n})$ for $u \in L^p({\mathbb R}^{2n})$, $j \in \{1,\dots, n\}$,  as $\delta \to 0^+$.

We start with the last term in~\eqref{e:4terms} and notice that, with $V_j := \partial_{x_j} V$, we have
\begin{align} 
	[\partial_{x_j} V\partial_{v_j} , \Xi_\delta] 
	& =  [V_j \partial_{v_j} , \hat \rho (\delta D _x) \chi(\delta x)  \hat \rho (\delta D _v) \chi(\delta v)] \nonumber \\
	& = \hat \rho (\delta D _x) \chi(\delta x)  \hat \rho (\delta D _v)  [V_j \partial_{v_j} , \chi(\delta v)] + [V_j , \hat \rho (\delta D _x)] \chi(\delta x)  \hat \rho (\delta D _v) \chi(\delta v)  \nonumber \\
	& = \hat \rho (\delta D _x) \chi(\delta x)  \hat \rho (\delta D _v) V_j  \delta (\partial_j \chi)(\delta v) +\frac{1}{\delta}[V_j , \hat \rho (\delta D _x)] \chi(\delta x)  \delta \partial_{v_j}  \hat \rho (\delta D _v)\chi(\delta v) \nonumber \\
	& = \hat \rho (\delta D _x) \chi(\delta x)V_j  \delta  \hat \rho (\delta D _v)  (\partial_j \chi)(\delta v) + \delta \partial_{v_j}  \hat \rho (\delta D _v)\chi(\delta v)
	\frac{1}{\delta}[V_j , \hat \rho (\delta D _x)] \chi(\delta x). \label{e:first-comm-to-be-est}
 \end{align}
From the subquadraticity assumption~\eqref{e:subquadratic} and the mean-value theorem, we deduce that there is $C>0$ such that  $|V_j(x)| \leq |\nabla V(x)|\leq C (1+|x|)$ for all $x \in {\mathbb R}^n$.
As a consequence,
\[
	|\chi(\delta x)V_j(x) \delta| \leq C (\delta+|\delta x|) \leq C \delta \|\chi\|_{L^\infty} + C \| |x|\chi(x)\|_{L^\infty} , \quad x \in {\mathbb R}^n , \delta \in(0,1) .
\]
Hence, the first term in the right-hand side of~\eqref{e:first-comm-to-be-est} is $L^p$ bounded uniformly in $\delta$.
The second  term in the right-hand side of~\eqref{e:first-comm-to-be-est} is $L^p$ bounded as a consequence of~\eqref{e:first-claim-convolution}  in Lemma~\ref{e:main-comm-estimate} together with the fact that the operator $\delta\partial_{v_j}  \hat \rho (\delta D _v)$ is $L^p$ bounded uniformly in $\delta$: indeed,  
\[
	\delta\partial_{v_j}  \hat \rho (\delta D _v) u =
	i ( \xi_j \hat \rho(\xi))\big|_{\xi = \delta D _v}u = \mathcal{F}(\partial_j \rho)(\delta D _v)u
	= \frac{1}{\delta^{n}}(\partial_j\rho)\left(\frac{\cdot}{\delta}\right) * u ,
\]
and from the Young inequality,
\begin{equation}\label{e:bounded-convol-young}
	\|\delta \partial_{v_j}  \hat \rho (\delta D _v) \|_{L^p\rightarrow L^p} \leq \| \partial_j \rho \|_{L^1} . 
\end{equation}
Moreover, once applied to a function $u \in L^p({\mathbb R}^{2n})$, the first term in the right-hand side of~\eqref{e:first-comm-to-be-est} converges to zero  since $(\partial_j \chi)(\delta v)$ converges to zero almost everywhere (and using dominated convergence) and all other terms are $L^p$ bounded (see Remark~\ref{r:tensor}).
Similarly, once applied to a function $u \in L^p({\mathbb R}^{2n})$, the second term in the right-hand side of~\eqref{e:first-comm-to-be-est} converges to zero  since $\frac{1}{\delta}[V_j , \hat \rho (\delta D _x)] w$ converges to zero in $L^p(\mathbb R^{2n})$ for $w \in L^p({\mathbb R}^{2n})$ according to~\eqref{e:main-claim-convolution} in Lemma~\ref{e:main-comm-estimate} (and the subquadraticity assumption~\eqref{e:subquadratic}) and all other terms are $L^p$ bounded  (see Remark~\ref{r:tensor}).
This concludes the proof for the last term in~\eqref{e:4terms}. 
 
The term $[ v_j \partial_{x_j} , \Xi_\delta]$ in~\eqref{e:4terms} is treated exactly the same way (the situation is even simpler).

Concerning the first term $[\partial_{v_j}^2 , \Xi_\delta]$  in~\eqref{e:4terms}, we write 
\[
	[\partial_{v_j}^2 , \Xi_\delta] 
	=  [\partial_{v_j}^2, \hat \rho (\delta D _x) \chi(\delta x)  \hat \rho (\delta D _v) \chi(\delta v)]
	=  \hat \rho (\delta D _x) \chi(\delta x)  \hat \rho (\delta D _v)  [\partial_{v_j}^2, \chi(\delta v)] ,
\]
and, on account to Remark~\ref{r:tensor}, it suffices to prove that $\hat \rho (\delta D _v)  [\partial_{v_j}^2, \chi(\delta v)]$ is $L^p$ bounded and converges strongly to zero.
To this aim, we write
\begin{align}
	\hat \rho (\delta D _v)  [\partial_{v_j}^2, \chi(\delta v)]  
	& = \hat \rho (\delta D _v) \partial_{v_j}  [\partial_{v_j} , \chi(\delta v)] 
	+ \hat \rho (\delta D _v)  [\partial_{v_j} , \chi(\delta v)] \partial_{v_j} \nonumber \\
	& =   2  \hat \rho (\delta D _v) \partial_{v_j}  [\partial_{v_j} , \chi(\delta v)] 
	- \hat \rho (\delta D _v)  [\partial_{v_j} ,  [\partial_{v_j} , \chi(\delta v)] ] \nonumber \\
	& =   2  \hat \rho (\delta D _v) \delta \partial_{v_j} (\partial_j\chi)(\delta v) 
	-\hat \rho (\delta D _v)  \delta^2 (\partial_j^2\chi)(\delta v) . \label{e:first-comm-to-be-est-2}
\end{align}
The second term in the right-hand side of~\eqref{e:first-comm-to-be-est-2} is bounded on $L^p(\mathbb R^{2n})$ and converges to zero in norm (it is of order $\delta^2$ times an operator bounded on $L^p(\mathbb R^{2n})$).
As for the first term, the operator $\hat \rho (\delta D _v) \delta\partial_{v_j}$ is $L^p$ bounded uniformly in $\delta$ from~\eqref{e:bounded-convol-young} and hence the first term is uniformly $L^p$ bounded.
Moreover, once applied to a function $u \in L^p({\mathbb R}^{2n})$, the first term converges to zero since $(\partial_j \chi)(\delta v)$ converges to zero almost everywhere (and using dominated convergence) and all other terms are $L^p$ bounded.

Finally, concerning the second term $[v_j^2 , \Xi_\delta]$  in~\eqref{e:4terms}, we write 
\[
	[v_j^2 , \Xi_\delta] 
	=  [v_j^2, \hat \rho (\delta D _x) \chi(\delta x)  \hat \rho (\delta D _v) \chi(\delta v)]
	=  \hat \rho (\delta D _x) \chi(\delta x)  [v_j^2, \hat \rho (\delta D _v)] \chi(\delta v) ,
\]
and, on account to Remark~\ref{r:tensor}, it suffices to prove that $ [v_j^2, \hat \rho (\delta D _v)] \chi(\delta v)$ is $L^p$ bounded and converges strongly to zero.
To this aim, we write
\begin{align} 
	[v_j^2, \hat \rho (\delta D _v)] \chi(\delta v) 
	& = v_j [v_j, \hat \rho (\delta D _v)] \chi(\delta v) 
	+ [v_j, \hat \rho (\delta D _v)] v_j \chi(\delta v) \nonumber  \\
	& =  [v_j , [v_j, \hat \rho (\delta D _v)]] \chi(\delta v)  
	+ \frac{2}{\delta}[v_j, \hat \rho (\delta D _v)] \delta v_j \chi(\delta v). \label{e:first-comm-to-be-est-3}
\end{align}
Concerning the second term in the right-hand side of~\eqref{e:first-comm-to-be-est-3}, we notice that the function $v\mapsto \delta v_j \chi(\delta v)$ converges to $0$ as $\delta\to 0^+$, for all $v \in {\mathbb R}^n$, and is uniformly bounded. The dominated convergence theorem shows that $ \delta v_j \chi(\delta v) u \to 0$ in $L^p({\mathbb R}^n)$ for all $u \in L^p({\mathbb R}^n)$, and, in view of Lemma~\ref{e:main-comm-estimate}, we obtain that the second term in the right-hand side of~\eqref{e:first-comm-to-be-est-3} is bounded in $L^p(\mathbb R^{2n})$ and converges strongly to zero.
As for the first term in the right-hand side of~\eqref{e:first-comm-to-be-est-3}, we notice that~\eqref{e:comm-2-bounded} in Lemma~\ref{l:commutateur} yields 
$\| [v_j , [v_j, \hat \rho (\delta D _v)]] \|_{L^p\rightarrow L^p} \leq C_\rho \delta^2$.
Hence the first term in the right-hand side of~\eqref{e:first-comm-to-be-est-3} is bounded on $L^p(\mathbb R^{2n})$ and converges to zero in norm (it is of order $\delta^2$ times an operator bounded on $L^p(\mathbb R^{2n})$).

This concludes the proof that $[\mathcal{P} , \Xi_\delta]$ maps $L^p({\mathbb R}^{2n})$  into itself and for any $u \in  L^p({\mathbb R}^{2n})$,  $[\mathcal{P} , \Xi_\delta] u \to 0$ in $L^p({\mathbb R}^{2n})$ as $\delta \to 0^+$.
 \end{proof}
 
\subsection{The adjoint operator \texorpdfstring{$\mathcal{P}^*$}{}}

Thanks to Proposition \ref{fermetureessentielle}, we can compute explicitly the adjoint  of the Fokker--Planck operator.
From now on, if $V$ satisfies the subquadraticity assumption~\eqref{e:subquadratic}, on account to Proposition~\ref{fermetureessentielle}, we will denote (with an abuse of notation), for a given $p \in (1,\infty)$,
\begin{align*}
	& \opec := \mathcal{P}_{\max} = \mathcal{P}_{\min} = \overline{\mathcal{P}_0} , \quad \hspace{4.5pt}\text{ on } L^p({\mathbb R}^{2n}), \\
	& \qpec := \mathcal{Q}_{\max} = \mathcal{Q}_{\min} = \overline{\mathcal{Q}_0} ,  \quad \text{ on } L^{p^{\prime}}({\mathbb R}^{2n}) ,
\end{align*}
in order to simplify notation. We also set 
\[
	D(\opec) = D(\mathcal{P}_{\min}) = D(\mathcal{P}_{\max})\quad\text{and}\quad D(\qpec) = D(\mathcal{Q}_{\min}) = D(\mathcal{Q}_{\max}).
\]
Recall that these operators coincide respectively with $\mathcal{P}_0$ and $\mathcal{Q}_0$ on $C^\infty_c({\mathbb R}^{2n})$. 

Here and below, $ A^* $ (resp. $^tA$) denotes the formal adjoint (resp. the formal transpose) in the sense of distributions. 
For clarity, we shall denote by $ B^* : L^{p^{\prime}}(\mathbb R^{2n}) \rightarrow L^{p^{\prime}}(\mathbb R^{2n}) $ the adjoint of a bounded linear map $ B : L^p(\mathbb R^{2n}) \rightarrow L^p(\mathbb R^{2n}) $.
For an operator with dense domain $(P,D(P))$ on $L^p(\mathbb R^{2n})$, we shall also write  $(P^*,D(P^*))$ the adjoint operator on $L^{p^{\prime}}(\mathbb R^{2n})$, where we recall (see e.g.~\cite[eq.~(10.1)]{MR710486})
\begin{equation}\label{e:def-domain-adj}
	D(P^*) = \big\{w\in L^{p^{\prime}}(\mathbb R^{2n})\ \vert\ \exists g\in L^{p^{\prime}}(\mathbb R^{2n}), \forall u \in D(P),\, \langle w , P u \rangle_{L^{p^{\prime}},L^p} =\langle g , u \rangle_{L^{p^{\prime}},L^p} \big\},
\end{equation}
and, for $w \in D(P^*)$, the function $g$ is unique and $P^* w :=g$. Here, we have used the notation~\eqref{e:duality-bracket} for the duality bracket.

\begin{coro}\label{16032018C2} 
Let $p \in(1,+\infty)$. Assume that $V$ satisfies the subquadraticity assumption~\eqref{e:subquadratic}. Then, $\opec^* = \qpec$ on $L^{p^{\prime}}(\mathbb R^{2n})$ and $\qpec^* = \opec$ on $L^{p}(\mathbb R^{2n})$.
\end{coro}

\begin{proof}
First, we let $u\in D(\opec)$ and $w\in D(\qpec)$.
From Proposition~\ref{fermetureessentielle}, $\mathcal{P}=\mathcal{P}_{\min}$ and $\mathcal{Q}=\mathcal{Q}_{\min}$, hence there exist sequences $(u_k)_k$ and $(w_k)_k$ in $C^\infty_c(\mathbb R^{2n})$ such that
\[
	u_k \to  u,\quad \opec u_k \to  \opec u\quad \text{in $L^p(\mathbb R^{2n})$} \quad \text{and}\quad
	w_k \to w,\quad \qpec w_k \to \qpec w\quad \text{in $L^{p^{\prime}}(\mathbb R^{2n})$.}
\]
It follows from an integration by parts that for all $k\in {\mathbb N}$, 
$\langle \opec u_k,w_k\rangle_{L^p,L^{p^{\prime}}} = \langle u_k,\qpec w_k\rangle_{L^p,L^{p^{\prime}}}$. 
Letting $k\to+\infty$, we deduce that $\langle \opec u,w\rangle_{L^p,L^{p^{\prime}}} = \langle u, \qpec w\rangle_{L^p,L^{p^{\prime}}}$. 
Recalling~\eqref{e:def-domain-adj}, this equality shows that $D(\qpec)\subset D( \opec^*)$ and $\opec^* w = \qpec w$ for all $w\in D(\qpec)$.  

Conversely, if $w\in D(\opec^*)$, we have from~\eqref{e:def-domain-adj} that  $\langle w , \opec u \rangle_{L^{p^{\prime}},L^p} =\langle g , u \rangle_{L^{p^{\prime}},L^p}$ 
for all $u \in D(\opec)$, with $g = \opec^* w \in L^{p^{\prime}}(\mathbb R^{2n})$.  
In particular, recalling~\eqref{e:grad-V-Dprime} for the definition of $\nabla V\cdot \partial_v$ in the sense of distributions, we deduce that for all $u\in C^\infty_c(\mathbb R^n)$,
\[
	\langle {\mathcal Q}w,u\rangle_{\mathcal{D}^{\prime},\mathcal{D}} = \int_{{\mathbb R}^{2n}}w \mathcal{P}u \, dxdv = \langle w , \mathcal{P} u \rangle_{L^{p^{\prime}},L^p} =\langle g , u \rangle_{L^{p^{\prime}},L^p} .
\]
This proves that 
${\mathcal Q}w = g =\opec^* w$ in $\mathcal{D}^{\prime}({\mathbb R}^{2n})$ and thus that ${\mathcal Q}w \in L^{p^{\prime}}({\mathbb R}^{2n})$. Using $\qpec=\mathcal{Q}_{\max}$ from Proposition~\ref{fermetureessentielle}, this yields $w \in D(\qpec)$ and concludes the proof of $D(\opec^*)\subset D(\qpec)$, and thus of the corollary.
\end{proof}
 
\subsection{Semigroup and well-posedness of the evolution equation}\label{subsec:gene}
 
The goal of this section is to prove that $\opec$ generates a strongly continuous semigroup, and deduce the well-posedness of the evolution PDE~\eqref{pourCauchyproblem} (being given by the semigroup).

\begin{prop} \label{p:generation-sgp}
Let $p \in (1, \infty)$. Assume that $V$ satisfies the subquadraticity assumption~\eqref{e:subquadratic}. Then, the operator $(\opec,D(\opec))$  generates a strongly continuous semigroup $(e^{-t\opec})_{t\geq0}$ on $L^p(\mathbb{R}^{2n})$, which satisfies for all $t\geq 0$ and $u\in L^p({\mathbb R}^{2n})$, with $c_0 := \frac{m\gamma n}{2}$,  
\begin{equation}\label{09052018E10}
	\big\Vert e^{-t\opec}u\big\Vert_{L^p}\leq e^{c_0 t}\Vert u\Vert_{L^p}.
\end{equation}
Respectively, the operator $(\qpec,D(\qpec))$ generates a strongly continuous semigroup $(e^{-t\qpec})_{t\geq0}$ on $L^{p^{\prime}}(\mathbb{R}^{2n})$, which satisfies for all $t\geq 0$ and $w\in L^{p^{\prime}}({\mathbb R}^{2n})$,
\[
	\big\Vert e^{-t\qpec}w\big\Vert_{L^{p^{\prime}}}\leq e^{c_0 t}\Vert w\Vert_{L^{p^{\prime}}}.
\]
\end{prop}

Note that the value of $c_0$ is not optimized here. For instance, for $p=2$, one can choose $c_0=0$ in~\eqref{09052018E10} on account to the Heisenberg uncertainty principle (hence $(e^{-t\opec})_{t\geq0}$ is a contraction semigroup on $L^2(\mathbb R^{2n})$). This is not relevant for the analysis in the present article since we focus on small-time behavior.

\begin{proof}
Lemma~\ref{l:accretif} and Corollary~\ref{16032018C2} show that the operator $\opec+c_0$ is accretive on $L^p(\mathbb R^{2n})$ and that $(\opec+c_0)^* = \qpec+c_0$ is accretive on $L^{p^{\prime}}(\mathbb R^{2n})$. Therefore, the Lumer-Phillips theorem (see e.g. ~\cite[Chap. 1, Corollary 4.4]{MR710486} or ~\cite[Corollary p.241]{ReedSimon2}) proves that $\opec+c_0$ generates a contraction semigroup $(e^{-t(\opec+c_0)})_{t\geq0}$ on $L^p(\mathbb R^{2n})$ and $\qpec+c_0$ generates a contraction semigroup $(e^{-t(\qpec+c_0)})_{t\geq0}$ on $L^{p^{\prime}}(\mathbb R^{2n})$. As a consequence, the operator $\opec$ generates a strongly continuous semigroup $(e^{-t\opec})_{t\geq0}$ on $L^p(\mathbb R^{2n})$, satisfying $e^{-t(\opec+c_0)}= e^{-tc_0}e^{-t\opec}$, and the result follows.  
\end{proof}

The next result allows in particular to justify the Duhamel formula.

\begin{prop}\label{p:existence-uniqueness}
Let $p \in (1, \infty)$. Assume that $V$ satisfies the subquadraticity assumption~\eqref{e:subquadratic}. Denote by $(S(t))_{t\geq0}=(e^{-t\opec})_{t\geq0}$ the semigroup generated by $\opec$ on $L^p({\mathbb R}^{2n})$, given by Proposition~\ref{p:generation-sgp}.
For all $u_0 \in L^p({\mathbb R}^{2n})$ and $f \in L^1_{\rm loc}({\mathbb R}_+ ; L^p({\mathbb R}^{2n}))$, the function 
\begin{equation}\label{e:duhamel-pas-duhamel}
	u(t) := S(t)u_0 + \int_0^t S(t-s)f(s)\, ds 
\end{equation} 
is the unique function satisfying
\begin{enumerate}
	\item $u \in C^0({\mathbb R}_+; L^p({\mathbb R}^{2n}))$,
	\item $u(0)=u_0$,
	\item \label{i:evol-PDE} $\partial_t u + \mathcal{P}u = f$ in $\mathcal{D}^{\prime}({\mathbb R}_+^* \times {\mathbb R}^{2n})$.
\end{enumerate}
Moroever, with $c_0 := \frac{m\gamma n}{2}$, we have for any $t\geq0$,
\begin{align}\label{e:energie-estimate}
	\|u(t)\|_{L^p}\leq e^{c_0 t}\left(  \|u_0\|_{L^p}+ \|f\|_{L^1(0,t;L^p)} \right) .
\end{align}
\end{prop}

Note that in the statement of Item~\ref{i:evol-PDE} (and in the proof), $\mathcal{P}u$ defines a distribution according to~\eqref{e:grad-V-Dprime} (and the discussion preceding this identity) integrated in time here.
 
To prove Proposition~\ref{p:existence-uniqueness}, we use the following general statement for semigroups.
 
\begin{theo}\label{theo.2.10} 
Let $E$ be a Banach space and let $(S(t))_{t\geq0}$ be a strongly continuous semigroup on $E$. Denote by $(P,D(P))$ the infinitesimal generator of $(S(t))_{t\geq0}$.
Then, for any $u_0\in E$ and $f \in L^1_{\rm loc}({\mathbb R}_+; E)$, the function $u(t) = S(t)u_0 + \int_0^t S(t-s)f(s)\, ds$ is the unique function satisfying
\begin{enumerate}
	\item $u \in C^0({\mathbb R}_+; E)$,
	\item $u(0)=u_0$,
	\item For any $\psi \in C^1_c({\mathbb R}_+^*;{\mathbb C})$, we have $\int^\infty _0 \psi(t) u(t)\, dt \in D(P)$ and 
	\[
		P\bigg(\int^\infty _0 \psi(t) u(t)\, dt\bigg)=\int^\infty _0 \psi^{\prime}(t) u(t)\, dt + \int^\infty _0 \psi(t) f(t)\, dt \quad \text{ in }E .
	\]
\end{enumerate}
\end{theo}

Theorem~\ref{theo.2.10} is proved in~\cite{BGL}   and can be deduced from ~\cite[Proposition~3.1.16 p.118]{ABHN}.
In this second reference, the (equivalent) result is proved with $\psi$ replaced by $\mathds{1}_{[0,\tau]}$ for all $\tau>0$ and the result of Theorem~\ref{theo.2.10} can be deduced from this one via an approximation argument.

\begin{proof}[Proof of Proposition~\ref{p:existence-uniqueness}]
Note first that the energy inequality~\eqref{e:energie-estimate} follows from the formula~\eqref{e:duhamel-pas-duhamel} and the fact that the semigroup $(S(t))_{t\geq0}$ satisfies~\eqref{09052018E10}.

The proof of the main part of the statement will follow from applying Theorem~\ref{theo.2.10} to $P= \opec$, and proving that, in this context, the last item of Theorem~\ref{theo.2.10} is equivalent to a statement in  $\mathcal{D}^{\prime}({\mathbb R}_+^* \times {\mathbb R}^{2n})$ (\textit{i.e.} the last statement of Proposition~\ref{p:existence-uniqueness}).

We start by assuming that for any $\psi \in C^1_c({\mathbb R}_+^*;{\mathbb C})$, we have $\int^\infty _0 \psi(t) u(t)\, dt \in D( \opec)$ and 
\[
	\opec\bigg(\int^\infty _0 \psi(t) u(t)\, dt\bigg) = \int^\infty _0 \psi^{\prime}(t) u(t)\, dt + \int^\infty _0 \psi(t) f(t)\, dt \quad \text{ in }L^p({\mathbb R}^{2n}) .
\]
Then, multiplying by a test function $\varphi \in C^\infty_c({\mathbb R}^{2n})$ and integrating on ${\mathbb R}^{2n}$, we deduce
\[
	\int_{{\mathbb R}^{2n}} \opec\bigg(\int^\infty _0 \psi(t) u(t)\, dt\bigg) \varphi(x)\, dx
	= \int_{{\mathbb R}^{2n}}\bigg(\int^\infty _0 \psi^{\prime}(t) u(t)\, dt\bigg) \varphi(x)\, dx+ \int_{{\mathbb R}^{2n}} \bigg(\int^\infty _0 \psi(t) f(t)\, dt\bigg)\varphi(x)\, dx.
\]
From the Fubini theorem, the right-hand side is equal to
\[
	\int_{{\mathbb R}_+^*\times {\mathbb R}^{2n}}(u \partial_t  (\psi \otimes \varphi ) + f  \psi \otimes \varphi)\, dt dx.
\]
On the other hand, that $\int^\infty _0 \psi(t) u(t)\, dt \in D(\opec) = \{ w \in L^p(\mathbb R^{2n})\ \vert \ \mathcal{P}w \in L^p(\mathbb R^{2n})\}$ implies that, by definition (note also that $^t\mathcal{P}=\mathcal{Q}$, together with~\eqref{e:grad-V-Dprime}),
\[
	\int_{{\mathbb R}^{2n}} \opec\bigg(\int^\infty _0 \psi(t) u(t)\, dt\bigg) \varphi(x)\, dx 
	= \int_{{\mathbb R}^{2n}} \int^\infty _0 \psi(t) u(t)   ^t \mathcal{P}\varphi\, dt dx =   \int_{{\mathbb R}_+^*\times {\mathbb R}^{2n}} u \,   ^t \mathcal{P}(\psi \otimes \varphi)\, dt dx ,
\]
where $^t \mathcal{P}$ is the transpose of $\mathcal{P}$, \textit{i.e.} $^t \mathcal{P}=\mathcal{Q}$, and where we have used the Fubini Theorem in the last equality. 
Combining the above two lines, we have now obtained for all $\psi \in C^\infty_c({\mathbb R}_+^*)$ and $\varphi \in C^\infty_c({\mathbb R}^{2n})$, 
\[
	\int_{{\mathbb R}_+^*\times {\mathbb R}^{2n}} u \,   ^t \mathcal{P}(\psi \otimes \varphi)\, dt dx  = \int_{{\mathbb R}_+^*\times {\mathbb R}^{2n}}(u \partial_t  (\psi \otimes \varphi ) + f  \psi \otimes \varphi)\, dt dx.
\]
The density of tensor products in $C^\infty_c({\mathbb R}_+^*\times {\mathbb R}^{2n})$ implies that for every $\phi \in C^\infty_c({\mathbb R}_+^*\times {\mathbb R}^{2n})$,
\[
	\int_{{\mathbb R}_+^*\times {\mathbb R}^{2n}} u \,   ^t \mathcal{P}\phi\, dt dx  = \int_{{\mathbb R}_+^*\times {\mathbb R}^{2n}}(u \partial_t\phi+ f  \phi)\, dt dx ,
\]
that is to say $\partial_t u + \mathcal{P}u = f$ in $\mathcal{D}^{\prime}({\mathbb R}_+^* \times {\mathbb R}^{2n})$.
 
Conversely, if $\partial_t u + \mathcal{P}u = f$ in $\mathcal{D}^{\prime}({\mathbb R}_+^* \times {\mathbb R}^{2n})$, the calculation shows that  we then have for any $\psi \in C^\infty_c({\mathbb R}_+^*)$ (then passing to $\psi \in C^1_c({\mathbb R}_+^*)$ by a density argument) and $\varphi \in C^\infty_c({\mathbb R}^{2n})$, 
\begin{align*} 
	\int_{{\mathbb R}^{2n}} \bigg (\int^\infty _0 \psi(t) u(t)\, dt\bigg) \ ^t\mathcal{P} \varphi(x)\, dx
	& = \int_{{\mathbb R}^{2n}}\bigg(\int^\infty _0 \psi^{\prime}(t) u(t)\, dt\bigg) \varphi(x)\, dx+ \int_{{\mathbb R}^{2n}}\bigg(\int^\infty _0 \psi(t) f(t)\, dt\bigg)\varphi(x)\, dx \\
	& = \int_{{\mathbb R}^{2n}} \bigg( \int^\infty _0 (\psi^{\prime}(t) u(t)+ \psi(t) f(t))\, dt \bigg) \varphi(x)\, dx.
\end{align*} 
But by assumption, $\left( \int^\infty _0 \psi^{\prime}(t) u(t)+ \psi(t) f(t) dt \right)  \in L^p({\mathbb R}^{2n})$, so this identity implies that, in the sense of $\mathcal{D}^{\prime}({\mathbb R}^{2n})$,  
\[
	\mathcal{P}  \bigg (\int^\infty _0 \psi(t) u(t)\, dt\bigg) = \bigg( \int^\infty _0 (\psi^{\prime}(t) u(t)+ \psi(t) f(t))\, dt \bigg)  \in L^p({\mathbb R}^{2n}).
\]
Since $ \opec = \mathcal{P}_{\max}$, this implies that $ \int^\infty _0 \psi(t) u(t) dt \in D(\opec)$ and this concludes the proof of the last item of Theorem~\ref{theo.2.10}.
\end{proof}

\subsection{The adjoint semigroup}

We denote by $ U $ the symmetry operator
\[
	(U \varphi ) (x,v) = \varphi (x,-v),\quad (x,v)\in\mathbb R^{2n}.
\]
The linear operator $U$ maps $C^\infty_c({\mathbb R}^{2n})$ and all spaces $ L^p({\mathbb R}^{2n})$, $p\in [1,\infty]$, continuously into themselves. It satisfies in addition $U U=\mbox{Id}$ hence $U^{-1} = U $. 
In this paragraph (and only here), we restore the dependence of the operators $\mathcal{P},\mathcal{Q}$ in the space $L^p(\mathbb R^{2n})$ and write  
\[
	\opec_p = \mbox{closed realization of} \ \mathcal{P} \ \mbox{on} \ L^p(\mathbb R^{2n}), \qquad 
	\qpec_p = \mbox{closed realization of} \ \mathcal{Q} \ \mbox{on} \ L^p(\mathbb R^{2n}),
\]
with domains $D( \opec_p) \subset L^p(\mathbb R^{2n})$ and $D(\qpec_p) \subset L^p(\mathbb R^{2n})$.
With this notation, the statement of Corollary~\ref{16032018C2}  reads
\[
	\opec_p^* := (\opec_p)^* = \qpec_{p^{\prime}}  \text{ on } L^{p^{\prime}}(\mathbb R^{2n}) ,
	\quad \qpec_{p^{\prime}}^* :=(\qpec_{p^{\prime}})^* = \opec_{p}  \text{ on } L^{p}(\mathbb R^{2n}),
	\quad p\in (1,\infty) .
\]
We notice that, for a function $w \in C^\infty_c({\mathbb R}^{2n})$, we have $\mathcal{P}w = U\mathcal{Q}U w$. 

\begin{prop} \label{adjointsansadjoint} 
Let $p \in (1, \infty)$. Assume that $V$ satisfies the subquadraticity assumption~\eqref{e:subquadratic}. Then, we have 
\begin{align}
	& U(D(\opec_p)) = D(\qpec_p)\quad \text{ and}  \quad U \opec_p w = \qpec_p Uw,\quad w \in D(\opec_p) , \label{e:conjug-op} \\
	& \big( e^{ -t\opec_p} \big)^* =   e^{- t \opec_{p}^*} = e^{- t \qpec_{p^{\prime}}} = U e^{-t \opec_{p^{\prime}}}U,\quad t>0.  \label{e:adjointsansadjoint}
\end{align}
\end{prop}

The proof relies on the following general lemma.

\begin{lemm}\label{l:conjug-sgp}
Let $E,F$ be two Banach spaces. Let $(P,D(P))$ be an operator on $E$ generating a strongly continuous semigroup $(e^{-tP})_{t\geq0}$ on $E$ and $(Q,D(Q))$ be an operator on $F$  generating a strongly continuous semigroup $(e^{-tQ})_{t\geq0}$ on $F$. Assume that $U \in \mathcal{L}(E;F)$ satisfies $U(D(P)) \subset D(Q)$ and for any $w \in D(P)$, $U P w = Q U w$. Then, we have $U e^{-tP} = e^{-tQ} U$ for all $t\geq 0$.
\end{lemm}

\begin{proof}
Assume first that $w \in D(P)$. Then, the function $w(t) := e^{-tP}w$ belongs to $C^0({\mathbb R}_+;D(P))\cap C^1({\mathbb R}_+;E)$ and satisfies $w^{\prime}(t)+Pw(t)=0$ for all $t\geq0$ (see e.g. the Item h) of ~\cite[Proposition~3.1.9, p.112]{ABHN}). 
Setting $z(t): = Uw(t) = U e^{-tP}w$, we have $z \in C^0({\mathbb R}_+; UD(P))\cap C^1({\mathbb R}_+;UE) \subset C^0({\mathbb R}_+; D(Q))\cap C^1({\mathbb R}_+;F)$ by assumption and $z^{\prime}(t) = Uw^{\prime}(t) = U (-P)w(t) = - QU w(t) = -Q z(t)$, where the penultimate  equality also follows from the assumption. Uniqueness of the solution 
(see ~\cite[Proposition~3.1.11, p.115]{ABHN}) implies that $z(t)=e^{-tQ}z(0)$, that is to say, $U e^{-tP}w = e^{-tQ} Uw$ for all $t\geq 0$. This is the sought result for $w \in D(P)$. The density of $D(P)$ in $E$ concludes the proof of the lemma.
\end{proof}

\begin{proof}[Proof of Proposition~\ref{adjointsansadjoint}]
First, we prove~\eqref{e:conjug-op}. To this aim, we let $w \in D( \opec_p) \subset L^p(\mathbb R^{2n})$. Using $D(\opec_p) = D(\opec_{p,\min})$, there is $w_n\in C^\infty_c({\mathbb R}^{2n})$ such that $w_n \to w$ in $L^p(\mathbb R^{2n})$ and $\mathcal{P} w_n\to \opec_pw$ in $L^p(\mathbb R^{2n})$. From the properties of $U$, we have $Uw \in L^p(\mathbb R^{2n})$ and $Uw_n \in C^\infty_c(\mathbb R^{2n})$ with $Uw_n\to Uw$ in $L^p(\mathbb R^{2n})$. In addition, since $w_n\in C^\infty_c(\mathbb R^{2n})$, a direct computation shows that $\mathcal{Q}_p U w_n=U\mathcal{P}_p w_n \to U\opec_p w$ in $L^p(\mathbb R^{2n})$. As a consequence, we have obtained that $Uw \in D(\qpec_{p,\min})=D(\qpec_{p})$, hence $UD(\opec_p)\subset D(\qpec_p)$ and  $\qpec_p U w =  U\opec_p w$.
The converse inclusion is proved similarly.

Secondly, the identity $( e^{ -t\opec_p} )^* =   e^{- t \opec_{p}^*}$ in~\eqref{e:adjointsansadjoint} is a consequence of~\cite[Corollary~10.6]{MR710486}. The second equality in~\eqref{e:adjointsansadjoint} is Corollary~\ref{16032018C2}. 
Finally, $e^{- t \qpec_{p^{\prime}}} = U e^{-t \opec_{p^{\prime}}} U$ in~\eqref{e:adjointsansadjoint} follows from Lemma~\ref{l:conjug-sgp} together with~\eqref{e:conjug-op} (and used with $p^{\prime}$ instead of $p$).
\end{proof}

\section{Elementary technical lemmas}\label{section:technical}

In this appendix, we collect technical lemmas, that are used along the paper. Section~\ref{s:schur} is devoted to the Schur test. In Section~\ref{s:regularization}, we give the proof of two commutator lemmas. Finally, in Section \ref{subsec:contLp}, we prove a continuity result in the $L^p$ spaces.

\subsection{The Schur test}\label{s:schur}

We recall that whenever $ K (z,z^{\prime}) $ is a continuous function on $ {\mathbb R}^N \times {\mathbb R}^{N} $ such that
\begin{equation} \label{conditionSchur}
	C_1 : = \sup_{z} \int_{\mathbb R^N} |K(z,z^{\prime})|\, dz^{\prime} < \infty, \qquad C_2:= \sup_{z^{\prime}} \int_{\mathbb R^N} |K(z,z^{\prime})|\, dz < \infty , 
\end{equation}
and if we define the operator $ {\mathcal K} $ by 
\[
	({\mathcal K} u )(z) = \int_{\mathbb R^N} K (z,z^{\prime})u(z^{\prime})\, dz^{\prime} , \quad z\in\mathbb R^N,
\]
we have for any $ p \in [1,\infty] $ and $ u \in L^p(\mathbb R^N) $,
\begin{equation}\label{estimationdeSchurdebase}
	\Vert{\mathcal K} u\Vert_{L^p} \leq C_1^{\frac{1}{p^{\prime}}} C_2^{\frac{1}{p}}\Vert u\Vert_{L^p}  \leq \max (C_1,C_2)\Vert u\Vert_{L^p},
\end{equation}
where $ p^{\prime} $ is the H\"older-conjugate exponent to $p$. The condition (\ref{conditionSchur}) is called the Schur test and (\ref{estimationdeSchurdebase}) shows it guarantees the boundedness of $ {\mathcal K} $ on each $ L^p(\mathbb R^N) $. For completeness, we recall the short proof of (\ref{estimationdeSchurdebase}). We assume that $p$ is finite, the case $p=\infty$ being simpler.  By the triangle inequality, we have
\[
	\big| ( {\mathcal K} u ) (z) \big| \leq \int_{\mathbb R^N} | K(z,z^{\prime}) |^{\frac{1}{p^{\prime}} + \frac{1}{p}} |u(z^{\prime})|\, dz^{\prime},
\]
hence, by the  H\"older inequality,
\[
	\big| ( {\mathcal K} u ) (z) \big|^p \leq \left( \int_{\mathbb R^N} | K(z,z^{\prime}) |\, dz^{\prime} \right)^{\frac{p}{p^{\prime}}}  \int_{\mathbb R^N} \big| K (z,z^{\prime}) \big| |u(z^{\prime})|^p\, dz^{\prime},
\]
so that, using (\ref{conditionSchur}) and the Fubini Theorem, we get
\[
	\Vert{\mathcal K} u\Vert_{L^p}^p \leq C_1^{\frac{p}{p^{\prime}}}  \int_{\mathbb R^N} \bigg( \int_{\mathbb R^N} | K (z,z^{\prime}) |\, dz \bigg) |u(z^{\prime})|^p\, dz^{\prime},
\]
which leads to (\ref{estimationdeSchurdebase}).

A special case of interest, in particular in this paper, is when $ K $ is the kernel of a pseudodifferential operator, namely is of the form
\begin{equation}\label{noyaupseudodifferentielSchur}
	K (z,z^{\prime}) = (2 \pi)^{-N} \int_{{\mathbb R}^N} e^{i (z-z^{\prime})\cdot \zeta} a (z,z^{\prime},\zeta)\, d \zeta, 
\end{equation}
with $a$ decaying fast enough in $ \zeta $, as below. Using that, for any integer $ M \geq 0 $,
\begin{align*}
	|z-z^{\prime}|^{2M} K (z,z^{\prime}) & = (2 \pi)^{-N} \int_{{\mathbb R}^N} (- i \Delta_{\zeta})^M e^{i (z-z^{\prime})\cdot \zeta} a (z,z^{\prime},\zeta)\, d \zeta \\
	& = (2 \pi)^{-N} \int_{{\mathbb R}^N}  e^{i (z-z^{\prime})\cdot \zeta} ( i \Delta_{\zeta})^M a (z,z^{\prime},\zeta)\, d \zeta,
\end{align*} 
we infer that
\begin{equation}\label{conditionSchursymbole}
	\big(1 + |z-z^{\prime}|^{2N} \big)|K (z,z^{\prime}) | \leq \sup_{z,z^{\prime}}\big(\Vert a (z,z^{\prime},\cdot)\Vert_{L^1_{\zeta}} + \Vert\Delta_{\zeta}^N a (z,z^{\prime},\cdot)\Vert_{L^1_{\zeta}}\big).
\end{equation}
Since $ (1 + |z|^{2N} )^{-1}  $ is integrable, the Schur test (\ref{conditionSchur}) is satisfied by (\ref{noyaupseudodifferentielSchur}) provided the right-hand side of (\ref{conditionSchursymbole}) is finite, and the constants $ C_1 , C_2 $ can be estimated by the right-hand side of (\ref{conditionSchursymbole}). In practice, the symbol $a$ may depend on  parameters; for instance if $ N = 2n $ and $ \zeta = (\xi,\eta) $, the symbol may  be a Schwartz function of   $ ( t^{\frac{3}{2}}\xi,  t^{\frac{1}{2}}\eta) $  as in (\ref{sanscomposition}). To cover such cases, we record the following proposition.

\begin{prop}  \label{referencesection}  
There exists $C>0$ depending only on $N$ such that for all  invertible $ N \times N $ matrix $ L  $ and all   symbol  $ a = a (z,z^{\prime},\zeta) $, the operator $ {\mathcal K}_{a,L} $ with kernel
\[
	K_{a,L} (z,z^{\prime}) = (2\pi)^{-N} \int_{\mathbb R^N} e^{i (z-z^{\prime})\cdot \zeta} a (z,z^{\prime}, L \zeta)\, d \zeta
\]
satisfies
\[
	\Vert{\mathcal K}_{a,L}\Vert_{L^p \rightarrow L^p} \leq C \big(  \sup_{z,z^{\prime}}\Vert a (z,z^{\prime},\cdot)\Vert_{L^1_{\zeta}} + \Vert\Delta_{\zeta}^N a (z,z^{\prime},\cdot)\Vert_{L^1_{\zeta}} \big).
\]
\end{prop}

\begin{proof} It suffices to rewrite
\[
	K_{a,L} (z,z^{\prime}) =|\mbox{det} (L)|^{-1} (2\pi)^{-N} \int_{\mathbb R^N} e^{i (z-z^{\prime})\cdot L^{-1} \theta} a (z,z^{\prime}, \theta)\, d\theta
\]
to obtain, as we obtained (\ref{conditionSchursymbole}), 
\[
	\big| K_{a,L} (z,z^{\prime}) \big|  \leq |\mbox{det} (L)|^{-1}  \big( 1 + | (L^{-1})^T (z-z^{\prime}) |^{2N} \big)^{-1}  
	\big(  \sup_{z,z^{\prime}}\Vert a (z,z^{\prime},\cdot)\Vert_{L^1_{\zeta}} + \Vert\Delta_{\zeta}^N a (z,z^{\prime},\cdot)\Vert_{L^1_{\zeta}} \big) .
\]
The result then follows from the Schur test and the fact that
\[
	\int_{\mathbb R^N}  |\mbox{det} (L)|^{-1}  \big( 1 + | (L^{-1})^T z |^{2N} \big)^{-1}\, dz =  \int_{\mathbb R^N}  \big( 1 + |  z |^{2N} \big)^{-1}\, dz
\]
is finite and independent on $L$.
\end{proof}

\subsection{Regularization in \texorpdfstring{$L^p$}{} and two commutator lemmas}\label{s:regularization}

In this subsection, we collect two elementary  lemmas  on regularization by convolution that are used in the proof of the fact that $\mathcal{P}_{\min}=\mathcal{P}_{\max}$ in Subsection \ref{subsectionPmax}, namely in the proof of Lemma~\ref{e:commutator}. These are folklore results in real analysis and we reproduce here the results for the convenience of the reader.

\begin{lemm}\label{l:commutateur}
Assume $W \in C^0({\mathbb R}^N)$ such that $\nabla W \in L^\infty({\mathbb R}^N)$ (in the sense of distributions) and let $\varphi \in \mathcal{S}({\mathbb R}^N)$. Then, for all $p \in [1,\infty]$ and $\delta>0$, we have 
\begin{align}
	& \big\|[\varphi(\delta D) , W]\big\|_{L^p\rightarrow L^p} \leq C_\varphi \delta \|\nabla W\|_{L^\infty} , \label{e:comm-bounded} \\ 
	& \big\|[[\varphi(\delta D) , W] ,W]\big\|_{L^p\rightarrow L^p} \leq \tilde{C}_\varphi \delta^2 \|\nabla W\|_{L^\infty}^2. \label{e:comm-2-bounded}
\end{align}
\end{lemm}

\begin{proof}
The operator $\varphi(\delta D)$ is the convolution by $\frac{1}{\delta^n}\check{\varphi}\left(\frac{\cdot}{\delta}\right)$ where $\check{\varphi}$ is the inverse Fourier transform of $\varphi$. Its kernel is therefore $\frac{1}{\delta^n}\check{\varphi}\left(\frac{x-y}{\delta}\right)$. As a consequence, the kernel of $[\varphi(\delta D) , W(x)]$ is $K_{\delta}(x,y)=\frac{1}{\delta^n}\check{\varphi}\left(\frac{x-y}{\delta}\right)\left(W(y)-W(x)\right)$. The Schur test and symmetry of the kernel in $(x,y)$ give, for all $p \in [1,\infty]$,
\begin{align*}
	\big\| [\varphi(\delta D) , W]\big\|_{L^p\rightarrow L^p}
	& \leq \max\big(\sup_{x\in {\mathbb R}^N}\| K_{\delta}(x,\cdot)\|_{L^{1}_y},\sup_{y\in {\mathbb R}^N}\| K_{\delta}(\cdot,y)\|_{L^{1}_x}\big)\\
	& \leq \frac{1}{\delta^N} \sup_{x\in {\mathbb R}^N}\int_{{\mathbb R}^N_{y}}\Big|\check{\varphi}\Big(\frac{x-y}{\delta}\Big)\Big||W(y)-W(x)|\,dy \\
	& \leq  \sup_{s\in {\mathbb R}^N}\int_{{\mathbb R}^N_{t}}|\check{\varphi}(s-t)||W(\delta t)-W(\delta s)|\, dt\\
	& \leq \delta \|\nabla W\|_{L^\infty}\sup_{s\in {\mathbb R}^N} \int_{{\mathbb R}^N_{t}}|\check{\varphi}(s-t)| |s-t|\, dt.
\end{align*}
This yields the sought inequality~\eqref{e:comm-bounded} with $C_\varphi=\| |t|\check{\varphi}(t)\|_{L^{1}_{t}}$. Next, the kernel of $[[\varphi(\delta D) , W],W]$ is $L_{\delta}(x,y)=\frac{1}{\delta^N}\check{\varphi}(\frac{x-y}{\delta})(W(y)-W(x))^2$, and the same proof as that of~\eqref{e:comm-bounded} yields~\eqref{e:comm-2-bounded} with $\tilde{C}_\varphi=\| |t|^2\check{\varphi}(t)\|_{L^{1}_{t}}$.
\end{proof}

We next consider an approximation of the identity $(\rho_\delta)_{\delta>0}$ in ${\mathbb R}^N$ and formulate a second commutator lemma.
We recall that properties of $\rho$ is defined in~\eqref{e:asspt-rho} and properties of $\hat \rho (\delta D_x)$ are recalled in~\eqref{e:mapping-convolution}--\eqref{e:convergence-convolution}.
   
\begin{lemm}\label{e:main-comm-estimate}
Let $W \in C^1({\mathbb R}^N)$ be such that $\nabla W\in L^\infty({\mathbb R}^N)$ and let $p \in [1,+\infty)$ and $\rho$ satisfying~\eqref{e:asspt-rho}. Then, the operator $[W , \hat \rho (\delta D) ] : L^p({\mathbb R}^N)\rightarrow L^p(\mathbb R^N)$ is bounded, with 
\begin{eqnarray}
	& \forall\delta\in(0,1],&\quad\frac{1}{\delta} \big\| [W , \hat \rho (\delta D) ] \big\|_{L^p\rightarrow L^p} \leq C,  \label{e:first-claim-convolution} \\
	& \forall u \in L^p({\mathbb R}^N), &\quad \frac{1}{\delta}[W , \hat \rho (\delta D) ] u \underset{\delta\rightarrow0^+}{\longrightarrow} 0 \quad \text{ in }L^p({\mathbb R}^N). \label{e:main-claim-convolution}
\end{eqnarray}
\end{lemm}

\begin{proof}
The statement~\eqref{e:first-claim-convolution} is a straightforward consequence of Lemma~\ref{l:commutateur} (it would also follow from~\eqref{e:dom-dom} below combined with the Young inequality).
Concerning the statement~\eqref{e:main-claim-convolution}, we first notice that, on account to the uniform estimate~\eqref{e:first-claim-convolution}, the density of $C^\infty_c({\mathbb R}^N)$ in $L^p({\mathbb R}^N)$, and the principle of continuation of convergence, it suffices to prove~\eqref{e:main-claim-convolution} for $u \in C^\infty_c({\mathbb R}^N)$. For $u\in C^\infty_c({\mathbb R}^N)$, using the definition of the convolution product, we have, for all $x \in {\mathbb R}^N$, 
\[
	\big( [W , \hat \rho (\delta D_x) ]u\big) (x) = \big( Wu_\delta - (Wu)_\delta\big)(x) = \int_{{\mathbb R}^N}\big( W(x) - W(x-\delta y)\big)u(x-\delta y)\rho(y)\,dy .
\]
With the Taylor formula,
\[
	W(x) - W(x-\delta y) = \delta \int_0^1 y \cdot\nabla W(x-t \delta y)\, dt , 
\]
and we deduce that for all $x \in {\mathbb R}^N$, 
\[
	f_\delta(x) : = \frac{1}{\delta}\big( [W , \hat \rho (\delta D_x) ] u\big) (x)  
 	=  \int_{{\mathbb R}^N}\bigg(\int_0^1 y \cdot\nabla W(x-t \delta y)\,dt\bigg) u(x-\delta y)\rho(y)\,dy .
\]
On the one hand, using that $W \in C^1({\mathbb R}^N)$ and $u \in C^\infty_c({\mathbb R}^N)$, we have the pointwise convergence
\begin{align*}
	f_\delta(x) \underset{\delta\rightarrow0^+}{\longrightarrow} 
	\int_{{\mathbb R}^N}\bigg(\int_0^1 y \cdot\nabla W(x)\, dt\bigg) u(x)\rho(y)\,dy & =  \int_{{\mathbb R}^N} y \cdot\nabla W(x)  u(x)\rho(y)\,dy  \\
	& = u(x)\nabla W(x)  \cdot\int_{{\mathbb R}^N} y \rho(y)\,dy = 0,
\end{align*}
since we have assumed that $\rho(x)= \rho(-x)$ in~\eqref{e:asspt-rho}. On the other hand, using that $\nabla W \in L^\infty({\mathbb R}^N)$, we have the bound
\begin{align}
	| f_\delta(x)| & \leq \int_{{\mathbb R}^N}\bigg( \int_0^1 |y| |\nabla W(x-t \delta y)|\,dt \bigg) |u(x-\delta y)| |\rho(y)|\,dy \nonumber \\
	& \leq  \|\nabla W\|_{L^\infty} \int_{B(0,1)}  |y| |u(x-\delta y)| \rho(y)\, dy, \label{e:dom-dom}
\end{align}
uniformly for $\delta \leq 1$. Since $u \in C^\infty_c({\mathbb R}^N)$, there is a compact subset $K \subset {\mathbb R}^N$ such that 
\[
	\forall\delta\in[0,1), \forall y\in B(0,1),\quad |u(x-\delta y)| \leq \|u\|_{L^\infty} \mathds{1}_K(x).
\]
This implies that for every $\delta \in (0,1]$ and $x \in {\mathbb R}^N$,
\[
	| f_\delta(x)|  \leq  \|\nabla W\|_{L^\infty} \|  |y| \rho \|_{L^1(B(0,1))}  \|u\|_{L^\infty} \mathds{1}_K(x)   \in L^p({\mathbb R}^n).
\]
We may thus apply the dominated convergence theorem to obtain that $f_\delta \to 0$ in $L^p({\mathbb R}^N)$ as $\delta\rightarrow0^+$, which is the sought statement for $u \in C^\infty_c({\mathbb R}^N)$. The general statement~\eqref{e:main-claim-convolution} follows by density as already mentioned.
\end{proof}

\subsection{Continuity in \texorpdfstring{$L^p$}{} spaces}\label{subsec:contLp}

The following classical continuity result is used in the proof of the accretivity of the Fokker--Planck operator in Lemma~\ref{l:accretif}.

\begin{lemm}\label{e:app-dual-continu}
For all $p \in (1,\infty)$ and $N \in {\mathbb N}^*$, the map $u\mapsto u^* := \mathds{1}_{\{u \neq 0\}} \overline{u} |u|^{p-2}$ is continuous  from $L^p({\mathbb R}^N)$ to $L^{p^{\prime}}({\mathbb R}^N)$. 
\end{lemm}
 
 We present here an elementary proof of Lemma~\ref{e:app-dual-continu} for the sake of completeness. One can actually prove more (but this is not needed here and the proof is more involved):  if $2 \leq p < +\infty$ (resp. $1<p\leq 2$), the map $u\mapsto u^* := \mathds{1}_{\{u \neq 0\}} \overline{u} |u|^{p-2}$ is Lipschitz continuous on bounded sets (resp. globally $(p-1)$-H\"older continuous) from $L^p({\mathbb R}^N)$ to $L^{p^{\prime}}({\mathbb R}^N)$. 
 
 \begin{proof}[Proof of Lemma~\ref{e:app-dual-continu}]
Let $(u_n)_n$ be a sequence in $L^p({\mathbb R}^N)$ and $u \in L^p({\mathbb R}^N)$ such that $u_n\to u$ in $L^p({\mathbb R}^d)$. As a consequence (see e.g.~\cite[Theorem~4.9 p.94]{Brezis}), there is a subsequence $(u_{n_k})_k$ and a function $h \in L^p({\mathbb R}^N)$ such that, for almost all $x \in {\mathbb R}^N$, $u_{n_k}(x)\to u(x)$ and $|u_{n_k}(x)|\leq h(x)$  for all $k \in {\mathbb N}$. On the one hand, we check that for almost all $x \in {\mathbb R}^N$,
\begin{equation}\label{e:claim-*}
	u^*_{n_k}(x) \to u^*(x) \quad\text{ as  }k \to +\infty.
\end{equation}
Indeed, either $u(x) \neq 0$, implying $u_{n_k}(x) \neq 0$ for $k$ sufficiently large, and then
\[
	u_{n_k}^*(x) = \mathds{1}_{\{u_{n_k} \neq 0\}}(x) \overline{u_{n_k}(x)} |u_{n_k}(x)|^{p-2} =  \overline{u_{n_k}(x)} |u_{n_k}(x)|^{p-2} \to  \overline{u(x)} |u(x)|^{p-2}= u^*(x),
\]
since $u_{n_k}(x)\to u(x)$. Or $u(x) = 0$, in which case 
\[
	| u_{n_k}^*(x)| = \big| \mathds{1}_{\{u_{n_k} \neq 0\}}(x) \overline{u_{n_k}(x)} |u_{n_k}(x)|^{p-2} \big| \leq  |u_{n_k}(x)|^{p-1} \to 0,
\]
since $p>1$ and $u_{n_k}(x)\to u(x) = 0= u^*(x)$. This proves~\eqref{e:claim-*}.

On the other hand, using the convexity inequality $(a+b)^{p^{\prime}} \leq 2^{p^{\prime}}(a^{p^{\prime}}+b^{p^{\prime}})$, which holds for every $a,b\geq0$, we have 
\begin{align*}
	| u_{n_k}^* - u^*|^{p^{\prime}} & = \big| \mathds{1}_{\{u_{n_k} \neq 0\}} \overline{u_{n_k}} |u_{n_k}|^{p-2} - \mathds{1}_{\{u \neq 0\}} \overline{u} |u|^{p-2} \big|^{p^{\prime}} \\
	& \leq 2^{p^{\prime}} \big( \big| \mathds{1}_{\{u_{n_k} \neq 0\}} \overline{u_{n_k}} |u_{n_k}|^{p-2}\big|^{p^{\prime}} +\big| \mathds{1}_{\{u \neq 0\}} \overline{u} |u|^{p-2} \big|^{p^{\prime}} \big) \\
	& \leq 2^{p^{\prime}} \big(   |u_{n_k}|^{(p-1)p^{\prime}} + |u|^{(p-1)p^{\prime}}  \big) \leq 2^{p^{\prime}}( h^p+ |u|^{p}) \in L^1({\mathbb R}^N) ,
\end{align*}
where we have used that $h \in L^p({\mathbb R}^d)$. Since $p^{\prime}>1$, we have $\|u_{n_k}^* - u^*\|_{L^{p^{\prime}}}\to 0$ as $k \to+\infty$ by the dominated convergence theorem.

Summarizing, we have proven that {\em any} sequence $(u_n)_n$ such that $u_n\to u$ in $L^p(\mathbb R^N)$ admits a subsequence $(u_{n_k})_k$ such that $u_{n_k}^*\to u^*$ in $L^{p^{\prime}}(\mathbb R^N)$. Since the limit does not depend on the subsequence, we have proven that the {\em full} sequence converges, namely $\|u_{n}^* - u^*\|_{L^{p^{\prime}}}\to 0$ as $n \to+\infty$. Indeed, if $(u_{n}^*)_n$ does not converge to $u^*$, it admits a subsequence $(u_{n_\ell}^*)_{\ell}$ converging to $w\neq u^*$, but, according to the above statement, we can extract a subsequence $(u_{n_{\ell_k}}^*)_k$ converging to $u^*$, which is a contradiction. This concludes the proof of the lemma.
\end{proof}

\section{Spectrum of the Fokker--Planck operator with a quadratic potential}\label{a:quadratic}
 
The global strategy towards the spectral description below was suggested to us by J\'er\'emy Faupin, and is directly inspired by the analysis of the harmonic oscillator.
Related computations are done in~\cite[Section~5.1.1]{HelfferNier}, who refer to Risken~\cite{Risken89}, and in~\cite{HerauSjostrandStolk05,HitrikPravdaStarov} in more general settings. 
Given a parameter $\cq>0$, we consider in dimension $n=1$ the case $V(x) = \frac{\cq}{2} x^2$, that is to say (with constants $m=\beta=\gamma=1$), the operator 
$$
    \opq := v  \partial_x - \cq x \partial_v   - \partial_v^2 + \frac{v^2}{4}  -\frac12 .
$$
We refrain ourselves from denoting the operator $ \ope_{\cq} $ to avoid any confusion with the index $p$ in  $ \ope_p $, and thus denote the operator by $ \opq $ throughout this section; but recall that $ \opq $ is just $ \ope $ with $ V (x) = \frac{\cq}{2}x^2 $.

The goal of the present section is to furnish an elementary proof of Theorem~\ref{t:quadratique}.

\begin{theo}\label{thm:spequad}
\label{t:quadratique}
Assume that $\cq>0$ and $\cq \neq \frac14$.
\begin{itemize}
    \item There exist positive constants $t_0\in(0,1)$ and $c_{\cq,0},c_{\cq,1}>0$ such that for every $t\in(0,t_0)$,
    \[
    	c_{\cq,0}t^{-4} \le\|e^{-t\opq}\|_{\tr} \le c_{\cq,1}t^{-4}.
    \]
    \item The trace of the evolution operator $e^{-t\opq}$ satisfies the following short-time asymptotics
    \begin{equation}\label{eq:tracequad}
    	\tr(e^{-t\opq}) \underset{t\to 0^+}{\sim} \frac1{\cq} t^{-2}.
    \end{equation}
    \item  Considering the counting function $\mathcal{N}_{\opq}$ defined in~\eqref{e:def-counting-function}, we have
    \[
    	\mathcal{N}_{\opq}(\bspe) \underset{\bspe\to +\infty}{\sim}  2\max\bigg(1,\frac1{4\cq}\bigg)\bspe^2.
    \]
\end{itemize}
\end{theo}

The first point is proved in Example \ref{ex:quadratic}. The second point follows from Theorem~\ref{t:trace-asymptotics}. Note that Assumption~\eqref{e:homogeneous-V} is satisfied with $\sigma=1$ and the function $a(\omega) = \cq$ (where the LHS refers to the notation in~\eqref{e:homogeneous-V} whereas the RHS refers to the notation of the present section).
Formula~\eqref{e:trace-equivalent-etP} in Theorem~\ref{t:trace-asymptotics} implies the equivalent \eqref{eq:tracequad}.

The rest of the section is devoted to the proof of the last point of Theorem~\ref{t:quadratique}. We actually first compute the spectrum and then compute the multiplicity of the eigenvalues. 
The purpose of the section is to provide a relatively simple example in which the spectrum and multiplicities can be computed explicitly. We thus  restrict ourselves to the cases $\cq>0$ and $\cq \ne 1/4$ to avoid degeneracies (that, however, could be handled with similar techniques).

We define the creation operators (in $v$ and $x$ respectively)
\[
    \mathsf{C} := \partial_v-\frac{v}{2}\quad \text{ and }\quad \mathsf{B} := \partial_x-\cq\frac{x}{2} .
\]
The associated adjoint (annihilation) operators are 
\[
    \mathsf{C}^* =-\Big( \partial_v + \frac{v}{2} \Big)\quad \text{ and }\quad \mathsf{B}^* =-\Big( \partial_x +\cq\frac{x}{2}\Big) ,
\]
and we have 
\[
    [\mathsf{C}^*,\mathsf{C}] = 1, \qquad [\mathsf{B}^*,\mathsf{B}] = \cq , \qquad [\mathsf{B},\mathsf{C}] = 0 , \qquad [\mathsf{B}^*,\mathsf{C}] = 0.
\]
Moreover, we see that 
\begin{align*}
    \opq & = \mathsf{C}\mathsf{B}^* -  \mathsf{B}\mathsf{C}^* +  \mathsf{C}\mathsf{C}^* , \\
    \opq^* & = - \big( \mathsf{C}\mathsf{B}^* -  \mathsf{B}\mathsf{C}^*\big) +  \mathsf{C}\mathsf{C}^* .
\end{align*}
As a consequence of the above commutation relations, we deduce 
\begin{align}
    [\opq,\mathsf{C}] &=\mathsf{C}-\mathsf{B} , \qquad [\opq,\mathsf{B}] = \cq \mathsf{C} , \label{e:comm-relat-P} \\
    [\opq^*,\mathsf{C}] &=\mathsf{C}+\mathsf{B} , \qquad [\opq^*,\mathsf{B}] = - \cq \mathsf{C} .\label{e:comm-relat-Pstar}
\end{align}
These relations imply the following lemma after solving a second order polynomial equation in $\vp$.

\begin{lemm}
\label{l:operateurs-D-lambda}
Assume that $\cq>0$ and $\cq \neq \frac14$. With $\mathsf{D}=\alpha_1\mathsf{B}+\alpha_2\mathsf{C}$,  and $\alpha_1,\alpha_2,\vp \in \C$, the equation $[\opq,\mathsf{D}] = \mu \mathsf{D}$  has a nontrivial (i.e. such that $(\alpha_1,\alpha_2)\neq (0,0)$) solution $(\vp,\mathsf{D})$ if and only if 
\begin{enumerate}
    \item $\cq \in (0,\frac14)$ and $\vp = \vp_\pm := \frac{1\pm \sqrt{1-4\cq}}{2}$ and $\mathsf{D}$ is a nonzero multiple of  $\mathsf{D}_\pm:=\mathsf{B}- \vp_\pm \mathsf{C}$;
    \item $\cq \in (\frac14,+\infty)$ and $\vp = \vp_\pm := \frac{1\pm i \sqrt{4\cq-1}}{2}$ and $\mathsf{D}$ is a nonzero multiple of  $\mathsf{D}_\pm:=\mathsf{B}- \vp_\pm \mathsf{C}$;
\end{enumerate}
and in both cases $\vp_+\neq \vp_-$ and we have for every $m,n \in \N$,
\begin{align}\label{e:key-algebra}
    [\opq,\mathsf{D}_\pm ] = \vp_\pm \mathsf{D}_\pm, 
    \qquad [\mathsf{D}_+, \mathsf{D}_-] = 0 , 
    \qquad [\opq,\mathsf{D}_+^m\mathsf{D}_-^n  ]  = (m \vp_+ + n \vp_-)\mathsf{D}_+^m\mathsf{D}_-^n.
\end{align}
Finally, setting $\mathsf{E}_\pm := \mathsf{B} + \vp_\pm \mathsf{C}$, we have for every $m,n \in \N$,
\begin{align}
\label{e:key-algebra-star}
    [\opq^*,\mathsf{E}_\pm ] = \vp_\pm \mathsf{E}_\pm , 
    \qquad [ \mathsf{E}_+, \mathsf{E}_-] = 0 , 
    \qquad [\opq^*,\mathsf{E}_+^m\mathsf{E}_-^n  ]  = (m \vp_+ + n \vp_-)\mathsf{E}_+^m\mathsf{E}_-^n.
\end{align}
\end{lemm}

Note that, as for the first part of the statement, $(\vp_\pm , \mathsf{E}_\pm)$ are the only nontrivial solutions $(\vp ,\mathsf{E})$ to $[\opq^*,\mathsf{E} ] = \vp \mathsf{E}$ of the form $\mathsf{E} = \beta_1\mathsf{B}+\beta_2\mathsf{C}$.

\begin{proof}
With $\mathsf{D}=\alpha_1\mathsf{B}+\alpha_2\mathsf{C}$ and using~\eqref{e:comm-relat-P}, we have 
\begin{align*}
    [\opq,\mathsf{D}] = \vp \mathsf{D} 
    & \quad \Longleftrightarrow \quad \alpha_1 [\opq,\mathsf{B}] + \alpha_2 [\opq,\mathsf{C}]  = \vp\alpha_1\mathsf{B}+\vp\alpha_2\mathsf{C} \\
    & \quad \Longleftrightarrow \quad \alpha_1 \cq \mathsf{C} + \alpha_2 (\mathsf{C}-\mathsf{B}) = \vp\alpha_1\mathsf{B}+\vp\alpha_2\mathsf{C} \\
    & \quad \Longleftrightarrow \quad 
    \left\{\begin{array}{rl}
        \alpha_1 \cq  + \alpha_2 = \vp\alpha_2 , \\
        - \alpha_2  = \vp\alpha_1, \\
    \end{array}\right.
    \quad \Longleftrightarrow \quad 
    \left\{\begin{array}{rl}
        \alpha_1 \cq  - \vp\alpha_1 = - \vp^2\alpha_1 , \\
        \alpha_2  =- \vp\alpha_1, \\
    \end{array}\right. \\
    & \quad \Longleftrightarrow \quad \left(\alpha_1=\alpha_2=0 ,\, \vp\in \R \right) \quad \text{ or }\quad 
    \left\{\begin{array}{rl}
        \vp^2  - \vp +  \cq  = 0 , \\
        \alpha_1\in \R^*,\quad \alpha_2  =- \vp\alpha_1 . \\
    \end{array}\right.
\end{align*}
The first case is considered trivial. As for the second, the roots of the polynomial $\vp^2  - \vp +  \cq$ are the $\vp_\pm$ in the statement (depending on the sign of $\cq-\frac14$), and, necessary $\mathsf{D} = \alpha_1(\mathsf{B}- \vp_\pm\mathsf{C}) =\alpha_1\mathsf{D}_\pm$ with $\alpha_1\neq0$ ($\mathsf{D}_\pm$ corresponding to the choice $\alpha_1=1$). 
This concludes the first part of the statement.

Concerning~\eqref{e:key-algebra}, the first property is satisfied by construction, the second follows from $[\mathsf{B},\mathsf{C}] = 0$, and the last property in~\eqref{e:key-algebra} follows by induction from the first two. Indeed, it is satisfied $(m,n)=(0,0)$ (and also if $(m,n)=(1,0)$ or $(m,n)=(0,1)$). Assuming it is satisfied for $(m,n)$, we prove it for $(m+1,n)$ (the proof for $(m,n+1)$ being the same):
\begin{align*}
    [\opq,\mathsf{D}_+^{m+1}\mathsf{D}_-^n  ]  
    & = [\opq,\mathsf{D}_+ \mathsf{D}_+^{m}\mathsf{D}_-^n  ]  
    = [\opq,\mathsf{D}_+]  \mathsf{D}_+^{m}\mathsf{D}_-^n + \mathsf{D}_+ [\opq, \mathsf{D}_+^{m}\mathsf{D}_-^n ] \\ 
    & = \vp_+ \mathsf{D}_+ \mathsf{D}_+^{m}\mathsf{D}_-^n +   \mathsf{D}_+  (m \vp_+ + n \vp_-)\mathsf{D}_+^m\mathsf{D}_-^n = ((m+1) \vp_+ + n \vp_-)\mathsf{D}_+^{m+1}\mathsf{D}_-^n , 
\end{align*}
where in the second line we have used the case $(m,n)=(1,0)$ and the induction assumption \textcolor{red}{on} $(m,n)$.

Finally, using the definition of $\mathsf{E}_\pm= \mathsf{B} + \vp_\pm \mathsf{C}$ with~\eqref{e:comm-relat-Pstar}, we have (using $\vp_\pm\neq0$)
\[
    [ \opq^* , \mathsf{E}_\pm ] = [ \opq^* ,  \mathsf{B}  ] +\vp_\pm [ \opq^* , \mathsf{C} ]
    = - \cq  \mathsf{C} +\vp_\pm (  \mathsf{C} +  \mathsf{B} )
    = \vp_\pm \Big(  \mathsf{B} + \Big(1- \frac{\cq}{\vp_\pm}  \Big)\mathsf{C} + \Big)
    = \mu_\pm \mathsf{E}_\pm ,
\]
where in the last equality we have used that $1- \frac{\cq}{\vp_\pm}=\vp_\pm$ since $\mu_\pm^2  - \mu_\pm +  \cq=0$.
This proves the first statement in~\eqref{e:key-algebra-star}, the other two being similar to the computations with $\mathsf{D}_\pm$.
\end{proof}

From Lemma~\ref{l:operateurs-D-lambda}, we may now construct eigenfunctions and eigenvalues of the operator $\opq$. First notice that 
\[
    \psi_{0,0} (x,v) :=  e^{-\frac{v^2}{4}} e^{- \cq \frac{x^2}{4}}
\]
satisfies 
\[
    \mathsf{C}^* \psi_{0,0} = 0, \quad \mathsf{B}^* \psi_{0,0} = 0,
\]
and hence
\begin{equation}\label{e:psi00}
    \opq \psi_{0,0} = 0 \quad\text{and}\quad 
    \mathsf{D}_\pm^* \psi_{0,0} = 0 ,  \quad  \text{ as well as }\quad 
    \opq^* \psi_{0,0} = 0 \quad\text{and}\quad 
    \mathsf{E}_\pm^* \psi_{0,0} = 0 .
\end{equation}
Concerning the operator $\opq$, we thus set $\psi_{m,n} := \mathsf{D}_+^m\mathsf{D}_-^n\psi_{0,0}$ and deduce from~\eqref{e:psi00} together with~\eqref{e:key-algebra} that
\[
    \opq\psi_{m,n} = \opq\mathsf{D}_+^m\mathsf{D}_-^n\psi_{0,0} 
    = [ \opq , \mathsf{D}_+^m\mathsf{D}_-^n] \psi_{0,0} 
    = (m \vp_+ + n \vp_-)\mathsf{D}_+^m\mathsf{D}_-^n\psi_{0,0} 
    = (m \vp_+ + n \vp_-)\psi_{m,n} .
\]
Concerning the operator $\opq^*$, we thus set $\phi_{m,n} := \mathsf{E}_+^m\mathsf{E}_-^n\psi_{0,0}$ and deduce from~\eqref{e:psi00} together with~\eqref{e:key-algebra-star} that
\[
    \opq^*\phi_{m,n} = \opq^*\mathsf{E}_+^m\mathsf{E}_-^n\psi_{0,0} 
    = [ \opq^* , \mathsf{E}_+^m\mathsf{E}_-^n] \psi_{0,0} 
    = (m \vp_+ + n \vp_-)\mathsf{E}_+^m\mathsf{E}_-^n\psi_{0,0} 
    = (m \vp_+ + n \vp_-)\phi_{m,n} .
\]

In particular, we have now proved that 
\begin{equation}
\label{e:sp-inclusion}
    \big\{ m \vp_+ + n \vp_-\ \vert \ (m,n)\in \N^2\big\}\subset \Sp(\opq),
\end{equation}
and we wish to prove that we have described all the spectral elements that way.
To this aim, we recall the following lemma, which concerns the bidimensional harmonic oscillator (recall  $\cq>0$)
\[
    \mathcal{H}_{\cq} : =  \mathsf{C}\mathsf{C}^* + \mathsf{B}\mathsf{B}^* 
    = - \partial_v^2 + \frac{v^2}{4} - \frac12 - \partial_x^2 + \cq^2\frac{x^2}{4} - \frac12\cq  , 
\] 
with domain 
\[
    D(\mathcal{H}_{\cq}):= H^2(\R^2)\cap L^2(\R^2, (x^2+v^2)dxdv) .
\]

\begin{lemm}
Assume $\cq>0$. Setting  
\[
    h_{k,\ell} := \mathsf{C}^k\mathsf{B}^\ell\psi_{0,0},\quad k,\ell\in\mathbb N,
\]
one has $h_{k,\ell} \in \mathcal{S}(\R^2)\subset D(\mathcal{H}_a)$, $\|h_{k,\ell}\|_{L^2}\neq 0$ and the family $\{\|h_{k,\ell}\|_{L^2}^{-1} h_{k,\ell}\ \vert\ k,\ell \in\N \}$ is a Hilbert basis of $L^2(\R^2)$, consisting of eigenfunctions of $\mathcal{H}_{\cq}$, namely, $ {\mathcal H}_{\cq} h_{k,\ell} = (k + \cq \ell ) h_{k,\ell} $.
\end{lemm}

This lemma is classical, and we do not reproduce a proof here.
For $N\in \N$, we now define the following subspaces of $L^2(\R^2)$:
\begin{align*}
    E_N & := \vect \big\{ h_{k,\ell}\ \vert \ k,\ell \in\N,\, k+\ell= N\big\} 
    = \vect \big\{\mathsf{C}^k\mathsf{B}^\ell\psi_{0,0}\ \vert \ k,\ell \in\N,\, k+\ell= N\big\}, \\
    F_N & := \vect \big\{ h_{k,\ell}\ \vert \ k,\ell \in\N,\, k+\ell \leq  N\big\} 
    = \vect\big\{ E_n \ \vert \ n \leq  N\big\} .
\end{align*}
Since $\{\|h_{k,\ell}\|_{L^2}^{-1} h_{k,\ell}\ \vert \ k,\ell \in\N \}$ is a Hilbert basis of $L^2(\R^2)$, we have the following Hilbert sum
\[
    L^2(\R^2) = \bigoplus_{N\in \N}^{\perp} E_N ,
\]
by which we mean
\begin{itemize}
    \item  $E_N \perp E_M$ if $N \neq M$,
    \item $\overline{\vect \{ E_N\ \vert \ N \in \N \}} = L^2(\R^2)$ (or equivalently $\vect\{ E_N \ \vert \ N \in \N \}^\perp = \{0\}$).
\end{itemize}
Since $F_N$ is finite-dimensional (hence closed), we also have 
\begin{equation}\label{e:FN-sum}
    F_N\oplus F_N^\perp=L^2(\R^2) .
\end{equation}
The following key lemma relates certain spectral subspaces of $\mathcal{H}_{\cq}$ (namely the $E_N$'s) and those of $\opq$.

\begin{lemm}\label{e:quad-fundamental}
Assume that $\cq>0$ and $\cq \neq \frac14$. Then for all $N\in \N$, we have 
\[
    \dim E_N = N+1,\qquad \dim F_N = \frac{(N+1)(N+2)}{2},
\]
and 
\begin{align}
    E_N & =\vect\big\{\psi_{m,n}\ \vert \ m,n\in\N,\, m+n= N\big\} = \vect\big\{ \phi_{m,n}\ \vert \ m,n \in\N,\, m+n= N\big\},\label{e:equality-EN} \\
    F_N & = \vect\big\{\psi_{m,n}\ \vert \ m,n\in\N,\,m+n\leq N\big\} =\vect\big\{\phi_{m,n}\ \vert \ m,n \in\N,\, m+n\leq N\big\} . \label{e:equality-FN}
\end{align}
The spaces $E_N$, $F_N$ are all stable by both $\opq$ and $\opq^*$, and we have 
\begin{align*}
    & \opq \big(E_N^\perp\cap D(\opq) \big) \subset E_N^\perp, \qquad 
    \opq \big(F_N^\perp\cap D(\opq) \big) \subset F_N^\perp, \\ 
    & \opq^* \big(E_N^\perp\cap D(\opq^*) \big) \subset E_N^\perp, \qquad 
    \opq^* \big(F_N^\perp\cap D(\opq^*) \big) \subset F_N^\perp.
\end{align*}
Finally, for all $z \in \rho(\opq)$, the spaces  $E_N , F_N, E_N^\perp, F_N^\perp$ are all stable by $(\opq-z)^{-1}$.
\end{lemm}

Note that in particular, this implies that the family $ (\psi_{m,n} )_{m,n\in\N}$ is complete, that is 
\[
    \overline{ \vect \big\{\psi_{m,n}\ \vert \ m,n\in\N \big\} } = L^2(\R^2).
\]
Note also that $E_N,F_N\subset \mathcal{S}(\R^2)$ since $h_{k,\ell} \in \mathcal{S}(\R^2)$. However, this is not the case for $E_N^\perp$ and $F_N^\perp$ (hence the need of the more precise stability statements in Lemma~\ref{e:quad-fundamental}).

\begin{proof}
First notice that by definition $(h_{k,\ell})_{k+\ell= N} = (h_{k,N-k})_{k\in \{0,\dots, N\}} $ is a basis of $E_N$, which contains $N+1$ vectors, and hence $\dim E_N = N+1$. The dimension of $F_N$ is given by 
\[
    \dim F_N = \sum_{n=0}^N\dim E_n = \sum_{n=0}^N(n+1)=\frac{(N+1)(N+2)}{2}.
\]

We now focus on the first equality in~\eqref{e:equality-EN}. Let us first prove the inclusion ($\supset$), namely  if $m+n=N$, then $\psi_{m,n}=\mathsf{D}_+^m\mathsf{D}_-^n\psi_{0,0} \in E_N$.
Since we have $\mathsf{D}_\pm=\mathsf{B}- \vp_\pm \mathsf{C}$ (see Lemma~\ref{l:operateurs-D-lambda}) with $[\mathsf{B},\mathsf{C}]=0$, we have with Newton's binomial formula
\[
    \mathsf{D}_+^m\mathsf{D}_-^n= (\mathsf{B}- \vp_+ \mathsf{C})^m(\mathsf{B}- \vp_- \mathsf{C})^n 
    = \sum_{i = 0}^m \alpha^+_{m,i}\mathsf{B}^i  \mathsf{C}^{m-i} \sum_{j = 0}^n \alpha^-_{n,j}\mathsf{B}^j \mathsf{C}^{n-j} 
    = \sum_{i = 0}^m  \sum_{j = 0}^n \alpha^+_{m,i}\alpha^-_{n,j} \mathsf{B}^{i+j} \mathsf{C}^{m+n-(i+j)},  
\]
with $\alpha^\pm_{m,i} =\binom{m}{i} (- \vp_\pm)^{m-i} \in \C$. Hence, for any $m,n \in \N$, we have 
\[
    \mathsf{D}_+^m\mathsf{D}_-^n\psi_{0,0} \in \vect\big\{ \mathsf{B}^{i+j} \mathsf{C}^{m+n-(i+j)}\psi_{0,0}\ \vert \ 0\leq i\leq m,\, 0\leq j \leq n \big\},
\]
and if $m+n=N$ then $\mathsf{B}^{i+j} \mathsf{C}^{m+n-(i+j)}\psi_{0,0} \in E_N$ since $i+j + m+n-(i+j) = m+n=N$. This proves that if  $m+n=N$, then $\mathsf{D}_+^m\mathsf{D}_-^n\psi_{0,0} \in E_N$, and hence the inclusion ($\supset$) in the first equality in~\eqref{e:equality-EN}. 

Conversely, we now prove the inclusion ($\subset$) in the first equality in~\eqref{e:equality-EN}. Since $\cq>0$ and $\cq \neq \frac14$, we have $\vp_+\neq \vp_-$ from Lemma~\ref{l:operateurs-D-lambda}, and hence 
\[
    \mathsf{C} = \frac{1}{\vp_--\vp_+}(\mathsf{D}_+-\mathsf{D}_-) , \quad 
    \mathsf{B} =  \frac{1}{\vp_--\vp_+}(\vp_- \mathsf{D}_+-\vp_+\mathsf{D}_-) ,
\]
with $[\mathsf{D}_+ ,\mathsf{D}_-] =0$, and the same argument with Newton's binomial formula proves that if $k+\ell= N$, then
\[
    \mathsf{C}^k\mathsf{B}^\ell\psi_{0,0}  \in \vect\big\{\mathsf{D}_+^m\mathsf{D}_-^n\psi_{0,0}\ \vert \ m+n= N\big\} ,
\]
and this concludes the proof of the first equality in~\eqref{e:equality-EN}. The second equality in~\eqref{e:equality-EN} follows the same.

The two equalities in~\eqref{e:equality-FN} follow from~\eqref{e:equality-EN} together with $F_N =\vect\{ E_n\ \vert \ n \leq  N\}$.

According to~\eqref{e:equality-EN} (resp.~\eqref{e:equality-FN}) together with the fact that $\psi_{m,n}$ are eigenfunctions of $\opq$, we obtain that $E_N$  (resp. $F_N$) is stable by $\opq$. Similarly, according to~\eqref{e:equality-EN} (resp.~\eqref{e:equality-FN}) and the fact that $\phi_{m,n}$ are eigenfunctions of $\opq^*$, we obtain that $E_N$ (resp. $F_N$) is stable by $\opq^*$.

If now $v \in E_N^\perp\cap D(\opq)$, we have, for all $w \in E_N$,
\[
    \langle\opq v, w\rangle_{L^2} = \langle v, \opq^* w\rangle_{L^2} = 0 , 
\]
since $\opq^* w \in E_N$ and $v \in E_N^\perp$. This proves that $ \opq v \in E_N^\perp$, and the proofs of the last three statements follow the same.

Let us finally take $z \in \rho(\opq)$ and prove that $F_N,F_N^\perp$ are stable by $(\opq-z)^{-1}$ (the proof for $E_N,E_N^\perp$ follows the same lines).
Let $v \in F_N$ and set $f := (\opq-z)^{-1} v \in D(\opq)$.
From the direct sum~\eqref{e:FN-sum}, we can write $f=f_N+g_N$ with $f_N \in F_N$ and $g_N\in F_N^\perp$. Since $F_N \subset \mathcal{S}(\R^2)$, we have $f_N \in D(\opq)$ and thus $g_N=f-f_N \in D(\opq)$. Applying $(\opq-z)$ to this identity yields
\[
    F_N \ni v =(\opq-z) f= (\opq-z)f_N+(\opq-z)g_N.
\]
Since $\opq(F_N) \subset F_N$ and $\opq (F_N^\perp\cap D(\opq)) \subset F_N^\perp$ we deduce that $(\opq-z)g_N=0$. But $z \in \rho(\opq)$ and hence $(\opq-z)$ is injective. Thus $g_N=0$ and $(\opq-z)^{-1} v = f =f_N \in F_N$ and $F_N$ is stable by $(\opq-z)^{-1}$.
The proofs for the other cases are similar and this concludes the proof of the lemma.
\end{proof}

The next lemma states that the spectral theory of $\opq$ reduces to that of its restrictions to the spaces $F_N$ (which is well-understood since $(\psi_{m,n})_{m+n \leq N}$ is a basis of $F_N$ which diagonalizes $\opq|_{F_N}$).

\begin{lemm}\label{l:reduction-PiN}
Assume that $\cq>0$ and $\cq \neq \frac14$.
Let $\gamma$ be a piecewise $C^1$ counterclockwise oriented Jordan curve in $\rho(\opq)$, and define the Riesz projector (as in \eqref{eq:projector})
\[
    \Pi_\gamma = \Pi_\gamma(\opq):= \frac{i}{2\pi} \oint_{\gamma}(\opq-z)^{-1}\, dz .
\]
Then, denoting by $Q_N$ the orthogonal projection onto $F_N$, the following three statements hold:
\begin{itemize}
	\item for all $v \in L^2(\R^2)$, $\|(I-Q_N) v\|_{L^2}\to 0$ as $N \to +\infty$, 
	\item for all $N\in \N$, $F_N$ and $F_N^\perp$ are stable by $\Pi_\gamma$, or equivalently, $\Pi_\gamma Q_N=Q_N \Pi_\gamma$,
	\item there is $N_0\in \N$ such that for any $N\geq N_0$,  $\Pi_\gamma=\Pi_\gamma Q_N = \Pi_\gamma(\opq|_{F_N}) Q_N$, where we have written 
$ \Pi_\gamma(\opq|_{F_N}):= \frac{i}{2\pi} \oint_{\gamma}(\opq|_{F_N}-z)^{-1}\, dz$ seen as a linear map $F_N\rightarrow F_N$.
\end{itemize}
\end{lemm}

\begin{proof}
The first claim follows from the definition of $F_N$ and the fact that $(\|h_{k,\ell}\|_{L^2}^{-1} h_{k,\ell})_{k,\ell\in\mathbb N}$ is a Hilbert basis of $L^2(\mathbb R^2)$. Namely, for any $v \in L^2(\R^2)$, we have  $v = \sum_{k,\ell \in \N}\|h_{k,\ell}\|_{L^2}^{-2} \langle v, h_{k,\ell}\rangle_{L^2} h_{k,\ell}$ with 
\[
	\|v\|^2_{L^2} = \sum_{k,\ell \in \N}\|h_{k,\ell}\|_{L^2}^{-2} \big| \langle v, h_{k,\ell}\rangle_{L^2}\big|^2  \quad \text{ and }\quad 
	\|(I-Q_N) v\|^2_{L^2} = \sum_{k,\ell \in \N, k+\ell >N}\|h_{k,\ell}\|_{L^2}^{-2} \big| \langle v, h_{k,\ell}\rangle_{L^2}\big|^2 .
\]
As a consequence, $ \|(I-Q_N) v\|^2_{L^2}$ is the remainder of a converging series hence converges to zero as $N\to +\infty$.

To prove the second claim, let $v \in F_N$ and notice from Lemma~\ref{e:quad-fundamental} that $(\opq-z)^{-1} v \in F_N$ for all $z \in \gamma \subset \rho(\opq)$. Since $\gamma$ is a compact curve in $\rho(\opq)$, we deduce that $\Pi_\gamma v = \frac{i}{2\pi} \oint_{\gamma}(\opq-z)^{-1}v \, dz  \in F_N$. The same proof holds for the stability of $F_N^\perp$ using again Lemma~\ref{e:quad-fundamental}.

As for the third claim, denote by $Q_N^\perp:=I- Q_N$ and notice that from the first point, we have $\Pi_\gamma Q_N^\perp=Q_N^\perp \Pi_\gamma$. As a consequence we have $(\Pi_\gamma Q_N^\perp)^2 =Q_N^\perp  \Pi_\gamma \Pi_\gamma Q_N^\perp = Q_N^\perp  \Pi_\gamma Q_N^\perp = \Pi_\gamma Q_N^\perp$ since $\Pi_\gamma$ and $Q_N^\perp$ are two commuting projectors. As a consequence, $\Pi_\gamma Q_N^\perp$ is a projector.

Next notice that $Q_N^\perp$ converges strongly to zero: for all $v \in L^2(\R^2)$, $\|Q_N^\perp v\|_{L^2}\to 0$ as $N \to +\infty$: 
this follows from the definition of $F_N$ and the fact that $(\|h_{k,\ell}\|_{L^2}^{-1} h_{k,\ell})_{k,\ell\in\mathbb N}$ is a Hilbert basis of $L^2(\mathbb R^2)$.
In addition, according to Proposition~\ref{prop:compactness}, $\opq$ has compact resolvent and, by assumption, $\gamma$ is a bounded contour so that  $\Pi_\gamma$ has finite rank (and in particular is compact). 
Since $Q_N^\perp$ converges strongly to zero and $\Pi_\gamma$ is compact, we deduce 
\[
	\|\Pi_\gamma Q_N^\perp\|_{L^2\rightarrow L^2} \to 0 , \quad \text{ as } N \to +\infty.
\]
Hence there is $N_0\in \N$ such that $\|\Pi_\gamma Q_N^\perp\|_{L^2\rightarrow L^2} \leq \frac12$ for all $N\geq N_0$. But we have proved that $\Pi_\gamma Q_N^\perp$ is a projector, and hence, necessarily, $\Pi_\gamma Q_N^\perp = 0$ for all $N\geq N_0$. This proves $\Pi_\gamma=\Pi_\gamma Q_N$ and, from the stability of $F_N$ by $\Pi_\gamma$, we have $\Pi_\gamma Q_N = \Pi_\gamma(\opq|_{F_N}) Q_N$ (here, we implicitly extend $ \Pi_{\gamma}(\opq|_{F_N}) Q_N $ by $ 0 $ on $ F_N^{\perp} $ since, strictly speaking, $ \Pi_{\gamma} (\opq|_{F_N}) $ is a linear map on $ F_N $).
This concludes the proof of the lemma.
\end{proof}

As a consequence of Lemma~\ref{l:reduction-PiN}, the spectral study of $\opq$ may be restricted to finite dimensional spaces.
Combined with above Lemmas~\ref{l:operateurs-D-lambda} and~\ref{e:quad-fundamental}, we have in particular obtained a complete description of the spectrum of the operator $\opq$, including multiplicities.

\begin{theo}
Assume that $\cq>0$ and $\cq \neq \frac14$. Then, we have
\begin{align}
	\label{e:SpPa}
	& \Sp(\opq) = \big\{ m \vp_+ + n \vp_-\ \vert \ (m,n)\in \N^2\big\} , \\
	\label{e:multPa}
	& \mult(z_j) = \# \big\{(m,n)\in \N^2 \ \vert\  m \vp_+ + n \vp_- = z_j \big\} , \quad z_j \in \Sp(\opq) , \\
	\label{e:NL=mn}
	& \mathcal{N}_{\opq}(\bspe) =  \#  \big\{(m,n)\in \N^2\ \vert \ \Re(m \vp_+ + n \vp_-) \leq \bspe\big\} .
\end{align}
\end{theo}
Note that if  $\cq>\frac14$, two different couples $(m,n)$ give rise to two different eigenvalues, hence the spectrum is simple (thanks to the fact that eigenvalues have nonzero imaginary part).
However, if $\cq\in (0,\frac14)$ is such that $\frac{1-\sqrt{1-4\cq}}{1+\sqrt{1-4\cq}} \in \Q$ (which happens since this function of $\cq$ takes all values in $(0,1)$ when $\cq\in (0,\frac14)$), there exist different couples $(m,n)$ giving rise to the same eigenvalue, hence the spectrum has multiplicity. 
In both cases, the operator $\opq$ has a complete basis of eigenfunctions, and no Jordan blocks (the latter would appear in the case $\cq=\frac14$).

\begin{proof}
We have already obtained an inclusion in~\eqref{e:sp-inclusion} and only need to prove the converse.
According to Proposition~\ref{prop:compactness}, $\opq$ has compact resolvent and hence discrete spectrum. In particular, if $z \in \opq$, there is $v \in D(\opq) \setminus \{0\}$ such that $(\opq-z)v=0$. For some $N$ large enough to be chosen below, we split $F_N \oplus F_N^\perp = L^2(\R^2)$,
and we write $v=Q_N v+ Q_N^\perp v$ with $Q_Nv \in F_N$ and $Q_N^\perp v \in F_N^\perp \cap D(\opq)$. Applying $(\opq-z)$ we deduce $0 =(\opq-z)Q_N v+ (\opq-z)Q_N^\perp v$. But from Lemma~\ref{e:quad-fundamental}, $(\opq-z)Q_N^\perp v \in F_N^\perp$, and hence applying $Q_N$ again we deduce $0 =(\opq-z)Q_N v$. Moreover, $\|v\|_{L^2}\neq 0$ so that from Lemma~\ref{l:reduction-PiN} there is $N$ such that $\|Q_N^\perp v\|_{L^2} \leq \frac12 \|v\|_{L^2}$, and for this $N$, $Q_N v \in F_N$ is an eigenfunction of $\opq$ inside of $F_N$. 
But  $(\psi_{m,n})_{m+n \leq N}$ is a basis of $F_N$ such that $\opq\psi_{m,n}= (m \vp_+ + n \vp_-)\psi_{m,n}$, and in particular $\Sp(\opq|_{F_N}) = \{ (m \vp_+ + n \vp_-)\ \vert \ m+n\leq N \}$, whence $z \in  \{ (m \vp_+ + n \vp_-)\ \vert \ m+n\leq N \}$. 
Recalling~\eqref{e:sp-inclusion}, this concludes the proof of~\eqref{e:SpPa}.

Next notice that, from the last item of Lemma~\ref{l:reduction-PiN}, 
for any piecewise $C^1$, counterclockwise oriented Jordan curve $\gamma$ in $\rho(\opq)$, there is $N=N_0 \in \N$ such that  
\[
	\rank_{L^2} \Pi_\gamma(\opq) = \rank_{L^2}( \Pi_\gamma(\opq )Q_N) =  \rank_{F_N} \Pi_\gamma(\opq|_{F_N}).
\]
In case $z_j \in \Sp(\opq)$ and $\gamma=\{z\in\C\ \vert \ |z-z_j|=\eps\}$, we thus obtain (note that $N$ depends on $z_j$)
\begin{align*}
	\mult(z_j,\opq) & =\rank_{L^2} \Pi_\gamma(\opq) = \rank_{F_N} \Pi_\gamma(\opq|_{F_N})=\mult(z_j,\opq|_{F_N} ) \\
	& =  \# \big\{(m,n)\in \N^2\ \vert \  m \vp_+ + n \vp_- = z_j \big\} ,
\end{align*}
where the last equality is a consequence of the fact that  $(\psi_{m,n})_{m+n \leq N}$ is a basis of $F_N$ such that $\opq\psi_{m,n}= (m \vp_+ + n \vp_-)\psi_{m,n}$.
This concludes the proof of~\eqref{e:multPa}.
Finally,~\eqref{e:NL=mn} follows from~\eqref{e:multPa} together with the definition of the counting function $\mathcal{N}_{\opq}$ in~\eqref{e:def-counting-function}.
\end{proof}

\begin{coro}\label{e:Weyl-quadratiq}
We have
\[
\begin{array}{ll}
	\mathcal{N}_{\opq}(\bspe)= \displaystyle\frac12(\lfloor  2\bspe \rfloor+1)\lfloor  2\bspe \rfloor \underset{\bspe\to+\infty}{\sim} 2\bspe^2, & \text{ if }\cq >\frac14 ,  \\
	\mathcal{N}_{\opq}(\bspe)\displaystyle\underset{\bspe\to+\infty}{\sim} \frac{\bspe^2}{2\cq}, & \text{ if }\cq \in \left(0,\frac14\right) \displaystyle, 
\end{array}
\]
where $  \lfloor  \cdot \rfloor $ denotes the integer part.
\end{coro}

Note that Corollary~\ref{e:Weyl-quadratiq} concludes the proof of the last point of Theorem~\ref{t:quadratique}.

\begin{proof}
If $\cq>\frac14$ we have $\Re(\vp_+)=\Re( \vp_-)=\frac12$ from Lemma~\ref{l:operateurs-D-lambda} and hence for $L\in \R_+$
\begin{align*}
	\mathcal{N}_{\opq}(\bspe) & = \# \big\{(m,n)\in \N^2\ \vert \  \Re(m \vp_+ + n \vp_-) \leq \bspe \big\} = \#\bigg\{(m,n)\in \N^2\ \vert \ \frac{m+n}{2}\leq \bspe\bigg\} \\
	& = \#\big\{(m,n)\in \N^2\ \vert \  m+n\leq \lfloor  2\bspe \rfloor\big\}  = \frac12(\lfloor  2\bspe \rfloor+1)\lfloor  2\bspe \rfloor \underset{L\to+\infty}{\sim} 2\bspe^2 .
\end{align*}
Now if $\cq \in \left(0,\frac14\right)$, writing $\alpha :=\sqrt{1-4\cq}$, we have $\vp_\pm := \frac{1\pm \alpha}{2}$ from Lemma~\ref{l:operateurs-D-lambda}. Hence for $\bspe\in \R_+$, denoting by $ \Leb $ the Lebesgue measure, we have
\begin{align*}
	\mathcal{N}_{\opq}(\bspe) & = \# \big\{(m,n)\in \N^2\ \vert \  \Re(m \vp_+ + n \vp_-) \leq \bspe \big\} 
	= \# \big\{(m,n)\in \N^2\ \vert \ m (1+ \alpha) + n (1- \alpha) \leq 2\bspe \big\}  \\
	& \underset{\bspe\to+\infty}{\sim} \Leb \big\{(x,y)\in \R_+^2\ \vert \  x (1+ \alpha) + y (1- \alpha) \leq 2\bspe\big\} 
	= \int_0^{\frac{2\bspe}{1+\alpha}} \bigg( \int_0^{\frac{2\bspe}{1-\alpha}-\frac{1+\alpha}{1-\alpha}x}\, dy \bigg)\, dx \\
	& = \int_0^{\frac{2\bspe}{1+\alpha}} \bigg( \frac{2\bspe}{1-\alpha}-\frac{1+\alpha}{1-\alpha}x \bigg)\, dx =  \frac{4\bspe^2}{1-\alpha^2} -\frac{1+\alpha}{1-\alpha}\frac{1}{2}\bigg(\frac{2\bspe}{1+\alpha}\bigg)^2 
	=  \frac{4\bspe^2}{1-\alpha^2} -\frac{1}{1-\alpha^2}2\bspe^2 \\
	& =  \frac{2\bspe^2}{1-\alpha^2} = \frac{\bspe^2}{2\cq}
\end{align*}
after having used $\alpha=\sqrt{1-4\cq}$.
\end{proof}

\small
\bibliographystyle{alpha}
\bibliography{biblio_alpha}
 
\medskip

{\small
\noindent\textsc{(Paul Alphonse) Institut de Math\'ematiques de Toulouse, Universit\'e de Toulouse, 118, route de Narbonne, F-31062 Toulouse Cedex 9, France} \\
\textit{Email address}: \verb|paul.alphonse@math.univ-toulouse.fr|
}
\newline

{\small
\noindent\textsc{(Jean-Marc Bouclet) Institut de Math\'ematiques de Toulouse, Universit\'e de Toulouse, 118, route de Narbonne, F-31062 Toulouse Cedex 9, France} \\
\textit{Email address}: \verb|jean-marc.bouclet@math.univ-toulouse.fr|
}
\newline

{\small
\noindent\textsc{(Matthieu L\'eautaud) Laboratoire de Math\'ematiques d'Orsay, Universit\'e Paris-Saclay, B\^atiment 307, 91405 Orsay Cedex France \& Institut Universitaire de France, Paris, France} \\
\textit{Email address}: \verb|matthieu.leautaud@universite-paris-saclay.fr|
}

\end{document}